\documentclass{amsart}

\usepackage{amsmath, amssymb, amsthm, amsfonts}
\usepackage[usenames,dvipsnames]{xcolor}
\newcommand{\Yun}[1]{{\color{red} {#1}}}

\input xy
\xyoption{all}

\usepackage{hyperref}

\swapnumbers
\numberwithin{equation}{section}

\theoremstyle{plain}
\newtheorem{theorem}[subsubsection]{Theorem}
\newtheorem{lemma}[subsubsection]{Lemma}
\newtheorem{prop}[subsubsection]{Proposition}
\newtheorem{cor}[subsubsection]{Corollary}
\newtheorem{conj}[subsubsection]{Conjecture}

\newtheorem*{claim}{Claim}

\theoremstyle{definition}
\newtheorem{defn}[subsubsection]{Definition}

\newtheorem{remark}[subsubsection]{Remark}
\newtheorem{exam}[subsubsection]{Example}

\def\AA{\mathbb{A}}
\def\BB{\mathbb{B}}
\def\CC{\mathbb{C}}
\def\DD{\mathbb{D}}

\def\FF{\mathbb{F}}
\def\GG{\mathbb{G}}

\def\JJ{\mathbb{J}}
\def\KK{\mathbb{K}}

\def\NN{\mathbb{N}}

\def\PP{\mathbb{P}}
\def\QQ{\mathbb{Q}}
\def\RR{\mathbb{R}}

\def\XX{\mathbb{X}}

\def\ZZ{\mathbb{Z}}

\def\calC{\mathcal{C}}

\def\calF{\mathcal{F}}
\def\calG{\mathcal{G}}

\newcommand\cA{\mathcal{A}}
\newcommand\cB{\mathcal{B}}
\newcommand\cC{\mathcal{C}}
\newcommand\cD{\mathcal{D}}
\newcommand\cE{\mathcal{E}}
\newcommand\cF{\mathcal{F}}
\newcommand\cG{\mathcal{G}}
\newcommand\cH{\mathcal{H}}

\newcommand\cK{\mathcal{K}}
\newcommand\cL{\mathcal{L}}
\newcommand\cM{\mathcal{M}}
\newcommand\cN{\mathcal{N}}
\newcommand\cO{\mathcal{O}}
\newcommand\cP{\mathcal{P}}

\newcommand\cS{\mathcal{S}}

\newcommand\cW{\mathcal{W}}
\newcommand\cX{\mathcal{X}}

\newcommand\cZ{\mathcal{Z}}

\newcommand\F{\mathcal{F}}

\def\bA{\mathbf{A}}

\def\bG{\mathbf{G}}

\def\bI{\mathbf{I}}
\def\bJ{\mathbf{J}}
\def\bK{\mathbf{K}}

\def\bP{\mathbf{P}}

\def\bR{\mathbf{R}}

\def\bT{\mathbf{T}}

\newcommand\frG{\mathfrak{G}}

\newcommand\frN{\mathfrak{N}}

\newcommand\frX{\mathfrak{X}}

\newcommand\fra{\mathfrak{a}}
\newcommand\frb{\mathfrak{b}}
\newcommand\frc{\mathfrak{c}}
\newcommand\frd{\mathfrak{d}}

\newcommand\frg{\mathfrak{g}}

\newcommand\frl{\mathfrak{l}}
\newcommand\frakm{\mathfrak{m}}
\newcommand\frn{\mathfrak{n}}
\newcommand\frp{\mathfrak{p}}

\newcommand\frt{\mathfrak{t}}

\newcommand\tilW{\widetilde{W}}

\def\dG{G^{\vee}}

\def\dT{T^{\vee}}

\newcommand\Aff{\textup{Aff}}

\newcommand\Av{\textup{Av}}

\newcommand{\Bun}{\textup{Bun}}
\newcommand{\can}{\textup{can}}
\newcommand{\ch}{\textup{char}}
\newcommand{\codim}{\textup{codim}}
\newcommand{\Coh}{\textup{Coh}}

\newcommand\dil{\textup{dil}}

\newcommand\ev{\textup{ev}}

\newcommand{\Fl}{\textup{Fl}}

\newcommand\Forg{\textup{Forg}}

\newcommand\Frac{\textup{Frac}}

\newcommand{\Gr}{\textup{Gr}}
\newcommand{\gr}{\textup{gr}}

\newcommand\IC{\textup{IC}}
\newcommand\id{\textup{id}}

\newcommand\inv{\textup{inv}}
\newcommand\Irr{\textup{Irr}}

\newcommand\Lie{\textup{Lie}\ }
\newcommand\Loc{\textup{Loc}}

\newcommand{\opp}{\textup{op}}

\newcommand\Perv{\textup{Perv}}
\newcommand{\Pic}{\textup{Pic}}
\newcommand\pr{\textup{pr}}

\newcommand\pt{\textup{pt}}
\newcommand\QCoh{\textup{QCoh}}

\newcommand\Rep{\textup{Rep}}
\newcommand{\Res}{\textup{Res}}

\newcommand\rs{\textup{rs}}

\newcommand\Spec{\textup{Spec}\ }
\newcommand\Spf{\textup{Spf}\ }
\newcommand\St{\textup{St}}

\newcommand\Stab{\textup{Stab}}
\newcommand\Supp{\textup{Supp}}
\newcommand\Sym{\textup{Sym}}

\newcommand\Tot{\textup{Tot}}

\newcommand\triv{\textup{triv}}

\newcommand{\univ}{\textup{univ}}

\newcommand{\Vect}{\textup{Vect}}

\newcommand\Aut{\textup{Aut}}
\newcommand\Hom{\textup{Hom}}
\newcommand\End{\textup{End}}

\newcommand{\RHom}{\bR\Hom}

\newcommand{\Ext}{\textup{Ext}}

\newcommand\GL{\textup{GL}}

\newcommand\SL{\textup{SL}}
\renewcommand\sl{\mathfrak{sl}}

\newcommand\Sp{\textup{Sp}}

\newcommand{\Gm}{\GG_m}
\def\Ga{\GG_a}

\newcommand{\Ad}{\textup{Ad}}

\newcommand\xch{\mathbb{X}^*}
\newcommand\xcoch{\mathbb{X}_*}

\newcommand{\leftexp}[2]{{\vphantom{#2}}^{#1}{#2}}
\newcommand{\pH}{\leftexp{p}{\textup{H}}}

\newcommand{\Ql}{\QQ_{\ell}}
\newcommand{\Qlbar}{\overline{\QQ}_\ell}

\newcommand{\twtimes}[1]{\stackrel{#1}{\times}}

\newcommand{\htimes}{\widehat{\times}}
\renewcommand{\j}[1]{\langle{#1}\rangle}
\newcommand{\wt}[1]{\widetilde{#1}}
\newcommand{\wh}[1]{\widehat{#1}}
\newcommand\quash[1]{}
\newcommand\un{\underline}

\newcommand{\ov}{\overline}
\newcommand{\bs}{\backslash}
\newcommand{\tl}[1]{[\![#1]\!]}
\newcommand{\lr}[1]{(\!(#1)\!)}
\newcommand\sss{\subsubsection}
\newcommand\xr{\xrightarrow}
\newcommand\op{\oplus}
\newcommand\ot{\otimes}

\newcommand{\qq}{\mathbin{/\mkern-6mu/}}
\renewcommand\c\circ
\newcommand\bt{\boxtimes}
\newcommand\vn{\varnothing}

\newcommand{\homog}[2]{\textup{H}_{#1}({#2})}  
\newcommand{\cohog}[2]{\textup{H}^{#1}({#2})}     
\newcommand{\cohoc}[2]{\textup{H}_{c}^{#1}({#2})}     
\newcommand{\hBM}[2]{\textup{H}^{\textup{BM}}_{#1}({#2})}  

\newcommand\upH{\textup{H}}

\newcommand{\incl}{\hookrightarrow}
\newcommand{\isom}{\stackrel{\sim}{\to}}

\newcommand{\surj}{\twoheadrightarrow}
\newcommand{\oll}[1]{\overleftarrow{#1}}
\newcommand{\orr}[1]{\overrightarrow{#1}}

\newcommand{\ari}{\ar@{^{(}->}} 

\renewcommand\a\alpha
\renewcommand\b\beta
\newcommand\g\gamma
\newcommand\G\Gamma
\renewcommand\d\delta
\newcommand\D\Delta
\newcommand{\e}{\epsilon}
\newcommand{\io}{\iota}
\renewcommand{\k}{\kappa}
\renewcommand{\th}{\theta}
\newcommand{\Th}{\Theta}
\newcommand{\ph}{\varphi}
\renewcommand\r\rho
\newcommand{\s}{\sigma}
\newcommand{\Sig}{\Sigma}
\renewcommand{\t}{\tau}

\newcommand{\y}{\eta}
\newcommand{\z}{\zeta}

\renewcommand{\l}{\lambda}
\renewcommand{\L}{\Lambda}
\newcommand{\om}{\omega}
\newcommand{\Om}{\Omega}
\newcommand\Up{\Upsilon}

\newcommand\hs{\heartsuit}
\newcommand\na{\natural}
\newcommand\sh{\sharp}

\DeclareMathOperator{\Kir}{Kir}
\newcommand\frAff{\mathfrak{Aff}}
\newcommand\Sh{\mathrm{Sh}}
\newcommand\muSh{\mu\mathrm{Sh}}
\newcommand\muhom{\mu\mathrm{hom}}

\newcommand{\LG}{G^{\vee}}
\newcommand{\LB}{B^{\vee}}
\newcommand{\Lg}{\frg^{\vee}}

\newcommand{\LT}{T^{\vee}}
\newcommand{\ft}{\frt}

\newcommand\sph{\mathrm{sph}}
\newcommand\Sat{\mathrm{Sat}}
\newcommand\Wak{\mathrm{Wak}}

\newcommand\lmod{\mathrm{-mod}}

\newcommand\gmod{\mathrm{-mod}^{\mathrm{gr}}}

\newcommand\Hk{\mathrm{Hk}}
\newcommand\mt{\mapsto}

\newcommand\J{\mathbf{G}^{1}_{\infty}}
\newcommand\Jker{\mathbf{G}^{\psi}_{\infty}}
\newcommand\mon{\mathrm{mon}}
\newcommand\hinf{\wh\infty}
\newcommand\Hinf{\mathcal{H}_{\infty}}
\newcommand\Him{\mathcal{H}^{\mon}_{\infty}}
\newcommand\hHim{\widehat{\cH}^{\mon}_{\infty}}
\newcommand\Hs{\cH^{\textup{sph}}_{0}}
\newcommand\Hsh{\cH^{\textup{sph},\hs}_{0}}
\newcommand\hi{\widehat{\infty}}
\newcommand\Cone{\mathrm{Cone}}
\newcommand\Fun{\mathrm{Fun}}
\newcommand\chk{\textup{char}(k)}
\newcommand\FT{\mathrm{FT}}
\newcommand\Grot{\GG_{m}^{\textup{rot}}}

\newcommand\lmon{\textup{-mon}}
\newcommand\Wang{\textup{Wang}}
\newcommand\dc{\frd} 
\newcommand\eqHinf{{}^{T}\cH_{\infty}^{T}}
\newcommand\Hinfhs{(\eqHinf)^{\hs}}
\newcommand\DG{\cD_{\psi, \bG}}

\newcommand{\rev}{\textup{rev}}
\renewcommand{\v}{\vee}

\newcommand{\tp}{\textup{top}}

\newcommand{\Jinf}{\bJ_{\infty}}

\newcommand{\Wa}{W_{\textup{aff}}}

\newcommand{\dg}{\mathfrak{g}^{\vee}}
\newcommand{\dt}{\mathfrak{t}^{\vee}}
\newcommand{\dN}{\mathcal{N}^{\vee}}
\newcommand{\dU}{\mathcal{U}^{\vee}}

\newcommand{\oB}{\operatorname{B}}

\newcommand{\oO}{\operatorname{O}}

\newcommand{\oU}{\operatorname{U}}

\newcommand{\oLoc}{\operatorname{Loc}}
\newcommand{\oLGr}{\operatorname{LGr}}

\newcommand{\Picard}{\operatorname{Pic}}
\newcommand{\Bet}{\operatorname{Bet}}

\newcommand{\IndCoh}{\operatorname{IndCoh}}

\newcommand{\cGm}{\check{\mathbb{G}}_m}
\newcommand{\cGa}{\check{\mathbb{G}}_a}

\newcommand{\sot}{\stackrel{!}{\otimes}}
\newcommand{\ssot}{\stackrel{!!}{\otimes}}

\title{Affine Springer fiber and the small quantum group}

\author{Roman Bezrukavnikov}
\email{}
\address{Department of Mathematics, Massachusetts Institute of Technology, 77 Massachusetts Ave, Cambridge, MA 02139}
\author{Pablo Boixeda Alvarez}
\email{p.boixedaalvarez@northeastern.edu}
\address{Department of Mathematics, Northeastern University, 567 Lake Hall, 360 Huntington Ave, Boston, MA 02115}
\author{Michael McBreen}
\email{mcb@math.cuhk.edu.hk}
\address{Department of Mathematics, Room 220, Lady Shaw Building, The Chinese University of Hong Kong, Shatin, N.T., Hong Kong
}
\author{Zhiwei Yun}
\email{zyun@mit.edu}
\address{Department of Mathematics, Massachusetts Institute of Technology, 77 Massachusetts Ave, Cambridge, MA 02139}

\thanks{R.B. is partly supported by the NSF grant DMS-2101507.
Z.Y. is supported by the Simon Foundation. M.M. is supported by the Research Grants Council of the Hong
Kong Special Administrative Region, Project No. CUHK 14307022, 14309323, 14307825.}


\begin{document}

\begin{abstract}
We find a new geometric incarnation for the principal block in the category of modules
over a quantum group at a root of unity, realizing it as a full subcategory of microsheaves
on a certain affine Springer fiber.  We also prove a related geometric Langlands type equivalence with wild ramification,
identifying the latter category with a category of coherent sheaves on the Springer resolution for the
dual group. This can also be viewed as a version of homological mirror symmetry for the Springer resolution.
\end{abstract}

\maketitle

\tableofcontents

\section{Introduction}
The results of this article have several sources of motivation belonging, respectively,
to representation theory, geometric Langlands duality and homological mirror symmetry.

From the representation theory perspective, the main result (see Theorem \ref{th:intro uq}) is a new geometric realization of 
a category of modules over the quantum group at a root of unity. Earlier theorems of this nature
identified a category of quantum group modules with perverse sheaves on the affine or 
semi-infinite flag manifold, which implies the relation between the appropriate version of the affine Kazhdan-Lusztig polynomials to numerical invariants of representations. 
Theorem \ref{th:intro uq} below relates quantum group modules to a different geometric object: it asserts
its equivalence with a full subcategory $\cP_{\psi}$ of {\em microsheaves} on a certain affine Springer fiber $\Fl_\psi$, where $\psi=t\psi_1\in \frg\lr{t}$ and $\psi_1\in \frg$ is regular semisimple;
or more precisely, microsheaves on a certain symplectic variety $\cM_\psi$ supported on
a conical Lagrangian subset homeomorphic to $\Fl_\psi$.
The known and expected applications (see \cite{FLH} and Theorem \ref{th:intro conseq K} below) are also different.

The category $\cP_{\psi}$ of microsheaves on $\Fl_\psi$ also admits a description in terms of perverse sheaves: it turns out to be equivalent to 
a full subcategory in the category of perverse sheaves on the affine flag variety $\Fl$ of $G$,
Whittaker-equivariant with respect to the first congruence subgroup $G[t^{-1}]^{1}$ in the group of negative loops $G[t^{-1}]$
equipped with the additive character given by $\psi$;
the space $\cM_\psi$ above is the shifted Hamiltonian reduction of $T^*\Fl$ by the action of $G[t^{-1}]^{1}$ (by a shifted Hamiltonian reduction we understand the quotient of the preimage of $d\psi=\psi_1dt$ under the moment map). 
We believe that the microlocal description of the category of (partially) Whittaker sheaves 
is a general phenomenon with strong potential applications in representation theory, see 
\ref{on usual} below. To make sense of Whittaker sheaves in an arbitrary sheaf-theoretic context (without necessarily the Artin-Schreier sheaves or exponential $D$-modules at our disposal), we use the notion of Kirillov category introduced by Gaitsgory \cite{Ga-Wh}.


The quotient of the affine flag variety by the congruence subgroup $G[t^{-1}]^{1}$ is the moduli stack
of $G$-bundles on ${\mathbb P}^1$ with level structures at zero and infinity, so the study of sheaves on this space belongs in geometric Langlands duality. Accordingly, we establish an equivalence (see Theorem \ref{th:intro Dpsi coh}) between the derived category of microsheaves and the derived category of coherent 
sheaves on a stack of Langlands parameters which can be identified with the derived category
of quantum group modules based on \cite{ABG} and \cite{Tanisaki}.

On the other hand, we can also realize the symplectic variety $\cM_\psi$ as the moduli space of Higgs bundles
on ${\mathbb P}^1$ with a regular singularities at zero and an irregular one at infinity (of order $2$), as explained in \cite{BBAMY}. There is a corresponding Betti space $\cM_{\Bet,\psi}$, diffeomorphic to a quotient of $\cM_\psi$ by an action of the lattice $\mathbb{X}_*(T)$ of cocharacters of a maximal torus $T \subset G$. The space $\cM_{\Bet,\psi}$ is (for suitable choices of parameters) a smooth affine variety of finite type, and thus comes equipped with a contractible space of Weinstein structures. Microsheaves on $\cM_{\Bet,\psi}$ may be identified with a version of the Fukaya category, using the general theory developed in \cite{ganatra2024microlocal}. One expects that the category of microsheaves on $\cM_\psi$ is equivalent to a suitable $\mathbb{X}_*(T)$-graded version of the Fukaya category of $\cM_{\Bet,\psi}$. 

Combining the above discussions with the geometric Langlands type equivalence in Theorem \ref{th:intro Dpsi coh}, one would thus establish homological mirror symmetry (HMS) between $\cM_{\Bet,\psi}$ and a space
of Langlands parameters. While this case of HMS was among our sources of motivation,
we do not work with Fukaya categories in this paper, thus 
this statement remains conjectural.

\subsection{Kirillov category and microsheaves on the affine Springer fiber}\label{ss:intro main}

\sss{Affine Springer fiber}
Let $G$ be a connected reductive group over $\CC$ with Lie algebra $\frg$. Let $\CC\lr{t}$ be the field of Laurent series in one variable $t$, and let $G\lr{t}$ be the associated loop group with Lie algebra $\frg\lr{t}=\frg\ot_\CC\CC\lr{t}$. Let $\bI\subset G\lr{t}$ be an Iwahori subgroup. Let $\psi_1\in\frg$ be regular semisimple and $\psi=t\psi_1\in \frg\lr{t}$. We denote by $\Fl_\psi=\{g\in G\lr{t}/\bI|\Ad(g^{-1}\psi\in \Lie \bI)\}$ the affine Springer fiber of $\psi$ in the affine flag variety $\Fl=G\lr{t}/\bI$. Note that $\Fl_\psi$ has irreducible components in bijection with the extended affine Weyl group $\tilW=\xcoch(T)\rtimes W$, and each component has dimension equal to that of the flag variety $\cB$ of $G$ (in fact some components can be identified with $\cB$). It also carries a $\Gm$-action given by loop rotation (scaling $t$).

Just as the usual Springer fibers, the affine Springer fiber $\Fl_\psi$ can be embedded into a symplectic variety $\cM_{\psi}$ as a Lagrangian. Moreover, the $\Gm$-action on $\Fl_\psi$ extends to $\cM_\psi$, and all points of $\cM_\psi$ contracts to points in $\Fl_\psi$ under this action. This is not a general structure for any affine Springer fiber but  crucially relies on the $\Gm$-symmetry on $\Fl_\psi$. A more general construction of this kind is done in our previous paper  \cite{BBAMY}.

\sss{Kirillov category}

In geometric representation theory we often encounter sheaf-theoretic versions of Whittaker models whose definition involves an Artin-Schreier sheaf on the affine line, so a priori only makes sense over a characteristic $p$ base field, or in the context of $D$-modules where we replace Artin-Schreier sheaf by the exponential $D$-module. Gaitsgory \cite{Ga-Wh} gives a way to define such geometric Whittaker categories in any sheaf-theoretic context (in particular in the Betti setting), when the $\Ga$-action on the space involved extends to an action of $\Ga\rtimes\Gm$ (where $\Gm$ acts on $\Ga$ by scaling). We apply his construction \footnote{We will need a variant of his construction where $\Gm$-equivariance is replaced by $\Gm$-monodromicity, which we develop in \S\ref{ss:Kir alt}.} to the moduli stack $\Bun_G(\Jker, \bI_0)$ of $G$
-bundles over $\PP^1$ with Iwahori level structure at $0$ and the level structure given by the group $\Jker$, the kernel of the linear character $\J\to \frg\xr{\psi_1}\Ga$, where $\J=\ker(G\tl{t^{-1}}\to G)$ is the first congruence subgroup at $\infty$. The $\Ga$-action on $\Bun_G(\Jker, \bI_0)$ comes from the residual symmetry of $\J/\Jker\cong \Ga$; the $\Gm$-action is by rotation on $\PP^1$ fixing $0$ and $\infty$. If we do not impose Iwahori level structure at $0$ we get a moduli stack $\Bun_G(\Jker)$. Let
\begin{equation*}
    \DG=\Kir(\Bun_G(\Jker)), \quad
    \wt\cD_\psi=\Kir(\Bun_G(\Jker, \bI_0))
\end{equation*}
be the Kirillov categories of $\Bun_G(\Jker, \bI_0)$; they are localizations of $\Gm$-equviariant sheaves on $\Bun_G(\Jker)$ and $\Bun_G(\Jker, \bI_0)$  by killing sheaves that are also $\Ga$-equivariant. We can work with various sheaf-theoretic contexts (see \S\ref{sss:sheaves}), with characteristic zero coefficient field $E$.

Let $\dG$ be the Langlands dual group of $G$, as a split reductive group over $E$, equipped with a maximal split torus $\dT$. 


Our first result is a geometric Langlands type equivalence beyond tame ramification. It gives a description of $\cD_\psi$ in terms of coherent sheaves on a stack of Langlands parameters.

\begin{theorem}[See Theorem \ref{th1}]\label{th:intro Dpsi coh} There is a canonical equivalence
\begin{equation}\label{Dpsi coh}
    \cD_\psi\cong D^{b}(\Coh^{\dT}_{\cB^{\vee}}(\wt\dN)).
\end{equation}
Here
\begin{itemize}
\item The left side is the full subcategory of $\wt\cD_\psi$ generated by the image of the pullback $\DG\to \wt\cD_\psi$ under the action of the affine Hecke category on $\wt\cD_\psi$ (by modifying the bundles at $0$).
\item On the right side, $\Coh^{\dT}_{\cB^{\vee}}(\wt\dN)$ is the abelian category of $\dT$-equivariant coherent sheaves on the Springer resolution $\wt\dN=T^{*}\cB^{\vee}$ of the nilpotent cone of $\dG$ with set-theoretic support on the zero section $\cB^{\vee}$ (the flag variety of $\dG$).
\end{itemize}
The equivalence is compatible with the following actions:
\begin{enumerate}
    \item The actions of the constructible and coherent affine Hecke categories on both sides, under the monoidal equivalence of the two realizations of the affine Hecke categories \cite{B}.
    \item The actions of $\Rep(\dT)$ on the right side by tensoring, and on the left side via a monoidal functor $\Rep(\dT)\to \hHim$ to a certain depth one Hecke category $\hHim$ that acts on $\cD_\psi$ via modification of bundles at $\infty$.
\end{enumerate}
\end{theorem}


Combining this result with known coherent realization of the category of modules over a quantum group we arrive at the following:

\begin{theorem}[See Corollary \ref{main_cor}(1)]\label{th:intro uq}
Let $\cP_\psi$ is the heart of the perverse $t$-structure on $\cD_{\psi}$ (see Theorem \ref{th:intro Dpsi coh}). Let $u_q^{\wh{0}}\lmod^{gr}$ be the principal block category in the category of $\xch(T)$-graded modules over the small quantum group $u_q$ at a root of unity. Then $\cP_\psi$ is equivalent 
to the category of finitely generated modules with nilpotent central character 
over the Koszul dual to $u_q^{\wh{0}}\lmod^{gr}$.
\end{theorem}


\subsection{Microsheaves and microlocal fiber functor for Satake category}

As mentioned in the beginning, $\cD_{\psi}$ is an incarnation of a category of microsheaves on $\Fl_{\psi}$ in the classical sheaf-theoretic context. We now make this statement more precise. 

Following Nadler and Shende \cite{NS,NS2}, we consider the category $\muSh_{\Fl_\psi}(\cM_\psi)$ of microsheaves on $\cM_\psi$ supported on $\Fl_\psi$. More details on microsheaves will be recalled in \S\ref{ss:micro} and \S\ref{ss:micro Mpsi}. Let $\muSh_{\Fl_{\psi}, fs}(\cM_{\psi})$ be the full subcategory of $\muSh_{\Fl_\psi}(\cM_\psi)$ consisting of microsheaves supported on finitely many irreducible components of $\Fl_\psi$ and with cohomologically bounded and finite-dimensional microstalks.




A general construction for Kirillov categories provides a microlocalization functor
\begin{equation*}
\wt M: \wt\cD_{\psi}\to \muSh(\cM_{\psi}).
\end{equation*}

\begin{theorem}[See Theorem \ref{th:fullfaith} and Proposition \ref{p:finite ss char Dpsi}]\label{th:intro mic}
    The functor $\wt M$ restricts to a full embedding of categories
    \begin{equation}\label{micro Dpsi}
        M: \cD_\psi\incl \muSh_{\Fl_{\psi}, fs}(\cM_{\psi}).
    \end{equation}
    Moreover, $\cD_{\psi}$ is the full subcategory of $\wt\cD_{\psi}$ consisting of sheaves whose image under $\wt M$ lies in $\muSh_{\Fl_{\psi}, fs}(\cM_{\psi})$.
\end{theorem}

We expect the full embedding \eqref{micro Dpsi} to be an equivalence. We plan to come back to this question later.

\sss{Microlocal fiber functor for the Satake category}
Along the way we also prove a result about the Satake category $\Hsh=\Perv(G\tl{t}\bs G\lr{t}/G\tl{t})$ that is simple to state and is of independent interest. For $\l\in \xcoch(T)$, we view $(t^\l, \psi)$ as a point in the cotangent bundle of $\Gr=G\lr{t}/G\tl{t}$ at the point  $t^\l$. Let $\mu_\l: \Hsh\to \Vect_E$ be the functor of taking microstalk (or local Morse group) at $(t^\l, \psi)$. This functor was considered by Evans and Mirkovi\'c \cite{EM} who showed that the dimension of $\mu_\l(\IC_\mu)$ (where $\IC_\mu\in \Hsh$ is the intersection complex of the Schubert variety $\Gr_{\le \mu}$) is the multiplicity of $\l$ in the irreducible $\dG$-module $V_\mu$ with highest weight $\mu$. Our next result gives a categorification of their numerical result; this was conjectured
(in a somewhat different form) in \cite{FK}.

\begin{theorem}
    The direct sum of the microstalks functors
    \begin{equation*}
        \mu=\bigoplus_{\l\in \xcoch(T)}\mu_\l: \Hsh\to \Rep(\dT)\cong \Vect_{\xcoch(T)}
    \end{equation*}
    has a canonical monoidal  structure. Under the geometric Satake equivalence $\Hsh\cong \Rep(\dG)$, $\mu$ is canonically isomorphic to the restriction functor $\Rep(\dG)\to \Rep(\dT)$ up to composing with the inversion on $\dT$. 
    
    In particular, composing $\mu$ with the forgetful functor, the resulting functor  $\mu': \Hsh\to\Vect$ has a canonical structure of a fiber  functor for the Tannakian category $\Hsh$.
\end{theorem}

\subsection{Applications}

\sss{Consequences for $K$-groups} Consequences of theorems in  \S\ref{ss:intro main} at the level of Grothendieck groups are
used in the recent work of Feng-Le Hung \cite{FLH} that solves cases of the Breuil-M\'ezard conjecture on the deformation space of Galois representations. The precise results used in {\em loc.cit.} are:

\begin{theorem}[See Theorem \ref{th:conseq K}]\label{th:intro conseq K}
    There is a canonical isomorphism of abelian groups
    \begin{equation*}        K_0(\Coh^{\dT}_{\cB^{\vee}}(\wt\dN))\to \homog{2N}{\Fl_\psi,\ZZ}
    \end{equation*}
    where $N=\dim \Fl_\psi$, 
    that intertwines both the left and right actions of the affine Weyl group $\tilW=\xcoch(T)\rtimes W$ on both sides:
    \begin{enumerate}
        \item the left side is equipped with a left action of $\tilW$ by tensoring with $\Rep(\dT)$ and the action of $N_{\dG}(\dT)$; it is also equipped with a right $\tilW$-action from convolution by the equivariant $K$-group of the Steinberg variety of $\dG$.
        \item The right side is equipped with a left $\tilW$-action by translation of the centralizer of $\psi$ on $\Fl_\psi$ and by the monodromy action by varying $\psi$; it is also equipped with a Springer type right $\tilW$-action constructed by Lusztig (up to tensoring by the sign character). 
    \end{enumerate}
\end{theorem}

\sss{Cohomology of affine Springer fiber and center of small quantum group}

Another application is to the  study of the center of small quantum group, see \cite{LQ} and \cite{BBSV}. The main result of \cite{BBSV} is an injective homomorphism (conjecturally an isomorphism) from $\upH^*(\Fl_\psi)$ to the center (endomorphisms of the identity functor) of $u_q^{\wh{0}}\lmod^{gr}$.

Theorem \ref{th:intro uq} provides a new way to construct and study this homomorphism. Namely, we have a natural map from $\upH^*(\Fl_\psi)$ to
endomorphisms of the identity functor in $\muSh_{\Fl_{\psi}, fs}(\cM_{\psi})$
preserving the homological degree. Using the purity of $\upH^*(\Fl_\psi)$
one can check that Koszul duality sends this action to an action
by degree zero endomorphisms of the Koszul dual triangulated category resulting
in a homomorphism $\upH^*(\Fl_\psi)\to \End(\textup{Id}_{u_q^{\wh{0}}\lmod^{gr}})$.
We plan to develop this application elsewhere.


\subsection{Further directions}

\sss{Coherent description of all microsheaves on $\Fl_\psi$}\label{sss:intro whole mush} We expect Theorem \ref{th:intro Dpsi coh} to extend to a full embedding from the larger category $\muSh_{\Fl_\psi}(\cM_\psi)$ to the category of coherent sheaves on a Betti moduli space for $\dG$, a special case of the Betti moduli spaces introduced in \cite{Trinh}, see also \cite[Section 4]{BBAMY}. 

Let $B^{\vee,+}$ and $B^{\vee, -}$ be a pair of opposite Borel subgroups of $\dG$ such that $B^{\vee,+}\cap B^{\vee,-}=\dT$. Consider the "big cell" 
\begin{equation*}
    \cW=B^{\vee,-} B^{\vee,+}\subset \dG.
\end{equation*}
Let $\dU\subset \dG$ be the unipotent variety and $\cW_0=\cW\cap \dU$. Let $\wt\cW_0=\wt\dU|_{\cW_0}$ be the restriction of the Springer resolution $\wt\dU$ to $\cW_0$. The maximal torus $\dT$ acts on $\cW_0$ and $\wt\cW_0$ by conjugation.

\begin{conj}\label{conj:mush-IndCoh}
    There is a full embedding
    \begin{equation}\label{conj W}
        \muSh_{\Fl_\psi}(\cM_\psi)\incl \QCoh^{\dT}(\wt\cW_0).
    \end{equation}
    extending the equivalence \eqref{Dpsi coh} in Theorem \ref{th:intro Dpsi coh} in the following sense:
\begin{itemize}
\item The Kirillov side $\cD_{\psi}$ of \eqref{Dpsi coh} is viewed as full subcategory of $\muSh_{\Fl_\psi}(\cM_\psi)$ by Theorem \ref{th:intro mic};
\item The coherent side $D^b(\Coh^{\dT}_{\cB^\vee}(\wt\dN))$ of \eqref{Dpsi coh} is identified with a full subcategory of $\QCoh^{\dT}(\wt\cW_0)$ supported on $\cB^\vee\incl \wt\cW_0$ (fiber over $1\in \cW_0$). 
\end{itemize}
Moreover, the embedding \eqref{conj W} intertwines the natural actions of $\xcoch(T)$ on both sides.
\end{conj}
The lattice $\xcoch(T)$ acts on the symplectic variety $\cM_\psi$ preserving its Lagrangian $\Fl_\psi$ (see Section \ref{ss:lattice action}), and the statement that it induces an action on $\muSh_{\Fl_\psi}(\cM_\psi)$ is part of the conjecture (see Section \ref{sss:latticeaction}). On the coherent side $\xcoch(T)$ acts by tensoring via the equivalence $\xcoch(T)=\Rep(\dT)$. 

\sss{Other homogeneous affine Springer fibers} In \cite[Conjecture 5.0.2]{BBAMY} we formulated a more general conjecture for any homogeneous element $\psi$ in the loop Lie algebra, of which Conjecture \ref{conj W} is a special case. To connect the two, we observe that $[\wt\cW_0/\Ad(\dT)]$ is the Betti moduli space denoted by $\cM_{\textup{Bet}, \psi}^{0,\dG}$ in {\em loc.cit.}, because the braid attached to the slope one element $\psi$ is the full twist $\b=w_0^2$ (see \cite[Example 4.1.3]{BBAMY}).

In a sequel to this work, we will prove generalizations of Theorem \ref{th:intro Dpsi coh} to homogeneous elements $\psi$ with slope of the form $1/m$.

\sss{Microsheaves on usual Springer fibers}\label{on usual} Let $e$ be a nilpotent element in $\frg$, $S_e\subset \cN$ be the (nilpotent part of the) Slodowy slice through $e$ and $\wt S_e$ be the preimage of $S_e$ under the Springer resolution. Then the Springer fiber $\cB_e$ is a conical Lagrangian in $\wt S_e$, and we may consider the category of microsheaves $\muSh_{\cB_e}(\wt S_e)$. Our construction in this paper suggests a sheaf-theoretic model for $\muSh_{\cB_e}(\wt S_e)$, namely a full subcategory in the category $\Kir(V_e\bs\cB)$, where $V_e$ is the kernel of
the character $\psi_e: U_e\to \Ga$
for a Premet subgroup $U_e$  and its character $\psi_e$ whose differential equals the linear functional defined by $e$. For example, if $e$ is even and $\{e,h,f\}$ is an $\sl_{2}$-triple, then $\Lie U_e$ is the sum of negative weight spaces of $\frg$ under $h$. 
We expect this category to be equivalent to the category of finite dimensional modules
over the corresponding $W$-algebra with an integral regular central character. Notice
that this category is closely connected to the $D$-modules (or Harish-Chandra bimodules) realization of the asymptotic Hecke category \cite{LO}  and to character
$D$-modules \cite{BFO}. We expect the above topological realization to yield a generalization
of these results to the constructible setting.

\sss{Characteristic $p$ version} We expect that our results can be extended to the setting when the coefficient field $E$ has positive characteristic. The analogue of Theorem \ref{th:intro uq} is a Koszul duality between the category of microsheaves and a regular block in the category
of $(\dG_1,\dT)$-modules where $\dG_1$ is the kernel of the Frobenius morphism for the reductive
group $\dG$ over $E$. Notice that a key ingredient in the proof of
Theorem \ref{th:intro uq}, the relation of modules over the Lusztig quantum group and perverse
sheaves on the affine Grassmannian \cite{ABG} has a positive characteristic analogue known as the Finkelberg--Mirkovi\' c conjecture, which was established
in \cite{BeRi} with some restrictions on $\ch(E)$. However, the proof of Theorem \ref{th:intro uq} relies also on the formality from weights
argument which is not available in positive characteristic, thus the generalization requires
a different technical implementation of the strategy outlined in Remark \ref{heur}.

\sss{Coherent duality for affine Springer fibers}
We briefly describe a heuristic relation of our results and the more general 
\cite[Conjecture 1.4.1]{BBAMY} to the conjectural equivalence between two coherent categorifications
of the double affine Hecke algebra (DAHA), which is a local counterpart of the classical limit
of the geometric Langlands duality. 

Varagnolo and Vasserot \cite[Theorem 2.5.6]{VV} introduced a categorification of DAHA coming from the derived category of equivariant coherent sheaves on the loop group analogue of the Steinberg variety (a similar construction
for the spherical, degenerate DAHA appeared in \cite{BFM}). One expects that the monoidal categories constructed this way from two Langlands dual group $G$, $G^\vee$ are naturally equivalent, see 
\cite[\S 7.9]{BFM} for the similar conjecture in the spherical case. Next, this equivalence of monoidal
categories should come with compatible equivalence of module categories as follows.

We fix an invariant quadratic form yielding a vector space isomorphism $\ft \cong \ft^\vee$ and 
an isomorphism of Chevalley spaces $\ft/W \cong \ft^\vee/W$. Let $e\in \g((t))$ and $e^\vee\in \g^\vee((t))$ be regular semisimple topologically nilpotent elements whose images in the Chevalley space
match under this isomorphism. Then one expects an equivalence
\[ \label{coh_dual}
DCoh(\Fl_e) \cong DCoh^{Z_{e^\vee}}(\Fl_{e^\vee})
\]
where $Z_{e^\vee}$ is the centralizer of $e^\vee$ and $\Fl_e$, $\Fl _{e^\vee}$ are the affine Springer fibers; \eqref{coh_dual} should be an equivalence of module
categories over the two coherent categorifications of DAHA.

The heuristic of $S^1$-localization proposed by Ben-Zvi and Nadler \cite{BZN} suggests that
\eqref{coh_dual} should imply another equivalence of categories obtained, respectively, from the left
hand side by quantization and from the right hand side by passing to the similar object for the
finite dimensional group $G^\vee$ replacing the corresponding loop group. In the special case when $e=ts$ as above, the two categories can be arguable identified with the two sided of the equivalence 
in Theorem \ref{th:intro Dpsi coh}.

\subsection{Notation and convention}

\sss{Sheaf-theoretic contexts}\label{sss:sheaves}
Let $X$ be an algebraic stack of finite type over an algebraically closed field $k$. Let $E$ be a field of characteristic zero depending on the sheaf-theoretic context. We will use the notation $D(X)$ to denote any of the following $E$-linear sheaf-theoretic categories in different contexts:
\begin{itemize}
    \item \'Etale: take $E=\ov\Ql$ where $\chk\ne \ell$. Let $D^b(X,E)$ be the bounded derived category of $\Qlbar$-complexes of sheaves for the \'etale topology of $X$. We will reserve the notation $D(X)$ for a variant of $D^b(X,E)$ by formally adding a half Tate twist as in \S\ref{sss:half Tate}.
    \item Betti: $k=\CC$ and $E$ is any field of characteristic zero field. We take $D(X)$ to be the bounded derived category of constructible sheaves on $X$ with $E$-coefficients for the classical topology on $X(\CC)$.
    \item De Rham: $\chk=0$ and $E=k$. We take $D(X)$ to be the bounded derived category of algebraic $D$-modules on $X$. 
\end{itemize}
The main body of the paper is written in the \'etale context, although all arguments in the paper work for the Betti and de Rham contexts equally well.

\sss{Half Tate twist}\label{sss:half Tate}
In the \'etale setting, we have Tate twist as a local system on $\Spec k$
\begin{equation*}
\ZZ_\ell(1)=\varprojlim_{n}\mu_{\ell^n}(k), \quad E(1):=E\ot_{\ZZ_\ell}\ZZ_\ell(1).
\end{equation*}
For any $\cF\in D^b(X,E)$ we denote $\cF(1)=\cF\ot_{E}E(1)\in D^b(X,E)$. 

We form the doubled category
\begin{eqnarray*}
    D(X):=D^b(X,E)\op D^b(X,E).
\end{eqnarray*}
Denote $(0,E)\in D(X)$ as $E(1/2)$, and in general $(\cF,\cG)\in D(X)$ as $\cF\op \cG(1/2)$. This way we have formally introduced a half Tate twist to $D^b(X,E)$.
Tensor product in $D(X)$ obey the usual rules $E(n/2)\ot E(m/2)=E((n+m)/2)$.

For $\cF\in D(X)$ we denote
\begin{equation*}
    \cF\j{n}=\cF[n](n/2)\in D(X).
\end{equation*}


\sss{Averaging functors} Suppose an algebraic group $G$ acts on s stack $X$,  we denote $D_G(X)=D(G\bs X)$. We have the forgetful functor $\om_{G}: D_{G}(X)\to D(X)$ to be $p^{*}[\dim G]\cong p^{!}[-\dim G]$, where $p: X
\to [X/G]$ is the quotient map. It has a left adjoint $\Av_{G,!}$ and a right adjoint $\Av_{G,*}$
\begin{equation*}
\xymatrix{ D_{G}(X)\ar[r]^{\om_{G}} & D(X)\ar@<1ex>[l]^{\Av_{G,*}}\ar@<-1ex>[l]_{\Av_{G,!}}
}
\end{equation*}
We have $\Av_{G,!}(-)\cong a_{!}(E\j{\dim G}\bt (-))$, and $\Av_{G,*}(-)\cong a_{*}(E\j{-\dim G}\bt (-))$ where $a:G\times X\to X$ is the action map.
  
The same convention applies to averaging along the $G$-action against a rank one character local system $\cL$ on $G$. Notation: $\Av_{G,\cL,!}$, or $\Av_{\cL,!}$ if $G$ is clear from the context.

\sss{} We use $\Vect$ to denote the abelian category of finite-dimensional $E$-vector spaces. For an algebraic group $H$ over $E$, we denote by $\Rep(H)$ the category of finite-dimensional algebraic representations of $H$ over $E$.

\section*{Acknowledgements}
The authors would like to thank Tony Feng and Ben Gammage for stimulating conversations.

\section{Kirillov category}

We elaborate on Gaitsgory's construction of the Kirillov category, which will serve as a model in the usual sheaf-theoretic context for talking about microsheaves. We also give an alternative construction of the Kirillov category in \S\ref{ss:Kir alt} that is more flexible for applications. 

We will freely use material about monodromic sheaves such as monodromic vanishing cycles and monodromic Fourier transform, which are reviewed and further developed in Appendix \ref{as:mon}.

\subsection{Kirillov category after Gaitsgory \cite{Ga-Wh}}\label{ss:Kir}
Let $\Aff=\Ga\rtimes\Gm$ be the group of affine transformations of $\AA^{1}$. Let $X$ be a stack with an action of $\Aff$. 

\begin{defn} The Kirillov category for the stack $X$ with $\Aff$-action is the quotient
\begin{equation*}
\Kir(X)=D_{\Gm}(X)/D_{\Aff}(X)
\end{equation*}
under the forgetful functor $D_{\Aff}(X)\incl D_{\Gm}(X)$.
\end{defn}
Denote the quotient functor by $\k: D_{\Gm}(X)\to \Kir(X)$.  

Let $\Kir_{*}(X)\subset D_{\Gm}(X)$ be the kernel of the averaging functor
\begin{equation*}
\Av_{\Ga,*}: D_{\Gm}(X)\to D_{\Aff}(X).
\end{equation*}
In other words, $\Kir_{*}(X)$ is the right orthogonal to the full subcategory $D_{\Aff}(X)\subset D_{\Gm}(X)$.  Let $\k_{*}: \Kir_{*}(X)\incl D_{\Gm}(X)$ be the inclusion, then $\k\c\k_{*}:\Kir_{*}(X)\to \Kir(X)$ is an equivalence. The composition $\Kir(X)\cong \Kir_{*}(X)\xr{\k_{*}} D_{\Gm}(X) $ is a right adjoint of $\k$.

Similarly $\k$ has a left adjoint $\k_{!}$ whose image $\Kir_{!}(X)$  is the left orthogonal of $D_{\Aff}(X)$ in $D_{\Gm}(X)$. We have the following recollement diagram
\begin{equation}\label{recollement}
\xymatrix{ D_{\Aff}(X)\ar[r] & D_{\Gm}(X) \ar[r]\ar@<1ex>[l]^{\Av_{\Ga,*}}\ar@<-1ex>[l]_{\Av_{\Ga,!}}& \Kir(X)\ar@<1ex>[l]^{\k_{*}}\ar@<-1ex>[l]_{\k_{!}}
}
\end{equation}

\sss{Functoriality}
Suppose $f:Y\to X$ is an $\Aff$-equivariant map.  The usual adjunctions $(f_{!},f^{!})$ and $(f^{*},f_{*})$ pass to the quotient and induce adjoint pairs between $\Kir(Y)$ and $\Kir(X)$.

\begin{exam} Let $X=\Aff$ with the left translation action. Then Fourier transform along $\Ga$ identifies $D_{\Gm}(X)$ with $D_{\Gm}(Y)$, where $Y\cong \AA^{1}\times \Gm$  with the action of $t\in \Gm$ by $t\cdot (b,s)=(t^{-1}b,ts)$. We may identify $D_{\Gm}(Y)\cong D(\AA^{1})$ by restricting to $\AA^{1}\times \{1\}\subset Y$.  Under these equivalence, $D_{\Aff}(X)\cong D(\{0\})\xr{i_{*}}D(\AA^{1})$, $\Kir(X)\cong D(\AA^{1}-\{0\})$ and the diagram \eqref{recollement} is the usual recollement diagram for the closed subset $\{0\}\subset \AA^{1}$ and its open complement $\AA^{1}-\{0\}$.
\end{exam}

\begin{exam} Let $X=\pt$. Then $D_{\Aff}(X)=D_{\Gm}(X)$ and $\Kir(X)=0$.
\end{exam}

\begin{exam}\label{ex:Kir A1} Let $X=\Ga=\AA^1$ with the tautological action of $\Aff$ by affine transformations. Then monodromic Fourier transform (see \S\ref{ss:FT}) homogeneous Fourier transform  
gives an equivalence $\FT_{\Ga}: D_{\Gm}(\Ga)\cong D_{\Gm}(\cGa)$ where $\cGa\cong \AA^{1}$ with the $\Gm$-action by inverse of dilation. Then $\FT_{\Ga}$ restricts to an equivalence $D_{\Aff}(\Ga)\cong D_{\Gm}(\{0\})\subset D_{\Gm}(\cGa)$. Passing to the quotients we get an equivalence 
\begin{equation}\label{FTGa}
\ov\FT_{\Ga}: \Kir(\Ga)\cong D_{\Gm}(\cGa-\{0\}).
\end{equation}
Restricting along $i_1: 1\incl \cGa-\{0\}$ gives an equivalence
\begin{equation}\label{res 1 shift}
i_1^*[-1]: D_{\Gm}(\cGa-\{0\})\isom D(\Vect).
\end{equation}
The composition of \eqref{FTGa} and \eqref{res 1 shift} gives an equivalence
\begin{equation}\label{FTGa eq}
\Phi_{\Ga}:\Kir(\Ga)\cong D(\Vect).
\end{equation}
Under this equivalence, $E\in \Vect$ corresponds to either $\d=i_{0*}E\in \Kir(\Ga)$ (where $i_{0}: \{0\}\incl \Ga$), or $j_{*}E[1](1)$ (where $j: \Ga-\{0\}\incl \Ga$), or $j_{!}E[1]$. Moreover, $\Phi_{\Ga}$ is monoidal with respect to the additive convolution on $\Kir(\Ga)$ and the tensor product on $D(\Vect)$. The diagram \eqref{recollement} is the usual recollement diagram for the closed subset $\{0\}\subset \cGa$ and its open complement. 

This example has an obvious generalization to the case where $X=Y\times\Ga$, with the tautological action of $\Aff$ just on the $\Ga$-factor. The conclusion is that we have an equivalence
\begin{equation*}
D(Y)\cong \Kir(Y\times \Ga)
\end{equation*}
given by sending $\cF\in D(Y)$ to either $\cF\bt \d$, or $\cF\bt j_{!}E[1]$, or $\cF\bt j_{*}E[1](1)$. This example will be further generalized in \S\ref{sss:Kir product A1} to allow nontrivial $\Gm$-action on the $Y$-factor.
\end{exam}

\sss{A product situation}\label{sss:Kir product A1}
Let $Y$ be a stack with $\Gm$-action. Let $X=Y\times \AA^{1}$ be equipped with the following action of $\Aff$: it acts on the $\AA^{1}$-factor in the usual way, and acts on $Y$ through its quotient $\Gm$ and the given $\Gm$-action on $Y$. If we want to emphasize the diagonal action of $\Gm$ on $Y\times \AA^1$, we write $\D\Gm$.


Note that restriction to $Y\times\{1\}$ induces an equivalence
\begin{equation*}
i_{1}^{*}[-1]: D_{\D\Gm}(Y\times(\AA^{1}-\{0\}))\isom D(Y).
\end{equation*}
Let $j: Y\times(\AA^{1}-\{0\})\incl Y\times \AA^{1}$ be the inclusion.  Define the functor 
\begin{equation*}
J: D(Y)\xr{(i_{1}^{*})^{-1}[1]}D_{\D\Gm}(Y\times(\AA^{1}-\{0\}))\xr{j_{*}}D_{\D\Gm}(Y\times \AA^{1})\to \Kir(X).
\end{equation*}

Let $D^{\mon}(Y)\subset D(Y)$ be the full subcategory given by the image of the forgetful map $D_{\Gm}(Y)\to D(Y)$. Let $D^{\mon}_{\D\Gm}(X)\subset D_{\D\Gm}(X)$ be the essential image of the forgetful map $D_{\Gm\times \Gm}(X)\to D_{\D\Gm}(X)$. Finally, let $\Kir^{\mon}(X)\subset \Kir(X)$ be the image of $D^{\mon}_{\D\Gm}(Y\times \AA^1)$ under the projection to $\Kir(X)$. Then the functor $J$ above restricts to a functor
\begin{equation*}
    J^{\mon}: D^{\mon}(Y)\to \Kir^{\mon}(X).
\end{equation*}

Consider the monodromic vanishing cycles functor with respect to the function $p_{2}: X=Y\times \AA^{1}\to \AA^{1}$ (see \S\ref{ss:mon van})
\begin{equation*}
\phi: D^{\mon}_{\D\Gm}(X)\to D^{\mon}(Y).
\end{equation*}
For $\cG\in D_{\Ga}(Y\times \AA^{1})$, $\phi(\cG)=0$, therefore $\phi$ factors through $\Kir^{\mon}(X)$:
\begin{equation*}
\Phi^{\mon}: \Kir^{\mon}(X)\to D^{\mon}(Y).
\end{equation*}


\begin{lemma}\label{l:Kir prod A1} In the above situation,  $J^{\mon}$ and $\Phi^{\mon}$ give inverse equivalences $D^{\mon}(Y)\cong\Kir^{\mon}(X)$. 
\end{lemma}
\begin{proof}

We view $\Gm\bs (Y\times\AA^{1})$ as a line bundle over $\Gm\bs Y$,  the monodromic Fourier transform reviewed in \S\ref{ss:FT} gives an equivalence
\begin{equation*}
\FT: D^{\Gm^{\dil}\lmon}_{\D\Gm}(Y\times \AA^{1})\isom D^{\Gm^{\dil}\lmon}_{\D\Gm}(Y\times \check\AA^{1})
\end{equation*}
Here $\check\AA^{1}$ denotes the dual line to $\AA^{1}$. The one-dimensional torus $\Gm^{\dil}$ acts on $\AA^{1}$ with weight $1$ and on $\check\AA^{1}$ with weight $-1$. 
Under the Fourier transform, $D^{\Gm^{\dil}\lmon}_{\Aff}(Y\times \AA^{1})$ (image of the pullback from $D_{\Gm}(Y)$) are sent to $D_{\Gm}(Y\times\{0\})$. Taking quotients, we get an equivalence
\begin{equation}\label{FT res}
\Kir^{\mon}(X)\xr{\ov\FT} D^{\Gm^{\dil}\lmon}_{\D\Gm}(Y\times (\check\AA^{1}-\{0\}))\xr{i_1^*[-1]} D^{\mon}(Y) 
\end{equation}
where the last equivalence is given up to a shift by restricting to $Y\times\{1\}$. By Lemma \ref{l:FT stalk} applied to the constant section $\xi: Y\to Y\times \check\AA^{1}$ given by $y\mt(y,1)$, we see that \eqref{FT res} is naturally isomorphic to $\Phi^{\mon}$. This shows that $\Phi^{\mon}$ is an equivalence. 

We give a natural isomorphism $\a: \Phi^{\mon}\c J^{\mon}\cong \id$.  For $\cF\in D^{\mon}(Y)$, we denote by $\wt\cF\in D_{\Gm}(Y\times (\AA^{1}-\{0\}))$ the unique object such that $i_{1}^{*}\wt\cF=\cF$. Then $J^{\mon}(\cF)=j_{!}\wt\cF[1]$. Let $a: Y\times\Gm\to Y\times\AA^1$ be given by $(y,s)\mt(sy,s)$, then $J^{\mon}\cF\cong j_!\wt\cF[1]=a_{!}(\cF\bt E)[1]$. Let $i_{t}:Y\times\{t\}\incl Y\times \AA^{1}$ be the inclusion for $t\in\{0,1\}$. In particular, $i_0^*(J^{\mon}\cF)=0$. By definition \eqref{star tri phi}, $\phi(J^{\mon}\cF)$ fits into the distinguished triangle $
\phi(J^{\mon}\cF)\to 0\to i_1^*(J^{\mon}\cF)\to$.
This gives a functorial isomorphism $$\phi(J^{\mon}\cF)\cong i_1^*(J^{\mon}\cF)[-1]=i_1^*a_{!}(\cF\bt E)=\cF,$$
which gives an isomorphism $\Phi^{\mon}\c J^{\mon}\cong\id$.
\end{proof}

\begin{remark}\label{r:Kir fun} Let $f:Y\to Y'$ be a $\Gm$-equivariant map. Then the following diagram is commutative
\begin{equation*}
\xymatrix{D^{\mon}(Y)\ar[d]^{f_{*}}\ar[r]^-{J^{\mon}} & \Kir^{\mon}(Y\times\AA^{1})\ar[d]^{(f\times\id)_{*}} \\
D^{\mon}(Y')\ar[r]^-{J^{\mon}} & \Kir^{\mon}(Y'\times\AA^{1}) 
}
\end{equation*}
The same is true for $f_{!}, f^{*}$ and $f^{!}$.
\end{remark}

\subsection{Comparison with Whittaker category}
Suppose we are in either of the following situations:
\begin{itemize}
\item $\chk=p>0$ and we work with $\Qlbar$-sheaves. We write $\cL_{\psi}$ for the Artin-Schreier sheaf on $\Ga$ associated to a fixed nontrivial character $\psi: \FF_{p}\to \Qlbar^{\times}$. 
\item $\chk=0$ and we work with $D$-modules. We write $\cL_{\psi}$ for the exponential $D$-module on $\Ga$. 
\end{itemize}
In either case, consider the category $D_{(\Ga,\cL_{\psi})}(X)$ of sheaves on $X$ equivariant under $\Ga$ with respect to the character sheaf $\cL_{\psi}$. The following proposition, due to Gaitsgory, was his motivation for introducing the Kirillov category.

\begin{prop}[Gaitsgory {\cite[1.6]{Ga-Wh}}]\label{p:K psi} Let $\om_{\psi}: D_{(\Ga,\cL_{\psi})}(X)\to D(X)$ be the forgetful functor. The image of the functor
\begin{equation*}
\xymatrix{D_{(\Ga,\psi)}(X)\ar[r]^-{\om_{\psi}}  & D(X) \ar[r]^-{\Av_{\Gm,*}} & D_{\Gm}(X)}
\end{equation*}
lands in $\Kir_{*}(X)$. The resulting functor
\begin{equation*}
\KK_{*}: D_{(\Ga,\psi)}(X)\to \Kir_{*}(X)
\end{equation*}
is an equivalence with inverse (left adjoint)
\begin{equation*}
\xymatrix{\KK^{*}: \Kir_{*}(X)\ar@{^{(}->}[r] & D_{\Gm}(X) \ar[r]^-{\om_{\Gm}} & D(X) \ar[r]^-{\Av_{\psi,!}} & D_{(\Ga,\psi)}(X).
}
\end{equation*}
\end{prop}
\begin{proof} For $\cF\in D_{(\Ga,\psi)}(X)$, $\Av_{\Gm,*}\om_{\psi}\cF$ is killed by $\Av_{\Ga,*}$ because $\Av_{\Ga,*}\Av_{\Gm,*}\om_{\psi}\cF=\Av_{\Aff,*}\om_{\psi}\cF=\Av_{\Gm,*}\Av_{\Ga,*}\om_{\psi}\cF=0$. Therefore $\Av_{\Gm,*}\om_{\psi}\cF\in \Kir_{*}(X)$.

Clearly $(\KK^{*},\KK_{*})$ are adjoint. Now we check the counit $c: \KK^{*}\c \KK_{*}\cong \Av_{\psi,!}\om_{\Gm}\Av_{\Gm,*}\om_{\psi}\to \id$ is an isomorphism. It suffices to check the same statement for a smooth $\Aff$-equivariant covering $X'\to X$. We may take the action map $a: X'=\Aff\times X\to X$ where the action of $\Aff$ on $X'$ is via the left translation on $\Aff$. Thus we reduce to the case $X=\AA^{1}\times Y$ with a diagonal action of $\Gm$. Performing Fourier transform in the $\AA^{1}$ direction (denote the dual affine line by $\check\AA^{1}$, and $\check\Gm=\check\AA^{1}\bs\{0\}$) we have a commutative diagram
\begin{equation*}
\xymatrix{D_{(\Ga,\psi)}(\AA^{1}\times Y)\ar[d]^{\FT}\ar[r]^-{\om_{\psi}}  & D(\AA^{1}\times Y) \ar[d]^{\FT}\ar[r]^-{\Av_{\Gm,*}} & D(\AA^{1}\times Y)\ar[d]^{\FT} \ar[r]^-{\Av_{\psi,!}\c\om_{\Gm}} & D_{(\Ga,\psi)}(\AA^{1}\times Y)\ar[d]^{\FT}\\
D(\{1\}\times Y) \ar[r]^-{i_{*}}& D(\check\AA^{1}\times Y)\ar[r]^-{\Av_{\Gm,*}} & D(\check\AA^{1}\times Y) \ar[r]^-{i^{*}} & D(\{1\}\times Y) }
\end{equation*}
Here $i:\{1\}\times Y\incl \check\AA^{1}\times Y$ is the closed embedding. All vertical functors are Fourier transform.  The composition of the top row is $\KK^{*}\c\KK_{*}$. We need to check that the composition of the bottom row is the identity. But for $\cG\in D(\{1\}\times Y)$, $\om_{\Gm}\Av_{\Gm,*}i_{*}\cG=u_{*}(E\bt \cG)$ where $u: \Gm\times Y\incl \check\AA^{1}\times Y$ is the open embedding $(t,y)\mapsto (t,ty)$. Restricting to $\{1\}\times Y$ which is in the image of $u$, we get $i^{*}u_{*}(E\bt \cG)\cong \cG$.

To show the unit map $\id \to \KK_{*}\c\KK^{*}$ is an isomorphism, the argument is similar. We reduce to the case $X=\AA^{1}\times Y$ with a diagonal $\Gm$-action. Under the Fourier transform, $\Kir_{*}(\AA^{1}\times Y)$ is identified with $D_{\Gm}((\check\AA^{1}-\{0\})\times Y)$, embedded into $D_{\Gm}(\check\AA^{1}\times Y)$ via $j_{*}$, where $j:(\check\AA^{1}-\{0\})\times Y\to \check\AA^{1}\times Y$ is the open embedding.  Now use the commutative diagram
\begin{equation*}
\xymatrix{\Kir_{*}(\AA^{1}\times Y) \ar@{^{(}->}[r]\ar[d]^{\FT} & D_{\Gm}(\AA^{1}\times Y)\ar[r]^-{\Av_{\psi,!}}\ar[d]^{\FT} & D_{(\Ga,\psi)}(\AA^{1}\times Y) \ar[r]^-{\Av_{\Gm,*}}\ar[d]^{\FT} & \Kir_{*}(\AA^{1}\times Y)\ar[d]^{\FT}\\
D_{\Gm}((\check\AA^{1}-\{0\})\times Y)\ar[r]^-{j_{*}} & D_{\Gm}(\check\AA^{1}\times Y)\ar[r]^-{i^{*}} & D(\{1\}\times Y) \ar[r]^-{\Av_{\Gm,*}} & D_{\Gm}((\check\AA^{1}-\{0\})\times Y)}
\end{equation*}
The compositum of the top row is $\KK_{*}\c\KK^{*}$, and the bottom row is clearly the identity.
\end{proof}

Similarly the functor
\begin{equation*}
\xymatrix{\KK_{!}: D_{(\Ga,\psi)}(X)\ar[r]^-{\om_{\psi}}  & D(X) \ar[r]^-{\Av_{\Gm,!}} & \Kir_{!}(X)
}
\end{equation*}
is an equivalence  with inverse  (right adjoint)
\begin{equation*}
\xymatrix{\KK^{!}: \Kir_{!}(X)\ar@{^{(}->}[r] & D_{\Gm}(X) \ar[r]^-{\om_{\Gm}} & D(X) \ar[r]^-{\Av_{\psi,*}} & D_{(\Ga,\psi)}(X).}
\end{equation*}

\subsection{Alternative construction}\label{ss:Kir alt}
We give an alternative construction of the Kirillov category that makes certain symmetries more transparent. The setup is the same as \S\ref{ss:Kir}. This construction is only used in the rest of paper to make sense of integral transforms on Kirillov categories given by kernel sheaves that are $\Ga$-equivariant and $\Gm$-monodromic (as opposed to $\Gm$-equivariant), see \S\ref{sss:int trans}. Such integral transforms will allow us to define Hecke symmetries on automorphic Kirillov categories (see \S\ref{ss:sph action} and \S\ref{sss:aff Hk action}). 

In this subsection alone, we work with cocomplete versions of sheaf categories (ind-constructible sheaves).

\sss{Monodromic category}
Let $D_{\Gm\lmon}(X)\subset D(X)$ be the full subcategory generated by the image of $D_{\Gm}(X)$ under the forgetful functor, i.e., unipotently monodromic complexes. By Verdier \cite[Section 5]{Ver}, every object $\cF\in D_{\Gm\lmon}(X)$ carries a unipotent action of the tame fundamental group $\pi^t_1(\Gm)$ (unipotent means a topological generator acts unipotently). Moreover, all morphisms in $D_{\Gm\lmon}(X)$ are $\pi^t_1(\Gm)$-equivariant. In other words, the identity functor of $D_{\Gm\lmon}(X)$ carries a unipotent action by $\pi^t_1(\Gm)$. 

Let $\frG_{E}=\Spf \wh{E[\pi^t_1(\Gm)]}$, where the completion is with respect to powers of the augmentation ideal of the group algebra $E[\pi^t_1(\Gm)]$. With a choice of a topological generator of $\pi^t_1(\Gm)$, $\frG_E$ is isomorphic to the formal multiplicative group over $E$. The category of ind-coherent sheaves $\IndCoh(\frG_E)$ is by construction equivalent to complexes of $E$-vector spaces with a unipotent action of $\pi^t_1(\Gm)$, under which convolution on $\IndCoh(\frG_E)$ (using the group structure of $\frG_E$) becomes tensor product of $\pi^t_1(\Gm)$-representations. We denote the convolution monoidal structure on $\IndCoh(\frG_E)$ by
\begin{equation*}
(\IndCoh(\frG_E), \star).
\end{equation*}
The monoidal unit of this monoidal structure is the skyscraper sheaf $\d_E$ supported at the reduced point of $\frG_E$.

On the other hand, $\IndCoh(\frG_E)$ carries another (symmetric) monoidal structure, given by the derived $!$-tensor product of $\cO_{\frG_E}$-modules. The monoidal unit in this case is the dualizing sheaf $\om_{\frG_E}$. We denote the resulting  monoidal category by 
\begin{equation*}
(\IndCoh(\frG_E), \sot).
\end{equation*}
We claim that there is canonical action of $(\IndCoh(\frG_E), \sot)$ on $D_{\Gm\lmon}(X)$. Indeed, let $\om_{\frG_E}$ act by identity, its endomorphism ring $\End_{\IndCoh(\frG_E)}(\om_{\frG_E})\cong \wh{E[\pi^t_1(\Gm)]}$ acts on the identity functor of $D_{\Gm\lmon}(X)$ by monodromy as we discussed before, and these structures determine the action of an arbitrary object in $\IndCoh(\frG_E)$ by writing it as a complex where terms are direct sums of $\om_{\frG_E}$.

\sss{Weak Kirillov category}
Let $D_{\Aff\lmon}(X)\subset D(X)$ be the full subcategory generated by the image of $D_{\Aff}(X)$ under the forgetful functor. It is also the full subcategory of $D_{\Gm\lmon}(X)$ consisting of $\Ga$-equivariant objects.

The {\em weak Kirillov} category of $X$ is defined as the quotient
\begin{equation*}
\Kir^w(X):=D_{\Gm\lmon}(X)/D_{\Aff\lmon}(X).
\end{equation*}
We have the adjunction
\begin{equation}\label{adj Kirw}
\xymatrix{\Kir(X) \ar@<1ex>[r]^-{\Forg} & \ar@<1ex>[l]^-{\Av_{*}} \Kir^w(X)}
\end{equation}
where $\Forg$ is the functor forgetting $\Gm$-equivariance, and $\Av_{*}$ is the $*$-averaging with respect to the $\Gm$-action.

\sss{Weak Kirillov category of $\Ga$}
Let $X=\Ga=\AA^1$ with the standard action of $\Aff$ by translation and dilation, the category $\Kir^w(\Ga)$ is the quotient of $D_{\Gm\lmon}(\Ga)$ by constant complexes. Note that $\Kir^w(\Ga)$ carries a monoidal structure under additive convolution on $\Ga$ (we use the $!$-version).

Monodromic Fourier transform as in \S\ref{ss:FT} gives an equivalence $\FT^w: D_{\Gm\lmon}(\Ga)\cong D_{\Gm\lmon}(\cGa)$, which induces an equivalence
\begin{equation}\label{FTGa mon}
\Phi^w_{\Ga}: \Kir^w(\Ga)\cong D_{\Gm\lmon}(\cGm).
\end{equation}
Taking stalk at $1\in \cGm$ and recording the monodromy action, we obtain an equivalence
\begin{equation*}
i^*_1: D_{\Gm\lmon}(\cGm)\cong \IndCoh(\frG_E).
\end{equation*}
Consider the composition
\begin{equation*}
\Psi: \Kir^w(\Ga)\xr{\Phi^w_{\Ga}} D_{\Gm\lmon}(\cGm)\xr{i^*_1[-1]}\IndCoh(\frG_E).
\end{equation*}
Note that we added a shift $[-1]$ in the second step, so that the skyscraper $\d\in \Kir^w(\Ga)$ maps to the skyscraper object $E$ supported at the reduced point of $\frG_E$. It is easy to see that $\Psi$ is carries a canonical structure of a monoidal equivalence
\begin{equation*}
\Psi: (\Kir^w(\Ga),\star)\cong (\IndCoh(\frG_E),\star).
\end{equation*}


\sss{} Back to the general setup of $X$ with $\Aff$-action. The action of $\Ga$ on $X$ gives an action of $D_{\Gm\lmon}(\Ga)$ (monoidal under additive convolution) on $D_{\Gm\lmon}(X)$. The action of the constant sheaf $\un{E}_{\Ga}$ on $D_{\Gm\lmon}(X)$ results in $\Ga$-equivariant sheaves on $X$, therefore the action passes to the quotients to give an action of $(\Kir^w(\Ga),\star)$ on $\Kir^w(X)$. Through the equivalence $\Psi$, we get an action of $(\IndCoh(\frG_E), \star)$ on $\Kir^w(X)$.

Next we spell out a compatibility structure between the actions of $(\IndCoh(\frG_E), \star)$ and $(\IndCoh(\frG_E), \sot)$ on $\Kir^w(X)$. For $\cK\in \IndCoh(\frG_E)=\Kir^{w}(\Ga)$ and $\cF\in \Kir^{w}(X)$, the monodromy action of $\g\in \pi_{1}(\Gm)$ on $\cK\star \cF$ is the same as the diagonal action on each factor: i.e., the following diagram is commutative (where $\mon_{\cK}$ denotes the monodromy action $\pi_{1}^{t}(\Gm)\to \Aut(\cK)$)
\begin{equation}\label{diag mono conv}
\xymatrix{\pi_{1}(\Gm)\ar@/_1pc/[rrr]_{\mon_{\cK\star\cF}} \ar[r]^-{\D} & \pi_{1}(\Gm)\times \pi_{1}(\Gm) \ar[r]^{\mon_{\cK}\times\mon_{\cF}}& \Aut(\cK)\times\Aut(\cF)\ar[r] & \Aut(\cK\star\cF)}
\end{equation}
and is functorial in $\cK$ and $\cF$.

This follows from the fact that the action map $\Ga\times X\to X$ is $\Gm$-equivariant. Now let 
\begin{equation*}
\a_{\star}: \IndCoh(\frG_{E})\ot \Kir^{w}(X)\cong \Kir^{w}(\Ga)\ot\Kir^{w}(X) \to \Kir^{w}(X)
\end{equation*}
be the action functor by convolution.  Reformulating the monodromy action in terms of the $(\IndCoh(\frG_E), \sot)$-action, we note that $\Kir^{w}(\Ga)\ot \Kir^{w}(X)$ carries an action of $(\IndCoh(\frG_E), \sot)\ot(\IndCoh(\frG_E), \sot)\cong (\IndCoh(\frG_E\times\frG_{E}), \sot)$. Let $m_{\frG}:\frG_{E}\times\frG_{E}\to \frG_{E}$ be the multiplication, then the diagram \eqref{diag mono conv} implies that $\a_{\star}$ is equivariant under $(\IndCoh(\frG_E), \sot)$, where it acts on the left side via the monoidal functor
\begin{equation}\label{m! frG}
m_{\frG}^{!}: (\IndCoh(\frG_E), \sot)\to (\IndCoh(\frG^{2}_E), \sot).
\end{equation}

\sss{Categorical statement}
Let $\cC$ be a cocomplete dg category over $E$ with an action of $(\IndCoh(\frG_E), \star)$. We may consider the $\frG_E$-equivariant category $\cC^{(\frG_E,\star)}$ equipped with adjoint functors
\begin{equation}\label{adj equiv}
\xymatrix{\cC^{(\frG_E,\star)} \ar@<-1ex>[r]_-{\Forg} & \ar@<-1ex>[l]_-{\Av_!} \cC}
\end{equation}
The category $\cC^{(\frG_{E},\star)}$ is defined as the totalization of the semi-cosimplicial diagram of dg categories
\begin{equation*}
\xymatrix{\cC \ar@<.5ex>[r]\ar@<-.5ex>[r] & \cC\ot \IndCoh(\frG_{E})\ar@<.5ex>[r]\ar[r]\ar@<-.5ex>[r] & \cC\ot \IndCoh(\frG_{E})\ot \IndCoh(\frG_{E})  \cdots}
\end{equation*}
where the functors use the co-action of $\IndCoh(\frG_{E})$ on $\cC$ (right adjoint to the action functor), and the $!$-pullback along the multiplication map of $\frG_{E}$. The monad $\Forg\c\Av_{!}$ acting on $\cC$ is given by the convolution $\om_{\frG_{E}}\star(-)$. Note that $\om_{\frG_{E}}$ is an algebra object in $(\IndCoh(\frG_E), \star)$. Therefore, $\cC^{\frG_{E}}$ can also be identified with the category of $\om_{\frG_{E}}\star(-)$-modules in $\cC$.

On the other hand, if $\cC$ is linear over $\frG_E$, i.e., it is equipped with an action of $(\IndCoh(\frG_E), \sot)$, we may consider its fiber $\cC_1$ over the reduced point $1\in \frG_E$, which fits into the adjunction
\begin{equation}\label{adj C1}
\xymatrix{\cC_1 \ar@<1ex>[r]^{i_*} & \ar@<1ex>[l]^-{i^!} \cC}
\end{equation}

We say that the $(\IndCoh(\frG_E), \star)$-action and the $(\IndCoh(\frG_E), \sot)$-action on $\cC$ are {\em compatible} if the convolution action functor
\begin{equation*}
\a_{\star}: \IndCoh(\frG_E)\ot \cC\to \cC
\end{equation*}
is equivariant under the $(\IndCoh(\frG_E), \sot)$-actions, where it acts on the left side via the monoidal functor \eqref{m! frG}.

\begin{exam}
Let $X$ be a formal scheme of that is smooth over $\frG_E$, with structure map $\pi: X\to\frG_E$. In this situation, letting $\cC=\IndCoh(X)$. Then $(\IndCoh(\frG_E), \sot)$ acts on $\cC$ where $\cK\in \IndCoh(\frG_E)$ acts by $\pi^!\cK\sot(-)$. Let $i: X_1\to X$ be the closed embedding of the fiber over $1\in \frG_E$, then $\cC_1$ is identified with $\IndCoh(X_1)$, and $i_*$ and $i^!$ have their usual meaning.

On the other hand, if $X$ admits an action of $\frG_E$, we have an action of $(\IndCoh(\frG_E), \star)$ on $\cC=\IndCoh(X)$ by convolution. If moreover the structure map $\pi$ is $\frG_E$-equivariant (which implies that $X\cong \frG_E\times X_1$), then the $(\IndCoh(\frG_E), \star)$ and $(\IndCoh(\frG_E), \sot)$-actions on $\cC$ are compatible.
\end{exam}

In our application, we shall take $\cC=\Kir^w(X)$ with the $(\IndCoh(\frG_E), \sot)$-action coming from $\Gm$-monodromy and the $(\IndCoh(\frG_E), \star)$-action by $\Kir^w(\Ga)$. Then we may identify $\cC_1$ with $\Kir(X)$, so that the adjunction \eqref{adj C1} becomes the adjunction \eqref{adj Kirw}.

\begin{lemma}\label{l:equiv vs fiber 1}
Let $\cC$ be a cocomplete dg category over $E$ admitting compatible actions of $(\IndCoh(\frG_E), \star)$ and $(\IndCoh(\frG_E), \sot)$. Then the composition 
\begin{equation*}
\Up: \cC_{1}\xr{i_{*}}\cC\xr{\Av_{!}}\cC^{\frG_{E}}
\end{equation*}
is fully faithful with right adjoint given by the composition
\begin{equation*}
\Up^{R}:\cC^{\frG_{E}}\xr{\Forg}\cC\xr{i^{!}}\cC_{1}.
\end{equation*}
If moreover $\cC$ is generated by the image of $\cC_{1}$ under $i_{*}$ under colimits, then $\Up$ and $\Up^{R}$ are inverse equivalences.

\end{lemma}
\begin{proof}
The adjunction $(\Up,\Up^{R})$ is the concatenation of the adjunctions \eqref{adj equiv} and \eqref{adj C1}. To show full-faithfulness of $\Up$, we need to show that the unit map $\id_{\cC_{1}}\to \Up^{R}\c\Up$ is an isomorphism. Let $\cF\in \cC_{1}$, we need to show that the canonical map
\begin{equation}\label{F to UpRUp}
\cF\to i^{!}(\Forg\c \Av_{!}(i_{*}\cF))
\end{equation}
is an isomorphism. Recall that the monad $\Forg\c \Av_{!}$ is given by the convolution with the algebra object $\om_{\frG_{E}}\in (\IndCoh(\frG_{E}), \star)$. The functor $i_{*}$ being conservative,  it suffices to show that
\begin{equation*}
i_{*}\cF\to i_{*}i^{!}(\om_{\frG_{E}}\star(i_{*}\cF))
\end{equation*}
is an isomorphism. Note that $i_{*}i^{!}(-)\cong \d\sot (-)$ where $\d\in (\IndCoh(\frG_{E}),\sot)$ is the skyscraper sheaf at $1$. We reduce to show that 
\begin{equation}\label{dsot}
i_{*}\cF\isom \d\sot(\om_{\frG_{E}}\star(i_{*}\cF)).
\end{equation}
The compatibility between the  actions of $(\IndCoh(\frG_E), \star)$ and $(\IndCoh(\frG_E), \sot)$ implies a canonical isomorphism
\begin{equation}\label{act m!}
\d\sot(\om_{\frG_{E}}\star(i_{*}\cF))\cong \a_{\star}(m^{!}_{\frG}\d\ssot(\om_{\frG_{E}}\bt i_{*}\cF)). 
\end{equation}
Here we use $\ssot$ to denote the action of $(\IndCoh(\frG_{E}^{2}),\sot)$ on $\IndCoh(\frG_E)\ot \cC$. Now $m^{!}_{\frG}\d\cong \om_{\D^{-}}$ where $\D^{-}\subset \frG_{E}\times\frG_{E}$ is the anti-diagonal. Consider $\IndCoh(\frG_E)\ot \cC$ as linear over $\frG_{E}^{2}$, the object $\om_{\frG_{E}}\bt i_{*}\cF$ is scheme-theoretically supported on $\frG_{E}\times\{1\}$, which intersects the anti-diagonal transversally at $(1,1)\in \frG_{E}^{2}$. Therefore we have a canonical isomorphism $\om_{\D^{-}}\ssot(\om_{\frG_{E}}\bt i_{*}\cF)\cong \d\bt i_{*}\cF\in \IndCoh(\frG_E)\ot\cC$, which implies
\begin{equation*}
\a_{\star}(m^{!}_{\frG}\d\ssot(\om_{\frG_{E}}\bt i_{*}\cF))\cong \a_{\star}(\om_{\D^{-}}\ssot(\om_{\frG_{E}}\bt i_{*}\cF))\cong \a_{\star}\d\bt i_{*}\cF\cong i_{*}\cF.
\end{equation*}
Combined with \eqref{act m!} we get the desired isomorphism \eqref{dsot}. We omit checking that the isomorphism thus obtained is the same as the one induced by the canonical map \eqref{F to UpRUp}. This proves that $\Up$ is fully faithful.

If $\cC$ is generated by $i_{*}\cC_{!}$, then $\Up$ is essentially surjective because $\cC^{\frG_{E}}$ is generated by the image of $\Av_{!}$ under colimits. Therefore  $\Up$ is an equivalence.
\end{proof}

Applying the above lemma to $\cC=\Kir^{w}(X)$, we obtain:
\begin{cor}\label{c:Kir alt} The functor
\begin{equation}\label{Kir to frG equiv}
\Kir(X)\xr{\Forg}\Kir^{w}(X)\xr{\Av_{!}}\Kir^{w}(X)^{(\frG_{E},\star)}
\end{equation}
is an equivalence with right adjoint
\begin{equation*}
\Kir^{w}(X)^{(\frG_{E},\star)}\xr{\Forg}\Kir^{w}(X)\xr{\Av_{*}}\Kir(X).
\end{equation*}
Moreover, the category $\Kir^{w}(X)^{(\frG_{E},\star)}$ is equivalent to the category of $\om_{\frG_{E}}\star(-)$-modules in $\Kir^{w}(X)$.
\end{cor}

\sss{Integral transform for Kirillov categories}\label{sss:int trans} By Corollary \ref{c:Kir alt}, we may henceforth use $\Kir^{w}(X)^{(\frG_{E},\star)}$ as an alternative definition of the Kirillov category $\Kir(X)$. The advantage of this alternative definition is that it allows us to see more symmetry on the Kirillov category. A typical situation is the following. Consider a correspondence
\begin{equation*}
\xymatrix{ & C\ar[dl]_{\oll{c}}\ar[dr]^{\orr{c}}\\
X & & Y}
\end{equation*}
in the category of stacks with $\Aff$-action. We construct a functor
\begin{equation}\label{int trans Kir}
D_{\Gm\lmon}(\Ga\bs C)\to \Fun(\Kir(X), \Kir(Y))
\end{equation}
as follows. For $\cK\in D_{\Gm\lmon}(C)$, we have the usual integral transform
\begin{eqnarray*}
\wt\phi_{\cK}:=\orr{c}_{!}(\oll{c}^{*}(-)\ot \cK): D_{\Gm\lmon}(X)\to D_{\Gm\lmon}(Y).
\end{eqnarray*}
When $\cK\in D_{\Gm\lmon}(\Ga\bs C)$, the functor $\wt\phi_{\cK}$ sends $D_{\Aff\lmon}(X)$ to $D_{\Aff\lmon}(Y)$, hence inducing a functor
\begin{equation*}
\phi^{w}_{\cK}: \Kir^{w}(X)\to \Kir^{w}(Y)
\end{equation*}
equivariant under the actions of $(\IndCoh(\frG_{E}),\star)$. Passing to $\frG_{E}$-equivariant objects, it induces a functor
\begin{equation*}
\phi_{\cK}: \Kir(X)\cong \Kir^{w}(X)^{(\frG_{E},\star)}\to \Kir^{w}(Y)^{(\frG_{E},\star)}\cong \Kir(Y).
\end{equation*}
Here we apply Corollary \ref{c:Kir alt} to $X$ and $Y$. The construction of \eqref{int trans Kir} is compatible with compositions of correspondences in the obvious sense.

\begin{remark}
    Let $\Kir^w(X)^\hs\subset \Kir^w(X)$ be the heart of the perverse t-structure. We also equip $\Kir^w(X)^{(\frG_E,\star)}$ with the t-structure pulled back from $\Forg: \Kir^w(X)^{(\frG_E,\star)}\to \Kir^w(X)$, whose heart is equivalent to $\om_{\frG_E}\star(-)$-modules in $\Kir^w(X)^\hs$ (note that $\om_{\frG_E}\star(-)$ is a t-exact endofunctor of $\Kir^w(X)$). One can check that the equivalence \eqref{Kir to frG equiv} is t-exact. We thus get an equivalence of abelian categories
    \begin{equation}\label{Kir frG equiv heart}
        \Kir(X)^\hs\isom \om_{\frG_E}\lmod(\Kir^w(X)^\hs).
    \end{equation}
    When $X$ is a scheme of finite type over $k$, the derived equivalence \eqref{Kir to frG equiv} can be obtained from the abelian category equivalence \eqref{Kir frG equiv heart} by passing to the derived categories of both sides and using Beilinson's theorem \cite[Theorem 1.3]{Bei}.
\end{remark}

\subsection{Relevant locus}\label{ss:rel}
Let $\mu_{\Ga}:T^{*}X\to (\Lie\Ga)^{*}$ be the moment map for the $\Ga$-action on $T^{*}X$. If $x$ is the coordinate on $\Ga$, then $dx$ gives a nonzero element in $(\Lie\Ga)^{*}$, which we denote by $1$.

 We define the Hamiltonian reduction of $T^{*}X$ under the $\Ga$-action with respect to the moment value $1$ as
\begin{equation*}
T^{*}X\qq_{1}\Ga:=\mu^{-1}_{\Ga}(1)/\Ga.
\end{equation*}

\begin{defn}\label{d:rel} The {\em relevant locus} for the $\Ga$-action on $X$ is the image of the projection $T^{*}X\qq_{1}\Ga=\mu^{-1}_{\Ga}(1)\to X$ (a constructible substack of $X$). By abuse of language we also call the image of the projection $T^{*}X\qq_{1}\Ga\to X/\Ga$ the relevant locus. A geometric point in the relevant locus of $X$ is called a {\em relevant point}; otherwise it is called {\em irrelevant}.
\end{defn}

\begin{lemma}\label{l:rel pt}
A geometric point $x\in X(K)$ (where $K$ is an algebraically closed field) is relevant if and only if the homomorphism
\begin{equation}\label{ax}
\Aut_{X/\Ga}(x)\to \GG_{a,K}
\end{equation}
is trivial, i.e., the stabilizer of $x$ under $\Ga$ is trivial. 

Equivalently, $x$ is irrelevant if and only if the homomorphism \eqref{ax} is nonzero, i.e., $\Ga$ fixes $x$.
\end{lemma}
\begin{proof}
By definition, $x$ is relevant if and only if the restriction of the moment map $\mu_{\Ga}$ to its cotangent fiber $\mu_{\Ga, x}: T^{*}_{x}X\to (\Lie\Ga)_{K}$ is nonzero. Dually, $x$ is relevant if and only if the infinitesimal action map $(\Lie\Ga)_{K}\to T_{x}X$ is nonzero (here $T_{x}X$ is the degree zero cohomology of the tangent complex of $X$ at $x$). This happens if and only if the stabilizer of $x$ under $\Ga$ has $0$ Lie algebra, which means the stabilizer is trivial. 
\end{proof}

\begin{lemma}\label{l:remove irr}
Let $i:Y\incl X$ be a locally closed $\Aff$-stable smooth substacks of $X$ such that all geometric points in $X\bs Y$ are irrelevant. Then
\begin{enumerate}
\item Both functors $i_{!}: \Kir(Y)\to \Kir(X)$ and $i_{*}: \Kir(Y)\to \Kir(X)$ are equivalences.
\item The natural transformation $i_{!}\to i_{*}: \Kir(Y)\to \Kir(X)$ is an isomorphism.
\end{enumerate}
\end{lemma}
\begin{proof} Decomposing $i$ as a closed embedding followed by an open embedding, we are reduced to treat the case where $i$ is either a closed embedding or an open embedding. The proof for both cases follows along the same lines, so in the following assume we are in either of the two cases.

Now let $X'=X-Y$, which consists of irrelevant points. We first show that $\Kir(X')=0$. Indeed, let $\cF\in D_{\Gm}(X')$.  Since all points in $U$ are irrelevant, by \eqref{l:rel pt}, all points in $U$ are fixed by $\Ga$. This implies that $\cF$ is $\Ga$-equivariant (which is a property since $\Ga$ is contractible), hence its $\Gm$-equivariant structure extends uniquely to an $\Aff$-equivariant structure, which implies $\cF$ becomes zero in $\Kir(X')=D_{\Gm}(X')/D_{\Aff}(X')$. 

We prove (2). For $\cF\in D_{\Gm}(Y)$, we have a distinguished triangle $i_{!}\cF\to i_{*}\cF\to k_{*}k^{*}i_{*}\cF$ where $k: X'\incl X$. Since $k^{*}i_{*}\cF\in \Kir(X')=0$, we see that  $i_{!}\cF\to i_{*}\cF$ becomes an isomorphism in $\Kir(X)$.

We prove (1). Since $i_{!}$ and $i_{*}$ are full embeddings, it suffices to show that they are essentially surjective. For $\cF\in D_{\Gm}(X)$, we have a distinguished triangle $i_{!}i^{!}\cF\to \cF\to k_{*}k^{*}\cF$. Since $k^{*}\cF\in \Kir(X')=0$, we conclude that $i_{!}i^{!}\cF\to \cF$ becomes an isomorphism in $\Kir(X)$. The same argument shows that $\cF\to i_{*}\cF$ becomes an isomorphism in $\Kir(X)$.

%
\end{proof}

\begin{cor}\label{c:Kir gen}
Suppose $\{Y_{\a}\}_{\a\in I}$ is a finite collection of disjoint, locally closed $\Aff$-stable smooth substacks of $X$. Assume there exists an $\Aff$-invariant stratification $X=\sqcup_{\a\in I} X_{\a}$ such that $Y_{\a}\subset X_{\a}$ and all geometric points in $X_{\a}\bs Y_{\a}$ are irrelevant (here by stratification we mean each $X_{\a}$ is locally closed and smooth, $I$ has a partial order such that $\ov X_{\a}=\cup_{\b\le \a}X_{\b}$ for all $\a\in I$). Then the images of $i_{\a!}: \Kir(Y_{\a})\to \Kir(X)$ generate $\Kir(X)$, and $\Kir(X)$ admits a semi-orthogonal decomposition into $\{\Kir(Y_{\a})\}_{\a\in I}$.
\end{cor}
\begin{proof}
Let $j_{\a}: X_{\a}\incl X$ be the inclusion. Since $\{X_{\a}\}_{\a\in I}$ form an $\Aff$-invariant stratification of $X$, the images of $j_{\a!}: \Kir(X_{\a})\to \Kir(X)$ generate $\Kir(X)$, and $\Kir(X)$ admits a semi-orthogonal decomposition into $\{\Kir(X_{\a})\}_{\a\in I}$. By Lemma \ref{l:remove irr}, the $!$-extension functors $\Kir(Y_{\a})\to \Kir(X_{\a})$ are equivalences, hence the same statements are true if $\Kir(X_{\a})$ is replaced with $\Kir(Y_{\a})$ for each $\a\in I$.
\end{proof}

\section{Recollections on microsheaves}\label{s:micro}

In this section we assume the base field $k=\CC$, and work with the analytic topology on various smooth complex varieties. We will review the basics of microsheaves, and recall a few properties which we will have use for later, as well as some general context and motivation. None of this material is original. 

\subsection{Microsheaves}\label{ss:micro}

\sss{Microsheaves on cotangent bundles}\label{sec:microsheaves} We follow \cite{KS} and \cite{NS} to recall the formalism of microsheaves.

Let $N$ be a smooth complex manifold, and let $\Sh(N)$ be the dg category of unbounded complexes of sheaves on $N$. For any subset $\L\subset T^{*}N$, let $\Sh_{\L}(N)$ be the full subcategory of sheaves with singular support contained in $\L$. For an open subset $\Om\subset T^{*}N$ (with respect to the usual topology on a complex manifold), let $\Sh(N;\Om)$ be the quotient $\Sh(N)/\Sh_{T^{*}N-\Om}(N)$. The assignment $\muSh^{pre}: \Om\mapsto \Sh(N;\Om)$ is a presheaf of stable $\infty$-categories. Let $\Om\mapsto \muSh(\Om)$  be the sheafification of $\muSh^{pre}$. It is a sheaf of stable $\infty$-categories.

For any open $\Om \subset T^{*}N$, an object $\cF\in \muSh(\Om)$ has a well-defined singular support $SS(\cF)$, which is a conical subset in $\Om$. For a subset $\L\subset T^{*}N$ and open $\Om\subset T^{*}N$, we let $\muSh_{\L}(\Om)$ denote the category of sections of $\muSh(\Om)$ with singular support in $\L\cap \Om$. 

For $\cF,\cG\in \Sh(N)$, in \cite{KS} a complex of sheaves $\muhom(\cF,\cG)$ on $T^{*}N$ is defined. We have 
\begin{equation*}
\Hom_{\muSh(\Om)}(\cF,\cG)=\cohog{0}{\Om, \muhom(\cF,\cG)}.
\end{equation*}
The support of $\muhom(\cF,\cG)$ is contained in $SS(\cF)\cap SS(\cG)$.

\sss{Extension to stacks} \label{sec:microsheavesonstacks} We refer to \cite[Section 4]{NS2} for more details on microsheaves on a stack. We will only need the case of a global quotient $\frX=[N/G]$ where $N$ is a smooth scheme over $\CC$. In this case, $T^{*}\frX=[\mu^{-1}(0)/G]$ where $\mu: T^{*}N\to \frg^{*}$ is the moment map for the $G$-action on $N$. For $\Om\subset \mu^{-1}(0)$ that is $G$-invariant, we can still form $\muSh^{pre}(\Om/G):=\Sh(\frX)/\Sh_{\mu^{-1}(0)-\Om}(\frX)$.

\begin{defn} \label{def:microsheavesonstacks} Define a sheaf of stable $\infty$-categories $\muSh$ over the stack $T^{*}\frX$ by sheafifying the above presheaf. We denote its global section category by $\muSh(T^{*}\frX)=\muSh(\mu^{-1}(0)/G)$.
\end{defn}

\sss{Microsheaves on exact symplectic manifolds} \label{sec:microsheavesonexact}
An exact sympectic manifold is a (smooth, real) manifold $X$ equipped with a real one-form $\lambda$, such that $\omega = d \lambda$ is a symplectic form. The image of $\lambda$ under the isomorphism of vector bundles $T^*X \to TX$ defined by $\omega$ is called the Liouville vector-field. The flow of this vector field is called the Liouville flow. 

\begin{exam}
Let $N$ be a manifold, and let $X = T^*N$ be equipped with the usual symplectic structure. The action of $\RR^{>0}$ dilating the cotangent fibers is the Liouville flow for the canonical Liouville form, which we denote $\lambda_N$. \end{exam}

Exact symplectic manifolds bear a close relation to contact manifolds. To an exact symplectic manifold $(X,\lambda)$, we associate the contact manifold $(Y, \alpha) = (X \times \RR, \lambda + dt)$.\footnote{One can also build a symplectic manifold from a contact manifold, though we will not need to do so here.}  

We will consider contact manifolds $(Y,\alpha)$ and symplectic manifolds $(X,\omega)$ equipped with an extra structure called {\em Maslov data}. This structure endows them with sheaves of categories locally modeled on the microsheaf categories defined by by Kashiwara and Schapira in \cite{KS}. We briefly sketch the definition of Maslov data below; a good reference for this material is \cite[Section 3]{CKNS}. 

\begin{remark}
We will use the same notation $\muSh$, and the same word microsheaves, for the sheaves of categories on a contact or symplectic manifold. As we recall below, there is a simple relationship between the two, and in practice this should not cause any confusion.
\end{remark}

Let $\mathcal{C}$ be a stable presentable symmetric monoidal category (eventually, this will be the category of coefficients for our microsheaves, and we may take it to be dg $\ZZ$-modules or dg $\CC$-vector spaces).

\begin{defn} Let $(X, \omega)$ be a real symplectic manifold. There is a canonical sequence of maps
      \begin{equation} \label{eq:maslovsequence} X \to \oB \oU \to \oB(\oU/\oO) \to \oB^2 \Picard(\mathcal{C}) \end{equation}
    where the first map classifies the stable tangent bundle to $X$ (viewed as a symplectic bundle), the second map is the delooping of the quotient map $\oU \to \oU/\oO$, and the last map is induced by a de-looping of the $J$-homomorphism. {\em Maslov data} for $(X,\lambda)$ is a choice of null-homotopy of the resulting map $X \to \oB^2 \Picard(\mathcal{C})$. 
\end{defn}

We will denote such a choice of nullhomotopy by the variable $\eta$ below. 

\begin{exam}
        When $\mathcal{C} = \ZZ-\operatorname{mod}$, $\oB^2\Picard(\mathcal{C})$ is given by $\oB^2(\ZZ) \oplus \oB^3(\ZZ/2\ZZ)$. A choice of Maslov data is canonically identified with a choice of trivialization of the complex line bundle $\det(T_\CC X)^{\otimes 2}$ and a class in $H^2(X,\ZZ/2\ZZ)$.

        Maslov data over other commutative rings $R$ can then be obtained by composition with $\Picard(\ZZ-\operatorname{mod}) \to \Picard(R-\operatorname{mod})$.
\end{exam}

Maslov data for contact manifolds is defined similarly to the symplectic case, where the contact distribution plays the role of the tangent bundle. In particular, Maslov data for an exact symplectic manifold $(X,\lambda)$ determines Maslov data for $(X \times \RR, \lambda + dt)$.

 One natural way of producing Maslov data is as follows.

\begin{defn}
A stable polarization of $X$ is a section of the Lagrangian Grassmanian bundle $\oLGr(TX \oplus \RR^{2n}) \to X$ for some $n \geq 0$. 
\end{defn}
We consider stable polarizations up to homotopy: they are classified by null-homotopies of the composition
\begin{equation} \label{eq:classifypol} X \to \oB\oU \to \oB(\oU/\oO) \end{equation}
of the first two maps in Equation \ref{eq:maslovsequence}. A choice of stable polarization defines Maslov data for $X$ by composition. 

\begin{exam} If $X = T^*N$ is a cotangent bundle, then it carries a natural polarization, and thus natural Maslov data, which we denote $\eta_N$. Similarly, if $Y = S^*N$ is the cosphere bundle of a manifold $N$, equipped with the canonical contact form $\alpha_N$, then it carries a natural (contact) polarization, and thus natural Maslov data, which we also denote $\eta_N$.\end{exam}

The papers \cite{S} and \cite[Theorem 1.1]{NS} endow any contact manifold $(Y, \alpha)$ equipped with a choice of Maslov data $(Y, \alpha, \eta)$ with a sheaf of categories $\muSh$, which we call microsheaves. 

Suppose $(Y,\alpha, \eta) = (S^*\RR^n, \alpha_N, \eta_N)$. Then the stalk of this sheaf at $p \in S^*\RR^n$ is the category 
\begin{equation} \label{eq:stalkcategory} \Sh(\RR^n)/\Sh_{S^* \setminus p}(\RR^n), \end{equation} where $\Sh(\RR^n)$ denotes sheaves on $\RR^n$ with coefficients in $\mathcal{C}$. This is essentially the category defined by Kashiwara and Schapira in \cite{KS}, except that one works here with coefficients in a presentable category, rather than with bounded complexes.

The same is true, non-canonically, for general $(Y, \alpha, \eta)$. Any homotopy of Maslov data $\eta \to \eta'$ induces an isomorphism on microsheaves. On the other hand, by the contact version of Darboux's theorem, $(Y,\alpha)$ is locally isomorphic, as a contact manifold, to the cosphere bundle $S^*\RR^n$ of a vector space. Thus the sheaf $\muSh$ is always locally isomorphic to \eqref{eq:stalkcategory}. 

 
\begin{defn} \label{def:microsheavesonexact}
Given an exact symplectic manifold equipped with a choice of Maslov data $(X, \lambda, \eta)$, we consider the associated contact manifold $(X \times \RR, \lambda + dt, \eta)$ equipped with the induced Maslov data. We then define $\muSh$ on $X$ as the pullback of $\muSh$ on $X \times \RR$ along the natural embedding $i_0 : X \to X \times 0 \subset X \times \RR$. 
\end{defn}

We call an object of $\muSh(X)$ a {\em microsheaf}. 

\begin{prop} The support of a microsheaf is a closed coisotropic subset of $X$ stable under the Liouville flow.
\end{prop}

Given any subset $S \subset X$, we can consider the subsheaf $\muSh_{S} \subset \muSh$ of objects supported on $S$. 

\begin{remark}
One could also pull back $\muSh$ on $X \times \RR$ along the graph of any function $f: X \to \RR$; this corresponds to shifting the Liouville form by $df$. 
\end{remark}

\begin{remark}
The dependence of $\muSh$ on the Maslov data $\eta$ is mild: any isotopy $\eta \to \eta'$ induces an equivalence on microsheaves. On the other hand, the stalks of $\muSh$ can depend quite non-trivially on the form $\lambda$. 

One of the main results of \cite{NS} is that in good situations, the global sections of $\muSh$ along the so-called {\em core} $\frc(X, \lambda) \subset X$ are invariant with respect to (sufficiently nice) exact isotopies of $\lambda$.
\end{remark}

\begin{prop}
Any isomorphism of $(X, \lambda, \eta)$ with an open subset $\Om \subset (T^*N, \lambda_N, \eta_N)$ identifies $\muSh$ with the sheaf defined in Section \ref{sec:microsheaves}.
\end{prop}





\sss{Comparing two constructions of microsheaves}
Suppose a Lie group $G$ acts on a manifold $N$, preserving an open subset $\Omega \subset T^*N$, and let $X = \Om /\!\!/_{\theta} G = \mu^{-1}(\theta) / G$ be a symplectic reduction at some moment value $\theta \in (\frg^{*})^G$. The cotangent polarization on $T^*N$ restricts to $\Om$. 

The pullback $\pi^*TX$ is the middle cohomology of the sequence $\frak{g} \to T\Om \to \frak{g}^*$ defined by the group action and moment map. As explained by Nadler and Shende in \cite{NS2}, any choice of $G$-equivariant splitting $T\Om \cong \pi^*TX \oplus \frak{g} \oplus \frak{g}^*$ identifies $G$-equivariant stable polarizations of $\Om$ and stable polarizations of the quotient. 

\begin{prop}\cite{NS2} \label{prop:nadlershendequotientisintrinsic}
Let $(W, \lambda)$ be a complex exact symplectic manifold, which occurs as an open subset of $T^*\frak{X} = T^*(X/G)$. Endow $W$ with the quotient polarization, with associated Maslov data $\eta$. Then Definition \ref{def:microsheavesonexact}) and Definition \ref{def:microsheavesonstacks} agree after restriction to $W$.\end{prop}

\sss{Objects from smooth lagrangians}

\begin{defn}
Let $(X,\lambda, \eta)$ be an exact symplectic manifold equipped with Maslov data. Let $L \subset X$ be a smooth Lagrangian submanifold, and let $\Omega \cong T^*L$ be a Weinstein neighborhood of $L$. {\em Secondary Maslov data} for $L$ is a choice of homotopy from the cotangent Maslov data on $T^*L$ to the restriction $\eta|_{\Omega}$.
\end{defn}

\begin{prop}\cite{NS}
Any Lagrangian submanifold $L \subset X$ such that $\lambda|_L = 0$, equipped with secondary Maslov data, defines a fully faithful functor $\Loc(L) \to \muSh(X)$.
\end{prop}
The condition $\lambda|_L = 0$ can be relaxed to $\lambda|_L = df$ under certain conditions : see Theorem \ref{thm:exactlagsgiveobjects}.

\subsection{Microsheaves and Fukaya categories}
In this section, we recall the relationship between microsheaves and the Fukaya category. This material is motivational, and will not be used anywhere else in this paper. 

An exact symplectic manifold $X$ is called Liouville if the Liouville flow is complete, and $X$ is convex, in the sense that it admits an exhaustion $\{ X^k \subset X | k \in \NN \}$ by compact submanifolds with boundaries, such that the Liouville vector field is outward pointing along each boundary. We call such a pair $(X^k, \lambda)$ a {\em Liouville domain}.

\begin{defn} \label{def:core}
    The core $\frak{c}(X,\lambda) \subset X$ (sometimes called the skeleton) of a Liouville manifold $(X, \lambda)$ is the maximal flow-invariant subset of $X^k$ (for $k \gg 0)$ which is disjoint from $\partial X^k$.
\end{defn}
This does not depend on the choice of exhaustion. The core is a compact subset of $X$.

\begin{prop}
Any compactly supported object in $\muSh(X)$ is supported in $\frak{c}(X,\lambda)$.
\end{prop}

\begin{defn} \label{def:globalmush}
Let $\frak{Sh}(X,\lambda) = \muSh_{\frak{c}(X,\lambda)}(X,\lambda)$.  
\end{defn}

A Liouville form $\lambda$ is {\em sufficiently Weinstein} if $\frak{c}(X,\lambda)$ is sufficiently isotropic \cite[Definition 9.10]{NS2}. This holds, for example, if it is Whitney-stratified by isotropic submanifolds. There is a corresponding notion of sufficiently Weinstein homotopies between sufficiently Weinstein Liouville manifolds.

It is proven in \cite{NS2} that a sufficiently Weinstein homotopy $\lambda \to \lambda'$ induces an equivalence of $\frak{Sh}(X,\lambda)$ with $\frak{Sh}(X,\lambda')$. One consequence of this is the following.
\begin{theorem} \cite[Theorem 1.2]{NS2} \label{thm:exactlagsgiveobjects}
Let $L$ be a compact exact Lagrangian submanifold of the sufficiently Weinstein manifold $(X,\lambda, \eta)$, equipped with secondary Maslov data. There is a fully faithful functor $\oLoc(L) \to \mathfrak{Sh}(X,\lambda)$.
\end{theorem}
\begin{theorem}\cite[Theorem 1.4]{ganatra2024microlocal} \label{thm:Fukayaismicro}
Let $(X,\lambda, \eta)$ be a Liouville manifold admitting homological cocores.\footnote{We refer to \cite{ganatra2024microlocal} for a definition of homological cocores.}
 Let $\frak{Sh}(X,\lambda)^c$ denote the category of compact objects. Then there is a canonical equivalence of dg categories 
\begin{equation*} 
\frak{Sh}(X,\lambda)^c \cong \operatorname{Perf} \mathcal{W}(X, \lambda)^{\operatorname{op}}
\end{equation*}
where $\operatorname{Perf} \mathcal{W}(X, \lambda)$ is the (opposite) category of the idempotent-completed pre-triangulated closure of the wrapped Fukaya category of $X$.
\end{theorem}

\subsection{Microlocalization of the Kirillov category}\label{ss:micro Kir}

We are back in the situation of \S\ref{ss:Kir}, where we further assume that $X=[N/G]$ is a quotient of a smooth scheme $N$ over $\mathbb{C}$ by an action of a group $G$. We further suppose that the $G$-action extends to an action of a group $\wt G$ that fits into an exact sequence $1\to G\to \wt G\to \Aff\to 1$.

Let $\frAff$ be the Lie algebra of $\Aff$.  Let $\frg_{m}$ and $\frg_{a}$ be the Lie algebras of $\Gm$ and $\Ga$ respectively. We have a decomposition $\frAff^{*}=\frg_{a}^{*}\op \frg_{m}^{*}\cong \CC\op \CC$ (as vector spaces). 

Let $\mu_{\Ga}:T^{*}X\to \frg_{a}^{*}, \mu_{\Gm}:T^{*}X\to \frg_{m}^{*}$ and $\mu_{\Aff}: T^{*}X\to \frAff^{*}$ be the moment maps for the $\Ga, \Gm$ and $\Aff$-actions. 

Consider the microlocalization functor
\begin{equation}\label{Gm mloc}
D_{\Gm}(X)\to \muSh(T^{*}(X/\Gm))=\muSh(\mu^{-1}_{\Gm}(0)/\Gm).
\end{equation}
For $\cF\in D_{\Aff}(X)$, we have $SS(\cF)\subset \mu^{-1}_{\Ga}(0)$. Therefore, restricting to the complement of the zero fiber of $\mu_{\Ga}$ induces a functor
\begin{equation}\label{micro nonzero Ga fiber}
\Kir(X)=D_{\Gm}(X)/D_{\Aff}(X)\incl \muSh(\mu_{\Aff}^{-1}((\frg_{a}^{*}\bs \{0\})\times \{0\})/\Gm).
\end{equation}
Note that the coadjoint action of $\Gm$ on $(\Lie\Ga)^{*}$ is by scaling, we have an isomorphism
\begin{equation*}
\mu^{-1}_{\Aff}((\frg_{a}^{*}\bs \{0\})\times \{0\})/\Gm\cong \mu^{-1}_{\Aff}(1,0).
\end{equation*}
As an open substack of the Hamiltonian reduction $T^{*}X\qq \Gm$, this also equips $\mu^{-1}_{\Aff}(1,0)$ with an exact symplectic structure. Combining with \eqref{micro nonzero Ga fiber}, we get a functor
\begin{equation}\label{micro mu10}
\Kir(X)\to \muSh(\mu^{-1}_{\Aff}(1,0)).
\end{equation}

\begin{lemma}
The inclusion $\mu^{-1}_{\Aff}(1,0)\incl \mu^{-1}_{\Ga}(1)\subset T^{*}X$ induces an isomorphism
\begin{equation}\label{shifted red mu10}
\mu^{-1}_{\Aff}(1,0)\cong T^{*}X\qq_{1}\Ga.
\end{equation}
\end{lemma}
\begin{proof}
We have $\mu^{-1}_{\Ga}(1)=\mu^{-1}(1\times\frg_{m}^{*})\subset T^{*}X$. The coadjoint action of $\Ga$ on $1\times \frg_{m}^{*}$ is by translation on the $\frg_{m}^{*}$-factor. Therefore we may identify $\mu^{-1}_{\Aff}(1\times\frg_{m}^{*})/\Ga$ with $\mu^{-1}_{\Aff}(1,0)$.
\end{proof}

Combining the isomorphism \eqref{shifted red mu10} with \eqref{micro nonzero Ga fiber}, we get a microlocalization functor
\begin{equation*}
M: \Kir(X)\incl \muSh(T^{*}X\qq_{1}\Ga).
\end{equation*}

For $\cF\in \Kir(X)$ that is the image of $\wt\cF\in D_{\Gm}(X)$, $SS(\wt\cF)\cap \mu^{-1}_{\Aff}(1,0)$ is independent of the choice of $\wt\cF$ and depends only on $\cF$;  we denote it by $SS(\cF)$, which is a conic subset of $T^{*}X\qq_{1}\Ga$.

For $\cF,\cG\in \Kir(X)$ that come from objects $\wt\cF,\wt\cG\in D_{\Gm}(X)$, we have the $\Gm$-equivariant sheaf $\muhom(\wt \cF, \wt\cG)$ on $\mu^{-1}_{\Gm}(0)$. Restricting $\muhom(\wt \cF, \wt\cG)$ to $\mu^{-1}_{\Aff}(1,0)$ we get a sheaf on $\mu^{-1}_{\Aff}(1,0)\cong T^{*}X\qq_{1}\Ga$ which is independent of the liftings $\wt\cF, \wt\cG$.  We denote this restriction by $\muhom(\cF, \cG)$, understood as a sheaf on $T^{*}X\qq_{1}\Ga$. Then we have
\begin{equation*}
\Hom(M(\cF), M(\cG))=\cohog{0}{T^{*}X\qq_{1}\Ga, \muhom(\cF,\cG)}.
\end{equation*}

\begin{lemma}\label{l:conserv}
The functor $M$ is conservative.
\end{lemma}
\begin{proof} We need to show that if $\wt\cF\in D_{\Gm}(X)$ and $M(\wt\cF)=0$ then $\wt\cF$ is $\Ga$-equivariant (which is a property rather than a structure, since $\Ga$ is contractible). Equivalently, we must show $\upH^{i}\wt\cF$ is constant along each $\Ga$-orbit for all $i$. The microstalk functor at a point $p\in T^{*}X$ is a factors through $\muSh(\Om)$ for any open $\Om$ containing $p$. Since $M(\wt\cF)=0$, the microstalks of $\wt\cF$ vanish along $\mu^{-1}_{\Aff}(1,0)$. Thus the singular support $SS(\wt\cF)$ is a $\Gm$-invariant subset of $T^{*}X$ that does not intersect $\mu^{-1}_{\Aff}(1,0)$. By the $\Gm$-invariance, we conclude that $SS(\wt\cF)\subset \mu^{-1}_{\Ga}(0)$. 

We need to show that $\upH^{i}\wt\cF$ is constant along each $\Ga$-orbit for all $i$.
Let $a: \Ga\times X\to X$ be the action map and $p_{X}: \Ga\times X\to X$ be the projection. The map $da: (T^{*}X)\times_{X,a}(\Ga\times X)\to T^{*}(\Ga\times X)=T^{*}\Ga\times T^{*}X$ takes the form $(t,x,\xi)\mapsto (t,\mu_{\Ga}(x,\xi), x,\xi)$ where $t\in \Ga, x\in X$ and $\xi\in T_{x}^{*}X$. The condition $SS(\wt\cF)\subset \mu^{-1}_{\Ga}(0)$ implies that $SS(a^{*}\cF)$, which lies in the image of $da(SS(\wt\cF)\times_{X,a}(\Ga\times X))$,  is contained in $\Ga\times\{0\}\times T^{*}X$, which is the image of $dp_{X}: T^{*}X\times_{X,p_{X}}(\Ga\times X)\to T^{*}(\Ga\times X)$ . Now apply \cite[Proposition 5.4.5]{KS} to the projection $p_{X}$. We conclude that $\upH^{i}\wt\cF$ is locally constant along fibers of $p_{X}$, which implies that $\upH^{i}\wt\cF$ is  constant along $\Ga$-orbits.
\end{proof}

\section{Affine Springer fiber and Hitchin moduli space}
Let $k$ be an algebraically closed field.
\subsection{Homogeneous affine Springer fiber with slope one}
Let $G$ be a connected reductive group over $k$ with Lie algebra $\frg$ and flag variety $\cB$. Let $\psi_{1}\in\frg$ be a regular semisimple element. Then $T=C_{G}(\psi_{1})$ is a maximal torus. Let $\Phi=\Phi(G,T)$ be the set of roots of $G$ with respect to $T$, and $W$ the Weyl group. For $\a\in \Phi$, let $\frg_{\a}\subset \frg$ be the corresponding root space.

Let $G\lr{t}$ be the loop group of $G$ with formal parameter $t$. Let $\bG=G\tl{t}$ be the formal arc group of $G$. In other words, $G\lr{t}(R)=G(R\lr{t})$ and $\bG(R)=G(R\tl{t})$ for $k$-algebras $R$.  Let $\Gr=G\lr{t}/\bG$ be the affine Grassmannian of $G$.  

Let $B\subset G$ be a Borel subgroup containing $T$, and $\bI\subset \bG$ the corresponding Iwahori subgroup. Then $\Fl=G\lr{t}/\bI$ is the affine flag variety of $G$. Let $\Phi^{+}\subset \Phi$ be the subset of positive roots defined by $B$. 

Write $\psi=t\psi_{1}\in \frg\lr{t}$. This element defines its affine Springer fiber $\Fl_{\psi}$ in the affine flag variety $\Fl$ and $\Gr_{\psi}$ in the affine Grassmannian $\Gr$ as in \cite{KL}:
\begin{eqnarray*}
\Fl_{\psi}=\{g\in \Fl=G\lr{t}/\bI_{0}|\Ad(g^{-1})(\psi)\in \Lie\bI_{0}\};\\
\Gr_{\psi}=\{g\in \Gr=G\lr{t}/\bG|\Ad(g^{-1})(\psi)\in \frg\tl{t}\}.
\end{eqnarray*}

Using the terminology of \cite[\S3.1]{OY}, the element $\psi$ is a homogeneous element of slope $1$ in the loop Lie algebra $\frg\lr{t}$.

\sss{Geometric properties}
By the dimension formula of Kazhdan-Lusztig-Bezrukavnikov \cite{KL} and \cite{B-ASF}, we have
\begin{equation*}
\dim \Fl_{\psi}=\dim \Gr_{\psi}=|\Phi^{+}|=\dim \cB.
\end{equation*}
Moreover, $\Fl_{\psi}$ and $\Gr_{\psi}$ are known to be equidimensional.  Now $\cB=G/B\subset \Fl$ is contained in $\Fl_{\psi}$, therefore it is an irreducible component of $\Fl_{\psi}$.

\sss{Symmetry}\label{sss:lattice on ASF}
Since $\psi$ has centralizer $T\lr{t}$ in $G\lr{t}$,  $T\lr{t}$ acts on $\Fl_{\psi}$ and $\Gr_{\psi}$ by left translation. In particular, we get a natural action of the lattice $\xcoch(T)$ on $\Fl_{\psi}$ and $\Gr_{\psi}$: $\l\in \xcoch(T)$ acts by left translation by $t^{\l}$.

\sss{Hessenberg varieties}\label{Hess_var}

Let $\Grot$ be the one-dimension torus acting on $G\lr{t}, \Fl, \Gr$ and $\frg^{*}\lr{t}$ by scaling $t$. Then $\Grot$ preserves $\Fl_{\psi}$ and $\Gr_{\psi}$. We describe the $\Grot$-fixed points following \cite[\S5.4]{OY}. 

The fixed points of $\Grot$ on $\Fl$ are of the form $Gt^{\l}\bI/\bI$ for $\l\in \xcoch(T)$. Let
\begin{equation*}
\cH_{\psi}(\l)=\Fl_{\psi}\cap (Gt^{\l}\bI/\bI).
\end{equation*}
Note that $Gt^{\l}\bI/\bI\cong G/B_{\l}$, and $B_{\l}=G\cap \Ad(t^{\l})\bI$ is a Borel subgroup of $G$ containing $T$, whose roots under $T$ are
\begin{equation*}
\{\a\in \Phi^{+}|\j{\l,\a}\le 0\}\cup\{\a\in \Phi^{-}|\j{\l,\a}\le -1\}.
\end{equation*}
In particular, $Gt^{\l}\bI/\bI$ is isomorphic to the flag variety of $G$. Let 
\begin{equation*}
V_{\l}^{+}=\frt\op\left(\bigoplus_{\a\in \Phi^{+}, \j{\l,\a}\le 1}\frg_{\a}\right)\op \left(\bigoplus_{\a\in \Phi^{-}, \j{\l,\a}\le 0}\frg_{\a}\right).
\end{equation*}
This is an $\Ad(B_{\l})$-submodule of $\frg$ containing $\frb_{\l}=\Lie B_{\l}$. Direct calculation shows that $\cH_{\psi}(\l)\subset G/B_{\l}$ is the closed subscheme defined by
\begin{equation*}
\cH_{\psi}(\l)=\{gB_{\l}\in G/B_{\l}|\Ad(g^{-1})\psi_{1}\in V_{\l}^{+}\}.
\end{equation*}
These are examples of Hessenberg varieties.  In particular, $\cH_{\psi}(0)=\cB$ is an irreducible component of $\Fl_{\psi}$. 

Under the $\Grot$-action on the affine Springer fiber $\Fl_{\psi}$, the attracting locus to $\cH_{\psi}(\l)$ is $\Fl_{\psi}\cap (\bG t^{\l}\bI/\bI)$, which is an affine space bundle over $\cH_{\psi}(\l)$.

We give some examples of Hessenberg varieties that appear as fixed points of $\Fl_{\psi}$.

\begin{exam} If $\l\in \xcoch(T)$ satisfies
\begin{equation*}
\mbox{$\l$ is regular and $\j{\l,\a}\ne 1$ for all $\a\in \Phi^{+}$,}
\end{equation*}
then $V_{\l}^{+}=\frb_{\l}$, and $\cH_{\psi}(\l)$ is canonically isomorphic to the Springer fiber of the regular semisimple element $\psi_{1}\in \frg$. We conclude that in this case, $\cH_{\psi}(\l)$ consists of points $wt^{\l}\bI/\bI$, where $w\in W$.
\end{exam}

\begin{exam} Consider the case $G=\GL_{n}$ and $T$ the diagonal torus, in which case we can write $\l=(\l_{1},\cdots,\l_{n})\in \ZZ^{n}$. Let $\mu_{1}<\cdots<\mu_{r}$ be the different values of $\l_{i}$, with $\mu_{1}$ appearing $n_{1}$ times, $\mu_{2}$ appearing $n_{2}$ times, etc.

Then under a suitable basis of $k^{n}$, $B_{\l}$ is the group of upper triangular matrices in $\GL_{n}$, we may arrange that $V_{\l}^{+}$ consists of block matrices with blocks of sizes $n_{1},n_{2},\cdots, n_{r}$ 
\begin{equation*}
\left(\begin{array}{cccc} * & * & * & *\\ ? & * & * & * \\ 0 & ? & * & * \\ 0 & 0 & ? & *\end{array}\right)
\end{equation*}
Here $*$ denotes an arbitrary block, while $?$ (at the $(i+1, i)$-block) is arbitrary only when $\mu_{i}+1=\mu_{i+1}$; otherwise $?$ is the zero block.
 \end{exam}

\begin{exam} Consider the case $G=\GL_{3}$ and $\psi_{1}$ is a regular diagonal element.  Let $\l=(-1,0,1)$. In this case, $V_{\l}^{+}$ takes the form $\left(\begin{array}{ccc} * & * & * \\  * & * & * \\ 0  & * & *\end{array}\right)$. We have
\begin{equation*}
\cH_{\psi}(\l)\cong \{V_{1}\subset V_{2}\subset V=k^{3}|\psi_{1}(V_{1})\subset V_{2}\}.
\end{equation*}
By convention, $\dim V_{i}=i$. Forgetting $V_{1}$ gives a projection
\begin{equation*}
\cH_{\psi}(\l)\to \PP^{2}=\{V_{2}\subset V\}.
\end{equation*}
This realizes $\cH_{\psi}(\l)$ as the blow-up of $\PP^{2}$ at three points $[1:0:0], [0:1:0]$ and $[0:0:1]$, which correspond exactly to the three planes in $V$ that are stable under $\psi_{1}$.  
\end{exam}

\subsection{Hitchin moduli space}
We review a special case of the construction from \cite[\S2.5-2.8]{BBAMY}, where we define a symplectic variety $\cM_{\psi}$ containing the affine Springer fiber $\Fl_{\psi}$ as a conic Lagrangian.

\sss{Global setup}
Let $X=\PP^{1}$ with affine coordinate $t$. The two points $0$ and $\infty$ will play essential roles in the following discussion. To avoid confusion between the point $0\in X$ and the number zero, we use $\un 0$ to denote the point in $X$ when necessary.

We again denote by $\Grot$ the one-dimensional torus acting on $X$ by scaling the coordinate $t$. 

The loop group $G\lr{t}$, arc group $\bG$ and Iwahori $\bI$ are now attached to the formal disk of  $X$ at $0$, with uniformizer $t$.  To emphasize that, we sometimes denote $\bG$ and $\bI$ by $\bG_{0}$ and $\bI_{0}$.

\sss{Moduli of bundles}
We use $\t=t^{-1}$ as the uniformizer of the local ring of $X$ at $\infty$. Let $\bG_{\infty}=G\tl{\t}$ be the formal arc group of $G$ at $\infty\in X$. Let $\bG^{1}_{\infty}=\ker(\ev_{\infty}: \bG_{\infty}\to G)$ be the first congruence subgroup at $\infty$. Let $\Bun_{G}(\bG^{1}_{\infty}, \bI_{0})$  be the moduli stack of $G$-bundles on $\PP^{1}$ with $\bG^{1}_{\infty}$ and $\bI_{0}$ level structures $\infty$ and $0$ respectively. Then $\Bun_{G}(\bG^{1}_{\infty}, \bI_{0})$ classifies $(\cE, \t_{\infty}, \cE_{0,B})$ where $\cE$ is a $G$-bundle over $\PP^{1}$, $\t_{\infty}$ a trivialization of the fiber $\cE_{\infty}$, and $\cE_{0,B}$ a $B$-reduction of the fiber $\cE_{0}$. 

\sss{Moduli of Higgs bundles}
Let $\cM(\bG^{1}_{\infty}, \bI_{0})$ be the classical cotangent bundle of $\Bun_{G}(\bG^{1}_{\infty}, \bI_{0})$. It classifies $(\cE,\t_{\infty}, \cE_{0,B}, \ph)$ where $(\cE, \t_{\infty}, \cE_{0,B})\in \Bun_{G}(\bG^{1}_{\infty}, \bI_{0})$ and $\ph$ is a section of the twisted coadjoint bundle $\Ad^*(\cE)\ot \om_{X}(\infty+\un 0)$ such that $\Res_{0}\ph\in \cE_{0,B}(\frb^{\bot})$, the vector space associated to the $B$-bundle $\cE_{0,B}$ at $\un 0$ and the $B$-module $\frb^{\bot}\subset \frg^{*}$. Under the Killing form of $\frg$, $\frb^{\bot}$ is identified with $\frn$. Since $\om_{X}(\infty+\un 0)$ has a canonical trivialization by $dt/t$, we may identify $\ph$ with a section of $\Ad^*(\cE)$ whose value at $\un 0$ lies in $\cE_{0,B}(\frn)$.

We may view $\psi$ as an additive character of $\bG^{1}_{\infty}$:
\begin{equation}\label{psi as char}
\psi: \bG^{1}_{\infty}\xr{\mod t^{-2}}t^{-1}\frg\xr{t}\frg\xr{\psi_{1}}\Ga.
\end{equation}
Let $\bG^{\psi}_{\infty}\subset \bG^{1}_{\infty}$ be the kernel of the above character. Let $\cM(\bG^{\psi}_{\infty}, \bI_{0})=T^{*}\Bun_{G}(\bG^{\psi}_{\infty}, \bI_{0})$. Then $\cM(\bG^{\psi}_{\infty}, \bI_{0})$ carries an action of $\Aff=\Ga\rtimes\Grot$ with $\Ga$ acting as $\bG^{1}_{\infty}/\bG^{\psi}_{\infty}$ and $\Grot$ by loop rotation that acts on the tangent space at $\infty\in\PP^{1}$ by scaling. Let
\begin{equation}\label{def Mpsi}
\cM_{\psi}:=\cM(\bG^{\psi}_{\infty}, \bI_{0})\qq_{1}\Ga
\end{equation}
be the shifted Hamiltonian reduction of $\cM(\bG^{\psi}_{\infty}, \bI_{0})$.   Then $\cM_{\psi}$ classifies $(\cE,\t_{\infty}, \cE_{0,B}, \ph)$ where $\cE,\t_{\infty}, \cE_{0,B}\in\Bun_{G}(\bG^{1}_{\infty}, \bI_{0})$, $\ph\in \Ad^*(\cE)\ot \om_{X}(2\cdot\infty+\un 0)$ satisfies
\begin{itemize}
\item Near $\infty$,  $\ph$ is allowed to have a double pole, so $\ph\in \Ad^*(\cE)dt$. Then we require $\ph\in\psi_{1} dt + \Ad^*(\cE)dt/t$. To make sense of ``the leading term is equal to $\psi_{1}$'', we use the trivialization $\t_{\infty}$ of $\cE_{\infty}$.
\item $\Res_{0}\ph\in \cE_{0,B}(\frn)$, i.e., near $0$ we have $\ph\in \cE_{0,B}(\frn)dt/t+\Ad^*(\cE)dt$.
\end{itemize}
Equivalently, we may consider the next congruence subgroup $\bG_{\infty}^{2}=\ker(\bG^{1}_{\infty}\to \t\frg)$, and $\cM(\bG_{\infty}^{2},\bI_{0})=T^{*}\Bun_{G}(\bG_{\infty}^{2},\bI_{0})$. Then 
\begin{equation*}
\cM_{\psi}\cong \cM(\bG_{\infty}^{2},\bI_{0})\qq_{\psi}\frg.
\end{equation*}
Here we use the natural action of $\frg=\bG^{1}_{\infty}/\bG^{2}_{\infty}$ on $\Bun_{G}(\bG^{2}_{\infty},\bI_{0})$.

Since $T=C_{G}(\psi_{1})$, it normalizes $\bG^{\psi}_{\infty}$. Therefore $\Bun_{G}(\bG^{\psi}_{\infty}, \bI_{0})$ carries an action of $T$ by changing the level structure at $\infty$.

\subsection{The Hitchin map}
Let $f_{1},\cdots, f_{n}$ be a chosen set of homogeneous free generators of the polynomial ring $\CC[\frg]^{G}$. Let $d_{i}=\deg f_{i}$. Let 
\begin{equation*}
\cA=\prod_{i=1}^{n}\cohog{0}{X, \om_{X}(2\cdot \infty+0)^{\ot d_{i}}\ot \cO_{X}(-\un 0)}\cong \prod_{i=1}^{n}\cohog{0}{X,\cO_{X}(d_{i}\cdot\infty-\un 0)}\cong \prod_{i=1}^{n}\cohog{0}{X, \cO(d_{i}-1)}
\end{equation*}
Namely, a point $(a_{1},\cdots, a_{n})\in \cA$ consists of a rational poly-differential form $a_{i}=\wt a_{i}(dt/t)^{d_{i}}$ on $X$, where $a_{i}$ is a regular function on $X-\{\infty\}$ with possibly a pole of order $\le d_{i}$ at $\infty$, and vanishing at $0$. We have a Hitchin map
\begin{equation}\label{Hitchin G2}
f:\cM(\bG_{\infty}^{2},\bI_{0})\to \cA
\end{equation}
sending $(\cE,\t_{2\infty}, \cE_{0,B}, \ph)$ to $(f_{1}(\ph),\cdots, f_{n}(\ph))$. Note that $\ph=\wt \ph dt/t$ for some $\wt\ph\in \cohog{0}{X,\Ad^*(\cE)(\infty)}$,  hence $f_{i}(\ph)$ can be written as $f_{i}(\wt \ph)(dt/t)^{d_{i}}$ where $f_{i}(\wt\ph)$ is a regular function on $X-\{\infty\}$ with  a pole of order $\le d_{i}$ at $\infty$;  $f_{i}(\wt \ph)$ vanishes at $0$ because $\wt\ph(0)=\Res_{0}\ph\in \cE_{0,B}(\frn)$. 
 
We also have the moment map for the $\frg$-action on $\cM(\bG_{\infty}^{2},\bI_{0})$. Together with $f$ we get
\begin{equation*}
\wt f=(f,\mu_{\frg}):\cM(\bG_{\infty}^{2},\bI_{0})\to \cA\times_{\frc}\frg^{*}.
\end{equation*}
Here $\frc=\frg^{*}\qq G$. Let $\chi: \frg^{*}\to \frc$ be the natural map. The map $\ev_{\infty}: \cA\to \frc$ is by taking the image of the polar part of $\ph_{\infty}$.

Let $\cA_{\psi}=\ev_{\infty}^{-1}\chi(\psi_{1})\subset \cA$. Then $\wt f$ restricts to a map
\begin{equation*}
f_{\psi}: \cM_{\psi}\to \cA_{\psi}.
\end{equation*}
There is a distinguished point $a_{\psi}\in \cA_{\psi}$ given as the image of $(\cE^{\triv}, \t^{\triv}_{\infty}, \cE^{\triv}_{0, B}, \psi_{1} dt)\in \cM_{\psi}$. Here $(\cE^{\triv},\t^{\triv}_{\infty}, \cE_{0,B}^{\triv})$ is the trivial $G$-bundle with its tautological trivialization at $\infty$ and $B$-reduction at $0$.

\begin{exam} Consider the case $G=\GL_{n}$. In this case $\psi_{1}$ is given by a diagonal matrix with entries $(\l_{1},\cdots, \l_{n})$. Then the fibers of the Hitchin map $f_{\psi}$ can be describe using spectral curves. A point $a=(a_{1},\cdots, a_{n})\in \cA_{\psi}$ (where $a_{i}\in \cohog{0}{X,\cO(i)}$ and $a_{i}(0)=0$) defines a spectral curve $Y_{a}\subset \Tot(\om_{X}(2\cdot \infty+\un 0))=\Tot(\cO(1))$ by the equation $y^{n}-a_{1}y^{n-1}+a_{2}y^{n-2}+\cdots+(-1)^{n}a_{n}=0$, such that over $\infty$ it passes exactly through the points $(\l_{i}, \infty)$ for $i=1,2,\cdots, n$, and its preimage over $\un 0$ is concentrated at $(0,\un 0)$. (To make sense of this it requires a trivialization of the fiber of $\om_{X}(2\cdot \infty+\un 0)$ at $\infty$, for which we use $dt$).

Then $f_{\psi}^{-1}(a)$ is isomorphic to the stack $\ov\cP(Y_{a};\infty)$ that classifies rank one torsion-free sheaves $\cF$ on $Y_{a}$ together with a trivialization of its fiber at every point over $\infty$. From this we see that $\Gm^{n}$ acts on $\ov\cP(Y_{a};\infty)$ by changing the trivializations, and the quotient $\ov\cP(Y_{a};\infty)/\Gm^{n}$ is the compactified Picard stack $\ov\Pic(Y_{a})$.

It is easy to see that the arithmetic genus of $Y_{a}$ is $(n-1)(n-2)/2$ by viewing it as a degree $n$ curve in $\PP^{2}$ (via the identification $\Tot(\cO(1))\cong \PP^{2}-\{[0,0,1]\}$). Therefore $\dim\ov\Pic(Y_{a}) =(n-1)(n-2)/2-1$ and $\dim \ov\cP(Y_{a};\infty)=n(n-1)/2$. On the other hand, $\dim\cA=1+2+\cdots +n=n(n+1)/2$; since $\cA_{\psi}\subset \cA$ is cut out by $n$ linear equations, $\dim \cA_{\psi}=n(n-1)/2=\dim \ov\cP(Y_{a})$.  

For the point $a_{\psi}$, the  corresponding $Y_{a_{\psi}}$ is defined by $(y-\l_{1}t)\cdots (y-\l_{n}t)$. It is isomorphic to a union of $n$ copies of $\PP^{1}$ intersecting over a unique point over $\un 0$ (it is still locally planar at the intersection). 

\end{exam}

\sss{$\Gm$-action}\label{sss:Gm action}
There is an action of $\Gm$ on $\cM_{\psi}$ by simultaneously scaling the Higgs field and rotating $\PP^{1}$. We take the inverse of the second action so that $\psi_{1} dt$ is invariant under $\Gm$. The induced action on $\cA_{\psi}$ is contracting to $a_{\psi}$. 

The basic geometric properties of $\cM_\psi$ is summarized below.
\begin{prop}[{\cite[Theorem 2.8.1, Lemma 2.15.6]{BBAMY}}]\label{p:geom Mpsi}
\begin{enumerate}
    \item The stack $\cM_\psi$ is a smooth algebraic space locally of finite type over $k$ of dimension $|\Phi|$ (the number of roots of $G$).
    \item The space $\cM_\psi$ carries a canonical symplectic structure with weight $1$ under the $\Gm$-action.
    \item The map $f_{\psi}$ is a completely integrable system. In particular, fibers of $f_{\psi}$ are Lagrangians in $\cM_{\psi}$.
    \item For any $\xi\in \cM_\psi$, the limit point $\lim_{s\to 0}s\cdot \xi$ under the $\Gm$-action exists.
\end{enumerate}
\end{prop}


\subsection{Central Hitchin fiber}
Let $\L_{\psi}=f_{\psi}^{-1}(a_{\psi})\subset \cM_{\psi}$ be the central fiber of $f_{\psi}$.

\sss{Open subset of $\L_{\psi}$ as a Lagrangian in $T^{*}\cB$}\label{ss:cotB}
There is an open subset of $\L_{\psi}$ that can be identified with a conic Lagrangian in $T^{*}\cB$. Let $\cM^{\c}_{\psi}\subset \cM_{\psi}$ be the open locus where the $G$-bundle $\cE$ is trivial. Then the moduli stack of $(\cE,\t_{\infty})$ with $\cE$ trivial is a point. Fixing $(\cE,\t_{\infty})$ to be the trivial bundle with the tautological trivialization at $\infty$, we see that  $\cM^{\c}_{\psi}$ classifies pairs $(B', \ph)$ where $B'\subset G$ is a Borel subgroup and $\ph=\psi_{1} t+A$ for $A\in \frg$, such that $A\in \frn_{B'}$, the nilpotent radical of $B'$. In other words we have a canonical isomorphism $\cM^{\c}_{\psi}\cong T^{*}\cB$. This is clearly a symplectomorphism.

Consider $\L_{\psi}^{\c}=\L_{\psi}\cap \cM^{\c}_{\psi}$. Under the symplectomorphism $\cM^{\c}_{\psi}\cong T^{*}\cB$, $\L^{\c}_{\psi}$ becomes a Lagrangian in $T^{*}\cB$. Explicitly, $\L_{\psi}^{\c}$ consists of pairs $\{(A,B')\in \frg\times \cB$ such that $A\in\frn_{B'}$ and $ \chi(\psi_{1} +sA)=\chi(\psi_{1} )\in \frc$ for all $s\in \CC$, where $\chi: \frg\to \frc$ is the quotient map.  Let 
\begin{equation*}
\cN_{\psi_{1}}=\{A\in \frg|\chi(\psi_{1} +sA)=\chi(\psi_{1}) \mbox{ for all $s\in \CC$}\}.
\end{equation*}
In other words $A\in \cN_{\psi_{1}}$ if and only if the line  $\ell_{\psi_{1},A}$ through $\psi$ in the direction of $A$ remains in the adjoint orbit of $\psi_{1}$. Clearly $\cN_{\psi_{1}}$ is a conical closed subset in the nilpotent cone. Let $\pi: T^{*}\cB\to \cN$ be the Springer resolution. Then
\begin{equation*}
\L_{\psi}^{\c}=\pi^{-1}(\cN_{\psi_{1}}).
\end{equation*}

The following is a special case of \cite[Theorem 2.8.1(4)]{BBAMY}.

\begin{prop}\label{p:ASF vs HF}
There is a canonical homeomorphism $\Fl_{\psi}\to \L_{\psi}$.
\end{prop}
We briefly indicate how the map $\Fl_{\psi}\to \L_{\psi}$ is constructed. We have a base point $(\cE^{\triv},\t^{\triv}_{\infty}, \cE^{\triv}_{0,B}, \psi_{1} dt)\in \ov\L_{\psi}$, where $(\cE^{\triv}, \t^{\triv}_{\infty}, \cE^{\triv}_{0,B})$ is the trivial $G$-bundle on $\PP^{1}$ with the tautological trivialization at $\infty$ and Borel reduction at $0$ given by $B$ itself. A point in $\Fl$ can be viewed as a triple $(\cF, \io, \cF_{0,B})$ where $\cF$ is a $G$-torsor over $\Spec \CC\tl{t}$, $\io$ a trivialization of $\cF|_{\Spec \CC\lr{t}}$ and $\cF_{0,B}$ a $B$-reduction of $\cF_{t=0}$.  Such a triple lies in $\Fl_{\psi}$ if and only if $\psi$, viewed as a section of $\Ad^{*}(\cF)|_{\Spec \CC\lr{t}}$, extends to a section of $\Ad^{*}(\cF)$ whose value at $t=0$ lies in $\cF_{0,B}(\frb^{\bot})$. The map $\Fl_{\psi}\to \L_{\psi}$ sends $(\cF,\io,\cF_{0,B})$ to the $G$-Higgs bundle on $\PP^{1}$ obtained by gluing $(\cE^{\triv}, \t^{\triv}_{\infty}, \cE^{\triv}_{0,B}, \psi_{1} dt)|_{\PP^{1}-\{0\}}$ with $(\cF, \cF_{0,B}, \psi_1 \frac{dt}{t})$ along their trivializations over $\Spec\CC\lr{t}$.

\subsection{Lattice action on $\cM_{\psi}$}\label{ss:lattice action}

\begin{lemma} The natural action of $\xcoch(T)$ on $\Fl_{\psi}\cong \L_{\psi}$ extends canonically to an action of $\xcoch(T)$ on $\cM_{\psi}$.
\end{lemma}
\begin{proof}
Following the construction in \cite[\S2.1]{Ngo}, let $\JJ\to [\frc/\Gm]$ be the regular centralizer group scheme for $G$.  Let $J\to \cA_{\psi}\times X$ be the group scheme whose restriction to $\{a\}\times X$ (where $a\in\cA_{\psi}$) is $a^{*}\JJ$, where $a$ is viewed as a map $X\to \frc/\Gm$. 
Consider the Picard stack $\cP$ over $\cA_{\psi}$ whose fiber at $a\in \cA_{\psi}$ is the moduli stack of $J_{a}=a^{*}\JJ$-torsors over $X$. 

For $a\in \cA_{\psi}$, the leading term of $a$ at $\infty$ is $\chi(\psi_{1})$, hence the fiber $J_{a}|_{\infty}$ is canonically identified with $T=C_{G}(\psi_{1})$. Let $\wh\cP\to \cA_{\psi}$ be the stack whose fiber over $a\in \cA_{\psi}$ is a $J_{a}$-torsor over $X$ together with a trivialization at $\infty$, so that $\wh\cP\to \cP$ is a $T$-torsor.  The action of $\BB\JJ$ on $[\frg/G]$ induces a natural action of $\wh\cP$ on $\cM_{\psi}$ over $\cA_{\psi}$. 

Consider the cameral cover $\pi: \wt X\to \cA_{\psi}\times X$ that is the pullback of $\frt^{*}/\Gm\to \frc/\Gm$ under the evaluation map $\cA_{\psi}\times X\to [\frc/\Gm]$. Let $D_{\infty}=\Spec \cO_{\infty}$ be the formal completion of $X$ at $\infty$, and let  $\pi_{\infty}: \wt X^{\wedge}_{\infty}\to \cA_{\psi}\htimes D_{\infty}$ be the base change of $\pi$.  Since $\cA_{\psi}\times\{\infty\}\to \frc$ is the constant map that lifts to $\psi\in (\frt^{*})^{\rs}$, $\pi_{\infty}$ is a $W$-torsor with a trivialization in the special fiber $\cA_{\psi}\times\{\infty\}$. Therefore the $W$-torsor $\pi_{\infty}$ has a canonical trivialization extending the trivialization of the special fiber, i.e., there is a canonical section $s: \cA_{\psi}\htimes D_{\infty}\to  \wt X^{\wedge}_{\infty}$ to $\pi_{\infty}$ given by $\psi$. Using this section we get a canonical isomorphism
\begin{equation}\label{JT}
J|_{\cA_{\psi}\htimes D_{\infty}}\cong T\times (\cA_{\psi}\htimes D_{\infty})
\end{equation}
of group schemes over $\cA_{\psi}\htimes D_{\infty}$.  Let $\wh\Gr_{J,\infty}\to \cA_{\psi}$ be the moduli space whose fiber over $a\in \cA_{\psi}$ classifies $J_{a}$-torsors over $X$ with a trivialization over $\infty$ and a trivialization over $X-\{\infty\}$. We have a canonical map $\t: \wh\Gr_{J,\infty}\to\wh\cP$ over $\cA_{\psi}$. Using \eqref{JT}, we can then identify $\wh\Gr_{J,\infty}$  with $\wh\Gr_{T}\times \cA_{\psi}$ where $\wh\Gr_{T}=T\lr{t^{-1}}/T\tl{t^{-1}}^{1}$, and $T\tl{t^{-1}}^{1}=\ker(\ev_{\infty}: T\tl{t^{-1}}\to T)$.  Thus we get a canonical homomorphism over $\cA_{\psi}$.
\begin{equation*}
\xcoch(T)\xr{\l\mapsto t^{\l}}\wh\Gr_{T}\times \cA_{\psi}\cong \wh\Gr_{J,\infty}\to \wh\cP.
\end{equation*}
Composing with the action of $\wh\cP$ on $\cM_{\psi}$ we get an action of $\xcoch(T)$ on $\cM_{\psi}$ that extends the natural action on $\L_{\psi}\cong\Fl_{\psi}$.

\end{proof}

\begin{exam} In the situation of $G=\GL_{n}$, the action of the standard basis elements $e_{i}\in\xcoch(T)$ on each Hitchin fiber $\ov\cP(Y_{a};\infty)$ is by sending $\cF$ to $\cF(-y_{i,\infty})$. Here $y_{i,\infty}=(\psi_{i}, \infty)\in Y_{a}$ are the points over $\infty$. Given a basis of $\cF$ at $y_{i,\infty}$, we have to specify a basis of $\cF(-y_{i,\infty})$ at $y_{i,\infty}$. This amounts to exhibiting a canonical basis of $\cO_{Y_{a}}(-y_{i,\infty})|_{y_{i,\infty}}$, or a basis of the cotangent line $T^{*}_{y_{i,\infty}}Y_{a}$. However,  the projection $p_{a}: Y_{a}\to X$ is \'etale over $\infty$,  the cotangent space $T^{*}_{y_{i,\infty}}Y_{a}$ is identified with $T^{*}_{\infty}X$ via $dp_{a}$, hence carries a basis $dt/t^{2}$.
\end{exam}

\subsection{Parahoric version}\label{ss:par}
For any standard parahoric subgroup $\bP\subset G\lr{t}$, we have the moduli stack $\Bun_{G}(\bG^{\psi}_{\infty},\bP_{0})$ (we use $\bP_{0}$ to emphasize that the $\bP_{0}$-level structure is at $0\in \PP^{1}$) and its cotangent bundle $\cM(\bG^{\psi}_{\infty},\bP_{0})$. They have actions of $\Aff=\Ga\rtimes\Grot$. Similarly we have $\Bun_{G}(\bG^{2}_{\infty},\bP_{0})$ and its cotangent bundle $\cM(\bG^{2}_{\infty},\bP_{0})$. They have actions by $\frg\rtimes\Grot$.  We define $\cM_{\psi,\bP}$ as
\begin{equation*}
\cM_{\psi,\bP}:=\cM(\bG^{\psi}_{\infty},\bP_{0})\qq_{1}\Ga=\cM(\bG^{2}_{\infty},\bP_{0})\qq_{\psi}\frg.
\end{equation*}

The Hitchin map
\begin{equation}\label{par H}
f_{\psi, \bP}: \cM_{\psi,\bP}\to \cA_{\psi}
\end{equation}
is defined similarly as $f_{\psi}$, but it is in general not surjective. Let $\L_{\psi, \bP}=f_{\psi,\bP}^{-1}(a_{\psi})$.

In the special case $\bP=\bG=G\tl{t}$, $\cM_{\psi,\bG}$ consists of triples  $(\cE,\t_{\infty},\ph)$ where $\cE$ is a $G$-bundle over $X$, $\t_{\infty}$ a trivialization of $\cE_{\infty}$ and $\ph$ is a section of  $\Ad^{*}(\cE)\ot \om_{X}(2\cdot\infty)\cong \Ad^{*}(\cE)$ whose leading term at $\infty$ is $\psi$.  The Hitchin map for $\cM_{\psi,\bG}$ is the constant map to $a_{\psi}\in \cA_{\psi}$ because $f_{i}(\ph)$ are constant functions on $X$.

To state an analogue of Proposition \ref{p:ASF vs HF}, we need to define affine Spaltenstein fibers for the parahoric $\bP$ and $\g\in \frg\lr{t}$:
\begin{equation*}
\Fl^{\bP}_{\g}=\{g\in G\lr{t}/\bP|\Ad(g^{-1})\g\in \Lie\bP^{+}\}.
\end{equation*}

The same argument of Proposition \ref{p:ASF vs HF} gives the following.
\begin{prop}\label{p:ASF vs HF par} There is a canoncial homeomorphism $\Fl^{\bP}_{\psi}\cong \L_{\psi, \bP}$ (the fiber of $a_{\psi}$ under the parahoric Hitchin map \eqref{par H}).
\end{prop}

\sss{The case $\bP=\bG$}\label{sss:MGr}
Consider $\Fl^{\bG}_{\psi}=\Gr_{\psi}$. For $\l\in \xcoch(T)$ we have $t^{\l}\in \Gr= G\lr{t}/G\tl{t}$. On the other hand, let  $\wt\cE_{\l}\in \Bun_{G}(\bG^{1}_{\infty})=\ker(G[t^{-1}]\to G)\bs G\lr{t}/G\tl{t}$ be the point corresponding to the double coset of $t^{\l}$. Denote the underlying $G$-bundle of $\wt\cE_{\l}$ by $\cE_{\l}$. Note $\psi_{1} dt$ defines a global section of $\Ad(\cE_{\l})\ot \om_{X}(2\cdot\infty)$, hence giving a point $(\wt\cE_{\l}, \psi_{1} dt)\in \cM_{\psi,\bG}$.

\begin{lemma}\label{l:Gr ASF discrete} We have $\Fl^{\bG}_{\psi}=\{t^{\l}|\l\in \xcoch(T)\}$. In particular, $\cM_{\psi,\bG}$ is a discrete set consisting of points $\{(\wt\cE_{\l}, \psi_1dt)|\l\in \xcoch(T)\}$. 
\end{lemma}
\begin{proof}
By definition, $\Fl^{\bG}_{\psi}=\{g\in \Gr|\Ad(g^{-1})t\psi_{1}\in\Lie \bG^{1}_{\infty}=t\frg\tl{t}\}$. Hence $\Fl^{\bG}_{\psi}=\Gr_{\psi_{1}}$. Since $\psi_{1}$ is regular semisimple, $\Gr_{\psi_{1}}=(\Gr)^{T}=\{t^{\l}|\l\in\xcoch(T)\}$. 
\end{proof}

\subsection{Microsheaves on $\mathcal{M}_\psi$}\label{ss:micro Mpsi}

We wish to define a sheaf of categories $\muSh$ on $\mathcal{M}_\psi$, in the spirit of Section \ref{sec:microsheavesonexact}. Although $\mathcal{M}_\psi$ is a finite dimensional manifold, it is not of finite type, and so this requires some care.

We make $\mathcal{M}_\psi$ into an exact symplectic manifold as follows. Consider the holomorphic symplectic form $\omega_\CC$ on $\mathcal{M}_\psi$. As our real symplectic form, we take the real part $\omega = \Re(\omega_\CC)$. Let $V$ be the vector field generated by the action of $\RR^* \subset \CC^*$ on $\mathcal{M}_\psi$. Then $V$ is a Liouville vector field for $\omega$, with associated Liouville form $\lambda$.

We can exhaust $\Bun_{G}(\Jker, \bI_0)$ by an increasing sequence of open substacks ${}^{\mu_i}\Bun_{G}(\Jker, \bI_0)$ parametrizing bundles of Harder-Narasimhan type bounded by some rational dominant coweight $\mu_i$, $i=1,2,\cdots$. Each ${}^{\mu_i}\Bun_{G}(\Jker, \bI_0)$ is a quotient of the smooth finite-dimensional scheme $N_i = {}^{\mu_i}\Bun_{G}(\bG^{n_i}_\infty, \bI_0)$ by the action of $\bK^\psi_i = \Jker/\bG_\infty^{n_i}$ for $n_i\in\NN$ sufficiently large (depending on $\mu_i$).


Let $\frak{N}^i = [N^i / \bK^{\psi}_i]$ be the quotient stack. Let $\frak{X}^i = T^*(\frak{N}^i / \GG_m)$. We apply Definition \ref{def:microsheavesonstacks} to obtain a sheaf of stable $\infty$-categories on $\frak{X}^i$.

The action of $\bK_i/\bK_i^\psi\cong \bG^1_\infty / \bG^{\psi}_\infty = \GG_a$ admits a moment map $\mu : T^*\frN^i \to (\Lie\GG_a)^*$, and the zero-fiber $\mu^{-1}(0)$ is $\GG_m$-invariant. Its complement defines a substack $\mathcal{M}_{\psi}^i \subset \frak{X}^i$, which embeds in $\mathcal{M}_\psi$ as an open subset. If we view $\mathcal{M}_\psi^i$ as a (real) symplectic manifold, it is naturally equipped with quotient Maslov data. We can apply Definition \ref{def:microsheavesonexact} to obtain a sheaf of stable $\infty$-categories on $\mathcal{M}_{\psi}^i$, which coincides with the restriction of the previous sheaf to $\frak{X}^i$ by Proposition \ref{prop:nadlershendequotientisintrinsic}.

We have an exhaustion
\[ \mathcal{M}_\psi = \bigcup_i \mathcal{M}_{\psi}^i. \]

Since the quotient Maslov data on each $\mathcal{M}^i_{\psi}$ is compatible with the inclusions $\mathcal{M}^i_{\psi} \subset \mathcal{M}^{i+1}_{\psi}$, the resulting sheaves are equal after restriction.

\begin{defn}
    Let $\muSh$ be the sheaf of stable $\infty$-categories on $\mathcal{M}_\psi$ whose restriction to $\mathcal{M}_{\psi}^i$ is defined as above.
\end{defn}


\begin{prop}
Any compactly supported object in $\muSh(\mathcal{M}_\psi, \lambda)$ is supported on a finite union of irreducible components of the central Hitchin fiber $\Lambda_\psi \subset \mathcal{M}_\psi$.
\end{prop}
\begin{proof}
The support of any such object is conical with respect to the Liouville flow. Any compact conical subset must be contained in $\Lambda_\psi$. Since the support is coisotropic, it must be a union of irreducible components of $\Lambda_\psi$. 
\end{proof}

For this and other reasons (see the discussion of Fukaya categories below), it is natural to restrict our attention to the subcategory 
$\muSh_{\Lambda_\psi} (\mathcal{M}_\psi,\lambda)$ of microlocal sheaves supported on $\Lambda_\psi$. We also sometimes abuse notation and write this as $\muSh_{\Fl_\psi} (\mathcal{M}_\psi,\lambda)$.

\sss{Lattice action} \label{sss:latticeaction}
The action of the lattice $\XX_*(T)$ on $\mathcal{M}_\psi$ does not commute with the action of $\CC^*$, and in particular does not preserve the one-form $\lambda$ and does not map conical subsets of $\mathcal{M}_\psi$ to conical subsets. For the same reason, the lattice does not
act on the category $\muSh (\mathcal{M}_\psi,\lambda)$. However, $\XX_*(T)$ does preserve the subset $\Lambda_\psi$, and we expect that it acts on $\muSh_{\Lambda_\psi} (\mathcal{M}_\psi,\lambda)$ by automorphisms. Indeed, for each $\sigma \in \XX_*(T)$, one may pick a homotopy between $\sigma^*\lambda$ and $\lambda$ such that $\Lambda_\psi$ is fiberwise conical. We expect that these homotopies, composed with the translation action of $\sigma \in \XX_*(T)$ induce an action of $\XX_*(T)$ on $\muSh_{\Lambda_\psi} (\mathcal{M}_\psi,\lambda)$, following the construction in \cite{NS}. We will not use this action in the present paper. It should be the microlocal counterpart of the action defined in Section \ref{subsec:heckeinf}.


\subsection{Speculations on Fukaya categories}

We now turn to the question of associating a Fukaya category to the space $\mathcal{M}_\psi$. The form $\lambda$ does {\em not} make $\mathcal{M}_\psi$ into a Liouville manifold, because $\mathcal{M}_\psi$ is not of finite type. There are various methods available to try and circumvent this problem; we sketch two approaches here.

\sss{Via localization to the core}

 Recall the Hitchin map $f: \mathcal{M}_\psi \to \mathcal{A}_\psi$. Fix a linear metric on $\mathcal{A}_\psi$, and consider the family of balls $B^k \subset \mathcal{A}_\psi$ of radius $k > 0$ centered at $a_\psi$. Set $\mathcal{M}_{\psi}^k = f^{-1}(B^k) \subset \mathcal{M}_\psi$. Then the flow of the vector field associated to $\lambda$ is outward pointing along $\partial \mathcal{M}_{\psi}^k$. The core of $(\mathcal{M}_\psi, \lambda)$ may be defined as in Definition \ref{def:core}, where we set $X^k = \mathcal{M}_{\psi}^k$.
\begin{lemma}
The core of $\mathcal{M}_\psi$ is the central fiber of the Hitchin fibration: $\frak{c}(\mathcal{M}_\psi, \lambda) = \Lambda_\psi$.
\end{lemma}
\begin{proof}
    The Liouville vector field dilates $\mathcal{A}_\psi$, with unique fixed point $a_\psi$. Hence $\frak{c}(\mathcal{M}_\psi, \lambda) \subset f^{-1}(a_\psi) = \Lambda_\psi$. The opposite containment follows from the fact that $\Lambda_\psi$ is a union of compact conical subsets.
\end{proof}

The fact that the core $\Lambda_\psi$ is ind-compact is a clue that $(\mathcal{M}_\psi, \lambda)$ in many respects behaves like a finite type Liouville manifold. For instance, the statement of Theorem \ref{thm:exactlagsgiveobjects} should hold for $(\mathcal{M}_\psi, \lambda)$. We may then seek to \emph{define} the Fukaya category in terms of $\mathfrak{Sh}(X,\lambda) := \muSh_{\Lambda_\psi} (\mathcal{M}_\psi,\lambda)$, converting Theorem \ref{thm:Fukayaismicro} into a definition.



\sss{Via the lattice quotient} A natural approach to producing a (finite type) Liouville manifold from $\mathcal{M}_\psi$ is to take a quotient by the action of the lattice $\XX_*(T)$, which preserves the symplectic form and the subset $\Lambda_\psi$. The lattice action is free and discontinuous on $\Lambda_\psi \subset \mathcal{M}_\psi$, and the quotient $\Lambda_\psi / \XX_*(T)$ is of finite type. The same should be true of a small tubular neighborhood of $\Lambda_\psi$. 

\begin{defn}
Fix a small ball $B^{\epsilon} \subset \mathcal{A}_\psi$ centered at $a_\psi$, and let $\mathcal{Y}_\psi = f^{-1}(B^{\epsilon}) / \XX_*(T)$.
\end{defn}
The space $\mathcal{Y}_\psi$ is a finite type manifold, which we wish to endow with a Liouville structure.\footnote{The space $\mathcal{Y}_\psi$ is a smooth affine complex variety and thus a Stein manifold. It therefore carries a natural family of Weinstein structures. It is difficult, however, to establish a direct relation between these Wenstein structures and the Liouville structure on $\mathcal{M}_\psi$ defined above. For the same reason, it is difficult to relate the Fukaya categories defined from these Weinstein structures to the categories considered in this paper, although we expect a close such relation exists.}  As mentioned above, the action of $\XX_*(T)$ on $\mathcal{M}_\psi$ does not preserve the Liouville form $\lambda$, and hence the latter does not descend to the quotient. 

One can instead try to deform $\lambda$ to something $\XX_*(T)$-invariant as follows. Suppose we are given a real-valued function $h: \mathcal{M}_\psi \to \mathbb{R}$, such that for all $\sigma \in \mathbb{X}_*(T)$ we have $\sigma^*dh = dh + \theta_\sigma$. Here $\theta_\sigma = \iota_{V_\sigma} \omega$ and $V_\sigma$ is the vector field generating the action of $\mathfrak{Re}(\mathbb{C}^*)$ defined by $\sigma$.

Then $\lambda_h = \lambda + df$ will be a $\XX_*(T)$-invariant Liouville form, which descends to $\mathcal{Y}_\psi$. 
One should further require that the associated isotopy of Liouville vector fields preserves $\Lambda_\psi$ and dilates $\mathcal{A}_\psi$. The result would then be a Liouville manifold $(\mathcal{Y}_\psi, \lambda_h)$ whose Fukaya category can be locally modelled by $\muSh_{\Lambda_\psi} (\mathcal{M}_\psi,\lambda)$. We do not at present know how to find such a function $h$. However, its existence would follow from the existence of a suitable hyperk\"ahler structure on $X$ by the same construction as in \cite{GMW}.
\quash{
\begin{conj}
There exists be a lattice-invariant hyperk\"ahler structure $g, (I,J,K)$ on $\mathcal{M}_\psi$ (or perhaps only on $\mathcal{M}_\psi^{\epsilon}$) preserved by the action of the compact subtorus $T_{c} \subset T$, such that $I$ is the Dolbeault complex structure, and the associated holomorphic symplectic form $\omega_J + I \omega_K$ coincides with the given one on $\mathcal{M}_\psi$.
\end{conj}

Then a Liouville structure on $\mathcal{Y}_\psi$ with skeleton $\Lambda_\psi / \XX_*(T)$ may be defined as in \cite{GMW}. One can check that this Liouville field makes $\mathcal{M}_\psi / \XX_*(T)$ into a Liouville manifold.
}
In this case, one could apply Theorem \ref{thm:Fukayaismicro} to identify the category $\muSh_{\Lambda_\psi} (\mathcal{M}_\psi,\lambda)$ with a $\XX^{*}(T)$-graded version of the Fukaya category of $\mathcal{Y}_\psi$. We leave such questions for future work.

\section{Affine Hecke category with level $(\J, \psi)$}

In this section we study the structure of versions of the affine Hecke category with deeper level structure $(\J, \psi)$. They will act on the categories $\DG$ and $\cD_\psi$, the main players in \S\ref{s:DGr} and \S\ref{s:D}. A more general type of affine Hecke categories were studied by Kamgarpour--Schedler \cite{KamSch} in the $\ell$-adic context over a positive characteristic base field. Our contribution here is: (1) we work with the Kirillov model so that the results hold in any sheaf-theoretic context; (2) for our purposes we need to work with a completed monodromic version of the affine Hecke category in the spirit of \cite[Appendix A]{BY}.

\subsection{The monoidal category $\Hinf$ and its variants}\label{ss:Hinf} Let $G\lr{\t}$ be the loop group of $G$ at $\infty\in \PP^{1}$, with uniformizer $\t=t^{-1}$. We would like to define a monoidal category that is informally sheaves on the double quotient  $(\J,\psi)\bs G\lr{\t}/(\J,\psi)$. More precisely, consider the ind-stack $\Jker\bs G\lr{\t}/\Jker$, which carries an action of $\Grot$ by loop rotation (scaling $t=\t^{-1}$) and $\Ga\times \Ga$ by left and translation coming from the fact that $\Ga=\J/\Jker$.  Our convention is that the left (resp. right) translation by $a\in \Ga$ is $g\mapsto \wt a g$ (resp. $g\mapsto g (\wt a)^{-1}$) where $\wt a\in \J$ is a lifting of $a$. On the quotient $\frac{\Jker\bs G\lr{\t}/\Jker}{\D(\Ga)}$ by the diagonal $\Ga$, there is an action of $\Aff=\Ga\rtimes\Grot$ where  $\Ga$ acts by the residue symmetry from $(\Ga\times\Ga)/\D(\Ga)\isom \Ga$ (via the first projection). We can thus consider the category
\begin{equation*}
\cH_{\infty}=\Kir\left(\dfrac{\Jker\bs G\lr{\t}/\Jker}{\D(\Ga)}\right).
\end{equation*}
  
\sss{Monoidal structure} To explain the monoidal structure on $\Hinf$, we consider the following general situation. Let $X_1, X_2$ and $Y$ be stacks with $\Aff$-actions. Let $X_{12}$ be another stack with an action of $\Aff^{(2)}:=(\Ga\times \Ga)\rtimes \Gm$ (diagonal $\Gm$-action on both copies of $\Ga$), together with maps
\begin{equation*}
    \xymatrix{& X_{12} \ar[dl]_{p_1}\ar[dr]^{p_2}\ar[rr]^{m} & & Y\\
    X_1 & & X_2}
\end{equation*}
such that $p_i$ is equivariant with respect to the homomorphism $\pi_i: \Aff^{(2)}=(\Ga\times \Ga)\rtimes \Gm\to \Aff=\Ga\rtimes\Gm$ given by $\pi_i(a_1,a_2,b)=(a_i, b)$, for $i=1,2$, and $m$ is equivariant with respect to the homomorphism $\s: \Aff^{(2)}\to \Aff$ given by $\s(a_1,a_2,b)=(a_1+a_2, b)$.

We claim that the functor
\begin{eqnarray*}
    D_{\Gm}(X_1)\times D_{\Gm}(X_2)&\to & D_{\Gm}(Y)\\
    (\cF_1,\cF_2)&\mt & m_!(p_1^*\cF_1\ot p_2^*\cF_2)
\end{eqnarray*}
induces a functor on the Kirillov categories
\begin{equation}\label{Kir conv}
    \Kir(X_1)\times \Kir(X_2)\to \Kir(Y).
\end{equation}
Indeed, $p_i^*$ sends $D_{\Aff}(X_i)$ to $D_{\Aff^{(2)}}(X_{12})$ for $i=1,2$, therefore the functor $(\cF_1,\cF_2)\mt p_1^*\cF_1\ot p_2^*\cF_2$ induces a functor 
\begin{equation}\label{K12}
    \Kir(X_1)\times \Kir(X_2)\to D_{\Gm}(X_{12})/D_{\Aff^{(2)}}(X_{12}).
\end{equation}
Moreover, $m_!$ sends $D_{\Aff^{(2)}}(X_{12})$ to $D_{\Aff}(Y)$, therefore it induces a functor
\begin{equation}\label{mKir}
    D_{\Gm}(X_{12})/D_{\Aff^{(2)}}(X_{12})\to \Kir(Y).
\end{equation}
The composition of \eqref{K12} and \eqref{mKir} then gives the desired functor \eqref{Kir conv}.

We now apply the above discussion to the diagram 
\begin{equation}\label{Hinf conv diag}
\xymatrix{& \dfrac{\Jker\bs G\lr{\t}\twtimes{\Jker}G\lr{\t}/\Jker}{\D\Ga}\ar[rr]^{m} \ar[dl]_{p_{1}}\ar[dr]^{p_{2}}&& \dfrac{\Jker\bs G\lr{\t}/\Jker}{\D\Ga}\\
\dfrac{\Jker\bs G\lr{\t}/\Jker}{\D\Ga} && \dfrac{\Jker\bs G\lr{\t}/\Jker}{\D\Ga}}
\end{equation}
Here $m$ is induced by the multiplication map of $G\lr{\t}$, and $p_1, p_2$ are the projections to the first and second factors. The $\D(\Ga)$-quotient in $\dfrac{\Jker\bs G\lr{\t}\twtimes{\Jker}G\lr{\t}/\Jker}{\D\Ga}$ is given by the action $a\cdot (g_{1},g_{2})=(\wt ag_{1}\wt a^{-1}, \wt ag_{2}\wt a^{-1})$. We equip $\dfrac{\Jker\bs G\lr{\t}/\Jker}{\D\Ga}$ with the $\Aff$-action explained in the beginning of \S\ref{ss:Hinf}. We equip $\dfrac{\Jker\bs G\lr{\t}\twtimes{\Jker}G\lr{\t}/\Jker}{\D\Ga}$ with the $\Aff^{(2)}$-action where $(a_1,a_2)\in \Ga\times \Ga$ acts by $(a_1,a_2)(g_1,g_2)=(\wt{a_1}g_1, g_2\wt{a_2})$ while $\Grot$ still acts by loop rotation. The discussion in the previous paragraph then gives a functor on the Kirillov categories
\begin{eqnarray*}
(-)\star(-): \cH_{\infty}\times\cH_{\infty}&\to&\cH_{\infty}\\
(\cF_{1},\cF_{2})&\mt& m_{!}(p_{1}^{*}\cF_{1}\ot p_{2}^{*}\cF_{2}).
\end{eqnarray*}
This defines a monoidal structure on $\cH_{\infty}$: the associativity structure follows from that of the monoidal structure of $D_{\Grot}(\dfrac{\Jker\bs G\lr{\t}/\Jker}{\D\Ga})$ under the usual convolution.

\sss{Perverse $t$-structure}
The $*$-pullback functor 
\begin{equation*}
D\left(\dfrac{\Jker\bs G\lr{\t}/\Jker}{\D\Ga}\right)\to D(G\lr{\t}/\Jker)
\end{equation*}
is fully faithful. The perverse $t$-structure on $D(G\lr{\t}/\Jker)$ restricts to a $t$-structure on $D(\dfrac{\Jker\bs G\lr{\t}/\Jker}{\D\Ga})$, which induces a $t$-structure on the quotient $\cH_{\infty}$. We call it the {\em perverse $t$-structure} on $\cH_{\infty}$. Denote the heart of the perverse $t$-structure by $\cH^{\hs}_{\infty}$.

\sss{Variants}\label{sss:var Hinf}
We will need the following variants of $\cH_{\infty}$.
\begin{enumerate}
\item We may impose right $T$-equivariance and consider the category
\begin{equation*}
\cH^{T}_{\infty}=\Kir\left(\dfrac{\Jker\bs G\lr{\t}/(\Jker T)}{\D(\Ga)}\right).
\end{equation*}
Similarly we may impose left $T$-equivariance to obtain the category ${}^{T}\cH_{\infty}$. We can also impose $T$-equivariance on both sides to form the monoidal category
\begin{equation*}
\eqHinf=\Kir\left(\dfrac{(\Jker T)\bs G\lr{\t}/(\Jker T)}{\D(\Ga)}\right).
\end{equation*}
The monoidal structure on $\eqHinf$ is defined using a diagram similar to \eqref{Hinf conv diag}
\begin{equation}\label{eqHinf conv diag}
\xymatrix{& \dfrac{(\Jker T)\bs G\lr{\t}\twtimes{\Jker T}G\lr{\t}/(\Jker T)}{\D\Ga}\ar[r]^-{m} \ar[dl]_{p_{1}}\ar[dr]^{p_{2}}& \dfrac{(\Jker T)\bs G\lr{\t}/(\Jker T)}{\D\Ga}\\
\dfrac{(\Jker T)\bs G\lr{\t}/(\Jker T)}{\D\Ga} && \dfrac{(\Jker T)\bs G\lr{\t}/(\Jker T)}{\D\Ga}}
\end{equation}

For both $\cH^{T}_{\infty}$ and ${}^{T}\cH^{T}_{\infty}$, we give them the perverse t-structure by viewing objects therein as sheaves on $G\lr{\t}/(\Jker T)$. Denote the hearts of the perverse t-structures by
\begin{equation*}
(\cH^{T}_{\infty})^{\hs}, \quad \Hinfhs.
\end{equation*}

\item We may consider the full dg subcategory $\cH^{\mon}_{\infty}\subset \cH_{\infty}$ generated by the essential image of the forgetful functor $\cH^{T}_{\infty}\to \cH_{\infty}$. It will be shown in Lemma \ref{l:Hmon same} that this is the same as the subcategory generated by the essential image of the forgetful functor ${}^{T}\cH_{\infty}\to\cH_{\infty}$. Therefore, $\cH^{\mon}_{\infty}$ is a monoidal full subcategory of $\cH_{\infty}$. We call it the $T$-monodromic subcategory of $\cH_{\infty}$ (with unipotent monodromy).

\item We may define a completed version $\wh\cH^{\mon}_{\infty}$ of $\cH^{\mon}_{\infty}$ by formally adding pro-objects that are universally monodromic along the left or right $T$-orbits. This can be done using the general framework in \cite[Appendix A]{BY}.
\end{enumerate}

\subsection{Relevant double cosets}

By \S\ref{ss:rel}, we can talk about relevant points in $\frac{\Jker\bs G\lr{\t}/\Jker}{\D(\Ga)}$ under the $\Ga$-action. We shall simply say $g\in G\lr{\t}$ is relevant if its image in $\frac{\Jker\bs G\lr{\t}/\Jker}{\D(\Ga)}$ is.  

\begin{lemma}\label{l:H inf rel}
The relevant locus in $G\lr{\t}$ is the union of $\J \t^{\l}T\J$ for all $\l\in \xcoch(T)$.
\end{lemma}
\begin{proof}
By definition, a point $x\in G\lr{\t}$ is relevant in this situation if and only if, under the $\J\times \J$-action by left and right translation, the projection
\begin{equation*}
\Stab_{\J\times \J}(x)\subset \J\times \J\xr{(\psi, \psi)}\Ga\times\Ga\xr{-}\Ga
\end{equation*}
is trivial (the last map is taking the difference of two coordinates). Here we view $\psi$ as the additive character on $\J$ via \eqref{psi as char}. The first projection identifies $\Stab_{\J\times \J}(g)$ with the subgroup $\J\cap x\J x^{-1}\subset \J$, and the above additive character becomes
\begin{eqnarray*}
\chi: \J\cap x\J x^{-1}&\to& \Ga\\
g&\mt& \psi(g)-\psi(x^{-1}gx).
\end{eqnarray*}
Therefore $x$ is relevant if and only if
\begin{equation}\label{x rel}
\psi(g)=\psi(x^{-1}gx) \mbox{for all $g\in \J\cap x\J x^{-1}$.}
\end{equation}

By the Cartan decomposition, we may assume $x\in \bG_{\infty}\t^{\l}\bG_{\infty}$ for some $\l\in\xcoch(T)$ (up to $W$). Let $P_{\l}$ be the parabolic subgroup of $G$ containing $T$ with roots $\{\a\in R|\j{\l,\a}\ge0\}$. Let $L_{\l}$ be its Levi subgroup containing $T$ with roots $\{\a\in R|\j{\l,\a}=0\}$, and let $U_{\l}$ be the unipotent radical of $P_{\l}$. Note that $P_{\l}=G\cap \t^{-\l}\bG_{\infty}\t^{\l}$.  Consider the $G\times G$-action on $\J\bs \bG_{\infty}\t^{\l}\bG_{\infty}/\J$; the stabilizer of the base point $t^{\l}$ is the subgroup of $P_{\l}\times_{L_{\l}}P_{-\l}$ (fiber product using the common quotient $L_{\l}$ of $P_{\l}$ and $P_{-\l}$). Therefore, up to left and right translation by $\J$, we may write $x=g_{1}\t^{\l}g_{2}^{-1}$, where $(g_{1},g_{2})\in G\times G$ is unique up to right translation by the subgroup $P_{\l}\times_{L_{\l}}P_{-\l}$. For this $x$, we have $ \J\cap x\J x^{-1}=\J\cap g_{1}\t^{\l}\J\t^{-\l}g_{1}^{-1}=g_{1}(\J\cap \t^{\l}\J \t^{-\l})g_{1}^{-1}$. Our goal is to show that the condition \eqref{x rel} implies that $x\in T\t^{\l'}$ for some $\l'\in W\l$, up to right multiplying $(g_{1},g_{2})$ by $P_{\l}\times_{L_{\l}}P_{-\l}$.

The condition \eqref{x rel} becomes
\begin{equation}\label{psi conj}
\psi^{g_{1}}(y)=\psi^{g_{2}}(\t^{-\l}y\t^{\l}) \quad\mbox{ for all $y\in \J\cap \t^{\l}\J \t^{-\l}$,}
\end{equation}
where for $g\in G$, $\psi^{g}\in \frg^{*}$ denote the linear function $v\mapsto \psi(\Ad(g)x)$ on $\frg$, viewed as a additive character on $\J$ as in \eqref{psi as char}. Now let $y$ go through elements in $\exp(\t^{n}\frg_{\a})$ for various affine roots $\a+n$ in that appear in $\J\cap \t^{\l}\J \t^{-\l}$, we see that \eqref{psi conj} implies
\begin{equation*}
\psi^{g_{1}}|_{\frn_{\l}}=0, \quad \psi^{g_{2}}|_{\frn_{-\l}}=0, \quad \psi^{g_{1}}|_{\frl_{\l}}=\psi^{g_{2}}|_{\frl_{\l}}.
\end{equation*}
Here $\frl_{\l}, \frn_{\l}$ and $\frn_{-\l}$ are the Lie algebras of $L_{\l}, U_{\l}$ and $U_{-\l}$ respectively. Using an invariant quadratic form on $\frg$ to identify $\frg$ with $\frg^{*}$,  we see that these conditions mean
\begin{equation}\label{psi conj 3}
\psi^{g_{1}}\in \frp_{\l}, \quad \psi^{g_{2}}\in \frp_{-\l}, \quad \psi^{g_{1}}\textup{ mod }\frn_{\l}=\psi^{g_{2}}\textup{ mod }\frn_{-\l}\in\frl_{\l}.
\end{equation}
Since $\psi^{g_{1}}$ is a semisimple element in $\frp_{\l}$, there exists $u_{1}\in U_{\l}$ such that $\psi^{g_{1}u_{1}}\in \frl_{\l}$. Changing $g_{1}$ to $g_{1}u_{1}$, we may assume $\psi^{g_{1}}\in \frl_{\l}$. Then there exists $\ell_{1}\in L_{\l}$ such that $\psi^{g_{1}\ell_{1}}\in \frt$. Since $\psi\in\frt$ is regular semisimple, we conclude that $g_{1}\ell_{1}\in N_{G}(T)$. So we may write $g_{1}=\dot w_{1}\ell_{1}^{-1}$ for some $\dot w_{1}\in N_{G}(T)$ with image $w_{1}\in W$. Similarly, we may assume $g_{2}=\dot w_{2}\ell_{2}^{-1}$ for some $w_{2}\in W$ and $\ell_{2}\in L_{\l}$. Now the third condition in \eqref{psi conj 3} implies $\psi^{g_{1}g_{2}^{-1}}=\psi$, hence $g_{1}g_{2}^{-1}\in T$ since $\psi$ is regular semisimple. In other words, $\dot w_{1}\ell_{1}^{-1}\ell_{2}\dot w_{2}^{-1}\in T$. 

Now we conclude
\begin{equation*}
x=g_{1}\t^{\l}g_{2}^{-1}=\dot w_{1}\ell_{1}^{-1}\t^{\l}\ell_{2}\dot w_{2}=\dot w_{1}\ell_{1}^{-1}\ell_{2}\dot w_{2}^{-1}\t^{w_{2}\l}\in T\t^{w_{2}\l}.
\end{equation*}
Here we use that $\t^{\l}$ commutes with $L_{\l}$.
\end{proof}

\sss{Orbit dimension} For $\l\in \xcoch(T)$, we have
\begin{equation}\label{d lam}
d_{\l}:=\dim \J \t^{\l} \J/\J=\dim \Jker \t^{\l}\Jker/\Jker=\sum_{\a\in R}\max\{-\j{\a,\l}, 0\}=\sum_{\a\in R^{+}}|\j{\a,\l}|.
\end{equation}
Indeed, $\J \t^{\l} \J/\J=\J/(\J\cap \Ad(\t^{\l})\J)$. The Lie algebra of $\J\cap \Ad(\t^{\l})\J$ is spanned by the affine root spaces $\t^{n}\frg_{\a}$ such that  $n>0$ and $n+\j{\a,\l}>0$, for each root $\a$ of $G$. From this we see the tangent space of $\J/(\J\cap \Ad(\t^{\l})\J)$ at identity is the direct sum of $\t^{n}\frg_{\a}$ for $-\j{\a,\l}\ge n>0$. 

In particular,  if $\l$ is dominant, and $2\r=\sum_{\a\in R^{+}}\a$, then $d_{\l}=\j{2\r, \l}$.

The next few lemmas will be used in the next subsection for computing the convolution of standard sheaves.

Choose a pair of opposite Borel subgroups $B, B^-\subset G$ containing $T$. Let $N$ and $N^-$ be the unipotent radical of $G$. Let
\begin{equation*}
    \bJ_N=\J\cap N\tl{\t}, \quad \bJ_T=\J\cap T\tl{\t}, \quad \bJ_{N^-}=\J\cap N^-\tl{\t}. 
\end{equation*}

\begin{lemma}\label{l:tri decomp}
    \begin{enumerate}
        \item Multiplication induces an isomorphism of $k$-schemes
    \begin{equation}\label{tri decomp for J}
        \bJ_N\times \bJ_T\times \bJ_{N^-}\isom \J.
    \end{equation}
        \item If $\l$ is dominant with respect to $B$, then
        \begin{equation*}
            \J \t^\l \J= \bJ_N \t^\l \J.
        \end{equation*}
    \end{enumerate}
\end{lemma}
\begin{proof}
    (1) It is well-known that the map $N
    \times T\times N^-\incl G$ is an open embedding. Part (1) of the lemma follows by taking the based arc spaces of both sides. 

    (2) Since $\J \t^\l \J/\J\cong \J/(\J\cap {}^{\l}\J)$, where ${}^{\l}\J=\t^{\l}\J\t^{-\l}$, we only need to show that the inclusion $\bJ_N\incl \J$ induces a surjection $\bJ_N\surj \J/(\J\cap {}^{\l}\J)$. Applying conjugation by $\t^\l$ to the decomposition \eqref{tri decomp for J}, ${}^\l\J$ admits a triangular decomposition ${}^{\l}\bJ_N\times \bJ_T\times {}^{\l}\bJ_{N^-}$, where ${}^{\l}\bJ_N=\t^{\l}\bJ_N\t^{-\l}={}^{\l}\J\cap N\tl{\t}$, etc. Since $\l$ is dominant, we have $\bJ_{N^-}\subset {}^{\l}\bJ_{N^-}$. Therefore every coset in $\J/(\J\cap {}^{\l}\J)$ can be represented by a point in $\bJ_N$ by (1).
\end{proof}

\begin{lemma}\label{l:rel in conv}
For any $\l,\mu\in\xcoch(T)$, 
\begin{enumerate}
\item The only relevant points in $\J \t^{\l} T\J \t^{\mu} T \J$ are $\J \t^{\l+\mu} T\J$.
\item The only relevant points in $\J \t^{\l} \J \t^{\mu} \J$ are $\J \t^{\l+\mu} \J$.
\item The only relevant points in $\Jker \t^{\l}\Jker \t^{\mu} \Jker$ are $\Jker \t^{\l+\mu}\Jker$.
\end{enumerate}
\end{lemma}
\begin{proof}
(1) follows from (2) because $\J \t^{\l} T\J \t^{\mu} T \J=\J \t^{\l} \J \t^{\mu} \J\cdot T$ and $\J \t^{\l+\mu} T\J=\J \t^{\l+\mu} \J\cdot T$.

(2) By Lemma \ref{l:H inf rel}, all relevant double cosets in $\J \t^{\l} \J \t^{\mu} \J$ are of the form $\J \t^{\nu} h\J$ for some $\nu\in \xcoch(T)$ and $h\in T$. We need to show that if $\t^{\nu} h\J$ has non-empty intersection with $\J \t^{\l} \J \t^{\mu}$, then $\nu=\l+\mu$ and $h=1$.

Choose a Borel $B\subset G$ containing $T$ such that $\l$ is dominant with respect to $B$. We apply Lemma \ref{l:tri decomp} to write
\begin{equation}\label{JlmJ}
    \J \t^{\l} \J \t^{\mu} =\bJ_N\t^{\l} \J \t^{\mu} \J=\bJ_N \t^\l (\bJ_N \bJ_T \bJ_{N^-}) \t^\mu\subset N\lr{\t}(\t^{\l+\mu}\bJ_T)N^-\lr{\t}.
\end{equation}
On the other hand, using \eqref{tri decomp for J} again, we have
\begin{equation}\label{tnuhJ}
    \t^\nu h \J\subset N\lr{\t}\cdot  (\t^\nu h \bJ_T)\cdot N^-\lr{\t}.
\end{equation}
If the left sides of \eqref{JlmJ} and \eqref{tnuhJ} intersect, so do the right sides, which by the uniqueness of triangular decomposition implies $\t^{\l+\mu}\bJ_T=\t^\nu h \bJ_T$. This forces $\nu=\l+\mu$ and $h=1$. 

(3) The proof is almost the same as (2), using the triangular decomposition $\Jker=\bJ_N\times \bJ^\psi_T\times \bJ_{N^-}$, where $\bJ_T^\psi=\bJ_T\cap \Jker$.

\end{proof}

%

\begin{lemma}\label{l:H inf conv fiber}
The fiber of the multiplication map $m^{\l,\mu}: \J \t^{\l}\J\twtimes{\J}\J \t^{\mu}\J\to G\lr{\t}$ over $\t^{\l+\mu}$ is an affine space of dimension $(d_{\l}+d_{\mu}-d_{\l+\mu})/2$.

The same is true when $\J$ is replaced by $\Jker$.
\end{lemma}
\begin{proof}
We prove the version for $\J$, and argument for the $\Jker$-version is the same.  Consider the subscheme $\wt Z$ of $(\J)^3$ consisting of triples $(x,y,z)$ such that $x\t^{\l} y \t^{\mu}z=\t^{\l+\mu}$. Rewriting this as $(\t^{-\l}x^{-1}\t^{\l})(\t^\mu z^{-1}\t^{-\mu})=y$ and make a change of variables $x'=\t^{-\l}x^{-1}\t^{\l}\in {}^{-\l}\J$ and $z'= \t^\mu z^{-1}\t^{-\mu}\in {}^{\mu}\J$, we have
\begin{equation*}
    \wt Z\cong \{(x',z')\in {}^{-\l}\J\times {}^{\mu}\J|x'z'\in \J\}.
\end{equation*}
There is an action of $\J\cap {}^{-\l}\J$ on $\wt Z$ by left translation on the $x'$-factor and an action of $\J\cap {}^{\mu}\J$ on it by right translation on the $z'$-factor. A simple calculation shows that 
\begin{equation}\label{fiber m xz}
    (m^{\l,\mu})^{-1}(\t^{\l+\mu})\cong (\J\cap {}^{-\l}\J)\bs \wt Z/(\J\cap {}^{\mu}\J)=:Z.
\end{equation}
Now choose a Borel subgroup $B$ containing $T$ (with $N$, $N^-$ having the usual meaning) such that $\l$ is dominant with respect to $B$. Then Lemma \ref{l:tri decomp}(2) (or rather the same argument using triangular decomposition) shows that $(\J\cap {}^{-\l}\J)\bs {}^{-\l}\J\cong \bJ_N\bs {}^{-\l}\bJ_N$. We can thus write the right side of \eqref{fiber m xz} as
\begin{eqnarray*}
    Z\cong \{(\xi,\z)\in \bJ_N\bs {}^{-\l}\bJ_N\times {}^{\mu}\J/(\J\cap {}^{\mu}\J)|\xi \z\in \J\}.
\end{eqnarray*}
Write $\z\in {}^{\mu}\J$ using the triangular decomposition $\z=\z^+\z^0\z^-$ where $\z^+\in {}^{\mu}\bJ_{N}, \z^0\in \bJ_T$ and $\z^-\in {}^{\mu}\bJ_{N^-}$. If $\xi\in {}^{-\l}\bJ_N$ is such that $\xi\z=(\xi\z^+)\z^0\z^-\in \J$, we must have $\z^-\in \bJ_{N^-}$ by the uniqueness of the triangular decomposition. This means for any point $(\xi,\z)\in Z$ we may replace $\z$ by an element in ${}^{\mu}\bJ_N$. Thus we get
\begin{equation*}
    Z\cong \{(\xi,\z)\in \bJ_N\bs {}^{-\l}\bJ_N\times {}^{\mu}\bJ_N/(\bJ_N\cap {}^{\mu}\bJ_N)|\xi\z\in \bJ_N\}.
\end{equation*}
We consider an analog of the right side for each root group $U_\a$ ($\a\in R^+$): let $\bJ_\a=\ker(U_\a\tl{\t}\to U_\a)=\J\cap U_\a\tl{\t}$, and consider 
\begin{eqnarray*}
    Z_{\a}=\{(\xi,\z)\in \bJ_\a\bs {}^{-\l}\bJ_\a\times {}^{\mu}\bJ_\a/(\bJ_\a\cap {}^{\mu}\bJ_\a)|\xi\z\in \bJ_\a\}.
\end{eqnarray*}
For $n\le 0$, let $L_n=\t^{n+1}\AA^1\tl{\t}/\t\AA^1\tl{\t}$, an additive group of dimension $-n$. An isomorphism $U_\a\cong \Ga$ allows us to identify $\bJ_\a\bs {}^{-\l}\bJ_\a$ with $L_{\j{-\l,\a}}$. We have two cases:
\begin{enumerate}
    \item If $\j{\mu,\a}\ge0$, then ${}^{\mu}\bJ_\a\subset \bJ_\a$ and $Z_\a=\pt$. 
    \item If $\j{\mu,\a}<0$, then ${}^{\mu}\bJ_\a/(\bJ_\a\cap {}^{\mu}\bJ_\a)\cong L_{\j{\mu,\a}}$. Under these isomorphisms, $Z_\a$ is the graph of the negation of either the natural embedding $L_{\j{\mu,\a}}\incl L_{\j{-\l,\a}}$ or the natural embedding $L_{\j{-\l,\a}}\incl L_{\j{\mu,\a}}$, depending on whether $\j{\l+\mu,\a}\ge0$ or not. In this case, $\dim Z_\a=\min\{\j{\l,\a}, \j{-\mu,\a}\}$.
\end{enumerate}
We can thus write uniformly
\begin{equation}\label{dim Za}
    \dim Z_\a=\frac{1}{2}\left(\j{\l,\a}+|\j{\mu,\a}|-|\j{\l+\mu,\a}|\right).
\end{equation}
Moreover, using the ordering of positive roots by height, $Z$ is an iterated fibration
\begin{equation*}
    Z=Z_{n}\xr{\pi_{n}} Z_{n-1}\xr{\pi_{n-1}} \cdots Z_{1}\xr{\pi_1} Z_{0}=\pt
\end{equation*}
such that each map $\pi_i$ is a torsor for the additive group $\prod_{\j{\r^\vee,\a}=i}Z_\a$. We conclude that $Z$, hence the fiber $(m^{\l,\mu})^{-1}(\t^{\l+\mu})$, is isomorphic to an affine space of dimension $\sum_{\a\in R^+}\dim Z_{\a}$. Using \eqref{dim Za} and the definition of $d_\l$ in \eqref{d lam}, we see that 
\begin{equation*}
    \dim Z=\sum_{\a\in R^+}\frac{1}{2}\left(\j{\l,\a}+|\j{\mu,\a}|-|\j{\l+\mu,\a}|\right)=\frac{1}{2}\left(d_\l+d_\mu-d_{\l+\mu}\right).
\end{equation*}
\end{proof}

\begin{remark}
    The above lemma strengthens and generalizes \cite[Proposition 30]{KamSch}.
\end{remark}

\subsection{Standard objects}
\sss{}

Fix $\l\in \xcoch(T)$.  Let
\begin{equation*}
\Hinf(\l)=\Kir\left (\dfrac{\Jker\bs\J \t^{\l}T\J/\Jker }{\D\Ga}\right)
\end{equation*}
Note that $\D\Ga$ acts trivially on $\Jker\bs \J \t^{\l}\J/\Jker$ since we may lift $\Ga$ to $\bT_{\infty}^{1}=\ker(T\tl{\t}\to T)$. Similarly define the monodromic and completed versions $\Him(\l)$ and $\hHim(\l)$, and the left (right) $T$-equivariant versions ${}^{T}\Hinf(\l)$ and $\Hinf^{T}(\l)$.

For $\l\in \xcoch(T)$, the $\Aff$-equivariant embeddings
\begin{equation*}
i_{\l}: \dfrac{\Jker\bs \J \t^{\l}T\J/\Jker}{\D\Ga}\incl\dfrac{\Jker\bs G\lr{\t}/\Jker}{\D\Ga}
\end{equation*}
gives a full embedding
\begin{equation*}
i_{\l!}: \Hinf(\l)\incl \Hinf
\end{equation*}
and induces a full embedding
\begin{equation*}
\wh i_{\l!}: \hHim(\l)\incl \hHim.
\end{equation*}

Similarly define $\Aff$-equivariant embedding
\begin{equation*}
i^{T}_{\l}: \dfrac{\Jker\bs \J \t^{\l}T\J/(\Jker T)}{\D\Ga}\incl\dfrac{\Jker\bs G\lr{\t}/(\Jker T)}{\D\Ga}
\end{equation*}
and its counterparts ${}^{T}i_{\l}$ when the $T$-quotient is on the left, and ${}^{T}i^{T}_{\l}$ when the $T$-quotient is on both sides.

\sss{Standard objects in $T$-equivariant variants}
For each $\l$,  the $\Aff$-equivariant map 
$$\k^{T}_{\l}: \Ga=\J/\Jker\xr{\t^{\l}} \Jker\bs\J \t^{\l}\J/\Jker=\Jker\bs\J \t^{\l}T\J/(\Jker T)$$ 
induces an equivalence on Kirillov categories
\begin{equation*}
\Hinf^{T}(\l)\isom \Kir(\Ga)
\end{equation*}
where $\Aff$ acts on $\Ga$ via affine transformations. The same is true for ${}^{T}\Hinf(\l)$.  Using the equivalence $\Kir(\Ga)\cong D^b(\Vect)$ from Example \ref{ex:Kir A1}, we see that
\begin{equation}\label{HT strat cat}
\Hinf^{T}(\l)\isom D^b(\Vect).
\end{equation}
For $\eqHinf$, we consider ${}^{T}{\k}^{T}_{\l}: (\pt/T)\times \Ga\cong T\bs (\Ga\times \t^{\l}T)/T\to (\Jker T)\bs\J \t^{\l}T\J/(\Jker T)$, and we have
\begin{equation}\label{THT strat cat}
{}^{T}\Hinf^{T}(\l)\isom \Kir((\pt/T)\times \Ga)\cong D(\pt/T).
\end{equation}

\begin{defn} Let $\l\in \xcoch(T)$.
\begin{enumerate}
\item We define $\cL^{T}_{\l}\in \Hinf^{T}(\l)$ to be the object corresponding to the one-dimensional vector space $E\in \Vect$ under the equivalence \eqref{HT strat cat}. Similarly define ${}^{T}\cL_{\l}\in {}^{T}\Hinf(\l)$. Define ${}^{T}\cL^{T}_{\l}\in {}^{T}\Hinf^{T}(\l)$ to correspond to the constant sheaf $E\in D(\pt/T)$ under the equivalence \eqref{THT strat cat}.

\item We define
\begin{eqnarray*}
\D^{T}_{\l}:=(i^{T}_{\l})_{!} \cL^{T}_{\l}\j{d_{\l}}\in \cH^{T}_{\infty},\\
{}^{T}\D_{\l}:=({}^{T}i_{\l})_{!}({}^{T}\cL_{\l})\j{d_{\l}}\in {}^{T}\cH_{\infty},\\
{}^{T}\D^{T}_{\l}:=({}^{T}i^{T}_{\l})_{!}({}^{T}\cL^{T}_{\l})\j{d_{\l}}\in {}^{T}\cH^{T}_{\infty},\\
\end{eqnarray*}
\end{enumerate}
\end{defn}

Since $i^T_\l$ and ${}^Ti^T_\l$ are affine morphisms, with our convention on the perverse t-structures, we have
\begin{equation*}
\D^{T}_{\l}\in(\cH^{T}_{\infty})^{\hs}, \quad {}^{T}\D^{T}_{\l}\in\Hinfhs.
\end{equation*}
Moreover, ${}^{T}\D^{T}_{0}\in {}^{T}\cH^{T}_{\infty}$ is the monoidal unit. 


\begin{lemma}\label{l:D1}
\begin{enumerate}
\item Consider the map
\begin{equation*}
\io^{T}_{\l}: \dfrac{\Jker\bs\Jker \t^{\l}T\Jker/(\Jker T)}{\D\Ga}\incl \dfrac{\Jker\bs G\lr{\t}/(\Jker T)}{\D\Ga}.
\end{equation*}
Then $\D^{T}_{\l}$ is the image of $\io^{T}_{\l!}E\j{d_{\l}}$ in $\Hinf^{T}$.
\item Consider the map
$$\k^{T, \c}_{\l}: \dfrac{\Jker\bs(\J\t^{\l}T\J-\Jker \t^{\l}T\Jker)/(\Jker T)}{\D\Ga}\incl \dfrac{\Jker\bs G\lr{\t}/(\Jker T)}{\D\Ga}.$$ 
Then $\D^{T}_{\l}$ is the image of $\k^{T,\c}_{\l!}E\j{d_{\l}}[1]$ in $\Hinf^{T}$. 
\end{enumerate}
\end{lemma}
\begin{proof}
(1) Under the equivalence $\Kir(\Ga)\cong D^b(\Vect)$, $E$ corresponds to the skyscraper sheaf $i_{!}E$ supported at $0$, therefore under the equivalence \eqref{HT strat cat}, $\cL^{T}_{\l}$ corresponds to the image of $('\io^{T}_{\l})_{!}E$ where $'\io^{T}_{\l}$ is the embedding $\frac{\Jker\bs\Jker \t^{\l}T\Jker/(\Jker T)}{\D\Ga}\incl \frac{\Jker\bs\J\t^{\l}T\J/(\Jker T)}{\D\Ga}$. The statement for $\D^{T}_{\l}$ then follows by applying $i^{T}_{\l!}$ to $\cL^{T}_{\l}\cong ('\io^{T}_{\l})_{!}E\in \Hinf^{T}(\l)$.

(2) is proved similarly, using that $E\in \Vect$ also corresponds to $j_{!}E[1]\in \Kir(\Ga)$, for $j:\Ga-\{0\}\incl \Ga$ the open inclusion .
\end{proof}

Let $\om^{T}: \cH^{T}_{\infty}\to \cH_{\infty}$ and ${}^{T}\om: {}^{T}\cH_{\infty}\to \cH_{\infty}$ be the forgetful functors. Then it is clear that 
\begin{equation}\label{pullback delta}
\om^{T}(\D^{T}_{\l})\cong {}^{T}{\om}({}^{T}\D_{\l}).
\end{equation}
We denote both objects by $\D_{\l}\in \Him(\l)$.

\begin{lemma}\label{l:Hmon same} The objects $\D^{T}_{\l}$ (resp. ${}^{T}\D_{\l}$) for $\l\in\xcoch(T)$ generate $\cH^{T}_{\infty}$ (resp. ${}^{T}\cH_{\infty}$). The essential images of ${\om}^{T}$ and ${}^{T}\om$ are the same, and they both generate $\cH^{\mon}_{\infty}$.
\end{lemma}
\begin{proof} First we prove the generation statement for $\Hinf^{T}$. By the equivalence \eqref{HT strat cat} and the definition of $\cL^{T}_{\l}$, $\cL^{T}_{\l}$ generates $\Hinf^{T}(\l)$.  Now by Lemma \ref{l:H inf rel} and Corollary \ref{c:Kir gen}, $i^{T}_{\l!}\cL^{T}_{\l}=\D^{T}_{\l}\j{-d_{\l}}$ for all $\l\in \xcoch(T)$ generate $\Hinf^{T}$.


The generation statement for ${}^{T}\cH_{\infty}$ is proved in the same way.  Hence the essential image of $\om^{T}$ (resp. ${}^{T}\om$) is generated by $\{\om^{T}(\D^{T}_{\l})|\l\in \xcoch(T)\}$ (resp. $\{{}^{T}{\om}({}^{T}\D_{\l})|\l\in \xcoch(T)\}$). By \eqref{pullback delta}, both essential images are generated by $\{\D_{\l}|\l\in \xcoch(T)\}$.
\end{proof}

\sss{Standard objects in the completed category}

Let $\Aff$ act on $T\times \Ga$ by 
\begin{equation*}
s\cdot_{\l}(h,b)=(s^{-\l}h, s^{-1}b), \quad a\cdot(h,b)=(h,b+a), \quad h\in T, b\in \Ga, (a,s)\in \Aff=\Ga\rtimes \Gm.
\end{equation*}
We call this the $\l$-action.  If we want to emphasize the $\l$-action of $\Gm$, we use $\Gm(\l)$. Consider the map
\begin{equation*}
\k_{\l}: T\times \Ga=T\times \J/\Jker\to \frac{\Jker\bs\J \t^{\l}T\J/\Jker }{\D\Ga}
\end{equation*}
sending $(h,a)\mapsto \t^{\l}h\wt a$ for $h\in T, a\in \Ga$.  This is $\Aff$-equivariant for the $\l$-action on $T\times\Ga$. The adjoint functors $(\k_{\l}^{*},\k_{\l*})$ give an equivalence on Kirillov categories
\begin{equation*}
\k_{\l}^{*}: \Hinf(\l) \isom \Kir_{\l}(T\times\Ga) : \k_{\l*}.
\end{equation*}
Here we use subscript $\l$ in $\Kir_{\l}(T\times\Ga)$ to emphasize the $\Aff$-action on $T\times \Ga$ is the $\l$-action. Similarly we get equivalences of the monodromic and completed versions
\begin{equation*}
\wh\k_{\l}^{*}: \hHim(\l) \isom \wh\Kir^{T\lmon}_{\l}(T\times\Ga) : \wh\k_{\l*}.
\end{equation*}
We put the superscript ``$T\lmon$'' in $\Kir^{T\lmon}_{\l}(T\times\Ga)$ to emphasize that the objects there are unipotently monodromic along $T$.

By Lemma \ref{l:Kir prod A1}, we have an equivalence
\begin{equation*}
J^{\mon}_{\l}: D^{\mon}(T)\isom \Kir^{T\lmon}_{\l}(T\times\Ga) 
\end{equation*}

Note a priori Lemma \ref{l:Kir prod A1} asserts a similar equivalence with the monodromicity condition under the $\Gm$ action on $T$ via $\l$ and the $\Gm$-scaling action on $\Ga$. Here we are restricting to the full subcategories of $T$-monodromic objects. The equivalence $J^{\mon}_{\l}$ induces an equivalence of the completed versions 
\begin{equation*}
\wh J_{\l}: \wh D^{\mon}(T)\isom \wh\Kir^{T\lmon}_{\l}(T\times\Ga) 
\end{equation*}
and the equivalence
\begin{equation*}
\wh\k_{\l*}\wh J_{\l}: \wh D^{\mon}(T)\xr{\wh J_{\l}} \wh\Kir^{T\lmon}_{\l}(T\times\Ga) \xr{\wh\k_{\l*}}\hHim(\l).
\end{equation*}

Let $\wh\cL\in \wh D^{\mon}(T)$ be the universal unipotent local system on $T$. As a representation of the fundamental group $\pi_{1}(T)$, it is the completion of $E[\pi_{1}(T)]$ along the augmentation ideal. 

The following lemma essentially follows from \cite[Lemma A.3.6]{BY}.
\begin{lemma}\label{l:L unit}
    The object $\wh\cL\j{2r}$ has a canonical structure of a monoidal unit in $\wh D^{\mon}(T)$ under $!$-convolution.
\end{lemma}



\begin{defn} Let $\l\in \xcoch(T)$.
\begin{enumerate}
\item Define 
\begin{equation*}
\wh\cL_{\l}:=\wh J_{\l}(\wh\cL)\in \wh\Kir^{T\lmon}_{\l}(T\times\Ga).
\end{equation*}
\item Let $r=\dim T$. Define the corresponding {\em free-monodromic standard object} in $\hHim$ as
\begin{equation*}
\wh\D_{\l}:= i_{\l,!}\wh\k_{\l*}\wh\cL_{\l}\j{d_{\l}+2r}\in \hHim.
\end{equation*}
\end{enumerate}
\end{defn}

Unraveling the definitions, we arrive at the following description of $\wh\k_{\l*}\wh J_{\l}$. 

\begin{lemma}\label{l:fm std}
Consider the maps
\begin{equation*}
\xymatrix{\dfrac{\Jker\bs(\J \t^{\l}T\J-\Jker \t^{\l}T\Jker)/\Jker}{\D\Ga}  \ar[d]_{q_{\l}}\ar@{^{(}->}[r]^-{j_{\l}} & \dfrac{\Jker\bs\J \t^{\l}T\J/\Jker }{\D\Ga}\\
T\times (\Ga-\{0\})=T\times\Gm\ar[r]^-{m_{\l}} & T }
\end{equation*}
Here $j_{\l}$ is the open embedding, $q_{\l}$ takes $g_{1}\t^{\l}hg_{2}$ (where $h\in T, g_{1}, g_{2}\in \J$) to $(h, \psi(g_{1}g_{2}))$, and $m_{\l}$ is the homomorphism sending $(h,s)\mt hs^{-\l}$ for $h\in T$ and $s\in \Gm$. 
Then for $\cF\in \wh D^{\mon}(T)$, $\wh\k_{\l*}\wh J_{\l}(\cF)$ is the image of $j_{\l!}q_{\l}^{*}m_{\l}^{*}\cF[1]$ in the Kirillov quotient $\hHim(\l)$.  In particular,
\begin{equation*}
    \wh\D_{\l}\cong \mbox{image of $i_{\l!}j_{\l!}q_{\l}^{*}m_{\l}^{*}\wh\cL\j{d_{\l}+2r}[1]$ in $\hHim$}.
\end{equation*}
\end{lemma}

When $\l=0$, the equivalence  $\wh\k_{0*}\wh J_{0}$ has a simpler description.

\begin{lemma}\label{l:H0} Let $\io_{0}: \frac{\Jker\bs \Jker T/\Jker}{\D\Ga}\incl \frac{\Jker\bs G\lr{\t}/\Jker}{\D\Ga}$ be the closed embedding. Consider 
the composition
\begin{equation}\label{J0 alt}
\io_*: \wh D^{\mon}(T)\cong \wh D^{\mon}\left(\dfrac{\Jker\bs \Jker T/\Jker}{\D\Ga}\right)\xr{\io_{0*}} \wh D^{\mon}_{\Grot}(\frac{\Jker\bs G\lr{\t}/\Jker}{\D\Ga})\to \hHim.
\end{equation}
Here, $\io_{0*}$ naturally lifts to the $\Grot$-equivariant category because $\cF\in \wh D^{\mon}(\frac{\Jker\bs \Jker T/\Jker}{\D\Ga})$,  there is a canonical $\Grot$-equivariant structure on $\io_{0*}\cF$ (for $\Grot$ acts trivially on $T$). 

Then there is a canonical isomorphism of functors
\begin{equation*}
    \io_*\cong i_{0!}\wh\k_{0*}\wh J_{0}: \wh D^{\mon}(T)\to \hHim.
\end{equation*}
Moreover, $\io_{*}$ has a natural structure of a monoidal functor. In particular, $\wh\D_{0}\cong \io_*\wh\cL\j{2r}$ is a monoidal unit of $\hHim$.
\end{lemma}
\begin{proof}
When $\l=0$, the map $T\to T\times\Ga$ given by $h\mt (h,0)$ is $\Gm$-equivariant with respect to the trivial action on $T$ and the $\l=0$-action on $T\times\Ga$ (which is trivial on $T$). The above description of $\wh J_{0}$ follows from Example \ref{ex:Kir A1}. 

The monoidal structure on the functor \eqref{J0 alt} is an easy diagram chase. Since $\wh\cL\j{2r}$ is a monoidal unit in $\wh D^{\mon}(T)$ by Lemma \ref{l:L unit}, $\wh\D_{0}$ is a monoidal unit of $\hHim$. 
\end{proof}



Let $\pi^{T}: \frac{\Jker\bs G\lr{\t}/\Jker}{\D\Ga}\to \frac{\Jker\bs G\lr{\t}/(\Jker T)}{\D\Ga}$ be the projection map. We have $\pi^{T}_{!}: \hHim\to \Hinf^{T}$. 

\begin{lemma}\label{l:right T equiv std} We have $\pi^{T}_{!}\wh\D_{\l}\cong \D^{T}_{\l}$ for all $\l\in \xcoch(T)$.
\end{lemma}
\begin{proof}
Comparing the description of $\wh\D_{\l}$ in Lemma \ref{l:fm std} and the description of $\D^{T}_{\l}$ in Lemma \ref{l:D1}(2), the isomorphism follows from the following statement: Let $p: T\times \Gm\to \Gm$ be the projection map, then
\begin{equation}\label{push fm}
p_{!}m_{\l}^{*}\wh\cL\j{2r}\cong \un E, \quad  \mbox{(constant sheaf on $\Gm$).}
\end{equation}
Since the map $(m_{\l}, p): T\times\Gm\to T\times \Gm$ is an isomorphism, we have $p_{!}m_{\l}^{*}\wh\cL\j{2r}\cong R\G_{c}(T, \wh\cL\j{2r})\ot \un E\cong \un E$. This proves \eqref{push fm}, and (1) follows.
\end{proof}

\subsection{Convolution}
The goal here is to describe the monoidal category $\hHim$ and its equivariant version $\eqHinf$ in terms of $T$. This will be achieved in Theorem \ref{th:Hinf} and Proposition \ref{p:Hinf equiv}. Such descriptions for deeper levels were obtained by Kamgarpour and Schedler \cite{KS} categorifying Roche's results on principal series Hecke algebras. Generalization of Proposition \ref{p:Hinf equiv} in another direction (replacing $\J$ with the unipotent radical of an arbitrary parahoric $\bP$) is worked out in  Jianqiao Xia's thesis \cite{Xia}.

\begin{lemma}\label{l:conv Kir TGa}
The following diagram is canonically commutative
\begin{equation*}
\xymatrix{\wh D^{\mon}(T)\times  \wh D^{\mon}(T) \ar[r]^-{\wh J_{\l}\times \wh J_{\mu}}\ar[d]^{\star_{T}} & \wh \Kir^{T\lmon}_{\l}(T\times\Ga)\times\wh \Kir^{T\lmon}_{\mu}(T\times\Ga)\ar[d]^{\star_{T\times\Ga}}\\
\wh D^{\mon}(T)\ar[r]^{\wh J_{\l+\mu}} & \wh \Kir^{T\lmon}_{\l+\mu}(T\times\Ga)
}
\end{equation*}
Here the $\star_{T}$ and $\star_{T\times\Ga}$ are the $!$-convolutions under the group structures on $T$ and $T\times\Ga$.
\end{lemma}
\begin{proof} It suffices to prove the version before completion. We will abbreviate all the $T\lmon$ by $\mon$. We also suppress the superscripts $\mon$ from $J^{\mon}$ and its inverse $\Phi^{\mon}$.

Let $\cF, \cG\in D^{\mon}(T)$ and let $\cF_{\l}=J_{\l}(\cF)\in D^{\mon}_{\Gm(\l)}(T\times \Ga)$ and $\cG_{\mu}=J_{\mu}(\cG)\in D^{\mon}_{\Gm(\mu)}(T\times \Ga)$. 
Our goal is to give a canonical isomorphism
\begin{equation*}
\cF_{\l}\star_{T\times\Ga}\cG_{\mu}\cong J_{\l+\mu}(\cF\star_{T}\cG)\in \Kir^{\mon}_{\l+\mu}(T\times\Ga).
\end{equation*}
We can rewrite this in terms of the inverse $\Phi_{\l+\mu}$ to $J_{\l+\mu}$:
\begin{equation}\label{conv Kir goal}
\Phi(\cF_{\l}\star_{T\times\Ga}\cG_{\mu})\cong \cF\star_{T}\cG\in D^{\mon}(T).
\end{equation}


We will use $F$ denotes the Fourier transform $F: D^{\mon}_{\Gm}(X\times \Ga)\isom D^{\mon}_{\Gm}(X\times \check\Ga)$ for various stacks $X$ with $\Gm$-action.   By definition, 
\begin{equation}\label{conv Kir1}
\cF\star_{T\times\Ga}\cG= m_{T!}(p_{13}^{*}\cF_{\l}\star_{\Ga}p_{23}^{*}\cG_{\mu}).
\end{equation}
Here $p_{13}, p_{23}: T\times T\times \Ga\to T\times\Ga$ are the projections,  $\star_{\Ga}$ denotes the convolution on $T\times T\times \Ga$ in the $\Ga$-factor, and $m_{T}: T\times T\to T$ is the multiplication map.  Applying Fourier transform to both sides of \eqref{conv Kir1}, and using that Fourier transform intertwines tensor product and additive convolution,  we get
\begin{equation*}
\FT(\cF_{\l}\star_{T\times\Ga}\cG_{\mu})=\FT(m_{T!}(p_{13}^{*}\cF_{\l}\star_{\Ga}p_{23}^{*}\cG_{\mu}))\cong m_{T!}(p_{13}^{*}\FT(\cF_{\l})\ot p_{23}^{*}\FT(\cG_{\mu})).
\end{equation*}
By the construction of the inverse $\Phi$ to $J$ in Lemma \ref{l:Kir prod A1},  $p_{13}^{*}\FT(\cF_{\l})$ is $\Gm(\l)$-equivariant with restriction equal to $\cF$ over $T\times\{1\}$; $p_{23}^{*}\FT(\cG_{\mu})$ is $\Gm(\mu)$-equivariant with restriction equal to $\cG$ over $T\times\{1\}$. Therefore, $p_{13}^{*}\FT(\cF_{\l})\ot p_{23}^{*}\FT(\cG_{\mu})$ is $\Gm$-equivariant on $T\times T\times \Ga$ with respect to the $(\l,\mu)$-action $s\cdot (h_{1},h_{2}, a)=(h_{1}s^{-\l}, h_{2}s^{-\mu}, s^{-1}a)$, and its restriction to to $T\times T\times \{1\}$ is $\cF\bt\cG$. In other words we have a canonical isomorphism
\begin{equation}\label{conv Kir2}
\Phi(p_{13}^{*}\FT(\cF_{\l})\ot p_{23}^{*}\FT(\cG_{\mu}))\cong \cF\bt\cG\in D^{\mon}(T\times T).
\end{equation}
Applying $m_{T!}$ to both sides of  \eqref{conv Kir2}, using the functoriality of the equivalences $J$ and $\Phi$ (see Remark \ref{r:Kir fun}), we get a canonical isomorphism
\begin{eqnarray*}
\Phi(m_{T!}(p_{13}^{*}\FT(\cF_{\l})\ot p_{23}^{*}\FT(\cG_{\mu})))&\cong& m_{T!}\Phi(p_{13}^{*}\FT(\cF_{\l})\ot p_{23}^{*}\FT(\cG_{\mu}))\\
&\cong& m_{T!}(\cF\bt\cG) =\cF\star_{T}\cG\in D^{\mon}(T).
\end{eqnarray*}
Combined with \eqref{conv Kir1} we get \eqref{conv Kir goal}. 
\end{proof}

\begin{prop}\label{p:fm std conv}
For $\l,\mu\in \xcoch(T)$, we have a canonical isomorphism
\begin{equation*}
\wh c_{\l,\mu}: \wh\D_{\l}\star \wh\D_{\mu}\cong \wh\D_{\l+\mu}.
\end{equation*}
\end{prop}
\begin{proof}
We restrict the diagram \eqref{Hinf conv diag} to the part relevant for the calculation of $\wh\D_{\l}\star \wh\D_{\mu}$:
\begin{equation*}
\xymatrix{& \dfrac{\Jker\bs \J \t^{\l}T\J\twtimes{\Jker }\J\t^{\mu}T\J/\Jker}{\D\Ga}\ar[r]^-{m_{\l,\mu}} \ar[dl]_{p_{1}}\ar[dr]^{p_{2}}& \dfrac{\Jker\bs G\lr{\t}/\Jker}{\D\Ga}\\
\dfrac{\Jker\bs \J\t^{\l}T\J/\Jker}{\D\Ga} && \dfrac{\Jker\bs \J\t^{\mu}T\J/\Jker}{\D\Ga}
}
\end{equation*}
We have $\wh\D_{\l}\star \wh\D_{\mu}=m_{\l,\mu!}(p^{*}_{1}\wh\cL_{\l}\ot p_{2}^{*}\wh\cL_{\mu})\j{d_{\l}+d_{\mu}+4r}$.

By Lemma \ref{l:rel in conv}, the only relevant points in the image of $m_{\l,\mu}$ are in the double cosets $\J \t^{\l+\mu}T\J$. By Lemma \ref{l:H inf rel} and Corollary \ref{c:Kir gen}, $\wh\D_{\l}\star \wh\D_{\mu}$ is a successive extension of $i_{\nu!}\cF_{\nu}$ for $\cF_{\nu}\in \hHim(\nu)$. Since the fiber of $m_{\l,\mu}$ over $\J \t^\nu T\J$ is empty for $\nu\ne\l+\mu$, we must have 
\begin{equation}\label{conv supp single orbit}
\wh\D_{\l}\star \wh\D_{\mu}=i_{\l+\mu!}\cF_{\l+\mu},
\end{equation}
where $\cF_{\l+\mu}= i_{\l+\mu}^{*}(\wh\D_{\l}\star \wh\D_{\mu})\in \hHim(\l+\mu)$. 

To compute $\cF_{\l+\mu}$, we restrict $m_{\l,\mu}$ over $\J \t^{\l+\mu} T\J$. Using the translation actions of $T$ and $\Ga=\J/\Jker$, we can write
\begin{equation}\label{TT conv}
\Jker\bs  \J\t^{\l} T\J\twtimes{\Jker}\J\t^{\mu}T\J/\Jker\cong (T\times\Ga)\times (\Jker\bs  \Jker\t^{\l} \Jker\twtimes{\Jker}\Jker\t^{\mu}\Jker/\Jker)\times(T\times \Ga).
\end{equation}
Similarly, using left translation by $T\times \Ga$, we may write
\begin{equation}\label{T triple J}
\Jker\bs  \J\t^{\l} T\J\t^{\mu}T\J/\Jker\cong T\times \Ga\times (\Jker\bs  \Jker\t^{\l} \Jker\t^{\mu}\Jker/\Jker).
\end{equation}
Let $F_{\l,\mu}$ be the fiber of the multiplication map 
$$m_{\l,\mu}': \Jker\t^{\l}\Jker\twtimes{\Jker}\Jker\t^{\mu} \Jker\to \Jker\t^{\l}\Jker\t^{\mu} \Jker$$
over $\t^{\l+\mu}$.  Then the isomorphisms \eqref{TT conv} and \eqref{T triple J} imply the following diagram is Cartesian
\begin{equation*}
\xymatrix{T\times\Ga\times  F_{\l,\mu} \times T\times \Ga \ar[r]^-{\a_{\l,\mu}}\ar[d]^{m_{T\times\Ga}\c \pr} &\Jker\bs  \J\t^{\l} T\J\twtimes{\Jker}\J\t^{\mu}T\J/\Jker\ar[d]^{m_{\l,\mu}}\\
T\times \Ga\ar[r]^-{\a_{\l+\mu}} &  \Jker\bs  G\lr{\t}/\Jker
}
\end{equation*}
Here $\a_{\l,\mu}(h_{1},a_{1}, x, t_{2}, a_{2})=h_{1}\wt a_{1}x h_{2}\wt a_{2}$ (where $h_{i}\in T$, $x\in F_{\l,\mu}$, $a_{i}\in \Ga$, and $\wt a_{1}, \wt a_{2}$ are liftings of $a_{1}$ and $a_{2}$ to $\bT_{\infty}^{1}$), and $\a_{\l+\mu}=i_{\l+\mu}\k_{\l+\mu}$, so that $\a_{\l+\mu}(h,a)=\t^{\l+\mu}h\wt a$.  The left vertical arrow is the composition of the projection $\pr: T\times\Ga\times  F_{\l,\mu} \times T\times \Ga\to (T\times\Ga)^{2}$ with the group law $m_{T\times\Ga}: (T\times\Ga)^{2}\to T\times\Ga$. Therefore by proper base change
\begin{equation*}
\wh\k_{\l+\mu}^{*}\cF_{\l+\mu}=\a^*_{\l+\mu}(\wh\D_\l\star\wh\D_\mu)\cong m_{T\times\Ga, !}\pr_{!}(\wh\cL_{\l}\bt \un E\bt \wh\cL_{\mu})\j{d_{\l}+d_{\mu}+4r}.
\end{equation*}
By Lemma \ref{l:H inf conv fiber}, $F_{\l,\mu}$ is an affine space of dimension $(d_{\l}+d_{\mu}-d_{\l+\mu})/2$. The fundamental cycle of $F_{\l,\mu}$ gives a canonical isomorphism
\begin{equation*}
    R\Gamma_{c}(F_{\l,\mu},E)\cong E\j{-d_{\l}-d_{\mu}+d_{\l+\mu}}.
\end{equation*}
Therefore we have a canonical isomorphism
\begin{equation*}
\pr_{!}(\wh\cL_{\l}\bt \un E\bt \wh\cL_{\mu})\j{d_{\l}+d_{\mu}+4r}\cong \wh\cL_{\l}\bt\wh\cL_{\mu}\j{d_{\l}+d_{\mu}+4r}\ot R\Gamma_{c}(F_{\l,\mu},E)\cong \wh\cL_{\l}\bt\wh\cL_{\mu}\j{d_{\l+\mu}+4r}.
\end{equation*}
Therefore
\begin{equation}\label{F pull TGa}
\wh\k_{\l+\mu}^{*}\cF_{\l+\mu}\cong  \wh\cL_{\l}\star_{T\times\Ga}\wh\cL_{\mu}\j{d_{\l+\mu}+4r}\in \wh\Kir^{T\lmon}_{\l+\mu}(T\times\Ga).
\end{equation}
By Lemma \ref{l:conv Kir TGa}, we have a canonical isomorphism
\begin{equation*}
\wh\cL_{\l}\star_{T\times\Ga}\wh\cL_{\mu}=\wh J_{\l}(\wh\cL)\star_{T\times\Ga}\wh J_{\mu}(\wh \cL)\cong \wh J_{\l+\mu}(\wh\cL\star_{T}\wh\cL).
\end{equation*}
By Lemma \ref{l:L unit},  we have a canonical isomorphism
\begin{equation*}
\wh\cL\star_{T}\wh\cL\cong \wh\cL\j{-2r} 
\end{equation*}
we get
\begin{equation*}
\wh\cL_{\l}\star_{T\times\Ga}\wh\cL_{\mu}\cong \wh J_{\l+\mu}(\wh\cL)\j{-2r}\cong \wh\cL_{\l+\mu}\j{-2r}.
\end{equation*}
Combined with \eqref{F pull TGa} we get a canonical isomorphism
\begin{equation*}
\wh\k_{\l+\mu}^{*}\cF_{\l+\mu}\cong\wh\cL_{\l+\mu}\j{d_{\l+\mu}+2r} \quad \mbox{ or } \quad \cF_{\l+\mu}\cong\k_{\l+\mu*}\wh\cL_{\l+\mu}\j{d_{\l+\mu}+2r}.
\end{equation*}
Combined with \eqref{conv supp single orbit}, we conclude with a canonical isomorphism  $\wh\D_{\l}\star \wh\D_{\mu}\cong \wh\D_{\l+\mu}$.

\end{proof}

Similar argument applied to the convolution in the equivariant version $\eqHinf$ shows the following.

\begin{prop}\label{p:Teq conv}
    For $\l,\mu\in \xcoch(T)$, there is a canonical isomorphism
    \begin{equation*}
        c_{\l,\mu}: {}^T\D^T_{\l}\star {}^T\D^T_{\mu}\cong {}^T\D^T_{\l+\mu}.
    \end{equation*}
\end{prop}

Let $\dT$ be the torus over $E$ dual to $T$, i.e., $\xch(\dT)=\xcoch(T)$. We have a monoidal equivalence $\Rep(\dT)\cong \Vect_{\xcoch(T)}$ ($\xcoch(T)$-graded finite-dimensional vector spaces over $E$).

\begin{lemma}\label{l:eq conv asso}
    The functor
    \begin{equation*}
        {}^T\phi^T: \Rep(\dT)\to (\eqHinf)^\hs
    \end{equation*}
    sending $E(\l)$ (the one-dimensional representation of $\dT$ with weight $\l$) to ${}^T\D^T_\l$ has a canonical structure of a monoidal functor.
\end{lemma}
\begin{proof}
Using the isomorphisms constructed in Proposition \ref{p:Teq conv}, it remains to check that these isomorphisms satisfy associativity: for $\l,\mu,\nu\in\xcoch(T)$, we need to show that $\a:=c_{\l+\mu,\nu}\c(c_{\l,\mu}\star\id_\nu)$ and $\b:=c_{\l,\mu+\nu}\c(\id_\l\star c_{\mu,\nu})$ are the same isomorphisms between ${}^T\D^T_{\l}\star {}^T\D^T_\mu\star {}^T\D^T_{\nu}$ and ${}^T\D^T_{\l+\mu+\nu}$. These two isomorphisms differ only by a scalar, therefore it suffices to show that $\a$ and $\b$ agree after taking stalks at $\t^{\l+\mu+\nu}$. 

As in the proof of Proposition \ref{p:fm std conv}, we use $F_{\l,\mu}$ to denote the fiber over $\t^{\l+\mu}$ of the convolution map of the $\t^\l$ and $\t^\mu$ double cosets. Let $F_{\l,\mu,\nu}$ be the fiber over $\t^{\l+\mu+\nu}$ of the triple convolution map
\begin{equation*}
    m_{\l,\mu,\nu}: \Jker \t^\l T\Jker\times^{T\Jker}\Jker \t^\mu T\Jker\times^{T\Jker}\Jker \t^\nu T\Jker\to G\lr{\t}.
\end{equation*}
Using the realization of ${}^T\D^T_\l$ as $\io^T_{\l!}E\j{d_\l}$ as in Lemma \ref{l:D1}, the object $(i^T_{\l+\mu+\nu})^*({}^T\D^T_{\l}\star {}^T\D^T_\mu\star {}^T\D^T_{\nu})\in \Hinf^T(\l+\mu+\nu)\cong D^b(\Vect)$ is the stalk of the direct image $m_{\l,\mu,\nu!}\un E\j{d_\l+d_\mu+d_\nu}$. This stalk is $R\G_c(F_{\l,\mu,\nu})\j{d_\l+d_\mu+d_\nu}$. The stalks of both 
$\a$ and $\b$ at $\t^{\l+\mu+\nu}$ give isomorphisms 
\begin{equation*}
z_\a, z_\b: R\G_c(F_{\l,\mu,\nu})\j{d_\l+d_\mu+d_\nu}\isom E\j{d_{\l+\mu+\nu}}.
\end{equation*}
It suffices to show that $z_\a=z_\b$.

Let $d=\frac{1}{2}(d_\l+d_\mu+d_\nu-d_{\l+\mu+\nu})$. Then the above isomorphisms show that $\dim F_{\l,\mu,\nu}=d$ and it has only one $d$-dimensional irreducible component. The isomorphisms $z_\a$ and $z_\b$ each give an $E$-basis of $\cohoc{2d}{F_{\l,\mu,\nu}}(d)$. All we need to show is that the two bases are the same.

Now let $m_{12}: F_{\l,\mu\,\nu}\to G\lr{\t}/(T\Jker)$ be the multiplication of the first two factors. Let $F'_{\l,\mu,\nu}=m_{12}^{-1}(\Jker \t^{\l+\mu}T\Jker/T\Jker)$. Then $m_{12}$ restricts to a surjective map $F'_{\l,\mu,\nu}\to F_{\l+\mu,\nu}$ whose fibers are isomorphic to $F_{\l,\mu}$. We see $\dim F'_{\l,\mu,\nu}=\dim F_{\l+\mu,\nu}+\dim F_{\l,\mu}=d$, and $R\G_c(F'_{\l,\mu,\nu})\cong E\j{-2d}$. In particular, the inclusion $F'_{\l,\mu,\nu}\incl F_{\l,\mu,\nu}$ induces an isomorphism on $R\G_c$, which implied $F'_{\l,\mu,\nu}=F_{\l,\mu,\nu}$. By construction, $z_\a\in \cohoc{2d}{F_{\l,\mu,\nu}}(d)$ then corresponds to the composition
\begin{equation*}    R\G_c(F_{\l,\mu,\nu})\j{d_\l+d_\mu+d_\nu}\isom R\G_c(F_{\l+\mu,\nu})\j{d_{\l+\mu}+d_\nu}\isom E\j{d_{\l+\mu+\nu}}.
\end{equation*}
Here the first map is capping with the relative fundamental class of $m_{12}: F_{\l,\mu,\nu}\to F_{\l+\mu, \nu}$ and the second one is capping with the fundamental class of $F_{\l+\mu,\nu}$. Altogether we see that $z_\a$ is capping with the fundamental class of $F_{\l,\mu,\nu}$.

Similarly, by considering the multiplication of the second and third factors in $F_{\l,\mu,\nu}$ we get a map $m_{23}:F_{\l,\mu,\nu}\to F_{\l,\mu+\nu}$ whose fibers are isomorphic to $F_{\mu,\nu}$. The map $z_\b$ can be similarly written as a composition, from which we see that it is also given by capping with the fundamental class of $F_{\l,\mu,\nu}$. This shows that $z_\a=z_\b$ and finishes the proof.
\end{proof}

\begin{cor}\label{c:mon gr vec to Hinf} Consider the additive functor 
\begin{equation}\label{phi RT}
\phi: \Rep(\dT)\to \hHim
\end{equation}
sending $E(\l)$ to $\wh\D_{\l}$. Then $\phi$ has a canonical structure of a monoidal functor.
\end{cor}
\begin{proof}
    The image of $\phi$ lies in the abelian category $\wh \cH^{\mon, \hs}_{\infty}[r]$, therefore it suffices to check that the canonical isomorphisms $\wh c_{\l,\mu}$ constructed in Proposition \ref{p:fm std conv} satisfy associativity: for $\l,\mu,\nu\in \xcoch(T)$, the two isomorphisms $\wh \a:=\wh c_{\l+\mu,\nu}\c(\wh c_{\l,\mu}\star\id_\nu)$ and $\wh\b:=\wh c_{\l,\mu+\nu}\c(\id_\l\star \wh c_{\mu,\nu})$ between $\wh \D_\l\star \wh \D_\mu\star\wh\D_\nu$ and $\wh\D_{\l+\mu+\nu}$ are the same. Let $\wh\g:=\wh \b\c \wh \a^{-1}\in \Aut(\wh\D_{\l+\mu+\nu})\cong \Aut(\wh\cL)$. By the usual spreading-out argument, we may reduced to the case where the base field $k$ is a finite field, and the relevant sheaves all carry canonical Weil structures. The canonical maps $\wh c_{\l,\mu}$ also respect Weil structures. In particular, $\wh\g$ respects the Weil structure of the universal unipotently monodromic local system $\wh\cL$ on $T$. By \cite[Lemma A.4.3]{BY}, we have
    \begin{equation*}
        \End(\wh\cL)\cong \wh{\Sym}_E(\homog{1}{T,E})
    \end{equation*}
    Here $\wh\Sym$ means completion with respect to the augmentation ideal. Since $\homog{1}{T,E}$ has weight $-2$,  we see that the only endomorphisms of $\wh\cL$ preserving Weil structures are scalars in $E$. Since $\wh\g$ is a scalar, it is the same scalar as its $T$-equivariant version, which is shown to be $1$ in Lemma \ref{l:eq conv asso}. This finishes the proof.
\end{proof}


Recall the monoidal full embedding $\io_{*}: \wh D^{\mon}(T)\incl \hHim$ from Lemma \ref{l:H0}. 

We denote by $\wh D^{\mon}(T)\ot \Rep(\dT)$ the category of $\xcoch(T)$-graded objects in $D(T)$. This is a monoidal category. 
\begin{theorem}\label{th:Hinf}
Consider the additive functor
\begin{equation*}
\Phi: \wh D^{\mon}(T)\ot \Rep(\dT)\to \hHim
\end{equation*}
sending $\cF(\l)$ (for  $\cF\in \wh D^{\mon}(T)$ put in degree $\l\in\xcoch(T)$) to $\io_{*}\cF\star\wh\D_{\l}$.  Then $\Phi$ has a canonical monoidal structure making it a monoidal equivalence.
\end{theorem}
\begin{proof}
Monoidal structure of $\Phi$ comes from the monoidal structure of $\io_{*}$ and $\phi$ (see Corollary \ref{c:mon gr vec to Hinf}). It remains to show that $\Phi$ is an equivalence of the underlying categories.

We show $\Phi$ is essentially surjective. Let $\cS_{\l}$ be the essential image of the full embedding $i_{\l!}: \hHim(\l)\incl \hHim$. From the convolution diagram we see that convolving with $\wh\D_{\l}$ gives an equivalence
\begin{equation*}
(-)\star\wh\D_{\l}: \cS_{0}\isom\cS_{\l}.
\end{equation*}
Since $\cS_{0}$ and $\wh\D_{\l}$ are clearly in the image of $\Phi$, $\cS_{\l}$ is in the essential image of $\Phi$ for all $\l\in \xcoch(T)$.  By Lemma \ref{l:H inf rel} and Corollary \ref{c:Kir gen}, $\hHim$ is generated by $\cS_{\l}$ for all $\l\in\xcoch(T)$, therefore  $\Phi$ is essentially surjective..

To show $\Phi$ is fully faithful, it suffices to show for $\cF,\cG\in \wh D^{\mon}(T)$ and $\l,\mu\in \xcoch(T)$
\begin{equation*}
\Hom_{\hHim}(\io_{*}\cF\star\wh\D_{\l},\io_{*} \cG\star\wh\D_{\mu})=\begin{cases}\Hom_{\wh D^{\mon}(T)}(\cF,\cG) & \l=\mu,\\ 0 & \l\ne\mu.\end{cases}
\end{equation*}
If $\l=\mu$, using that $\wh\D_{\l}$ is invertible, we get
\begin{equation*}
\Hom_{\hHim}(\io_{*}\cF\star\wh\D_{\l},\io_{*} \cG\star\wh\D_{\l})\cong \Hom_{\cH_{\infty}}(\io_{*}\cF,\io_{*} \cG)=\Hom_{\wh D^{\mon}(T)}(\cF,\cG)
\end{equation*}
because $\io_{*}$ is fully faithful.  

If $\l\ne\mu$, we may convolve both objects by $\wh\D_{-\mu}$ and get
\begin{equation*}
\Hom_{\hHim}(\io_{*}\cF\star\wh\D_{\l},\io_{*} \cG\star\wh\D_{\mu})\cong \Hom_{\hHim}((\io_{*}\cF)\star\wh\D_{\l-\mu},\io_{*} \cG).
\end{equation*}
Now $\io_{*}\cG$ and $(\io_{*}\cF)\star\wh\D_{\l-\mu}$ are both $!$-extensions from disjoint strata $\Jker\bs \J T/\Jker$ and $\Jker\bs \J \t^{\l-\mu}T\J/\Jker$ of $\Jker\bs G\lr{\t}/\Jker$, their Hom space is zero. This finishes the proof that $\Phi$ is an equivalence.

\end{proof}

\begin{cor} The natural transformation of functors $i_{\l!}\to i_{\l*}\in \Fun(\hHim(\l), \hHim)$ is an isomorphism.
\end{cor}
\begin{proof} We use the notation $\cS_{\l}$ from the proof of Theorem \ref{th:Hinf}, which is the essential image of $i_{\l!}$. Theorem \ref{th:Hinf} implies that $\hHim$ is the direct sum of $\cS_{\l}$ for $\l\in \xcoch(T)$. In particular, for $i_{\l!}\cF\in \cS_{\l}$ and $i_{\mu!}\cG\in \cS_{\mu}$, and $\l\ne\mu$,  we have $\RHom(i_{\mu!}\cG,i_{\l!}\cF)=0$. This implies $i_{\l!}\cF\isom i_{\l*}\cF$ for any $\cF\in \hHim(\l)$.
\end{proof}

\sss{$T$-equivariant version}
We briefly record the description of the bi-$T$-equivariant monoidal category ${}^{T}\cH^{T}_{\infty}$ parallel to the discussions on $\hHim$. The proofs are similar to the monodromic case, which we omit.

The map ${}^{T}\io^{T}: \pt/T\to \frac{(\Jker T)\bs (\J T)/(\Jker T)}{\D\Ga}\incl \frac{(\Jker T)\bs G\lr{\t}/(\Jker T)}{\D\Ga}$ induces a monoidal full subcategory ${}^{T}\io^{T}_{*}: D(\pt/T)\cong {}^{T}\Hinf^{T}(0)\incl {}^{T}\Hinf^{T}$.  

\begin{prop}\label{p:Hinf equiv}
\begin{enumerate}
\item The natural transformation of functors ${}^{T}i^{T}_{\l!}\to {}^{T}i^{T}_{\l*}\in \Fun({}^{T}\Hinf^{T}(\l), {}^{T}\Hinf^{T})$ is an isomorphism. 
\item The functor
\begin{equation*}
{}^{T}\Phi^{T}: D(\pt/T)\ot \Rep(\dT)\to {}^{T}\Hinf^{T}
\end{equation*}
sending $\cF(\l)$ to $\io_{*}\cF\star({}^{T}\D^{T}_{\l})$ has a canonical monoidal structure making it into a monoidal equivalence.
\item The functor ${}^{T}\Phi^{T}$ restricts to a monoidal equivalence
\begin{equation}\label{RepT Hinf}
\Rep(\dT)\isom \Hinfhs
\end{equation}
sending $E(\l)$ to ${}^{T}\D^{T}_{\l}$ for all $\l\in \xcoch(T)$.
\end{enumerate}
\end{prop}

\section{The category $\DG$}
\label{s:DGr}

In this section we study the category $\DG$ that carries commuting actions of $\Hinf$ and the Satake category $\Hs$. We determine the structure of $\DG$ as module categories under both. The actions on $\DG$ give a monoidal functor $\Hs\to \Rep(\dT)$ which we interpret as taking microstalks. As a byproduct, we construct a fiber functor on the Satake category using microstalks, categorifying an old result of Evans and Mirkovic \cite{EM}.

Some of the proofs in \S\ref{ss:sph action} are postponed to the next section.

\subsection{Definitions}\label{ss:D psi Gr} 
Recall $X=\PP^{1}$. Consider the moduli stack $\Bun_{G}(\bG_{\infty}^{\psi})$. Using uniformization at $\infty$, we may write 
\begin{equation}\label{BunG double coset}
\Bun_{G}(\bG_{\infty}^{\psi})=\Jker\bs G\lr{\t}/G[t].
\end{equation}
The additive group $\Ga=\J/\Jker$ acts on $\Bun_{G}(\Jker)$ by left translation under the above double coset description. On the other hand, $\Grot$ acts on $\Bun_{G}(\Jker)$ by rotation on $\PP^{1}$.
These actions combine to give an action of $\Aff=\Ga\rtimes \Grot$ on $\Bun_{G}(\Jker)$. Consider the corresponding Kirillov category
\begin{equation*}
\cD_{\psi, \bG}=\Kir(\Bun_{G}(\bG_{\infty}^{\psi})).
\end{equation*}
The heart of the perverse t-structure of $\cD_{\psi, \bG}$ (inherited from $D(\Bun_{G}(\bG_{\infty}^{\psi}))$) is denoted $\cD^{\hs}_{\psi, \bG}$.


\sss{Relevant points}\label{sss:rel DGr}
Recall from \S\ref{ss:par} that $\cM_{\psi,\bG}=T^{*}\Bun_{G}(\bG_{\infty}^{\psi})\qq_{1}\Ga$.  By definition, the relevant points for $\Bun_{G}(\Jker)$ are the images of the projection $\cM_{\psi, \bG}\to \Bun_{G}(\Jker)$.

For $\l\in \xcoch(T)$, let  $\wh\cE_{\l}\in \Bun_{G}(\bG^{\psi}_{\infty})$ be the point corresponding to the double coset of $t^{\l}$ under the isomorphism \eqref{BunG double coset}; its image in $\Bun_{G}(\J)$ is the point $\wt\cE_{\l}$ defined in \S\ref{sss:MGr}. The $\Ga$-orbits of $\wh\cE_{\l}$ are exactly the relevant points of $\Bun_{G}(\Jker)$ under the $\Aff$-action.

Let $A_{\l}=\J\cap \Ad(t^{\l})G[t]=\Jker\cap \Ad(t^{\l})G[t]$ be the automorphism group of $\wh\cE_{\l}$, which is also the automorphism group of $\wt\cE_{\l}$. We have the locally closed substack 
\begin{equation*}
\Bun^{\l}_{G}(\J)=\J\bs \J t^{\l}G[t]/G[t]\cong \pt/A_{\l}\incl \Bun_{G}(\J)
\end{equation*}
consisting of one point $\{\wt\cE_{\l}\}$. Taking its preimage in $\Bun_{G}(\Jker)$ we get the locally closed substack which is the $\Ga$-orbits of $\wh\cE_{\l}$
\begin{equation}\label{equ_beta}
\ov\b_{\l}: \Bun^{\l}_{G}(\Jker)=\Jker\bs \J t^{\l}G[t]/G[t]\cong (\pt/A_{\l})\times \Ga
\incl \Bun_{G}(\Jker).
\end{equation}
For varying $\l\in \xcoch(T)$, these are exactly the relevant strata for $\Bun_{G}(\Jker)$. The map $\ov\b_{\l}$  sends $a\in \Ga=\J/\Jker$ to $\Jker \wt a t^{\l}G[t]$ in the double coset description (for any lifting $\wt a$ of $a$ in $\J$). Then $\ov\b_{\l}$ is $\Aff$-equivariant with respect to the tautological action on $\Ga$, and the loop rotation action on $A_{\l}$.

\begin{lemma}
$\dim A_{\l}=d_{\l}$.
\end{lemma}
\begin{proof}
The Lie algebra of $A_{\l}$ is spanned by affine root spaces in $\Lie\J$ that are not in $\J\cap \Ad(t^{\l})\J$.  Hence it has the same dimension as $\J/\J\cap \Ad(t^{\l})\J\cong \J t^{\l}\J/\J$. Comparing with \eqref{d lam} we get $\dim A_{\l}=d_{\l}$.
\end{proof}

\sss{Standard objects}
The embedding $\ov\b_{\l}$ gives  functors
\begin{eqnarray*}
\b_{\l,!}: \Kir(\Ga)\xr{\pr_{\Ga}^{*}} \Kir((\pt/A_{\l})\times \Ga) \xr{\ov\b_{\l!}}\cD_{\psi, \bG},\\
\b_{\l,*}:\Kir(\Ga)\xr{\pr_{\Ga}^{*}} \Kir((\pt/A_{\l})\times \Ga) \xr{\ov\b_{\l*}} \cD_{\psi, \bG}. 
\end{eqnarray*}
Here the pullback $\pr_{\Ga}^{*}$ is an equivalence because $A_{\l}$ is a unipotent group. 

Let $\d\in \Kir(\Ga)$ be the skyscraper sheaf at $\{0\}$. We define the standard  objects in $\DG$: 
\begin{eqnarray*}
\d_{\l}=\b_{\l,!}\d\j{-d_{\l}}. 
\end{eqnarray*}
Since the embedding $\ov\b_{\l}$ is affine, we have $\d_{\l}\in \DG^{\hs}$. By Corollary \ref{c:Kir gen}, $\{\d_{\l}\}_{\l\in\xcoch(T)}$ generate $\DG$.

\subsection{Hecke action at $\infty$} \label{subsec:heckeinf}
From the double coset description \eqref{BunG double coset}, we see there is a left action of $\Hinf$ on $\DG$ using the $\Aff$-equivariant diagram
\begin{equation}\label{Hinf action}
\xymatrix{&\Jker\bs G\lr{\t}\twtimes{\Jker}G\lr{\t}/G[t]\ar[rr]\ar[dl]\ar[dr] && \Jker \bs G\lr{\t}/G[t]\\
\dfrac{\Jker\bs G\lr{\t}/\Jker}{\D\Ga} && \Jker \bs G\lr{\t}/G[t]}
\end{equation}
We denote this action by 
\begin{equation*}
(-)\star_{\infty}(-): \Hinf\times \DG\to \DG.
\end{equation*}

\begin{lemma}\label{l:std action DGr}
\begin{enumerate}
\item Restricting the $\Hinf$-action to $\Him$, the action extends to an action of $\hHim$ on $\DG$. 
\item Moreover, for $\l,  \mu\in \xcoch(T)$, there is a canonical isomorphism
\begin{equation*}
\wh\D_{\mu}\star_{\infty}\d_{\l}\cong \d_{\l-\mu}.
\end{equation*}
\end{enumerate}
\end{lemma}
\begin{proof}
Since $\{\wh\D_{\mu}\}_{\mu\in \xcoch(T)}$ generate $\hHim$ and $\{\d_{\l}\}_{\l\in\xcoch(T)}$ generate $\DG$, (2) implies (1). We now prove (2).  

First we prove the  special case where $\mu=\l$, i.e.,
\begin{equation}\label{std conv opp}
\wh\D_{\l}\star_{\infty}\d_{\l}\cong \d_{0}.
\end{equation}
Restricting to the relevant part of the convolution diagram \eqref{Hinf action}:
\begin{equation*}
m_{\l,\l}:  \Jker\bs \J \t^{\l}T\J\twtimes{\Jker}\Jker t^{\l}G[t]/G[t] \to \Jker\bs \J \t^{\l} T \J t^{\l} G[t]/G[t].
\end{equation*}
Consider the subgroups $\bK_{+}:=\Ad(t^{\l})G[t]\cap \J=\Ad(t^{\l})G[t]\cap \Jker$ and $\bK_{-}:=\Ad(t^{\l})\J\cap \J$ of $\J$ (the first one is finite-dimensional, the second finite-codimensional). Then the multiplication map $\bK_{+}\times \bK_{-}\to \J$ is an isomorphism. This can be seen by checking at the level of Lie algebras (using that $\J$ is pro-unipotent). This implies $\J\t^{\l}T\J \t^{-\l}G[t]=\J\t^{\l}T\bK_{+}\bK_{-}\t^{-\l}G[t]=\J T \Ad(\t^{\l})\bK_{+}\Ad(\t^{\l})\bK^{-}G[t]=\J G[t]$. Moreover, identifying $\Jker t^{\l}G[t]/G[t]$ with $\Jker/\bK_{+}$ (as $\Jker$-homogeneous spaces), and identifying $\Jker\bs \J \t^{\l}T\J$ with $\bK_{-}^{\psi}\bs \J\times T$ (as $\J\times T$ homogeneous spaces, where $\bK_{-}^{\psi}=\bK_{-}\cap \Jker$), we can fit $m_{\l,\l}$ into a commutative diagram
\begin{equation*}
\xymatrix{\Jker\bs \J \t^{\l}T\J\twtimes{\Jker}\Jker t^{\l}G[t]/G[t]\ar[r]^-{m_{\l,\l}}\ar[d]^{\wr} &  \Jker\bs \J \t^{\l} T \J t^{\l} G[t]/G[t]\ar[d]^{\wr}\\
(\bK^{\psi}_{-}\bs \J\twtimes{\Jker}\J/\bK_{+})\times T\ar[r]\ar[d]^{\wr} &  \Jker\bs \J G[t]/G[t]\ar[d]^{\wr}\\
(\bK_{-}^{\psi}\bs \J/\bK_{+})\times T \ar[d]^{\wr}\ar[r]& \J/\Jker\ar[d]^{\wr}\\
\Ga\times T\ar[r]^-{\pr_{\Ga}} & \Ga 
}
\end{equation*}
Using this diagram, we see that
\begin{equation}\label{std act D0}
\wh\D_{\l}\star_{\infty}\d_{\l}\cong \b_{0!}(\pr_{\Ga,!}(\wh\cL_{\l}\j{2r}))
\end{equation}
where $\wh\cL_{\l}\in \wh \Kir^{\mon}_{\l}(T\times \Ga)$ corresponds to the universal local system $\wh\cL\in \wh D^{\mon}(T)$ under the equivalence $\wh J_{\l}$. Note that the shift-twists $\j{-d_{\l}}$ in $\d_{\l}$ and $\j{d_{\l}}$ in $\wh\D_{\l}$ cancel each other. By the functoriality in Remark \ref{r:Kir fun}, $\pr_{\Ga,!}(\wh\cL_{\l}\j{2r})$ corresponds to $R\G_{c}(T, \wh\cL\j{2r})=E$ under the equivalence $\Kir(\Ga)\cong D(\Vect)$, which implies that $\pr_{\Ga,!}(\wh\cL_{\l}\j{2r})$ is the image of $\d\in D_{\Gm}(\Ga)$ in $\Kir(\Ga)$. Hence the right side of \eqref{std act D0} is canonically isomorphic to $\b_{0!}\d=\d_{0}$. This proves \eqref{std conv opp}.

Applying $\wh\D_{-\l}\star_{\infty}$ to both sides of \eqref{std conv opp}, and using that $\wt\D_{0}$ is the monoidal unit,  we get
\begin{equation}\label{D0 conv}
\wh\D_{-\l}\star_{\infty}\d_{0}\cong \wh\D_{-\l}\star_{\infty}\wh\D_{\l}\star \d_{\l}\cong \wh\D_{0}\star_{\infty}\d_{\l}\cong \d_{\l}.
\end{equation}

For general $\l,\mu\in\xcoch(T)$, by Proposition \ref{p:fm std conv} we have a canonical isomorphism $\wh\D_{\mu}\cong \wh\D_{\l}\star\wh\D_{-\l+\mu}$. Hence we get canonical isomorphisms
\begin{equation*}
\wh\D_{\mu}\star_{\infty}\d_{\l}\cong \wh\D_{\mu}\star_{\infty}\wh\D_{-\l}\star_{\infty}\d_{0}\cong \wh\D_{\mu-\l}\star_{\infty}\d_{0}\cong \d_{\l-\mu}.
\end{equation*}
where we use \eqref{D0 conv} twice, first for $\l$ and then for $\l-\mu$.
\end{proof}

Recall from Corollary \ref{c:mon gr vec to Hinf} the monoidal functor $\phi: \Rep(\dT)\to \hHim$, which is conservative but not fully faithful.  The action of $\hHim$ on $\DG$ makes $\cD_{\psi, \bG}$ into a $\Rep(\dT)$-module category, hence also a $D^{b}(\Rep(\dT))$-module category.

\begin{prop}\label{p:DGr clean}
\begin{enumerate}
\item For any $\l\in \xcoch(T)$, the natural transformation $\b_{\l,!}\to \b_{\l,*}\in \Fun(\Kir(\Ga), \DG)$ is an isomorphism.
\item The functor
\begin{equation*}
\op_{\l\in \xcoch(T)}\b_{\l!}: \bigoplus_{\l\in\xcoch(T)}\Kir(\Ga)\cong \bigoplus_{\l\in\xcoch(T)}D^{b}(\Vect)\cong D^b(\Rep(\dT))\to \DG
\end{equation*}
is an equivalence.
\item We have a canonical  t-exact equivalence $\DG\cong D^b(\DG^\hs)$.
\end{enumerate}
\end{prop}
\begin{proof}
(1) 
Since $\Kir(\Ga)$ is generated by $\d$, equivalently we need to show
\begin{equation}\label{Hom std van}
\RHom_{\DG}(\d_{\l}, \d_{\mu})=0, \quad \forall \l\ne\mu.
\end{equation}
Since $\hHim$ acts on $\DG$ and $\wh\D_{\l}\in \hHim$ is invertible, it suffices to show
\begin{equation*}
\RHom_{\DG}(\wh\D_{\l}\star_{\infty}\d_{\l}, \wh\D_{\mu}\star_{\infty}\d_{\l})=0.
\end{equation*}
By Lemma \ref{l:std action DGr}, the above is $\RHom(\d_{0}, \d_{\mu-\l})$, which by adjunction is $\RHom_{\Kir(\Ga)}(\d, \ov\b_{0}^{!}\d_{\mu-\l})$. Now $\ov\b_{0}$ is an open embedding $\Ga\cong \Jker\bs \J G[t]/G[t]\incl \Bun_{G}(\Jker)$ and $\d_{\mu-\l}$ is supported on the complement of the open stratum $\Jker\bs \J G[t]/G[t]$, we have $\ov\b_{0}^{!}\d_{\mu-\l}=\ov\b_{0}^{*}\d_{\mu-\l}=0$ and the desired vanishing follows.

(2) Each $\b_{\l!}$ is fully faithful. By \eqref{Hom std van}, the functor $\op \b_{\l!}$ is also fully faithful. By Corollary \ref{c:Kir gen}, $\op \b_{\l!}$ is essentially surjective.

(3) follows from (2).
\end{proof}

\begin{cor}\label{c:free over lattice}
\begin{enumerate}
\item As a $D^{b}(\Rep(\dT))$-module category, $\cD_{\psi, \bG}$ is free of rank one, and $\d_{\l}$ is a generator for any $\l\in \xcoch(T)$.
\item The action of $\Rep(\dT)$ on $\DG$ perserves the heart $\cD^{\hs}_{\psi,\bG}$. As a $\Rep(\dT)$-module category, $\cD^{\hs}_{\psi, \bG}$ is free of rank one, and $\d_{\l}$ is a generator for any $\l\in \xcoch(T)$.
\end{enumerate}
\end{cor}
\begin{proof}
(1) follows directly from Proposition \ref{p:DGr clean}(2). (2) follows from Lemma \ref{l:std action DGr}(2).
\end{proof}

\subsection{$T$-equivariant version}
Consider the variant of $\DG$ by further imposing $T$-equivariance at $\infty$
\begin{equation*}
{}^{T}\DG:=\Kir(\Bun_{G}(\Jker T)).
\end{equation*}

The relevant strata for ${}^{T}\cD_{\psi, \bG}$ are 
\begin{equation*}
{}^{T}\ov\b_{\l}: \Bun^{\l}_{G}(\Jker T)=(\Jker T)\bs \J t^{\l}G[t]/G[t]\cong \pt/(A_{\l}T)\times \Ga
\incl \Bun_{G}(\Jker T).
\end{equation*}
We also have the $T$-equivariant version of $\b_{\l,!}$ and $\b_{\l,*}$:
\begin{eqnarray}\label{Tbl!}
{}^{T}\b_{\l,!}: \Kir(\pt/T\times \Ga)\xr{{}^{T}\pr_{\Ga}^{*}} \Kir(\pt/(A_{\l}T)\times \Ga) \xr{{}^{T}\ov\b_{\l!}}{}^{T}\cD_{\psi, \bG},\\
\label{Tbl*}
{}^{T}\b_{\l,*}:\Kir(\pt/T\times \Ga)\xr{{}^{T}\pr_{\Ga}^{*}} \Kir(\pt/(A_{\l}T)\times \Ga) \xr{{}^{T}\ov\b_{\l*}} {}^{T}\cD_{\psi, \bG}. 
\end{eqnarray}
If we denote by ${}^{T}\d\in \Kir(\pt/T\times \Ga)$ the skyscraper sheaf at $\pt/T\times\{0\}$,  then
\begin{equation*}
{}^{T}\d_{\l}={}^{T}\b_{\l,!}({}^{T}\d)\j{-d_{\l}}
\end{equation*}
is a canonical $T$-equivariant lifting of $\d_{\l}$.

The monoidal category ${}^{T}\Hinf^{T}$ acts on ${}^{T}\cD_{\psi ,\bG}$.  We still denote this action by $\star_{\infty}$. 

There is a $\Rep(\dT)$-action on $\DG$ via the monoidal functor $\phi: \Rep(\dT)\to \hHim$ in Corollary \ref{c:mon gr vec to Hinf}, where $E(\l)$ acts by $\wh\D_\l\star_\infty(-)$ for $\l\in \xcoch(T)$. On the other hand, Proposition \ref{p:Hinf equiv}(3) gives an equivalence $\Rep(\dT)\isom(\eqHinf)^\hs$, which in particular gives an action of $\Rep(\dT)$ on ${}^T\DG$.

\begin{lemma}\label{l:TDG as Rep T mod}
    The forgetful functor $\om: {}^T\DG\to \DG$ has a canonical structure of a functor of $\Rep(\dT)$-module categories.
\end{lemma}
\begin{proof}
    Recall the functor $\pi^T_!: \hHim\to \Hinf^T$. We have the convolution $\cH^T_\infty\times {}^T\DG\to \DG$. By Lemma \ref{l:right T equiv std}, $\pi^{T}_{!}\wh\D_{\mu}\cong \D^{T}_{\mu}$. 
    For any $\cF\in {}^T\DG$ and any $\mu\in \xcoch(T)$, we have a canonical isomorphism
    \begin{equation}\label{hDmu acts on TDG}
        \wh\D_\mu\star_\infty \om (\cF)\cong (\pi^T_!\wh\D_\mu)\star_\infty \cF\cong \D^T_\mu\star_\infty\cF\in \DG.
    \end{equation}
    Note that $\D^T_\mu$ is the image of ${}^T\D^T_\mu$ under the forgetful functor $\Hinf^T\to {}^T\Hinf^T$. Therefore
    \begin{equation*}
        \D^T_\mu\star_\infty\cF\cong \om({}^T\D^T_\mu\star_\infty \cF).
    \end{equation*}
    Combined with \eqref{hDmu acts on TDG} we get a canonical isomorphism functorial in $\cF$
    \begin{equation}\label{hDmu F}
        \wh\D_\mu\star_\infty \om (\cF)\cong\om({}^T\D^T_\mu\star_\infty \cF).
    \end{equation}
    It is easy to check that this isomorphism is compatible with convolutions of the $\wh\D_\mu$'s and the ${}^T\D^T_\mu$'s.  This proves that $\om$ has a canonical structure of a functor of $\Rep(\dT)$-modules.
\end{proof}

We give ${}^{T}\DG$ the perverse t-structure inherited from the forgetful functor to $\DG$. 

\begin{lemma}\label{l:eqDG equiv heart}
The forgetful functor $\om^\hs: {}^{T}\DG^{\hs}\to \DG^{\hs}$ is an equivalence.
\end{lemma}
\begin{proof}
    Since $T$ is connected, $\om^\hs$ is fully faithful. Since any object in $\DG^{\hs}$ is a direct sum of $\{\d_{\l}\}_{\l\in\xcoch(T)}$, and each $\d_{\l}$ is the image of ${}^{T}\d_{\l}$ under $\om^\hs$, we conclude that $\om^\hs$ is also essentially surjective, hence an equivalence.
\end{proof}

\begin{cor}\label{c:Teq act}
    For $\l,\mu\in \xcoch(T)$, there is a canonical isomorphism
\begin{equation}\label{Dmu act on dl}
{}^{T}\D_{\mu}^{T}\star_{\infty}{}^{T}\d_{\l}\cong {}^{T}\d_{\l-\mu}.
\end{equation}
\end{cor}
\begin{proof}
    Lemma \ref{l:TDG as Rep T mod} applied to the actions of $\wh\D_{\mu}$ and ${}^T\D^T_\mu$ implies a canonical isomorphism
    \begin{equation*}
        \wh\D_{\mu}\star_\infty \d_\l\cong \om({}^T\D^T_\mu\star_\infty {}^T\d_\l)
    \end{equation*}
    By Lemma \ref{l:std action DGr}, the left side above is canonically isomorphic to $\d_{\l-\mu}$. In particular both sides are in $\DG^\hs$, and we have a canonical isomorphism in $\DG^\hs$
    \begin{equation}\label{DT act dl}
    \om^{\hs}({}^T\D^{T}_{\mu}\star_\infty {}^T\d_\l)\cong \d_{\l-\mu}\cong \om^{\hs}({}^T\d_{\l-\mu}).
    \end{equation}
    Since $\om^\hs$ is an equivalence by Lemma \ref{l:eqDG equiv heart}, \eqref{DT act dl} comes from a unique isomorphism ${}^T\D^{T}_{\mu}\star_\infty {}^T\d_\l\cong {}^T\d_{\l-\mu}$, as desired.

\end{proof}

\begin{cor}\label{c:TDG eq DG as Rep mod}
    The action of $(\eqHinf)^\hs\cong \Rep(\dT)$ on ${}^T\DG$ preserves ${}^T\DG^\hs$.  The forgetful functor $\om^\hs: {}^T\DG^\hs\to \DG^\hs$ is an equivalence of $\Rep(\dT)$-module categories.
\end{cor}
\begin{proof}
    The fact that $(\eqHinf)^{\hs}$ preservves ${}^T\DG^{\hs}$ follows from Corollary \ref{c:Teq act} and the fact that $\{{}^T\d_\l\}_{\l\in\xcoch(T)}$ generate ${}^T\DG^{\hs}$. The rest follows from Lemma \ref{l:TDG as Rep T mod} and Lemma \ref{l:eqDG equiv heart}.
\end{proof} 

\begin{cor}\label{c:TDGr clean}
\begin{enumerate}
\item For any $\l\in \xcoch(T)$, the natural transformation ${}^{T}\b_{\l,!}\to {}^{T}\b_{\l,*}\in \Fun(\Kir((\pt/T)\times\Ga), {}^{T}\DG)$ is an isomorphism.
\item The functor
\begin{equation*}
\op_{\l\in \xcoch(T)}{}^{T}\b_{\l!}: \bigoplus_{\l\in\xcoch(T)}\Kir((\pt/T)\times \Ga)\cong \bigoplus_{\l\in\xcoch(T)}D(\pt/T)\to {}^{T}\DG
\end{equation*}
is an equivalence.

\item\label{TDG free over eqHinf} As a $\eqHinf$-module category, ${}^{T}\cD_{\psi, \bG}$ is free of rank one, and ${}^{T}\d_{\l}$ is a generator for any $\l\in \xcoch(T)$.

\item As a $\Hinfhs$-module category, ${}^{T}\DG^{\hs}$ is free of rank one, and  ${}^{T}\d_{\l}$ is  a generator for any $\l\in \xcoch(T)$.

\end{enumerate}
\end{cor}
\begin{proof}
(1) follows from the non-$T$-equivariant version Proposition \ref{p:DGr clean}(1), because the forgetful functor $\DG\to {}^{T}\DG)$ is conservative. 

(2) follows from (1) as in the proof of Proposition \ref{p:DGr clean}(2). 

(3) Proposition \ref{p:Hinf equiv}(2) identifies $\eqHinf$ with $D(\pt/T)\ot \Rep(\dT)$ and part (2) above identifies ${}^T\DG$ also with $D(\pt/T)\ot \Rep(\dT)$. By 
Corollary \ref{c:Teq act}, the action functor $(-)\star_{\infty}{}^{T}\d_{0}: \eqHinf\to {}^{T}\DG$ can be identified with the auto-equivalence of $D(\pt/T)\ot \Rep(\dT)$ given by the identity on $D(\pt/T)$ and the inversion map on $\dT$. This shows that ${}^{T}\DG$ is free of rank one over $\eqHinf$ with generator ${}^T\d_0$. For general $\l$, ${}^{T}\d_{\l}\cong {}^T\D^T_{-\l}\star_\infty {}^T\d_0$ by Corollary \ref{c:Teq act}, and ${}^T\D^T_{-\l}$ is invertible, hence ${}^{T}\d_{\l}$ is also a free generator for ${}^T\DG$.

(4) The action map $(-)\star_{\infty}{}^{T}\d_{\l}: \eqHinf\to {}^{T}\DG$ is an equivalence by (3) and is t-exact by Corollary \ref{c:Teq act}. Therefore ${}^{T}\DG^{\hs}$ is free of rank one over $\Hinfhs$, with generator ${}^{T}\d_{\l}$.

\end{proof}


\subsection{Spherical Hecke action}\label{ss:sph action}

Consider the affine Grassmannian $\Gr=G\lr{t}/G\tl{t}$ with the action of $G\tl{t}\rtimes \Grot$ (where $\Grot$ acts on $G\tl{t}$ by scaling $t$) by left translation and loop rotation. 
Let $\cH^{\sph}_{0}$ be the category of $G\tl{t}$-equivariant constructible complexes on $\Gr$ supported on finitely many $G\tl{t}$-orbits (the subscript $0$ in $\cH^{\sph}_{0}$ emphasizes that $t$ is the local coordinate at $0\in X$). Note that all objects in $\cH^{\sph}_{0}$ are automatically $\Grot$-monodromic. Let $\cH^{\sph, \hs}_{0}\subset \cH^{\sph}_{0}$ be the heart of the perverse t-structure (as sheaves on $\Gr$). The geometric Satake equivalence of Lusztig, Ginzburg, Beilinson-Drinfeld and Mirkovic-Vilonen gives a canonical symmetric monoidal structure on $\cH^{\sph, \hs}_{0}$ (extending the monoidal structure given by the convolution product) and an equivalence of symmetric monoidal categories
\begin{equation*}
\Sat: \cH^{\sph, \hs}_{0}\cong \Rep(\dG).
\end{equation*}

Via modification of bundles at $0\in X$, we get an action of $\Hs$ on $\cD_{\psi, \bG}$. More precisely, consider the diagram
\begin{equation}\label{Hk0}
\xymatrix{& \Hk_{G, 0}(\bG^{\psi}_{\infty})\ar[rr]^{\ev_{0}}\ar[dl]_-{h_{1}}\ar[dr]^{h_{2}} && G\tl{t}\bs G\lr{t}/G\tl{t}\\
\Bun_{G}(\bG^{\psi}_{\infty})  && \Bun_{G}(\bG^{\psi}_{\infty})
}
\end{equation}
where $\Hk_{G, 0}(\bG^{\psi}_{\infty})$ classifies a pair of $G$-bundles on $X$ with $\bG^{\psi}_{\infty}$ level structures $\cE_{1},\cE_{2}\in \Bun_{G}(\bG^{\psi}_{\infty})$ and an isomorphism $\t: \cE_{1}|_{X\bs \{0\}}\isom \cE_{2}|_{X\bs \{0\}}$. The map $h_{i}$ records $\cE_{i}$, and $\ev_{0}$ records the relative position of $\cE_{1}$ and $\cE_{2}$ on the formal disk around $0$. The maps $h_{1},h_{2}$ are $\Aff$-equivariant.  The map $\ev_{0}$ is $\Ga$-invariant, and it descends to a $\Grot$-equivariant map
\begin{equation*}
\ov\ev_{0}: \Hk_{G, 0}(\J)\to G\tl{t}\bs G\lr{t}/G\tl{t}.
\end{equation*}
The construction of integral transform on Kirillov categories \S\ref{sss:int trans} gives a functor
\begin{equation*}
D_{\Grot\lmon}(\Hk_{G, 0}(\J))\to \End(\cD_{\psi, \bG}).
\end{equation*}
Precomposing it with $\ov\ev_{0}^{*}$ we get a functor
\begin{equation*}
\cH^{\sph}_{0}\xr{\ov\ev_{0}^{*}}D_{\Grot\lmon}(\Hk_{G, 0}(\J))\to \End(\cD_{\psi, \bG}).
\end{equation*}
This defines an action of $\cH^{\sph}_{0}$ on $\cD_{\psi,\bG}$. We denote the action of $\cK\in\cH^{\sph}_{0}$  on $\cD_{\psi,\bG}$ by $(-)\star_{0}\cK$, which is given by the formula  $(-)\star_{0}\cK=h_{2!}(h_{1}^{*}(-)\ot\ev^{*}_{0}\cK)$.

Using the coset descriptions, we can rewrite the  diagram \eqref{Hk0} as
\begin{equation}\label{Hk0 cosets}
\xymatrix{& \Jker\bs G\lr{\t}\twtimes{G[t]}G[t,t^{-1}]/G[t] \ar[r]^-{\pr_{2}}\ar[dl]_{\pr_{1}}\ar[dr]^{m} &  G[t]\bs G[t,t^{-1}]/G[t] \ar[r] & G\tl{t}\bs G\lr{t}/G\tl{t}\\
\Jker\bs G\lr{\t}/G[t] && \Jker\bs G\lr{\t}/G[t]}
\end{equation}
Here $\pr_{1}, \pr_{2}$ are projections to the first and second factors, and $m$ is the multiplication map on $G\lr{\t}$.

Similarly, for the $T$-equivariant version, we have an action of $\cH^{\sph}_{0}$ on ${}^{T}\cD_{\psi, \bG}$, still denoted by $\star_{0}$.

\sss{The monoidal functor $\rho$} 
\label{sss:def rho}

The action of $\hHim$ on $\DG$ commutes with the action of $\Hs$ since they are defined by modifications at different points. The spherical action thus gives a monoidal functor
\begin{equation}\label{predef rho}
\Hs\to \End_{\hHim}(\DG)^{\rev}\to \End_{D^b(\Rep(\dT))}(\cD_{\psi, \bG})^{\rev}.
\end{equation}
Here $(-)^{\rev}$ means reversing the order of the monoidal operation. 
By Corollary \ref{c:free over lattice}, $\DG$ is free of rank one under $D^b(\Rep(\dT))$, with generator $\d_{0}$. Therefore we get a canonical monoidal equivalence
\begin{equation}\label{End DG}
\End_{D^b(\Rep(\dT))}(\cD_{\psi, \bG})^{\rev}\cong D^b(\Rep(\dT))
\end{equation}
The equivalence \eqref{End DG} sends $F: \cD_{\psi, \bG}\to \cD_{\psi, \bG}$ to the unique object $V=\op_{\l\in\xcoch(T)}V(\l)\in D^b(\Rep(\dT))$ defined by
\begin{equation*}
    V(\l)=\RHom_{\DG}(\d_\l, F(\d_0))\in D^b(\Vect).
\end{equation*}
Composing \eqref{predef rho} and \eqref{End DG} we get a monoidal functor $\rho$
\begin{equation}\label{def rho}
\r: \Hs\to D^b(\Rep(\dT)).
\end{equation}
Unraveling definitions, for $\cK\in \Hs$, we have
\begin{equation*}
\r(\cK)=\op_{\l\in \xcoch(T)}\r_{\l}(\cK)\in D^b(\Rep(\dT)),
\end{equation*}
where $\r_{\l}(\cK)\in D^b(\Vect)$ is defined by
\begin{equation}\label{D0K}
\r_{\l}(\cK)=\RHom_{\DG}(\d_\l, \d_{0}\star_{0}\cK).
\end{equation}

\sss{Equivariant version of $\r$}\label{sss:eq rho}
For the $T$-equivariant version, the commuting action of $\eqHinf$ and $\Hs$ on ${}^{T}\DG$ gives a monoidal functor
\begin{equation}\label{def Trho}
{}^T\r: \Hs\to \End_{\eqHinf}({}^{T}\DG)^{\rev}\cong \eqHinf.
\end{equation}
Here we use the fact that ${}^T\DG$ is free of rank one over $\eqHinf$ by Corollary \ref{c:TDGr clean}\eqref{TDG free over eqHinf}. Writing ${}^T\r=\op_{\l\in \xcoch(T)}{}^T\r_\l$ where ${}^T\r_\l: \Hs\to D(\pt/T)$, we have
\begin{eqnarray}\label{compute Trho}
    {}^T\r_\l=\RHom_{{}^T\DG}({}^T\d_\l, {}^T\d_0\star_0\cK)\in D(\pt/T).
\end{eqnarray}
Thus $\rho$ is naturally isomorphic to the composition
\begin{equation}\label{Trho and rho}
\Hs\xr{{}^T\r}\eqHinf\cong D^b(\Rep(\dT))\ot D(\pt/T)\xr{p^*} D^b(\Rep(\dT))
\end{equation}
where the last functor is the pullback along $p:\pt\to \pt/T$.

\begin{theorem}\label{th:rho comm}
\begin{enumerate}
\item The action of $\Hsh$ on $\cD_{\psi, \bG}$ preserves the heart $\cD^{\hs}_{\psi, \bG}$. In particular, $\r|_{\Hsh}$ gives a monoidal functor
\begin{equation*}
\r^{\hs}: \Hsh\cong \Rep(\dG)\to \End_{\Rep(\dT)}(\DG^{\hs})^{\rev}\cong \Rep(\dT).
\end{equation*}
Here the last equivalence uses Corollary \ref{c:free over lattice}(2).
\item  There is a canonical isomorphism of monoidal functors
\begin{equation*}
\r^\hs\cong \inv\c\Res^{\dG}_{\dT}: \Rep(\dG)\xr{\Res^{\dG}_{\dT}} \Rep(\dT)\xr{\inv}\Rep(\dT).
\end{equation*}
Here $\inv: \Rep(\dT)\to \Rep(\dT)$ is the tensor equivalence induced by the inversion on $\dT$.
\end{enumerate}
\end{theorem}

The proof will be given in \S\ref{ss:pf sph action}, after some preparation on the Iwahori-level analogue of $\DG$. Another possible approach to Theorem \ref{th:rho comm}(2) is to prove that $\rho^\hs$ is compatible with the commutativity constraints of $\Hsh$ and $\Rep(\dT)$ using Hecke action at multiple moving points on $\PP^1$.

\subsection{Microlocalization of $\DG$}
In this subsection we work over $k=\CC$ and the analytic topology.

By the discussion in \S\ref{s:micro}, we have a microlocalization functor
\begin{equation}\label{micro DGr}
M_{\bG}:  \cD_{\psi, \bG}\to \muSh(T^{*}\Bun_{G}(\bG_{\infty}^{\psi})\qq_{1}\Ga)=\muSh(\cM_{\psi,\bG}).
\end{equation}
By Lemma \ref{l:Gr ASF discrete}, $\cM_{\psi,\bG}$ is a discrete set indexed by $\xcoch(T)$. Let $\muSh_{fs}(\cM_{\psi,\bG})\subset \muSh(\cM_{\psi,\bG})$ be the full subcategory of objects supported on finitely many points. Then we have a canonical equivalence
\begin{equation*}
\muSh_{fs}(\cM_{\psi,\bG})\cong D^{b}(\Vect_{\xcoch(T)}).
\end{equation*}
Here $\Vect_{\xcoch(T)}$ is the category of finite-dimensional $\xcoch(T)$-graded vector spaces. Composing with this equivalence, we can rewrite the microlocalization functor $M_{\bG}$ as
\begin{equation*}
M_{\bG}: \cD_{\psi, \bG}\to \muSh_{fs}(\cM_{\psi,\bG})\cong D^{b}(\Vect_{\xcoch(T)}).
\end{equation*}

\sss{Thick Grassmannian}\label{sss:thick Gr}
To study the microlocalization functor, it will be convenient to work with a uniformization of $\Bun_{G}(\Jker)$ by a pro-smooth scheme. 

Let $\Bun_{G,\wh\infty}$ be the moduli space of pairs $(\cE,\s)$, where $\cE$ is a $G$-bundles on $\PP^{1}$ and $\s: \cE|_{D_{\infty}}\isom \cE_{\triv}|_{D_{\infty}}$ is a trivialization of $\cE$ over the formal disk $D_{\infty}$ at $\infty$. We may identify $\Bun_{G,\wh\infty}(k)$ with cosets
\begin{equation*}
\Bun_{G,\wh\infty}\cong  G\lr{\t}/G[t].
\end{equation*}
The scheme $\Bun_{G,\wh\infty}$ has an open covering by projective limits of smooth varieties with smooth transition maps. 

There is a left action of the loop group $G\lr{\t}$ on $\Bun_{G,\hi}$, whose quotient by $\J$ (resp. $\Jker$) is canonical identified with $\Bun_{G}(\J)$ (resp. $\Bun_{G}(\Jker)$). The quotient map $\pi_{\J}: \Bun_{G,\hi}\to \Bun_{G}(\J)$ is a $\J$-torsor. 
Recall we have the locally closed substack $\Bun^{\l}(\J)=[\{\wt\cE_{\l}\}/A_{\l}]$, which corresponds to the double coset $\J t^{\l}G[t]$. Let
\begin{equation*}
\Sig_{\l}=\pi^{-1}_{\J}(\Bun^{\l}(\J))=\J t^{\l}G[t]/G[t]\subset \Bun_{G,\hi}.
\end{equation*}
Then $\Sig_{\l}$ is pro-smooth. Similarly define
\begin{equation*}
\Sig^{\psi}_{\l}= \Jker t^{\l}G[t]/G[t]\subset \Bun_{G,\hi}
\end{equation*}
which is the preimage of $[\{\wh\cE_{\l}\}/A_{\l}]\subset \Bun_{G}(\Jker)$.

The proof of the following result uses ideas from the work of Evans--Mirkovic \cite{EM} that shows the vanishing of microstalks with the help of torus action.

\begin{theorem}\label{th:MG equiv} 
\begin{enumerate}
    \item For $\l\in \xcoch(T)$, $M_{\bG}(\d_\l)$ is supported at  $(\cE_{\l}, \psi_{1} dt)\in \cM_{\psi,\bG}$ with one-dimensional microstalk in degree zero. 
    
    \item The microlocalization functor $M_{\bG}$ in \eqref{micro DGr} is an equivalence
\begin{equation*}
M_{\bG}: \DG\isom \muSh_{fs}(\cM_{\psi,\bG})\cong D^{b}(\Vect_{\xcoch(T)}).
\end{equation*}
\end{enumerate}
\end{theorem}
\begin{proof}
In view of Proposition \ref{p:DGr clean}, (1) implies (2). 


We work with the uniformized pro-smooth moduli space $\Bun_{G,\hi}$ and $\Jker\rtimes\Gm$-equivariant complexes on it. The microstalk of $\cF\in \DG$ at $(\cE_{\nu}, \psi=\psi_{1} dt)$ is computed as follows. 
We have the microlocalization functor along $\Sig^{\psi}_{\nu}$:
\begin{equation*}
\mu_{\nu}:=\mu_{\Sig_{\nu}^{\psi}/\Bun_{G,\hi}}: D_{\Jker\rtimes\Gm}(\Bun_{G,\hi})\to D_{\Jker\rtimes\Gm}(T^{*}_{\Sig^{\psi}_{\nu}}\Bun_{G,\hi}).  
\end{equation*}
Viewing $\psi=\psi_1 dt$ as an element in the conormal fiber of $T^{*}_{\Sig^{\psi}_{\nu}}\Bun_{G,\hi}$ at $t^{\nu}\in \Sig^{\psi}_{\nu}$, the microstalk $M_{\bG}(\cF)_{\nu}$ is the stalk of $\mu_{\nu}(\cF)$ at $(t^{\nu}, \psi)$.
 
By Lemma \ref{l:res slice}, due to the $\J$-equivariance, the restriction of $\mu_{\nu}$ to $t^{\nu}$ can be calculated by restricting to a transversal slice $\wt S_{\nu}$ to $\Sig_{\nu}^{\psi}$ through $t^{\nu}$ and microlocalizing along the point $t^{\nu}$. A transversal slice to $\Sig_{\nu}$ at $t^{\nu}$ is given by
\begin{equation*}
S_{\nu}=U_{\nu}[t]t^{\nu}G[t]/G[t]\subset \Bun_{G,\hi}.
\end{equation*}
Here $U_{\nu}\subset G$ is the unipotent subgroup whose Lie algebra is the direct sum of root spaces $\frg_{\a}$ such that $\j{\a,\nu}>0$. The fact that $S_{\nu}$ is a transversal slice to $\Sig_{\nu}$ at $t^{\nu}$ can be checked by computing the tangent spaces.  To get a transversal slice to $\Sig^{\psi}_{\nu}$, we choose a line $\fra\subset \frt=\Lie T$ such that $\psi|_{\fra}: \fra\to \AA^{1}$ is an isomorphism. Let $\wt\fra\subset \bT_{\infty}^{1}$ be the subgroup consisting of $1+a\t$ for $a\in \fra$. Then
\begin{equation*}
\wt S_{\nu}=\wt\fra\cdot S_{\nu} \cong \fra\times S_{\nu}
\end{equation*}
is a transversal slice to $\Sig^{\psi}_{\nu}$ at $t^{\nu}$. The covector $\psi_{1} dt$ on $(N_{\Sig^{\psi}_{\nu}}\Bun_{G,\hi})_{t^{\nu}}$, after identifying the latter with $T_{t^{\nu}}\wt S_{\nu}$, is proportional to the projection to the first factor $\pr_{\fra}: T_{t^{\nu}}\wt S_{\nu}\to T_{0}\wt\fra$. Therefore we have
\begin{equation}\label{ms nu}
M_{\bG}(\cF)_{\nu}=\FT(\Sp_{t^{\nu}/\wt S_{\nu}}(\cF))|_{\psi_1 dt}=\phi(\pr_{\fra*}\Sp_{t^{\nu}/\wt S_{\nu}}(\cF))
\end{equation}
where $\phi(-)$ denotes the monodromic vanishing cycles on the line $\fra$ towards zero.

Let $\Gm(-\nu)=\Gm$, with the action on $\Bun_{G,\hi}$ by right translation (equivalently adjoint action) via $-\nu: \Gm\to T$. This action stabilizes $\wt S_{\nu}$ and contracts $\wt S_{\nu}$ to $\fra_{\nu}:=t^{\nu}\cdot \wt\fra\cong \fra$. Then $\Gm(-\nu)$ also acts on $T_{t^{\nu}}\wt S_{\nu}$ and contracts it to $T_{t^{\nu}}(t^{\nu}\cdot\wt\fra)\cong T_{0}\fra$. Let $k: \fra_{\nu}\incl \wt S_{\nu}$ be the inclusion, whose tangent map at $t^{\nu}\in \fra_{\nu}$ is $k_{0}: T_{t^{\nu}}(\fra_{\nu})\cong T_{0}\fra\incl T_{t^{\nu}}\wt S_{\nu}$. Moreover, $\cF$ is $T$-equivariant, hence $\Gm(-\nu)$-equivariant, and so is $\Sp_{t^{\nu}/\wt S_{\nu}}(\cF)$. By the contraction principle, we have a canonical isomorphism
\begin{equation}\label{sp proj fra}
\pr_{\fra*}\Sp_{t^{\nu}/\wt S_{\nu}}(\cF)\cong k_{0}^{*}\Sp_{t^{\nu}/\wt S_{\nu}}(\cF).
\end{equation}
By Lemma \ref{l:sp contr}, we have
\begin{equation}\label{sp res k0}
k_{0}^{*}\Sp_{t^{\nu}/\wt S_{\nu}}(\cF)\cong \Sp_{t^{\nu}/\fra_{\nu}}(k^{*}\cF).
\end{equation}
The deformation from $\fra_{\nu}$ to its tangent space $T_{t^{\nu}}\fra_{\nu}$ can be trivialized using the $\Grot$-action on $\fra_{\nu}\cong \fra$ by scaling.  Using that $\cF$ is $\Grot$-equivariant, we can identify $\Sp_{t^{\nu}/\fra_{\nu}}(k^{*}\cF)$ with $k^{*}\cF$, viewed as an object in $D_{\Grot}(\fra)$. Combining \eqref{ms nu}, \eqref{sp proj fra} and \eqref{sp res k0}, we have
\begin{equation}\label{ms nu final}
M_{\bG}(\cF)_{\nu}\cong \phi(k^{*}\cF)=\phi(\cF|_{\fra_{\nu}}).
\end{equation}

When $\cF=\d_{\nu}$, $\cF|_{\fra_{\nu}}$ is the skyscraper sheaf at 
$0$ by definition, therefore $M_{\bG}(\d_{\nu})_{\nu}$ is one-dimensional.

Now consider the case $\cF=\d_{\l}$ where $\l\ne\nu$. Since $\fra_{\nu}=t^{\nu}\cdot \wt\fra$ does not intersect $\Sig^{\psi}_{\l}$, which is the support of $\cF$, we see that $\d_{\l}|_{\fra_{\nu}}=0$. We conclude from \eqref{ms nu final} that $M_{\bG}(\cF)_{\nu}=0$.
\end{proof}

\subsection{Microlocal fiber functor for the Satake category}\label{ss:mic fiber functor}
The focus of this subsection is the Satake category $\Hs$ itself. We give an interpretation of the functor \eqref{def rho} as a microstalk functor on the affine Grassmannian. The main result in this subsection are not used elsewhere in the paper but we believe it is of independent interest.

Since the affine Grassmannian cannot be written as an increasing union of smooth varieties in general, we will need to embed it into something smooth in order to make sense of microstalks. 

\sss{Embedding $\Gr$ into the thick Grassmannian}

Viewing the affine Grassmannian $\Gr=G\lr{t}/G\tl{t}$ as classifying $G$-bundles on $\PP^{1}$ with trivialization on $\PP^{1}\bs\{0\}$, we get an ind-closed embedding into the thick Grassmannian introduced in \S\ref{sss:thick Gr}
\begin{equation}\label{emb Gr}
i_{\Gr}: \Gr\incl \Bun_{G,\wh\infty}.
\end{equation}
Using the coset descriptions, $i_{\Gr}$ can be written as
\begin{equation*}
G\lr{t}/G\tl{t}\cong G[t,t^{-1}]/G[t]\incl G\lr{\t}/G[t]
\end{equation*}
corresponding to the inclusion $G[t,t^{-1}]\subset G\lr{\t}$. For an object $\cK\in D_{G\tl{t}}(\Gr)=\cH^{\sph}_{0}$, $i_{\Gr, *}\cK\in D(\Bun_{G,\wh\infty})$, and it makes sense to talk about microstalks of $i_{\Gr, *}\cK$ since $\Bun_{G,\wh\infty}$ is pro-smooth.

Recall the subscheme $\Sig_{\l}\subset \Bun_{G,\hi}$ from \S\ref{sss:thick Gr}. A simple calculation of tangent spaces shows:

\begin{lemma}\label{l:trans slice} For $\l\in \xcoch(T)$, $\Sig_{\l}$ is a transversal slice to $\Gr_{\l}$ at $t^{\l}$, both viewed as subschemes of the pro-smooth scheme $\
\Bun_{G,\hinf}$.
\end{lemma}

Let $\Sig^{\Gr}_{\l}=\Sig_{\l}\cap \Gr$ and $\s_{\l}: \Sig^{\Gr}_{\l}\incl \Gr$ be the locally closed embedding. The linear character $\psi$ on $\J$ induces a function
\begin{equation*}
\wt\psi_{\l}: \Sig_{\l}\cong \J/A_{\l}\xr{\psi} \Ga,
\end{equation*}
and restricts to a function
\begin{equation*}
\psi_{\l}: \Sig^{\Gr}_{\l}\subset \Sig_{\l}\xr{\wt\psi_{\l}} \Ga.
\end{equation*}
The function $\psi_{\l}$ is $T$-invariant with respect to the left translation action on $\Sig^{\Gr}_{\l}$. 

Recall the functors $\r_\l$ and its equivariant version ${}^T\r_\l$ defined in \S\ref{sss:def rho} and \S\ref{sss:eq rho}.

\begin{lemma}\label{l:rho lam}
There is a canonical isomorphism of functors
\begin{equation}\label{rho as phi}
{}^{T}\r_{\l}\j{2d_\l}\cong \phi(\psi_{\l!}\s_{\l}^{*}(-)): \Hs\to D(\pt/T)
\end{equation}
Here $\phi(-)$ denotes the monodromic vanishing cycles functor on $D^{\mon}(\pt/T\times \Ga)$. 

In particular, there is a canonical isomorphism for the non-equivariant version
\begin{equation*}
\r_{\l}\j{2d_\l}\cong \phi(\psi_{\l!}\s_{\l}^{*}(-)): \Hs\to D^b(\Vect)
\end{equation*}
\end{lemma}
\begin{proof}
We use \eqref{compute Trho} to compute ${}^T\r_\l(\cK)$ for $\cK\in \Hs$.

We first compute ${^T}\d_0\star_0\cK$. By definition, ${}^{T}\d_{0}$ is the extension by zero from the base point of $\Bun_{G}(\Jker)$. We restrict the diagram \eqref{Hk0 cosets} defining the spherical Hecke action to the locus where the first factor $\Jker\bs G\lr{\t}$ is replaced by $\Jker\bs \Jker G[t]$, the support of ${}^{T}\d_{0}$:
\begin{equation}\label{Hk0 act on D0}
\xymatrix{\Jker\bs \Jker G[t] \twtimes{G[t]}G[t,t^{-1}]/G[t]\ar[dr]^{m}\ar[rr]^-{\pr_{2}} &&  G\tl{t}\bs G\lr{t}/G\tl{t}\\
& \Jker\bs G\lr{\t}/G[t]
}
\end{equation}
Now using the isomorphism $\Jker\bs \Jker G[t] \twtimes{G[t]}G[t,t^{-1}]/G[t]\cong G[t,t^{-1}]/G[t]\cong \Gr$, the map $m$ becomes the composition of the embedding $i_{\Gr}$ followed by the quotient map $\pi: \Bun_{G,\hinf}\to \Bun_{G}(\Jker)$:
\begin{equation*}
\Gr=G[t,t^{-1}]/G[t]\xr{i_{\Gr}} G\lr{\t}/G[t]=\Bun_{G,\hinf}\xr{\pi} \Jker\bs G\lr{\t}/G[t]=\Bun_{G}(\Jker).
\end{equation*}
The map $\pr_{2}$ in \eqref{Hk0 act on D0} becomes the projection $\pr: G[t,t^{-1}]/G[t]=\Gr\to G\tl{t}\bs \Gr$. Therefore
\begin{equation}\label{d0K}
{}^{T}\d_{0}\star_{0}\cK\cong (\pi \c i_{\Gr})_{!}\cK.
\end{equation}
Here we view $\cK\in \Hs$ as a $T$-equivariant complex of sheaves on $\Gr$. 

To compute ${}^{T}\r_{\l}(\cK)=\RHom({}^T\d_\l, {}^{T}\d_{0}\star_{0}\cK)$, we need to restrict ${}^{T}\d_{0}\star_{0}\cK$ via the functor
\begin{equation*}
    {}^{T}b^{*}_{\l}: {}^T\DG=\Kir(\Bun_{G}(\Jker T))\xr{{}^T\ov\b_\l} \Kir(\pt/(A_\l T)\times\Ga)\xr{\pr^*}\Kir(\pt/T\times \Ga)
\end{equation*} 
Consider the following diagram where all squares are Cartesian
\begin{equation}\label{slice diagram}
\xymatrix{
& \Sig^{\Gr}_{\l} \ari[r]^-{\s_{\l}} \ar[ddl]_-{\psi_{\l}}\ari[d]^{i_{\Sig_{\l}}} & \Gr\ari[d]^{i_{\Gr}}\\
& \Sig_{\l}\ari[r]^-{\wt\s_{\l}}\ar[dl]_{\wt\psi_{\l}}\ar[d]^{\ov\psi_{\l}=\pr\c\wt\psi_\l} & \Bun_{G,\hi}\ar[d]^{\pi}\\
\Ga \ar[r]^-{\pr} & \Ga\times (\pt/A_{\l}) \ari[r]^{\ov\b_{\l}} & \Bun_{G}(\Jker)
}
\end{equation}
By proper base change
\begin{equation*}
{}^{T}\b^{*}_{\l}(\pi\c i_{\Gr})_{!}(\cK)\cong\pr^{*}\ov\psi_{\l!}\wt\s_{\l}^{*}i_{\Gr!}\cK\cong \pr^{*}\pr_{!}\psi_{\l!}\s_{\l}^{*}\cK\in \Kir(\pt/T\times \Ga).
\end{equation*}
Since $\pr$ is the quotient map by the trivial action of the unipotent group $A_{\l}$, we have $\pr^{*}\pr_{!}\cong \id\ot \cohoc{*}{A_{\l}}=\id\j{-2d_{\l}}$. Combining with \eqref{d0K} we get
\begin{equation*}
{}^{T}\b^{*}_{\l}({}^{T}\d_{0}\star_{0}\cK)\cong {}^{T}\b_{\l}^{*}(\pi\c i_{\Gr})_{!}(\cK)\cong\psi_{\l!}\s_{\l}^{*}\cK\j{-2d_{\l}}\in \Kir(\pt/T\times \Ga).
\end{equation*}
Finally, we apply $\phi$ extracts the multiplicity of $\d\in \Kir(\pt/T\times \Ga)$ in the above expression to get the desired formula \eqref{rho as phi}. 
\end{proof}

%
%
%
%

\sss{Renormalized microstalks}\label{sss:ren mic}
The cotangent bundle $T^{*}\Bun_{G,\wh\infty}$ classifies triples $(\cE, \s, \ph)$ where $(\cE, \s)\in \Bun_{G,\wh\infty}$ and $\ph\in \G(\AA^{1}, \Ad(\cE)\ot \om_{\PP^{1}})$. Expanding $\ph$ at $\infty$ using the trivialization $\s$ of $\cE$ near $\infty$, we get a canonical map
\begin{equation*}
T^{*}\Bun_{G,\wh\infty}\to \frg\lr{\t}d\t
\end{equation*}
which is the moment map for the $G\lr{\t}$-action on $\Bun_{G,\wh\infty}$.

The union of conormals
\begin{equation*}
\L_{\Gr}:=\bigcup_{\l\in \xcoch(T)/W}T^{*}_{\Gr_{\l}}\Bun_{G,\hi}\subset T^{*}\Bun_{G,\wh\infty}
\end{equation*}
can be described as follows. A triple $(\cE, \s, \ph)\in T^{*}\Bun_{G,\wh\infty}$ lies in $\L_{\Gr}$ if and only if $\s$ extends to an isomorphism $\wt\s:\cE|_{\PP^{1}\bs\{0\}}\isom \cE_{\triv}|_{\PP^{1}\bs\{0\}}$ 
(which means $(\cE,\t)\in \Gr$), and that $\ph$, viewed as a section of $\Ad(\cE_{\triv})\ot\om|_{\AA^{1}\bs\{0\}}$ via the isomorphism $\wt\s$, extends to a section of $\Ad(\cE_{\triv})\ot\om|_{\AA^{1}}$.  

In particular, for the point $t^{\l}\in \Bun_{G,\hi}$, the Higgs field $\psi=\psi_{1}dt$ gives a point $(t^{\l}, \psi)\in \L_{\Gr}$, which lies in the conormal fiber of $T^{*}_{\Gr_{\l}}\Bun_{G,\hi}$ at $t^{\l}$.

We would like to apply the microlocalization functor (see \S\ref{ss:mic}) for the subvartiety $\Gr_{\l}$ of $\Bun_{G,\hi}$. Since $\Bun_{G,\hi}$ is infinite-dimensional, we need to renormalize the shift as follows.

Let $\Bun_{G, N\infty}$ be the moduli stack of $G$-bundles on $\PP^1$ with a trivialization on the $(N-1)$-th infinitesimal neighborhood of $\infty$ (so that $\Bun_{G, N\infty}=\Bun_G(\J)$). Let $\pi_N: \Bun_{G, \hi}\to \Bun_{G, N\infty}$ be the projection. Each $\Bun_{G,N\infty}$ contains a maximal schematic (open) locus which we denote by $X_N$. Note that
\begin{equation}\label{XN cover}
    \bigcup_{N\ge1}\pi_N^{-1}(X_N)=\Bun_{G,\hi}.
\end{equation}
Let $i_{\Gr, N}: \Gr\xr{i_\Gr} \Bun_{G,\hi}\xr{\pi_N} \Bun_{G,N\infty}$. 

Fix $\l\in \xcoch(T)$.  For each $\cK\in \Hs$, its support is contained in $\Gr_{\le \nu}$ for some dominant coweight $\nu$. We enlarge if necessarily such that $t^\l\in \Gr_{\le\nu}$. For sufficiently large $N$ (depending on $\nu$ hence on $\cK$),  $i_{\Gr}(\Gr_{\le \nu})\subset \pi_N^{-1}(X_N)$ by \eqref{XN cover}, hence $i_{\Gr,N}(\Gr_{\le\nu})\subset X_N$. Moreover for sufficiently large $N$ we can arrange that $i_{\Gr,N}|_{\Gr_{\le\nu}}: \Gr_{\le\nu}\to X_N$ is an embedding. 

We may then apply the microlocalization functor to the sheaf $\cK_N=i_{\Gr, \cK,!}\cK\in D(X_N)$ on $X_N$ with respect to its subscheme $\Gr_\l$ (embedded using $i_{\Gr,N}$):
\begin{equation*}
    \mu_{\Gr_\l/X_N}: D_{G\tl{t}}(\Gr_{\le \nu})\to D(T^*_{\Gr_\l}X_N).
\end{equation*}
We will use a shifted version of microlocalization: it amounts to removing the shift in the Fourier transform step.  In particular we consider
\begin{eqnarray*}
    \mu^\na_{\Gr_\l/X_N}:=\mu_{\Gr_\l/X_N}\c i_{\Gr,N,!}[-\codim_{X_N}\Gr_\l]: D_{G\tl{t}}(\Gr_{\le \nu})\to D(T^*_{\Gr_\l}X_N).
\end{eqnarray*}

The conormal bundles $T^*_{\Gr_\l}X_{N}$ form an increasing union of vector bundles over $\Gr_\l$ as $N$ increases, whose union is $T^*_{\Gr_\l}\Bun_{G,\hi}$. Concretely, the conormal fiber of $T^*_{\Gr_\l}X_{N}$ at $(\cE, \tau: \cE|_{\AA^1}\cong G\times \AA^1)\in \Gr_\l$ is $\cohog{0}{\PP^1, \Ad^*(\cE)\ot \om_{\PP^1}(N\infty)}$. For $N<N'$ let
\begin{equation}\label{embed conormal}
    \io_{N,N'}: T^*_{\Gr_\l}X_{N}\incl T^*_{\Gr_\l}X_{N'}.
\end{equation}
be the natural embeddings of conormal bundles over $\Gr_\l$.

\begin{lemma}\label{l:mu pullback}
    Fix a dominant coweight $\nu$ such that $t^\l\in \Gr_{\le\nu}$. The functors $\{\mu^\na_{\Gr_\l/X_N}\}$ are compatible with pullbacks along the embeddings of the conormal bundles in \eqref{embed conormal}, i.e., there are canonical isomorphisms of functors for $N<N'$
    \begin{equation*}        \io_{N,N'}^*\c\mu^\na_{\Gr_\l/X_{N'}}\cong \mu^\na_{\Gr_\l/X_{N}}:D_{G\tl{t}}(\Gr_{\le\nu})\to D(T^*_{\Gr_\l}X_N)
    \end{equation*}
    (whenever $N$ are large enough so that $\Gr_\l$ embeds to $X_N$), compatible with pullbacks in multiple steps.
\end{lemma}
\begin{proof} 
    Let $X=X_N$. Although $X_{N'}$ does not necessarily map to $X_N$, let $Y\subset X_{N'}$ be the preimage of $X_N$. Let $\pi: Y\to X$ be the projection, which is a torsor for an affine algebraic group, hence is smooth. 
    
    Let $S=\Gr_{\le\nu}$ with an embedding $i_Y: S\incl Y$ (given by $i_{\Gr,N'}$) such that $i_X=\pi\c i_Y: S\to X$ is also an embedding. Let $Z=\Gr_\l$ be embedded into $Y$ and $X$ by the restrictions of $i_Y$ and $i_X$.  
    Let $\io: T^*_{Z}X\incl T^*_{Z}Y$ be the map induced by $d\pi$. 
    
    The lemma asks for a canonical isomorphism of functors
    \begin{equation}\label{isom muZ}
        \mu_{Z/X}i_{X,!}[-\codim_XZ]\cong \io^*\c \mu_{Z/Y}i_{Y,!}[-\codim_YZ]: D(S)\to D(T^*_{Z}X).
    \end{equation}
    Let $\t: N_{Z}Y\to N_ZX$ be the dual of $\io$. We first claim that the specializations to the normal bundles satisfy
    \begin{equation}\label{isom Sp}
        \Sp_{Z/X}\c i_{X,!}\cong \t_!\c\Sp_{Z/Y}\c i_{Y,!}: D(S)\to D^{\mon}(N_ZX).
    \end{equation}
    Indeed, let $M_ZS$ be the total space of the deformation to the normal cone of $Z$ in $S$. Then we have embeddings $k_Y: M_{Z}S\incl M_ZY$ and $k_X: M_ZS\incl M_ZX$ compatible with the natural map $M_ZY\to M_ZX$. Now both sides of \eqref{isom Sp} can be identified with $k_{X,0,!}\Sp_{Z/S}$, where $k_{X,0}:N_ZS\to N_ZX$ is the special fiber of $k_X$.

    The isomorphism \eqref{isom muZ} then follows by applying the monodromic Fourier transform to both sides of \eqref{isom Sp} and applying Lemma \ref{l:FT push} to the pair of maps $(\t,\io=\t^\v)$ between normal and conormal bundles.
\end{proof}

By Lemma \ref{l:mu pullback}, for fixed $\nu$, the functors $\{\mu^\na_{\Gr_\l/X_N}\}$ together with a functor
\begin{equation*}
    \mu^\na_\l: D_{G\tl{t}}(\Gr_{\le\nu})\to \varprojlim_N D(T^*_{\Gr_\l}X_{N}):=D^*(T^*_{\Gr_\l}\Bun_{G,\hi}).
\end{equation*}
As we enlarge $\nu$ to $\nu'$, the functor $\mu^\na_\l$ above is compatible with the embeddings $D_{G\tl{t}}(\Gr_{\le\nu})\incl D_{G\tl{t}}(\Gr_{\le\nu'})$. Passing to the colimit for $\nu$, we get the renormalized microlocalization functor for the Satake category $\Hs$ along $\Gr_\l$,  still denoted by the same notation
\begin{eqnarray*}
    \mu^\na_\l: \Hs\to D^*(T^*_{\Gr_\l}\Bun_{G,\hi}).
\end{eqnarray*}
In particular, it makes sense to evaluate the $*$-stalk of $\mu_\l(\cK)$ ($\cK\in \Hs$) at any point of $\xi\in T^*_{\Gr_\l}\Bun_{G,\hi}$, and call it the renormalized microstalk of $\cK\in \Hs$ at $\xi$. We are interested in the case $\xi=(t^\l, \psi)$.


\begin{prop}\label{p:rho ms} For $\cK\in \Hs$, there is a canonical isomorphism functorial in $\cK$
\begin{equation*}
\r_{\l}(\cK)\cong \mu^\na_{\l}(\cK)|_{(t^{\l}, \psi)}\j{-2d_\l}.
\end{equation*}
\end{prop}
\begin{proof} 
In view of Lemma \ref{l:rho lam}, it remains to give a canonical isomorphism
\begin{equation}\label{mic Gr as phi}
\mu^\na_{\l}(\cK)|_{(t^{\l}, \psi)}\cong \phi(\psi_{\l!}\s_{\l}^{*}\cK)
\end{equation}
for $\cK\in \Hs$. Below we shall freely use notations from the proof of Lemma \ref{l:rho lam}.

Suppose $\Supp(\cK)\subset \Gr_{\le\nu}$  and $t^\l\in \Gr_{\le\nu}$. By definition of the renormalized microlocalization functor in \S\ref{sss:ren mic}, we have
\begin{equation}\label{mu XN}
    \mu^\na_{\l}(\cK)|_{(t^{\l}, \psi)}=\mu_{\Gr_\l/X_N}(i_{\Gr,N!}\cK)|_{(t^\l, \psi)}[-\codim_{X_N}\Gr_\l]
\end{equation}
for sufficiently large $N$.

Let $\Sig_{\l,N}$ be the image of $\Sig_\l$ in $\Bun_{G,N\infty}$. We have $\Sig_{\l,N}\cong \bG_\infty^N\bs \J/A_\l$ (where $\bG^N_\infty$ is the $N$th congruence subgroup of $G\tl{\t}$). Hence for $N$ sufficiently large $\Sig_{\l,N}$ is a scheme therefore $\Sig_{\l,N}\subset X_N$. For $N\ge2$, the map $\wt\psi_\l: \Sig_\l\to \Ga$ factors through a map $\wt\psi_{\l,N}: \Sig_{\l,N}\to \Ga$. We thus have an analog of the diagram \eqref{slice diagram}, of which we only record a part
\begin{equation}\label{N slice diagram}
    \xymatrix{
& \Sig^{\Gr}_{\l,N} \ari[r]^-{\s_{\l,N}} \ar[dl]_-{\psi_{\l,N}}\ari[d]^{i_{\Sig_{\l,N}}} & \Gr_{\le\nu}\ari[d]^{i_{\Gr,N}}\\
\Ga & \Sig_{\l,N}\ari[r]^-{\wt\s_{\l,N}}\ar[l]_{\wt\psi_{\l,N}} & X_N
}
\end{equation}
Since $\cK$ is supported on $\Gr_{\le\nu}$, we have
\begin{equation}\label{phi Sig vs SigN}    \phi(\psi_{\l!}\s_{\l}^{*}\cK)\cong \phi(\psi_{\l,N,!}\s_{\l,N}^*\cK).
\end{equation}

We then apply Proposition \ref{p:aff space van cycle} to the smooth scheme $\Sig_{\l,N}$ equipped with the inverse of the $\Grot$-action.  The inverse of the $\Grot$-action contracts $\Sig_{\l,N}$ to $t^\l$, and the function $\psi_{\l,N}: \Sig_{\l,N}\to \AA^1$ has weight $1$ under the inverse of $\Grot$-action, which is the smallest among the  weights on $T^*_{t^\l}\Sig_{\l,N}$. Proposition \ref{p:aff space van cycle} is thus applicable, and it gives a canonical isomorphism for $\cF\in D^{\Grot\lmon}(\Sig_{\l,N})$.
\begin{equation*}
    \mu_{\{t^\l\}/\Sig_{\l,N}}(\cF)|_{(t^\l, \psi)}\cong \phi(\wt\psi_{\l,N,!}\cF)[\dim\Sig_{\l,N}].
\end{equation*}
We apply this to $\cF=i_{\Sig_{\l,N}!}\s^*_{\l,N }\cK$ to get
\begin{equation*}
    \mu_{\{t^\l\}/\Sig_{\l,N}}(i_{\Sig_{\l,N}!}\s^*_{\l,N} \cK)|_{(t^\l, \psi)}\cong \phi(\wt\psi_{\l,N,!}i_{\Sig_{\l,N}!}\s^*_{\l,N} \cK)[\dim \Sig_{\l,N}]\cong\phi(\psi_{\l,N,!}\s^*_{\l,N} \cK)[\dim \Sig_{\l,N}]. 
\end{equation*}
Combined with \eqref{phi Sig vs SigN} we get
\begin{equation}\label{phi as mu Sig}
\phi(\psi_{\l!}\s_{\l}^{*}\cK)\cong \mu_{\{t^\l\}/\Sig_{\l,N}}(i_{\Sig_{\l,N}!}\s^*_{\l,N} \cK)|_{(t^\l, \psi)}[-\dim\Sig_{\l,N}].
\end{equation}
By Lemma \ref{l:trans slice}, $\Sig_{\l,N}$ is a transversal slice to $\Gr_\l$ in $\Sig_{\l,N}$, hence
\begin{equation}\label{codim in XN}
    \dim \Sig_{\l,N}=\codim_{\Gr_\l}X_N.
\end{equation}
Using \eqref{mu XN}, \eqref{phi as mu Sig} and the identity \eqref{codim in XN}, to give \eqref{mic Gr as phi}, it reduces to give a canonical isomorphism
\begin{equation}\label{res to slice}
    \mu_{\Gr_\l/X_N}(i_{\Gr,N!}\cK)|_{(t^\l, \psi)}\isom \mu_{\{t^\l\}/\Sig_{\l,N}}(i_{\Sig_{\l,N}!}\s^*_{\l,N} \cK)|_{(t^\l, \psi)}\cong \mu_{\{t^\l\}/\Sig_{\l,N}}(\wt\s_{\l,N}^*i_{\Gr,N!}\cK)|_{(t^\l, \psi)}
\end{equation}
for any $\cK\in D_{G\tl{t}}(\Gr_{\le\nu})$. This follows by applying Lemma \ref{l:res slice} to $X=X_N$, $Y=\Gr_{\le\nu}$ with the action of a finite-dimensional quotient $H$ of $G\tl{t}$ by left translation, the $H$-orbit $Z=\Gr_\l\incl Y$ through $z=t^\l$, with transversal slice $S=\Sig_{\l,N}$.
\end{proof}


Combining Proposition \ref{p:rho ms} and Theorem \ref{th:rho comm}, we conclude that the microstalks functor on the Satake category gives a fiber functor. More precisely,

\begin{cor}\label{c:mic fiber functor}
    The functor
    \begin{equation*}
        \om_\psi=\bigoplus_{\l\in \xcoch(T)}\mu^\na_\l|_{(t^\l,\psi)}\j{-2d_\l}: \Hs\to D^b(\Vect)
    \end{equation*}
    has a canonical monoidal structure and is t-exact. Moreover, its restriction to $\Hsh$ can be equipped with the structure of a fiber functor for the Tannakian category $\Hsh$. 
\end{cor}

\section{The category $\cD_{\psi}$}\label{s:D}

We introduce and study the main player of this paper, a conventional sheaf category on the affine flag variety that turns out to be equivalent to microsheaves on the affine Springer fiber $\Fl_{\psi}$ with finite type supports.

\subsection{Definition of $\wt\cD_{\psi}$ and $\cD_{\psi}$}
Consider the stack $\Bun_{G}(\bG^{\psi}_{\infty}, \bI_{0})$ with the canonical $\Aff$-action. Define
\begin{equation*}
\wt\cD_{\psi}:= \Kir(\Bun_{G}(\bG^{\psi}_{\infty}, \bI_{0})).
\end{equation*}
By Proposition \ref{p:K psi}, when $\ch(k)=p$, we may identify $\wt\cD_{\psi}$ with the category of $(\Ga,\psi)$-equivariant complexes on $\Bun_{G}(\bG^{\psi}_{\infty}, \bI_{0})$, or equivalently $(\frg,\psi)$-equivariant complexes on $\Bun_{G}(\bG^{2}_{\infty}, \bI_{0})$.

\sss{Relevant points for $\wt\cD_\psi$}
By Definition \ref{d:rel}, the relevant locus for the Kirillov category $\wt\cD_\psi$ is the image of the forgetful map $\cM_\psi\to \Bun_G(\J, \bI_0)$. 



To describe the relevant locus of $\wt\cD_\psi$, consider the map
\begin{equation}
    \wt \b_\l: G t^\l\bI_0/\bI_0\to \Bun_{G}(\J, \bI_{0})=G[\t]^1\bs \Fl.
\end{equation}
We identify the source of the above map as a closed variety of $\Fl$, which is a fixed point component under $\Grot$. Recall the Hessenberg variety $\cH_\psi(\l)\subset G t^\l\bI_0/\bI_0$ introduced in \S\ref{Hess_var}.



\begin{lemma}\label{Relv pts Fl}
    The relevant locus for $\wt\cD_\psi$ is the union of $\wt\b_\l(\cH_\psi(\l))$ for all $\l\in \xcoch(T)$. 
\end{lemma}
\begin{proof} 
Using the uniformization
\begin{equation*}
    G[\t]^1\bs \Fl\isom \Bun_G(\J, \bI_{0}),
\end{equation*}
the relevant $k$-points for $\wt\cD_\psi$ are the image of those $g\bI_0\in \Fl$ such that the composition
\begin{equation}\label{res psi to stab}
    G[\t]^1\cap \Ad(g)\bI_0\subset G[\t]^1\surj t^{-1}\frg\xr{t}\frg\xr{\j{\psi_1,-}} k 
\end{equation}
is zero. Up to left translation by $G[\t]^1$ we may assume $g\bI_0\in G t^\l\bI_0$. Then a root-theoretic calculation shows that \eqref{res psi to stab} is trivial if and only if $g\bI_0\in \cH_\psi(\l)$. The lemma follows.

\end{proof}

\sss{The category $\wt\cD_{\psi, \bP}$}
For any standard parahoric subgroup $\bP\subset G\lr{t}$ we may consider $\Bun_{G}(\bG_{\infty}^{\psi}, \bP_{0})$ and the corresponding Kirillov category $\wt \cD_{\psi, \bP}$.  In particular, for $\bP=\bG=G\tl{t}$ we have 
$\wt \cD_{\psi, \bG}=\cD_{\psi, \bG}$ defined in \S\ref{ss:D psi Gr}. 

The projection $\pi_{\bP}: \Bun_{G}(\bG_{\infty}^{\psi}, \bI_{0})\to \Bun_{G}(\bG_{\infty}^{\psi}, \bP_{0})$ induces functors
\begin{equation*}
\xymatrix{\wt\cD_{\psi}\ar@<-.5ex>[r]_{\pi_{\bP*}} & \wt\cD_{\psi,\bP} \ar@<-.5ex>[l]_{\pi^{*}_{\bP}}}
\end{equation*}
For $\bP=\bG$, we denote $\pi_{\bG}$ by $\pi$.

We give $\wt\cD_{\bP}$ the perverse t-structure inherited from the middle perversity on the stack $\Bun_{G}(\Jker, \bP_{0})$. Denote its heart by $\wt\cD_{\bP}^{\hs}$.

\sss{Affine Hecke action}\label{sss:aff Hk action}
Let $\cH_{0}=D(\bI_{0}\bs G\lr{t}/\bI_{0})$ be the affine Hecke category of $G$. Again all objects in $\cH_{0}$ are automatically $\Grot$-monodromic with unipotent monodromy, since all $\bI_{0}$-double cosets in $G\lr{t}$ are stable under loop rotation. Similar to the spherical action discussed in \S\ref{ss:sph action}, there is a right action of $\cH_{0}$ on $\wt\cD_{\psi}$ by Hecke modification at $0$, using the construction of integral transforms in \S\ref{sss:int trans}. For $\cK\in \cH_{0}$, we write its action on $\wt\cD_{\psi}$ by $(-)\star_{0}\cK$. It will be convenient to use the double coset description (uniformization at $0$)
\begin{equation*}
\Bun_{G}(\Jker, \bI_{0})\cong G[\t]^{\psi}\bs G\lr{t}/\bI_{0}=G[\t]^{\psi}\bs \Fl
\end{equation*}
when working with the action of $\cH_0$. Here $G[\t]^{\psi}=G[\t]\cap \Jker\subset G[\t]$.
We also introduce the notation
\begin{equation*}
    G[t^{-1}]^1=G[\t]^{1}:=\ker(\ev_{\infty}: G[\t]\to G).
\end{equation*}
For $w\in \tilW$, let $\IC_{w}\in \cH_{0}$ be the middle extension of the constant sheaf $E\j{\ell(w)}$ from the Schubert cell indexed by $w$.

\begin{defn}Let $\cD_{\psi}\subset \wt\cD_{\psi}$ be the full subcategory generated under the $\cH_{0}$-action by the image of $\pi^{*}: \DG\to \wt\cD_{\psi}$. 
\end{defn}


Recall the objects $\d_{\l}\in \DG$ for $\l\in \xcoch(T)$. Let
\begin{equation*}
\cF_{\l}=\pi^{*}\d_{\l}\in \wt\cD_{\psi}.
\end{equation*}
Since $\DG$ is generated by $\{\d_{\l}\}$,  $\{\cF_{\l}; \l\in \xcoch(T)\}$ generate $\cD_{\psi}$ as a $\cH_{0}$-module. 

\begin{remark}
    Note that by Theorem \ref{th:rho comm}, the action of central sheaves in $\cH_0$ on $\cF_{0}$ generates all $\cF_\l$ after idempotent completion. Thus $\cD_{\psi}$ is generated under the $\cH_{0}$-action by only the object $\cF_{0}$ after idempotent completion.
\end{remark}



\sss{The category $\cD_{\psi,\bP}$}\label{sss:DpsiP}
For general $\bP$, let $\cD_{\psi,\bP}\subset \wt \cD_{\psi,\bP}$ be the full subcategory of $\cF\in \wt\cD_{\psi,\bP}$ such that $\pi^{*}_{\bP}\cF\in \cD_{\psi}$.  Equivalently, $\cD_{\psi,\bP}$ is generated by the summands of the image of $\pi_{\bP*}: \cD_{\psi}\to \wt\cD_{\psi,\bP}$. When $\bP=\bG$, this is consistent with the earlier definition of $\DG=\wt\cD_{\psi, \bG}$.

\sss{Hecke action at $\infty$}
As in the case of $\cD_{\psi, \bG}$, the monoidal category $\Hinf$ acts on $\wt\cD_{\psi}$ by modifications of bundles at $\infty$.  We denote the action of $\cK\in \Hinf$ by $\cK\star_{\infty}(-)$. This action commutes with the action of $\cH_{0}$, and therefore the $\Hinf$-action preserves the subcategory $\cD_\psi$. It will be convenient to use the double coset description
\begin{equation*}
\Bun_{G}(\Jker, \bI_{0})\cong \Jker\bs G\lr{\t}/I_{\AA^{1}}
\end{equation*}
when working with the action of $\Hinf$. Here $I_{\AA^{1}}\subset G[t]$ is the preimage of $B$ under the evaluation map $\ev_{t=0}: G[t]\to G$.

By Lemma \ref{l:std action DGr}, we have a canonical isomorphism
\begin{equation*}
\wh\D_{\mu}\star_{\infty}\cF_{\l}\cong \cF_{\l-\mu}, \quad \forall \l,\mu\in \xcoch(T).
\end{equation*}
In particular, the action of $\Him$ on $\cD_{\psi}$ extends to an action of the completion $\hHim$.

\sss{$T$-equivariant versions}
We can similarly consider the stack $\Bun_{G}(\bG^{\psi}_{\infty}T, \bI_{0})$ with the action of $\Aff$. We then can define 
\begin{equation*}
	{}^T\wt\cD_{\psi}:= \Kir(\Bun_{G}(\bG^{\psi}_{\infty}T, \bI_{0})).
\end{equation*}
We have adjoint functors
\begin{equation*}
\xymatrix{{}^T\wt\cD_{\psi}\ar@<-.5ex>[r]_{{}^T\pi} & {}^T\DG \ar@<-.5ex>[l]_{{}^T\pi^{*}}}
\end{equation*}

The category ${}^T\wt\cD_{\psi}$ still carries an action of $\cH_{0}$ of the Hecke category at $0$, and an action of $\eqHinf$ at infinity (see \S\ref{sss:var Hinf}).

We give ${}^T\wt\cD_{\psi}$ the perverse t-structure inherited from the forgetful functor to $\wt\cD_\psi$, which in turn comes from the middle perversity on the stack $\Bun_{G}(\Jker, \bI_{0})$. We denote its heart by ${}^T\wt\cD^{\hs}_{\psi}$.

By Proposition \ref{p:Hinf equiv}, we have a monoidal equivalence of abelian categories $\Rep(\dT)\cong \Hinfhs$, under which the character $\l\in \xch(\dT)=\xcoch(T)$ corresponds to the standard sheaf ${}^T\D^T_{\l}\in \Hinfhs$.

\begin{lemma}\label{l:texact inf} For any $\cK\in \Hinfhs$, the endofunctor $\cK\star_{\infty}(-):{}^{T}\wt\cD_{\psi}\to {}^{T}\wt\cD_{\psi}$ is t-exact for the perverse  t-structure on ${}^{T}\wt\cD_{\psi}$.
\end{lemma}
\begin{proof} 
We first show that for any $\l\in \xcoch(T)$, the endofunctor ${}^{T}\D_{\l}^{T}\star_{\infty}(-)$ of ${}^{T}\wt\cD_{\psi}$ is t-left-exact. Indeed, ${}^{T}\D_{\l}^{T}\star_{\infty}(-)\cong p_{\l!}[d_{\l}]$, where $p$ is the following map induced by multiplication
\begin{equation*}
p_{\l}: (\Jker T)\bs\Jker \t^{\l}\Jker T\times^{\Jker T}\Fl\cong ((\Jker T)\cap \Ad(t^{\l})(\Jker T))\bs \Fl\to (\Jker T)\bs \Fl.
\end{equation*}
Since $p_{\l}$ is affine with relative dimension $d_{\l}$ by \eqref{d lam}, we conclude that ${}^{T}\D_{\l}^{T}\star_{\infty}(-)$ is t-left-exact. 

Now ${}^{T}\D_{-\l}^{T}\star{}^{T}\D_{\l}^{T}\cong {}^{T}\D_{0}^{T}$ by Proposition \ref{p:Hinf equiv}, which acts by identity on ${}^{T}\wt\cD_{\psi}$. This forces ${}^{T}\D_{\l}^{T}\star_{\infty}(-)$ to be t-exact.  Finally, since  $\{{}^{T}\D_{\l}^{T}\}_{\l\in\xcoch(T)}$ generate the abelian category $\Hinfhs$ by Lemma \ref{l:H inf rel}, $\cK\star_{\infty}(-)$ is t-exact for any $\cK\in \Hinfhs$.

\end{proof}

\subsection{The category ${}^{T}\cC$ and action by Wakimoto sheaves}\label{sec:Wak_action}
Consider the embedding
\begin{equation*}
{}^{T}\ov\b^{\bI_{0}}_{\l}: \Bun^{\l}_{G}(\Jker T,\bI_{0}):=(G[\t]^{\psi}T) \bs G[t]^{1}  t^{\l}\bI_{0}/\bI_{0}= (\Jker T) \bs \J t^{\l} I_{\AA^{1}}/I_{\AA^{1}}
	\incl \Bun_{G}(\Jker T,\bI_{0}).
\end{equation*}
analogous to the embedding in \ref{equ_beta}. Note that here these are not all the relevant strata in $\Bun_{G}(\Jker T ,\bI_{0})$. Similar to the calculations in \S\ref{sss:rel DGr}, we have
\begin{equation*}
\Bun^{\l}_{G}(\Jker T,\bI_{0})=(\Jker T)\bs \J t^{\l} I_{\AA^{1}}/I_{\AA^{1}}\cong (\pt/\wt A_{\l})\times \Ga
\end{equation*}
where $\wt A_{\l}=(\Jker T)\cap \Ad(t^{\l})I_{\AA^{1}}$. The embedding ${}^{T}\ov\b^{\bI_{0}}_{\l}$ gives a functor
\begin{eqnarray*}
	{}^{T}\b^{\bI_{0}}_{\l,!}: \Kir(\Ga)\xr{\pr_{\Ga}^{*}}\Kir((\pt/\wt A_{\l})\times \Ga)\xr{{}^{T}\ov\b^{\bI_{0}}_{\l,!}} \wt\cD_{\psi}
\end{eqnarray*}

Let
\begin{equation*}
\dc_{0}={}^{T}\b^{\bI_{0}}_{0,!}\d, \quad \dc_{\l}={}^{T}\D^{T}_{-\l}\star_{\infty}\dc_{0}, \quad \forall \l\in\xcoch(T).
\end{equation*}
By Lemma \ref{l:texact inf}, we see that
\begin{equation*}
\dc_{\l}\in {}^{T}\wt\cD^{\hs}_{\psi}, \quad \mbox{for all }\l\in \xcoch(T).
\end{equation*}

\begin{defn} Let ${}^{T}\calC\subset {}^{T}\wt\cD^{\hs}_{\psi}$ be the full subcategory whose objects are finite successive extensions of $\{\dc_{\l}\}_{\l\in\xcoch(T)}$. Further denote by  ${}^{T}\calC^{ss}$
the full subcategory of ${}^{T}\cC$ whose objects are finite direct sums of $\{\dc_{\l}\}_{\l\in\xcoch(T)}$. 
\end{defn}

\begin{remark}
    The category ${}^T\cC$ is not a subcategory of ${}^T\cD_\psi$. We will mention in \S\ref{r:whole mush} how this category fits into the coherent description of $\cD_\psi$. 
\end{remark}

\begin{lemma}\label{l:Tpi dc to d}
    For any $\l\in \xcoch(T)$ we have a canonical isomorphism
    \begin{equation*}
        {}^T\pi_*(\dc_\l)\cong {}^T\d_\l\in {}^T\DG.
    \end{equation*}
\end{lemma}
\begin{proof}
    For $\l=0$ it follows from the definition of $\dc_0$ and $\d_0$ because the  projection $\Bun_G^0(\Jker T, \bI_0)\to \Bun_G^0(\Jker T)$ is an isomorphism. Now we apply convolution by ${}^T\D^T_{-\l}$ to the isomorphism $\dc_0\cong {}^T\d_0$ and use Lemma \ref{Dmu act on dl} to conclude for general $\l$. 
\end{proof}

\begin{lemma}\label{l:dcl dom}
For $\l\in\xcoch(T)^{+}$ we have $\dc_{\l}\cong{}^T\b^{\bI_{0}}_{\l,!}\d\j{-d_\l}$. 
\end{lemma}
\begin{proof}
We show that for $\l\in\xcoch(T)^{+}$ there is a canonical isomorphism
\begin{equation}\label{Dl conv bl}
{}^T\D_{\l}^T\star_{\infty}{}^T\b^{\bI_{0}}_{\l,!}\d\j{-d_\l}\cong {}^T\b_{0!}^{\bI_{0}}\d=\dc_{0}.
\end{equation}
The result then follows by convolving with ${}^T\D_{-\l}^T$, which is the inverse of ${}^T\D_{\l}^T$.
	
	To see \eqref{Dl conv bl}, we only need to note that the multiplication map gives an isomorphism
	\begin{equation*}
		\Jker  \t^{\l}\Jker T\times^{\Jker T}\Jker  t^{\l}I_{\AA^{1}}/I_{\AA^{1}}\cong\Jker I_{\AA^{1}}/I_{\AA^{1}},
	\end{equation*}
for $\l\in\xcoch(T)^{+}$. Indeed, we can identified the left side with $\Jker \t^{\l}\Jker\times^{\Jker}\Jker  t^{\l}I_{\AA^{1}}/I_{\AA^{1}}$, so it suffices to show that
\begin{equation}\label{left2}
\Jker\bs \Jker \t^{\l}\Jker \times^{\Jker}\Jker  t^{\l}I_{\AA^{1}}/I_{\AA^{1}}
\end{equation}
is a point. This has essentially been proved in Lemma \ref{l:std action DGr}(1). We observe that 
$\Jker  t^{\l}I_{\AA^{1}}/I_{\AA^{1}}=\Jker /(\Jker\cap \Ad(t^{\l})I_{\AA^{1}})$. For $\l$ dominant, both $\Jker\cap \Ad(t^{\l})I_{\AA^{1}}$ and $\bK_{+}=\Jker\cap \Ad(t^{\l})G[t]$ are equal to $\J\cap \Ad(t^{\l})N[t]$ (where $N$ is the unipotent radical of $B$).  We also have $\Jker\bs \Jker  \t^{\l}\Jker=\bK^{\psi}_{-}\bs \Jker$, where $ \bK^{\psi}_{-}=\Jker\cap \Ad(t^{\l})\Jker$. Therefore \eqref{left2} can be identified with $\bK_{-}^{\psi}\bs \Jker/\bK_{+}$. Finally we observe that the multiplication map $\bK^{\psi}_{-}\times \bK_{+}\to \Jker$ is an isomorphism, which proves that \eqref{left2} is isomorphic to a point.
\end{proof}

\begin{lemma}\label{l:Css free over Hinf}
\begin{enumerate}
\item The category ${}^{T}\cC^{ss}$ is a semisimple abelian category with simple objects $\{\dc_{\l}\}_{\l\in \xcoch(T)}$. 
\item As a $\Hinfhs$-module category,  ${}^{T}\cC^{ss}$ is free of rank one.
\end{enumerate}
\end{lemma}
\begin{proof}
(1) By definition, $\dc_{0}$ is a simple object in ${}^{T}\wt\cD^{\hs}_{\psi}$. Since the action by ${}^{T}\D^{T}_{-\l}\star_{\infty}(-)$ is an auto-equivalence of ${}^{T}\wt\cD^{\hs}_{\psi}$ for any $\l\in \xcoch(T)$, we see that $\dc_{\l}={}^{T}\D^{T}_{-\l}\star_{\infty}\dc_{0}$ is also a simple object in ${}^{T}\wt\cD^{\hs}_{\psi}$. This shows that ${}^{T}\cC^{ss}$ is a semisimple abelian category. To show (1), we only need to show that for $\l\ne\l'$, $\dc_{\l}$ is not isomorphic to $\dc_{\l'}$. Translate by ${}^{T}\D^{T}_{-\mu}$ for $\mu$ sufficiently dominant, we may arrange that $\l+\mu$ and $\l'+\mu$ are both dominant. In this case, we have $\dc_{\l+\mu}\not\cong\dc_{\l'+\mu}$ by Lemma \ref{l:dcl dom}. Translating back by ${}^{T}\D^{T}_{\mu}$, we conclude that $\dc_{\l}\not\cong\dc_{\l'}$.

(2) Since $\Hinfhs$ is semisimple with simple objects $\{{}^{T}\D^{T}_{\l}\}_{\l\in \xcoch(T)}$,  (2) now follows from (1) and the definition of $\dc_{\l}$.
\end{proof}

\begin{lemma}\label{l:Css equiv DGr}
The functor ${}^{T}\pi_{*}: {}^{T}\wt\cD_{\psi}\to {}^{T}\cD_{\psi, \bG}$ restricts to an equivalence of $\Hinfhs$-module categories
\begin{equation*}
{}^{T}\pi_{*}|_{{}^{T}\cC^{ss}}: {}^{T}\cC^{ss}\isom {}^{T}\DG^{\hs}.
\end{equation*}
\end{lemma}
\begin{proof} The functor ${}^{T}\pi_{*}$ is clearly equivariant under $\eqHinf$. By Lemma \ref{l:Tpi dc to d}, ${}^{T}\pi_{*}$ sends simple objects to simple objects. By Lemma \ref{l:Css free over Hinf} and Corollary \ref{c:TDGr clean}(4), ${}^{T}\pi_{*}$ induces a bijection between isomorphisms classes of simple objects (both are indexed by $\xcoch(T)$). The statement follows.
\end{proof}

\begin{lemma}\label{l:filC}
	\begin{enumerate}
	    \item Each $\calG\in {}^{T}\cC$ has a unique filtration $F_{\le\l}\calG$ (by subobjects in ${}^{T}\cC$) indexed by the poset $(\xcoch(T)), \le)$, such that the associated graded pieces $\gr_{\l}\cG=F_{\le\l}\calG/F_{<\l}\calG$ is a direct sum of copies of $\dc_{\l}$. 
        \item All morphisms in ${}^T\cC$ preserve the filtration above. Passing to the associated graded gives a functor
	$$\gr:{}^{T}\cC\rightarrow {}^{T}\cC^{ss}.$$
        \item For any $\l,\mu\in \xcoch(T)$ and $\cG\in {}^T\cC$, we have
        \begin{equation*}
            F_{\le\l}(^T\D^T_{\mu}\star_\infty \cG)=^T\D^T_{\mu}\star_\infty F_{\le \l+\mu}\cG
        \end{equation*}
        as subobjects of $^T\D^T_{-\mu}\star_\infty \cG$.
	\end{enumerate}
\end{lemma}
\begin{proof}
	Existence part of (1). We call a filtration $F_{\le \l}\cG$ on $\cG\in {}^{T}\cC$ indexed by the poset $(\xcoch(T), \le )$ {\em admissible} if it satisfies the condition specified in  the lemma. We first remark that $F_{\le \l}\cG$ is a filtration of $\cG$ satisfying the required condition, then for any $\mu\in \xcoch(T)^{+}$, $F_{\le \l}({}^{T}\D^{T}_{\mu}\star_{\infty}\cG):={}^{T}\D^{T}_{\mu}\star_{\infty}F_{\le \l+\mu}\cG$ is an admissible filtration for ${}^{T}\D^{T}_{\mu}\star_{\infty}\cG$.
		
	Now we construct an admissible filtration on any $\cG\in {}^{T}\cC$. By the  definition of ${}^{T}\cC$, each object $\calG\in {}^{T}\cC$ has a finite filtration with associated graded in $\{\dc_{\l}\}$. We first consider the case where all  the associated graded of $\calG$ are $\dc_{\l}$ for $\l\in\xcoch(T)^{+}$. We will define a filtration $F_{\le \l}\cG$ indexed by $\l\in \xcoch(T)^{+}$ by truncation of the support as follows. 
	
	We have a stratification $\Bun_{G}=\cup_{\l\in\xcoch(T)^{+}}\Bun_{G}^{[\l]}$ where  $\Bun_{G}^{[\l]}$ has only one point corresponding to the double coset  $\bG_{\infty} t^{\l}G[t]$.  Note that for $\l,\l'\in \xcoch(T)^{+}$, $\Bun^{[\l']}_{G}$ is in the closure of $\Bun_{G}^{[\l]}$ if and only if $\l\le \l'$. In particular, $\Bun_{G}^{\le[\l]}:=\cup_{\l'\le \l}\Bun_{G}^{[\l']}$ is open in $\Bun_{G}$. Let $j^{\le [\l]}: \Bun_{G}^{\le[\l]}(\J T, \bI_{0})\incl \Bun_{G}(\J T, \bI_{0})$ be the preimage of $\Bun_{G}^{\le[\l]}$. Let $j^{[\l]}: \Bun_{G}^{[\l]}(\J T, \bI_{0})\incl \Bun_{G}(\J T, \bI_{0})$ be the preimage of $\Bun_{G}^{[\l]}$. Define $F_{\le \l}\cG=j^{\le [\l]}_{!}j^{\le [\l],*}\cG$. The associated graded $F_{\le \l}\cG/F_{<\le \l}\cG$ is then isomorphic to $j^{[\l]}_{!}j^{[\l]*}\cG$, which is isomorphic to ${}^{T}\b_{\l,!}^{\bI_{0}!}{}^{T}\b_{\l}^{\bI_{0}*}\cG$ because $\Bun_{G}^{\l}(\J T,\bI_{0})$ is the only relevant stratum in $\Bun_{G}^{[\l]}(\J T, \bI_{0})$ in the support of $\cG$. By Lemma \ref{l:dcl dom}, $\dc_{\l}\cong{}^T\b^{\bI_{0}}_{\l,!}\d\j{-d_\l}$, therefore $F_{\le \l}\cG/F_{<\l}\cG$ is isomorphic to a direct sum of copies of $\dc_{\l}$. This gives an admissible filtration on $\cG$.
	
	For general $\cG\in {}^{T}\cC$, by convolving with some ${}^T\D_{-\mu}^T$ where $\mu\in\xcoch(T)^{+}$ is dominant enough, we can arrange that $\cG'={}^T\D_{-\mu}^T\star_{\infty}\cG$ is a successive extension of $\dc_{\l}$ for $\l\in \xcoch(T)^{+}$. Let $F_{\le \l}\cG'$ be the filtration constructed in the previous paragraph. Then $F_{\le \l}\cG:={}^T\D_{\mu}^T\star_{\infty}F _{\le \l+\mu}\cG'$ is an admissible filtration on $\cG$, by the remark in the first paragraph.	
	
	(2) We need to show that if  $\cG_{1}, \cG_{2}\in {}^{T}\cC$ are equipped with admissible filtrations $F_{\le \l}\cG_{1}$ and $F_{\le \l}\cG_{2}$, and $\ph: \cG_{1}\to \cG_{2} $ is a morphism in ${}^{T}\cC$, then induces a canonical map $\ph_{\le\l}: F_{\le \l}\cG_{1}\to F_{\le\l }\cG_{2}$ compatible with $\ph$ and compatible with the filtrations in the obvious sense.

	By convolving both $\cG_{1}$ and $\cG_{2}$ with ${}^{T}\D^{T}_{-\mu}$ for sufficiently dominant $\mu\in \xcoch(T)^{+}$, we may assume that $ \cG_{1}, \cG_{2}$ are both successive extensions of $\{\dc_{\l}\}_{\l\in \xcoch(T)^{+}}$. Let $\l\in \xcoch(T)^{+}$. By Lemma \ref{l:dcl dom}, $F_{\le \l}\cG_{1}$ is the $!$-extension from the open substack $\Bun^{\le[\l]}_{G}(\J T, \bI_{0})$, and $\Cone(F_{\le \l}\cG_{2}\to \cG_{2})$ is supported on the closed complement of $\Bun^{\le[\l]}_{G}(\J T, \bI_{0})$. Therefore $\RHom(F_{\le \l}\cG_{1}, \Cone(F_{\le \l}\cG_{2}\to \cG_{2}))=0$. This implies that $\ph|_{F_{\le \l}\cG_{1}}$ has a canonical factorization as $ F_{\le\l} \cG_{1}\xr{\ph_{\le \l}}F_{\le \l}\cG_{2}\to \cG_{2}$.

	The uniqueness of the admissible filtration on $\cG$ in (1) follows by applying the functoriality in (2) to the identity map of $\cG$.

    It remains to show (3). Let $\cG'={}^T\D^T_{-\mu}\star_\infty\cG$. Define another filtration $F'$ on $\cG'$, also indexed by $(\xcoch(T), \le)$, by
    \begin{equation*}
        F'_{\le \l}\cG':={}^T\D_{\mu}\star_\infty F_{\l+\mu}\cG.
    \end{equation*}
	Note that the right side is a subobject of $\cG'={}^T\D_{\mu}\star_\infty\cG$ because convolution with ${}^T\D_{\mu}\star_\infty$ is t-exact. The associated graded of the filtration $F'$ are $\gr'_{\l}\cG'\cong {}^T\D_{\mu}\star_\infty \gr_{\l+\mu}\cG$ which is isomorphic to copies of ${}^T\D_{\mu}\star_\infty\dc_{\l+\mu}\cong \dc_{\l}$. Therefore $F'$ is also an admissible filtration on $\cG'$. By the uniqueness of the admissible filtration, we have $F'_{\le \l}\cG'=F_{\le \l}\cG'$, which proves (3).
\end{proof}

\sss{Wakimoto sheaves} 

For $\l\in \xcoch(T)$, denote by $J_{\l}\in\cH_{0}$ the Wakimoto sheaf defined in \cite[Corollary 1]{AB}. Recall that for $\l\in \xcoch(T)^{+}$, the dominant cone, $J_{-\l}\cong j_{-\l!}E\j{\ell(\l)}$, where $j_{-\l}: \bI t^{-\l}\bI/\bI\incl \Fl$ is the embedding. 

Denote by $\Wak\subset \cH_{0}$ the full subcategory whose objects are finite successive extensions of $\{J_{\l}\}_{\l\in\xcoch(T)}$. Further denote by $\Wak^{ss}$ the full subcategory of $\cH_{0}$ consisting of finite direct sums of  $\{J_{\l}\}_{\l\in \xcoch(T)}$. By \cite[Corollary 1]{AB},  both $\Wak$ and $\Wak^{ss}$ are monoidal subcategories of $\cH_{0}$. Moreover, we have a canonical monoidal equivalence
\begin{equation}\label{WakRepT}
\Rep(\dT)\isom \Wak^{ss}
\end{equation}
sending the one-dimensional representation $\l\in \xch(\dT)=\xcoch(T)$ to $J_{\l}$.

 
Consider the partial order of $\xcoch(T)$ given by $\l\leq\mu$ if $\mu-\l$ a sum of positive coroots. By \cite[3.6.5]{AB}, each object $\calF\in \Wak$ has a {\em canonical} filtration $F_{\leq\l}\calF$ such that $F_{\leq\l}\calF/F_{<\l}\calF$ is a direct sum of copies of $J_{\l}$. Moreover, by \cite[Proposition 6(a)]{AB}, the associated graded functor
\begin{equation*}
\gr:\Wak\rightarrow \Wak^{ss}\cong \Rep(\dT)
\end{equation*}
is a monoidal functor.

\begin{lemma}\label{l:Wak act dom}
Let $\xcoch(T)^{++}$ be the set of coweights $\l$ such that $\j{\l,\a}\ge1$ for each positive root $\a$. Then for $\l\in\xcoch(T)^{++}$ and  $\mu\in\xcoch(T)^{+}$, we have a canonical isomorphism $\dc_{\l}\cong\dc_{\l+\mu}\star_{0}J_{-\mu}$. 
\end{lemma}
\begin{proof} This follows from the fact that the multiplication  map gives an isomorphism
\begin{equation}\label{lm isom}
	G[\t]^{\psi}  t^{\l+\mu}\bI_{0}\times^{\bI_{0}}\bI_{0}t^{-\mu} \bI_{0}/\bI_{0}\cong G[\t]^{\psi}  t^{\l}\bI_{0}/\bI_{0}
\end{equation}
for $\l\in \xcoch(T)^{++}$ and $\mu\in \xcoch(T)^{+}$. Below we use abbreviated notation: for a subgroup $\bJ\subset G\lr{t}$ and $\nu\in \xcoch(T)$, we write 
\begin{equation*}
    {}^{\nu}\bJ:=\Ad(t^\nu)\bJ.
\end{equation*}
To prove \eqref{lm isom}, observe that its left side is $G[\t]^\psi$-equivariantly isomorphic to 
\begin{equation}\label{lm1}
G[\t]^{\psi}\times^{G[\t]^{\psi}\cap {}^{\l+\mu}\bI_0}({}^{\l+\mu}\bI_0)/({}^{\l+\mu}\bI_0\cap {}^\l\bI_0).
\end{equation}
To show \eqref{lm isom}, we only need to two things
\begin{enumerate}
    \item[(a)] \eqref{lm1} is homogeneous under $G[\t]^\psi$;
    \item[(b)] the stabilizer of the identity coset in ${}^{\l+\mu}\bI_0/({}^{\l+\mu}\bI_0\cap {}^\l\bI_0)$ in \eqref{lm1} is $G[\t]^\psi\cap {}^{\l+\mu}\bI_0
\cap {}^\l\bI_0$. 
\end{enumerate} 
Now (a) holds if and only if
\begin{equation*}
    {}^{\l+\mu}\bI_0=(G[\t]^{\psi}\cap {}^{\l+\mu}\bI_0)( {}^{\l+\mu}\bI_0\cap {}^\l\bI_0)
\end{equation*}
Conjugating by $t^{-\l-\mu}$ the above becomes
\begin{equation}\label{Jtrans}
    \bI_0=({}^{-\l-\mu}G[\t]^{\psi}\cap \bI_0)( \bI_0\cap {}^{-\mu}\bI_0).
\end{equation}
Let $x$ be the barycenter of the fundamental alcove of the building of $G\lr{t}$ corresponding to $\bI_0$.  Comparing affine roots of both sides of \eqref{Jtrans}, we see it is equivalent to that for any real affine roots $\a$ such that $\a(x)>0$, we have either $\a(x+\mu)>0$ or $\a(\l+\mu)<0$. In other words, to show (a), we only need to show that
\begin{eqnarray}\label{aff root sep 3 points}
    \mbox{No real affine root $\a$ satisfies $\a(x)>0, \a(x+\mu)< 0$ and $\a(\l+\mu)\ge0$.}
\end{eqnarray}
On the other hand, (b) is equivalent to 
\begin{equation*}
    G[\t]^\psi\cap {}^{\l+\mu}\bI_0
\cap {}^\l\bI_0=G[\t]^\psi\cap{}^\l\bI_0
\end{equation*}
or equivalently
\begin{equation*}
    {}^{-\l-\mu}G[\t]^\psi\cap{}^{-\mu}\bI_0\subset \bI_0.
\end{equation*}
Comparing affine roots on both sides, it is further equivalent to
\begin{equation}\label{aff root sep 3 points neg}
    \mbox{No real affine root $\a$ satisfies $\a(x)<0, \a(x+\mu)> 0$ and $\a(\l+\mu)<0$.}
\end{equation}
It remains to show \eqref{aff root sep 3 points} and \eqref{aff root sep 3 points neg} for $\l\in \xcoch(T)^{++}$ and $\mu\in \xcoch(T)^+$. The two statements are clearly equivalent, so we only need to show \eqref{aff root sep 3 points}. Write $\a=\ov\a+n$ as an affine function on the standard apartment, where $\ov\a$ is a finite root. The stated properties of $\a$ imply that
\begin{equation}\label{ovalp}
    \j{\mu,\ov\a}<0, \mbox{ and } \j{\l-x, \ov\a}>0.
\end{equation}
Since $\mu$ is dominant, we see $\ov\a<0$. Since $\l\in \xcoch(T)^{++}$, $\l-x$ still lies in the interior of the dominant cone, hence $\j{\l-x, \ov\a}< 0$, which  contradicts the second inequality in \eqref{ovalp}.  This proves \eqref{aff root sep 3 points}, which implies (a) and (b).
\end{proof}

\begin{prop}\label{Prop:wakvsHinfty}
\begin{enumerate}
\item The monoidal category $\Wak^{ss}$ acts on the category ${}^{T}\cC^{ss}$.

\item There are canonical isomorphisms
\begin{equation*}
    c_{\l,\mu}: \dc_{\l}\star_{0}J_{\mu}\cong\dc_{\l+\mu}
\end{equation*}
for all $\l,\mu\in \xcoch(T)$ satisfying the following compatibilities
\begin{enumerate}
    \item[(a)] For $\l, \mu,\mu'\in \xcoch(T)$, $c_{\l,\mu+\mu'}$ is the composition
    \begin{equation*}
        \dc_{\l}\star_{0}J_{\mu+\mu'}\cong \dc_{\l}\star_{0}J_{\mu}\star_0 J_{\mu'}\xr{c_{\l,\mu}\star\id(J_{\mu'})}\dc_{\l+\mu}\star_0 J_{\mu'}\xr{c_{\l+\mu, \mu'}}\dc_{\l+\mu+\mu'}.
    \end{equation*}
    \item[(b)] For $\l,\mu,\nu\in \xcoch(T)$, $c_{\nu+\l,\mu}$ is equal to the composition
    \begin{equation*}
        \dc_{\nu+\l}\star_{0}J_{\mu}={^T}\D^T_{-\nu}\star_\infty \dc_{\l}\star_0 J_{\mu} \xr{\id({^T}\D^T_{-\nu})\star_\infty c_{\l,\mu}}{}^T\D^T_{-\nu}\star_\infty \dc_{\l+\mu}=\dc_{\nu+\l+\mu}.
    \end{equation*}
\end{enumerate}
\item The $\Wak^{ss}$-module category ${}^{T}\cC^{ss}$ is free of rank one.
\end{enumerate}
\end{prop}
\begin{proof}
	(1) will follow from the existence of the isomorphisms $c_{\l,\mu}$ in (2).
    
    (2) For $\mu\in \xcoch(T)^{-}$ (anti-dominant) and $\l+\mu\in \xcoch(T)^{++}$, we let $c_{\l,\mu}: \dc_{\l}\star_0 J_{\mu}\isom \dc_{\l+\mu}$ be the isomorphism constructed in Lemma \ref{l:Wak act dom}. The compatibility in (a) holds whenever $\mu,\mu'\in \xcoch(T)^-$ and $\l+\mu+\mu'\in \xcoch(T)^{++}$. Indeed, the equation to be checked holds up to a nonzero scalar in $\Qlbar$, and it is enough to check it for the maps between stalks at $t^{\l+\mu+\mu'}$, for which the equality follows from the isomorphism \eqref{lm isom} that gives the isomorphisms in Lemma \ref{l:Wak act dom}. Similarly, (b) holds whenever $\mu\in \xcoch(T)^{-}$, $\l+\mu\in \xcoch(T)^{++}$ and $\nu\in \xcoch(T)^+$, as can be checked on the stalks at $t^{\nu+\l+\mu}$.
    
    Now we define $c_{\l,\mu}$ for $\mu\in \xcoch(T)^{-}$ and general $\l$. Choose $\nu\in \xcoch(T)$ such that $\nu+\l+\mu\in \xcoch(T)^{++}$. Now $c_{\nu+\l,\mu}: \dc_{\nu+\l}\star_0 J_\mu\cong\dc_{\nu+\l+\mu}$ is already defined.  Acting on both sides by ${}^T\D_{\nu}^T$ we denote the resulting isomorphism $\dc_{\l}\star_0 J_{\mu}={}^T\D^T_\nu\star_\infty\dc_{\nu+\l}\star_{0}J_{\mu}\cong {}^T\D^T_\nu\star_\infty\dc_{\nu+\l+\mu}=\dc_{\l+\mu}$ by $c_{\l,\mu}$. The compatibility (b) for $c_{\l,\mu}$ in the previous paragraph shows that $c_{\l,\mu}$ thus defined is independent of the choice of $\nu$. Moreover, the $c_{\l,\mu}$ defined so far satisfies the compatibility (a) whenever the terms are defined. 
    
    Next we define $c_{\l,\mu}$ for general $\l,\mu\in \xcoch(T)$. Choose $\mu'\in\xcoch(T)^{-}$ such that $\mu+\mu'\in\xcoch(T)^{-}$. Now $c_{\l,\mu+\mu'}$ and $c_{\l+\mu, \mu'}$ are both defined in the previous paragraph, and they induce an isomorphism
    \begin{equation*}
        \dc_{\l}\star_0 J_{\mu}\star_0 J_{\mu'}\cong \dc_{\l}\star_0 J_{\mu+\mu'}\xr{c_{\l,\mu+\mu'}}\dc_{\l+\mu+\mu'}\xr{c_{\l+\mu, \mu'}^{-1}} \dc_{\l+\mu}\star_0 J_{\mu'}.
    \end{equation*}
    Applying $\star_0 J_{-\mu'}$ to both sides gives the desired isomorphism $c_{\l,\mu}$. The partial compatibility (a) shows that $c_{\l,\mu}$ is independent of the choice of $\mu'$. Finally one can show that the fully defined $\{c_{\l,\mu}\}$ satisfy the compatibilities (a) and (b), by reducing to the partial compatibilities (a) and (b) we checked in the beginning.  
    
	
	
    
    (3) The functor $\dc_{0}\star_{0}(-): \Wak^{ss}\to {}^{T}\cC^{ss}$ sends simple objects $J_{\l}$ of $\Wak^{ss}$ to simple objects $\dc_{\l}$. Since both $\Wak^{ss}$ and ${}^{T}\cC^{ss}$ are semisimple, this functor is an equivalence, proving that ${}^{T}\cC^{ss}$ is free of rank one as a $\Wak^{ss}$-module category.
    \end{proof}

\begin{cor}\label{Cor:wakvsHinfty} 
\begin{enumerate}
\item The actions of $\Hinfhs$ and  $\Wak^{ss}$ on ${}^{T}\cC^{ss}$ induce equivalences of monoidal categories
\begin{equation*}
a_{\Wak}^{ss}: \Wak^{ss}\isom\End_{\Hinfhs}({}^{T}\cC^{ss}).
\end{equation*}

\item Using the free generator $\dc_0$ to identify ${}^T\cC^{ss}$ with $\Hinfhs$ as a left $\Hinfhs$-module, we have a monoidal equivalence $\End_{\Hinfhs}({}^{T}\cC^{ss})\cong \Hinfhs$.  Under the equivalences \eqref{WakRepT} and \eqref{RepT Hinf}, there is a canonical isomorphism of monoidal functors making the following diagram commutative
\begin{equation*}
    \xymatrix{\Wak^{ss}\ar[d]^{\cong}\ar[r]^-{a_{\Wak}^{ss}} & \End_{\Hinfhs}({}^{T}\cC^{ss}) \ar[r]^-{\sim} & \Hinfhs\ar[d]^{\cong}\\
    \Rep(\dT) \ar[rr]^-{\inv} && \Rep(\dT)}
\end{equation*}
Here $\inv$ is the auto-equivalence of $\Rep(\dT)$ induced from the inversion on $\dT$, i.e., the functor that sends $E(\mu)$ to $E(-\mu)$.
\end{enumerate}
\end{cor}
\begin{proof}
(1) The category ${}^{T}\cC^{ss}$ carries commuting left action by $\Hinfhs$ and right action by $\Wak^{ss}$, and it is free of rank one over both by Lemma \ref{l:Css free over Hinf} and Proposition \ref{Prop:wakvsHinfty}. 

(2) It suffices to give canonical isomorphisms $\g_\mu: a^{ss}_{\Wak}(J_\mu)\cong {}^T\D^{T}_{-\mu}$ for all $\mu\in \xcoch(T)$ compatible with convolutions. By definition, such an isomorphism $\g_\mu$ for an individual $\mu$ amounts to the same thing as a collection of isomorphisms $\dc_\l\star_0 J_\mu\cong \dc_{\l+\mu}$ for all $\l\in \xcoch(T)$ compatible with the left action by ${}^T\D^T_\nu$ for any $\nu\in\xcoch(T)$. Such a collection is given by $c_{\l,\mu}$, and the required compatibility boils down to Proposition \ref{Prop:wakvsHinfty}(2)(a). This defines $\g_\mu$ for each $\mu\in \xcoch(T)$. The fact that these $\g_\mu$ are compatible with convolutions follows from Proposition \ref{Prop:wakvsHinfty}(2)(b).

	
\end{proof}

Let $\End({}^{T}\cC; \gr)$ be the category of triples $(\Phi, \Phi_{\gr}, \a)$ where $\Phi:{}^{T}\cC\to {}^{T}\cC$ is an endo-functor of ${}^{T}\cC$, $\Phi_{\gr}: {}^{T}\cC^{ss}\to {}^{T}\cC^{ss}$ is an endo-functor of ${}^{T}\cC^{ss}$, and $\a: \gr\c\Phi\stackrel{\sim}{\Rightarrow} \Phi_{\gr}\c\gr$ is a natural isomorphism of functors making the following diagram commutative
\begin{equation*}
\xymatrix{{}^{T}\cC\ar[d]_{\gr}\ar[r]^-{\Phi} & {}^{T}\cC\ar[d]^{\gr}_{}="b"\\
{}^{T}\cC^{ss}\ar[r]_-{\Phi_{\gr}}^{}="a" & {}^{T}\cC^{ss} \ar@{=>} "b";"a" _-{\a}}
\end{equation*}
A morphism two such triples $(\Phi, \Phi_{\gr}, \a)\to (\Phi', \Phi'_{\gr}, \a')$ is a pair of natural transformations $\b: \Phi\Rightarrow\Phi'$, $\b_{\gr}: \Phi_{\gr}\Rightarrow\Phi'_{\gr}$ such that $\gr\c\b\in \End(\gr\c\Phi)$ and $\b_{\gr}\c\gr\in \End(\Phi_{\gr}\c\gr)$ are intertwined by the isomorphism $\a$. Then $\End({}^{T}\cC; \gr)$ has an obvious monoidal structure under composition of functors.

Let $\End_{\Hinfhs}({}^{T}\cC; \gr)\subset \End({}^{T}\cC; \gr)$ be monoidal subcategory with further compatibility data with the action of $\Hinfhs$ on ${}^{T}\cC$ and on ${}^{T}\cC^{ss}$ (any morphisms are required to be compatible with the $\Hinfhs$-action as well). We have forgetful functors
\begin{equation*}
\xymatrix{\End_{\Hinfhs}({}^{T}\cC) & 
\ar[l]_-{\om} \End_{\Hinfhs}({}^{T}\cC; \gr)\ar[r]^-{\om_{\gr}} &  \End_{\Hinfhs}({}^{T}\cC^{ss})}.
\end{equation*}

\begin{prop}
	The category $\Wak$ acts on ${}^{T}\cC$ compatibly with the action of $\Wak^{ss}$ on ${}^{T}\cC^{ss}$ via the canonical grading functors on both categories. In other words, the action functor $\Wak\to \End_{\Hinfhs}({}^{T}\cC)$ has a canonical lift to a monoidal functor
\begin{equation*}
a_{\Wak}: \Wak\to \End_{\Hinfhs}({}^{T}\cC;\gr).
\end{equation*}

\end{prop}
\begin{proof}
	The action of $\Wak$ on ${}^{T}\cC$ follows immediately from the action of $\Wak^{ss}$ on ${}^{T}\cC^{ss}$.

    We show that for any $\cG\in {}^T\cC$ and any $\mu\in \xcoch(T)$, convolution with $J_\mu$ satisfies
    \begin{equation}\label{fil conv J}
        (F_{\le\l}\cG)\star_0 J_{\mu}=F_{\le \l+\mu}(\cG\star_0 J_\mu)
    \end{equation}
    as subobjects of $\cG\star_0 J_{\mu}$. 

    To show \eqref{fil conv J}, the argument is the same as the proof of Lemma \ref{l:filC}(3): define another filtration $F'$ on $\cG'=\cG\star_0 J_{\mu}$ by $F'_{\le\l}\cG'=F_{\le\l-\mu}\cG\star_0 J_{\mu}$. We check that $F'$ is also an admissible filtration because its associated graded are $\gr'_{\l}\cG'\cong \gr_{\l-\mu}\cG\star_0 J_{\mu}$, which is isomorphic to a direct sum of copies of $\dc_{\l-\mu}\star_0 J_{\mu}\cong \dc_{\l}$ by Proposition \ref{Prop:wakvsHinfty}(2).  Hence $F'$ has to be equal to the filtration $F$ on $\cG$ by the uniqueness of the admissible filtration, proving \eqref{fil conv J}.

    Now consider any objects $\cK\in \Wak$ and $\cG\in {}^T\cC$, each equipped with canonical filtrations $F_{\le \mu}\cK$ (whose associated graded are direct sums of $J_\mu$) and $F_{\le \l}\cG$. Then \eqref{fil conv J} implies that 
    \begin{equation*}
        F_{\le\l}(\cG\star_0 \cK)=\sum_{\l'+\mu\le \l}(F_{\le\l'}\cG)\star_0 (F_{\le\mu}\cK) 
    \end{equation*}
    as subobjects of $\cG\star_0 \cK$. Passing to the associated graded we get a canonical functorial isomorphism
    \begin{equation*}
    (\gr_{\l'}\cG)\star_0(\gr_{\mu} \cK)\cong \gr_{\l'+\mu}(\cG\star_0 \cK)
    \end{equation*}
    which gives a lifting of $\cK$ to $\wt\cK\in \End({}^T\cC; \gr)$. Moreover, the compatibility of the canonical filtrations on objects in ${}^T\cC$ with convolution by $\Hinfhs$ as in Lemma \ref{l:filC}(3) shows that $\wt\cK$ indeed lands in the subcategory $\End_{\Hinfhs}({}^T\cC; \gr)$. One easily checks that the assignment $\cK\mt \wt\cK$ extends to a monoidal functor $a_{\Wak}: \Wak\to \End_{\Hinfhs}({}^T\cC; \gr)$, as desired. 

\end{proof}

\subsection{Proof of Theorem \ref{th:rho comm}}
\label{ss:pf sph action}

\sss{Reduction to equivariant version}
Since $\r$ can be factored as in \eqref{Trho and rho} as the composition of ${}^T\r$ and the forgetful functor $\eqHinf\to D^b(\Rep(\dT))$, it suffices to prove the version of Theorem \ref{th:rho comm} for ${}^T\r$. Below we will work systematically with the $T$-equivariant versions ${}^T\DG$ and ${}^T\wt \cD_\psi$, with the actions of $\eqHinf$.

\sss{Central functor}
Recall Gaitsgory's central functor \cite{Ga}
\begin{equation*}
\cZ: \cH^{\sph}_{0}\to \cH_{0}.
\end{equation*} 
By \cite[Proposition 5(b)]{AB}, we have $\cZ(\cK)\in \Wak$ for $\cK\in \cH^{\sph,\hs}_{0}$. This way we get a monoidal functor $\cZ_{\Wak}: \cH^{\sph,\hs}_{0}\to \Wak$. By passing to associated graded, we get a monoidal functor with a canonical central structure
\begin{equation*}
\cZ^{ss}_{\Wak}:  \cH^{\sph,\hs}_{0}\xr{\cZ_{\Wak}} \Wak\xr{\gr}\Wak^{ss}.\end{equation*}
By \cite[Theorem 6]{AB}, after identifying both sides above with $\Rep(\dG)$ (via $\Sat$) and $\Rep(\dT)$ (see \eqref{WakRepT}), there is a natural isomorphism of monoidal functors 
\begin{equation}\label{ZWak Res}
\cZ^{ss}_{\Wak}\cong \Res^{\dG}_{\dT}: \Rep(\dG)\to \Rep(\dT).
\end{equation}
Now consider the commutative diagram
\begin{equation}\label{Hsph action Wak}
\xymatrix{  \cH^{\sph,\hs}_{0}\ar[r]^-{\cZ_{\Wak}} \ar[dr]_{\cZ^{ss}_{\Wak}}& \Wak \ar[r]^-{a_{\Wak}}\ar[d]^{\gr} & \End_{\Hinfhs}({}^{T}\cC;\gr)^{\rev}\ar[d]^{\om_{\gr}}\\
& \Wak^{ss}\ar[r]^-{a^{ss}_{\Wak}}_-{\sim} & \End_{\Hinfhs}({}^{T}\cC^{ss})^{\rev}\ar[rr]^-{\End({}^{T}\pi_{*}|{}^{T}\cC^{ss})}_-{\sim} 
&&\End_{\Hinfhs}({}^{T}\DG^{\hs})^{\rev} }
\end{equation}
Here the bottom right equivalence is induced by ${}^{T}\pi_{*}|_{{}^{T}\cC^{ss}}$ from Lemma \ref{l:Css equiv DGr}. 

\begin{lemma}\label{l:Wak diagram}
The composition of the arrows through the top  row of \eqref{Hsph action Wak} down to $\End_{\Hinfhs}({}^{T}\DG^{\hs})^{\rev}$ is naturally isomorphic to the monoidal functor defined using the action of $\cH^{\sph}_{0}$ on ${}^{T}\cD_{\psi, \Gr}$ in \eqref{def Trho}. 
\end{lemma}
\begin{proof} For $\cK\in \cH_{0}^{\sph,\hs}$ and $\cG\in {}^{T}\cC$, we need to give a functorial isomorphism
\begin{equation*}
{}^T\pi_{*}(\cG\star_{0}\cZ_{\Wak}(\cK))\cong ({}^T\pi_{*}\cG)\star_{0}\cK.
\end{equation*}
We have
\begin{equation*}
\pi_{*}(\cG\star_{0}\cZ_{\Wak}(\cK))\cong \cG\star_{0}\pi_{\Gr,*}\cZ_{\Wak}(\cK)\cong \cG\star_{0}({}^{\bI_{0}}\cK).
\end{equation*}
Here $\pi_{\Gr,*}$ is the direct image functor $\cH_{0}\to D^{b}(\bI_{0}\bs \Gr)$; ${}^{\bI_{0}}\cK\in D^{b}(\bI_{0}\bs \Gr)$ is the same as $\cK$ but with the left $\bG$-equivariance forgotten down to $\bI_{0}$-equivariance. We are using here that $\pi_{*}\cZ(\cK)\cong {}^{\bI_{0}}\cK$, which follows directly from the definition of $\cZ$ by nearby cycles \cite{Ga}. Since ${}^{\bI_{0}}\cK$ comes from the $\bG$-equivariant object $\cK$, we have
\begin{equation*}
\cG\star_{0}({}^{\bI_{0}}\cK)\cong ({}^T\pi_{*}\cG)\star_{0}\cK
\end{equation*}
where on the right side the convolution contracts $\bG$-equivariance while on the left it contracts $\bI_{0}$-equivariance. 
\end{proof}

\sss{Proof of Theorem \ref{th:rho comm}(1)}
Lemma \ref{l:Wak diagram} in particular shows that the action of $\cH^{\sph,\hs}_{0}$ on ${}^T\DG$ preserves the heart ${}^T\DG^\hs$. This also implies the action of $\cH^{\sph,\hs}_{0}$ on $\DG$ preserves the heart $\DG^\hs$, because the forgetful functor ${}^T\DG^\hs\to \DG^\hs$ is essentially surjective by Lemma \ref{l:eqDG equiv heart}. This proves part (1) of Theorem \ref{th:rho comm}.

\sss{Proof of Theorem \ref{th:rho comm}(2)}
Lemma \ref{l:Wak diagram} together with the commutative diagram \eqref{Hsph action Wak} gives a factorization of $\r^{\hs}$ as the composition
\begin{equation*}
\cH^{\sph,\hs}_{0}\xr{\cZ_{\Wak}^{ss}}\Wak^{ss}\xr{a_{\Wak}^{ss}}\End_{\Hinfhs}({}^{T}\cC^{ss})^{\rev} \isom (\eqHinf)^{\hs}.
\end{equation*}
By \eqref{ZWak Res} and Corollary \ref{Cor:wakvsHinfty}, the above diagram fits into a commutative diagram with canonical isomorphisms of monoidal functors
\begin{equation*}
    \xymatrix{\cH^{\sph,\hs}_{0}\ar[r]^{\cZ_{\Wak}^{ss}}\ar[d]^{\cong} & \Wak^{ss}\ar[d]^{\cong}\ar[r]^-{a_{\Wak}^{ss}} & \End_{\Hinfhs}({}^{T}\cC^{ss})^{\rev}\ar[r]^-{\sim} & (\eqHinf)^{\hs}\ar[d]^{\cong}\\
    \Rep(\dG) \ar[r]^-{\Res^{\dG}_{\dT}} & \Rep(\dT)\ar[rr]^-{\inv} && \Rep(\dT)
    }
\end{equation*}
This finishes the proof of Theorem \ref{th:rho comm}. \qed

\subsection{Examples}\label{Eg_SL_n}

We describe some objects in $\cD_{\psi}$ when $G$ is of type $A$. We ignore Tate twists here. We will identify the Kirillov category on the open $G/B\hookrightarrow\Bun_{G}(\J, \bI_{0})$ with sheaves on $G/B$ and similarly for other parahoric level structures at $0$.

\begin{exam}[$\SL_{n}$] We consider the pushforward of $\mathcal{F}_0$ to $\mathcal{D}_{\psi,\mathbf{P}}$ for certain parahorics $\mathbf{P}$ and consider some irreducible summands of these when $G=\SL_{n}$.

The only cases of interest are when $s_0\in W_\mathbf{P}$. Let $\mathbf{P}'$ be the parahoric whose Weyl group is given by $W\cap W_\mathbf{P}$, and let $P\subset G$ be the corresponding parabolic subgroup. Then the map $\pi_\mathbf{P}$ factors through $\pi_{\mathbf{P}'}$. Thus we can simply compute the pushforward along the map $\pi_{\mathbf{P}',\mathbf{P}}$. 

Assume that $W_\mathbf{P}$ includes $s_i$ for $i<k$ and $i>j$ and does not include $s_j$ and $s_k$. Then the relevant orbits in $G/P\rightarrow \Bun_G(\J,\mathbf{P})$ are given by 
$$\cH_{\mathbf{P}}(0) = \{(V_i)\in G/P|\psi V_k\subset V_j\}.$$

This is a Hessenberg variety in $G/P$, generalizing the variety $\cH(0)$ from Section \ref{Hess_var}. Our category thus contains $\pi_{\mathbf{P}}^*E_{\cH_{\mathbf{P}}(0)}[\dim(\cH_{\mathbf{P}}(0)+\dim(P/B)]$, the pullback of the constant sheaf on the Hessenberg variety $\cH_{\mathbf{P}}(0)$, as one of the irreducible objects. 
	
\end{exam}

\begin{exam}[$\SL_{4}$]
	For $\SL_3$ the previous construction essentially gives all irreducibles. Here we consider the category $\mathcal{D}_\psi$ in the case of $\SL_4$, where some more interesting irreducibles appear. Of particular interest will be some non-trivial local systems appearing in the microlocalization of certain irreducible objects. 
	
	To construct these objects we will act with $\cH_{0}$ on some of the objects of Example \ref{Eg_SL_n}. Recall that there is an open inclusion $G/B\hookrightarrow\Bun_{G}(\J, \bI_{0})$. We will describe the restriction of some irreducibles to this $G/B$. From Example \ref{Eg_SL_n} we have: 
    \begin{align*}
        \cF_{0}&:= \IC(G/B)\\
        \cF_{1}&:= \IC(\{(V_i)\in G/B|\psi V_1\subset V_3\})\\
        \cF_{2}&:= \IC(\{(V_i)\in G/B|\psi V_2\subset V_3\})\\
        \cF_{3}&:= \IC(\{(V_i)\in G/B|\psi V_1\subset V_2\}).
    \end{align*}
    
    Denote by $\alpha_0,\dots,\alpha_3$ the simple roots, where $\alpha_1,\dots,\alpha_3$ denote the finite simple roots. For a simple root $\alpha_i$ we denote by $\mathbf{P}_{\alpha_i}=\mathbf{P}_i$ the parahoric generated by the Iwahori $\mathbf{I}$ and the simple reflection $s_{\alpha_i}=s_i$. Finally denote $\pi_i=\pi_{\mathbf{P}_i}$. Then we have
	
	\begin{equation*}
	    \cF_{4}:=\cF_{3}\star_{0} \IC_{s_1}\cong \cF_{2}\star_{0} \IC_{s_3}=\IC(\{(V_i)\in G/B|\psi V_2\cap V_2\neq\{0\}\}).
	\end{equation*}

    Finally we compute $\cF_{4}\star_{0} \IC_{s_2}=\pi_2^*(\pi_2)_*\cF_{4}[1]$. Here note that when projecting $\{(V_i)\in G/B|\psi V_2\cap V_2\neq\{0\}\}$ to the partial flag variety, the fibers are cut out by a quadratic equation on $\mathbb{P}^1$ determined by $\psi V_2\cap V_2\neq\{0\}$. The support of $\cF_{4}$ is irreducible and over the partial flag variety $G/P_2$ is a branched double cover. It follows that on some open $U\subset G/P_2$, $(\pi_2)_*\cF_{4}\cong \cL[5]\oplus E_U[5]$ for some non-trivial local system $\cL$. We thus get our final irreducible

    \begin{equation*}
	    \cF_{5}:=\IC(\pi_2^*(\cL)).
	\end{equation*}

    The microlocalisation of this object to $G/B$ produces a non-trivial local system on an open subset of the latter. It also gives rise to a trivializable local system on a different irreducible component of $\Fl_{\psi}$.
    
    The above irreducibles are in fact all the irreducibles up to the extended action of $\cH_{\infty}$ and the length zero elements in $\cH_{0}$ of $\GL_{4}$.
\end{exam}

\subsection{Full faithfulness of microlocalization functor}

\sss{Singular support}

By the discussion of \S\ref{s:micro}, we have a microlocalization functor
\begin{equation*}
\wt M: \wt\cD_{\psi}\to \muSh(\cM_{\psi}).
\end{equation*}
Here we are using \eqref{def Mpsi}.

\begin{lemma}\label{l:SS finite}
For any $\cF\in \cD_{\psi}$, $\wt M(\cF)$ is supported on finitely many components of the Lagrangian $\L_{\psi}$.
\end{lemma}
\begin{proof}
Since $\cD_{\psi}$ is generated under $\cH_{0}$ by $\cF_{\l}$, and $\cH_{0}$ is generated by $\IC_{w}$ for $\ell(w)=0$ or $1$ as a monoidal category, we need to check
\begin{enumerate}
\item $SS(\cF_{\l})\subset \L_{\psi}$ and has finite-type support;
\item If $\cF\in \wt\cD_{\psi}$ has finite-type  support in $\L_{\psi}$, so does $\cF\star\IC_{w}$ for any $\ell(w)=0$ or $1$, $w\in \tilW$. 
\end{enumerate}
Consider the Lagrangian correspondence attached to $\pi=\pi_{\bG}$
\begin{equation}\label{dpi}
\xymatrix{\cM_{\psi} & \L_{\psi}\cong \Fl_{\psi}  \ar@{_{(}->}[l] & \cM_{\psi,\bG}\times_{\Bun_{G}(\bG^{1}_{\infty})}\Bun_{G}(\J,\bI_{0})\ar[l]_-{d\pi}\ar[r]^-{\pi_{\na}} & \cM_{\psi,\bG}\cong \Gr_{\psi}}
\end{equation}

(1) Use the formula of singular support under smooth pullback $\pi^{*}$, we have
\begin{equation*}
SS(\cF_{\l})=SS(\pi^{*}\d_{\l})=d\pi \pi_{\na}^{-1}(SS(\d_{\l})).
\end{equation*}
From the diagram \eqref{dpi} we see $SS(\cF_{\l})\subset \L_{\psi}$. More precisely,  by Theorem \ref{th:MG equiv}(1), $SS(\d_{\l})$ is the point $t^{\l}\in \Gr_{\psi}$. Therefore $SS(\cF_{\l})=d\pi \pi_{\na}^{-1}(t^{\l})$, which is equal to $t^{\l}G\tl{t}/\bI_0\subset \Fl_{\psi}$, and is isomorphic to the flag variety of $G$. 

(2)  The case $\ell(w)=0$ is clear since $w$ induces an automorphism of $\Bun_{G}(\Jker, \bI_{0})$ that preserves the Hitchin map hence preserves $\L_{\psi}$.

We check the case $w=s$ is a simple reflection. Let $\bP_{s}\supset \bI$ be the standard parahoric subgroup corresponding to $s$, and $\pi_{s}=\pi_{\bP_{s}}$. We have $\cF\star \IC_{s}\cong \pi_{s}^{*}\pi_{s*}\cF\j{1}$.  Consider the Lagrangian correspondence attached to $\pi_{s}$
\begin{equation}\label{Lag pi s}
\xymatrix{\cM_{\psi}  & \cM_{\psi,\bP_{s}}\times_{\Bun_{G}(\J,\bP_{s,0})}\Bun_{G}(\J,\bI_{0})\ar[l]_-{d\pi_{s}}\ar[r]^-{\pi_{s,\na}} & \cM_{\psi,\bP_{s}}
}
\end{equation}
Since $\pi_{s}$ is smooth and proper, we have
\begin{equation*}
SS(\pi_{s}^{*}\pi_{s*}\cF)\subset d\pi_{s}\pi_{s,\na}^{-1}\pi_{s,\na}(d\pi_{s})^{-1}(SS(\cF)).
\end{equation*}
Now the diagram \eqref{Lag pi s} is compatible with the Hitchin maps $f_{\psi}$ and its analogue for $\cM_{\psi,\bP}$. Therefore $f_{\psi}SS(\pi_{s}^{*}\pi_{s*}\cF)$ is contained in $f_{\psi}SS(\cF)$, which is the point $a_{\psi}$.  Hence $SS(\pi_{s}^{*}\pi_{s*}\cF)\subset \L_{\psi}$. Moreover,  $d\pi_{s}$ is a closed embedding and $\pi_{s,\na}$ is a fibration with fibers isomorphic to the flag variety of $L_{\bP}$, both pushforward and pullback under $\pi_{s,\na}$ and $d\pi_{s}$ send finite-type algebraic  cycles to finite-type algebraic cycles. Therefore $SS(\pi_{s}^{*}\pi_{s*}\cF)$ is of finite type. This finishes the proof.
\end{proof}

\begin{defn}\label{def:mush fs}
    Let $\muSh_{\L_{\psi},fs}(\cM_{\psi})\subset \muSh_{\L_{\psi}}(\cM_{\psi})$ be the full subcategory consisting of objects supported on finitely many irreducible components of $\L_{\psi}$ such that the microstalks are finite dimensional bounded complexes of vector spaces.  
\end{defn}

Lemma \ref{l:SS finite} then shows that $\wt M|_{\cD_{\psi}}$ lands in $\muSh_{\L_{\psi},fs}(\cM_{\psi})$.

\begin{theorem}\label{th:fullfaith} For $\cF\in \cD_{\psi}$ and $\cG\in \wt\cD_{\psi}$, the natural map
\begin{equation}\label{hom to check}
\Hom_{\wt\cD_{\psi}}(\cF,\cG)\to \cohog{0}{\cM_{\psi}, \muhom(\cF,\cG)}=\Hom_{\muSh(\cM_{\psi})}(\wt M(\cF), \wt M(\cG))
\end{equation}
is an isomorphism. In particular, $\wt M$ restricts to a fully faithful functor
\begin{equation*}
M: \cD_{\psi}\incl\muSh_{\L_{\psi}, fs}(\cM_{\psi}).
\end{equation*}
\end{theorem}
\begin{proof} Let $w\in \tilW$ be either of length $1$ or length $0$. We have the adjunction
\begin{equation}\label{adj ICw Dpsi}
\Hom_{\wt\cD_{\psi}}(\cF\star\IC_{w},\cG)\cong \Hom_{\wt\cD_{\psi}}(\cF,\cG\star\IC_{w^{-1}}).
\end{equation}
We claim that there is a similar adjunction 
\begin{equation}\label{adj ICw Mpsi}
\cohog{0}{\cM_{\psi}, \muhom(\cF\star\IC_{w},\cG)}\cong \cohog{0}{\cM_{\psi}, \muhom(\cF,\cG\star\IC_{w^{-1}})}
\end{equation}
that is compatible with \eqref{adj ICw Dpsi} under the map \eqref{hom to check}. Indeed, when $\ell(w)=0$ this is clear since $\star\IC_{w}$ is induced by an automorphism of $\Bun_{G}(\bI_{0}, \bG^{\psi}_{\infty})$. When $w=s$ is a simple reflection, then $\star\IC_{s}=\pi_{s}^{*}\pi_{s*}\j{1}$.  The desired adjunction follows from 
\begin{eqnarray*}
\cohog{0}{\cM_{\psi}, \muhom(\pi_{s}^{*}\pi_{s*}\cF[1],\cG)}&\cong&\cohog{0}{\cM_{\psi,\bP_{s}}, \muhom(\pi_{s*}\cF[1],\pi_{s*}\cG)}\\
&\cong&\cohog{0}{\cM_{\psi}, \muhom(\cF[1],\pi_{s}^{!}\pi_{s*}\cG)[-1]}\\
&\cong&\cohog{0}{\cM_{\psi}, \muhom(\cF,\pi_{s}^{*}\pi_{s*}\cG[1])}.
\end{eqnarray*}
Here the first and second isomorphisms follow by taking sections of the isomorphisms in \cite[Proposition 4.4.6 (iii)(iv)]{KS} over $\cM_{\psi}$ ( using that $\pi_{s}$ is smooth and proper). 

Since $\{\cF_{\l}\star\IC_{w_{1}}\star\cdots\star\IC_{w_{n}}; \l\in \xcoch(T), \ell(w_{i})\le 1, n\in \NN\}$ generate $\cD_{\psi}$ by definition,  it suffices to check \eqref{hom to check} for $\cF$ and $\cG$ in this collection.  By the adjunctions \eqref{adj ICw Dpsi} and \eqref{adj ICw Mpsi}, we reduce to the case $\cF=\cF_{\l}$ and $\cG$ arbitrary. In this case, using adjunctions analogous to \eqref{adj ICw Dpsi} and \eqref{adj ICw Mpsi} for $\pi=\pi_{\bG}$ instead of $\pi_{s}$, we get
\begin{equation*}
\Hom_{\wt\cD_{\psi}}(\cF_{\l}, \cG)= \Hom_{\wt\cD_{\psi}}(\pi^{*}\d_{\l}, \cG)\cong \Hom_{\DG}(\d_{\l}, \pi_{*}\cG)
\end{equation*}
and
\begin{equation*}
\cohog{0}{\cM_{\psi}, \muhom(\cF_{\l}, \cG)}=\cohog{0}{\cM_{\psi}, \muhom(\pi^{*}\d_{\l}, \cG)}\cong \cohog{0}{\cM_{\psi,\bG}, \muhom(\d_{\l}, \pi_{*}\cG)}.
\end{equation*}
Therefore we reduce to showing that the natural map 
\begin{equation*}
\Hom_{\DG}(\cF',\cG')\to \cohog{0}{\cM_{\psi,\bG}, \muhom(\cF', \cG')}=\Hom_{\muSh(\cM_{\psi,\bG})}(M_{\bG}(\cF'), M_{\bG}(\cG'))
\end{equation*}
is an isomorphism for $\cF',\cG'\in \DG$. This follows from Theorem \ref{th:MG equiv}, which says that $M_{\bG}$ is an equivalence.
\end{proof}

\subsection{Singular support characterization of $\cD_\psi$}
In this subsection we show the converse to Lemma \ref{l:SS finite}, characterizing
$\cD_\psi$ as a full subcategory in the Kirillov category $\wt\cD_{\psi}$ specified by a condition on the singular support.

We will need an auxiliary construction. For $\l \in \xcoch(T)$ let
$(G/B)_\lambda=t^\l G\tl{t}/\bI_0\subset \Fl_\psi$ be the corresponding irreducible component isomorphic to the flag variety $G/B$. For an object $\F$ in the Kirillov category $\wt\cD_{\psi}$ we set $\mu_\lambda(\F)=i_\psi^*\mu_{(G/B)_\lambda}(\F) $ where $\mu$ denotes microlocalization
and $i_\psi$ is the section of the conormal bundle to $(G/B)_\lambda$  defined by $\psi$. 

\begin{lemma}\label{micro_action}
    \label{mu l} For any $\cF\in \wt\cD_\psi$ and $\l\in \xcoch(T)$, we have:
\begin{enumerate}
    \item $\mu_\lambda(\F)\cong \muhom(\pi^*\delta_\l,\F)$. 
    \item The functors $\mu_\l$ are compatible with the Hecke action at infinity in the sense that there are canonical isomorphisms
$$\mu_{\l+\nu}(\wh\D_{\nu}\star_{\infty}\F)\cong \mu_\l (\F), \quad \forall \nu\in \xcoch(T).$$
\end{enumerate}
\end{lemma}
\begin{proof}
    (1) follows from the definition.

    (2) is a consequence of Proposition \ref{p:aff space van cycle}. To see this, note that it is enough to prove the case of $\nu=-\l$, as $\wh\D_{\nu}\star_{\infty}\wh\D_{\l}\cong \wh\D_{\nu+\l}$. 
    
    Consider the $G[\t]=G[t^{-1}]$-orbit $U_0=G[\t]I_0/I_0\subset \Fl$. This is an open neighbourhood of $(G/B)_0$ and the character $\psi$ induces a map $\psi_0:U_0\rightarrow\bA^1\times (G/B)_0$. We can translate this open by multipliying with $t^\l$ to obtain $U_\l=t^\l U_0$ and a map $\psi_\l:U_\l\rightarrow \bA^1\times (G/B)_\l$. We have 
    $$\wh\D_{-\l}\star_{\infty}\F|_{U_0}\cong \psi^*_0(\psi_\l)_!(\F|_{U_\l})\j{d_{-\l}}.$$

    As $\psi_0$ is a smooth map, we can rewrite $\mu_{0}(\wh\D_{-\l}\star_{\infty}\F)\cong \phi((\psi_\l)_!\F)$, where $\phi$ is taking nearby cycles along the $\bA^1$-factor in $\bA^1\times (G/B)_\l$. This is precisely the right hand side of the isomorphism in \ref{p:aff space van cycle}. The left hand side is by construction $\mu_\lambda(\F)$ and we thus obtain the required isomorphism.
\end{proof}

For an object $\F\in \wt \cD_\psi$, we define its {\em combinatorial 
support} to be the set of $\lambda\in \xcoch(T)$ for which $\mu_\l(\F)\ne 0$.

\begin{prop}\label{p:finite ss char Dpsi}
The following properties are equivalent for an object $\F\in \wt\cD_{\psi}$: 

 a) $\F\in \cD_\psi$.

 b) $SS(\F)$ consists of finitely many irreducible components of $\L_\psi$.

 c) $\F$ is $T$-monodromic and the combinatorial support of $\F$ is finite.
 
\end{prop}

\proof We check that a) $\Rightarrow$ b) $\Rightarrow$ c) $\Rightarrow$ a). 

Here a) $\Rightarrow$ b) is the content of Lemma \ref{l:SS finite}. The implication b) $\Rightarrow$ c) follows from Lemma \ref{mu l}a) since support of $\muhom$ between 
two constructible complexes is contained in the intersection of their singular supports and $SS(\F)\subset \Fl_\psi$ implies $\F$ is $T$-monodromic.
It remains to check that c) $\Rightarrow$ a).

We first relate the support and the combinatorial support for a $T$-monodromic sheaf $\F$. Consider the stratification of $\Bun_{G}(\bG^{\psi}_{\infty}, \bI_{0})$ by preimages of points in $\Bun_{G}(\bG_{\infty}, \bI_{0})$ which are in bijection with $\xcoch(T)$. For $\l\in \xcoch(T)$ the set of relevant points are $\cH_\psi(\lambda)$ by Lemma \ref{Relv pts Fl}. We show that:
\begin{equation}\label{open Hess supp}
\mbox{If $\cH_\psi(\lambda)$ contains an open substack in the support of $\F$, then $\mu_{\l'}(\F)\neq 0$ for some $\l'\in W\cdot\l$.}
\end{equation}
It follows that if the combinatorial support vanishes for a $W$ invariant set $S\subset \xcoch(T)$ such that $\bigcup_{\l\in S} G[\t] t^\l$ is open, then $\F$ vanishes in this open.

To show \eqref{open Hess supp}, consider $U_{\l'}$ as in the proof of Lemma \ref{micro_action}. Let $U_{\l,\l'}$ the intersection of $\cH_\psi(\lambda)$ and $U_{\l'}$, which is an open in $\cH_\psi(\lambda)$. The intersection of $(G/B)_{\l'}$ and $\cH_\psi(\lambda)$ is a Levi flag variety $P/B\subset \cH_\psi(\lambda)$. We can describe $U_{\l,\l'}$ as the attracting locus of this Levi flag variety $P/B$ under the action of some cocharacter of the maximal torus $T$. It follows that the $G[\t]^{1}$-orbit of $U_{\l,\l'}$ is contained in the attracting cell of this $P/B$ under the cocharacter $\Gm\xr{\Delta} \Gm\times\Gm\rightarrow T\times \Grot$ and thus contained in the open $U_{\l'}$, as this is an open $T\times\Grot$ stable neighbourhood of $P/B$.

As $U_{\l'}$ an affine bundle over $(G/B)_{\l'}$ and the $G[\t]^{1}$-orbit of $U_{\l,\l'}$ is contained in $U_{\l'}$, we have that the $G[\t]^{1}$-orbit of $U_{\l,\l'}$ lives over some locally closed subvariety $V_{\l,\l'}\subset (G/B)_{\l'}$. The result follows from $\mu_\lambda(\F)$ inducing a conservative functor $D^b(U_{\l,\l'})\rightarrow D^b(V_{\l,\l'})$. To see this, note that the by Proposition \ref{micro_action} this functor can be written as the following composition

\begin{equation}\label{compfunct}
    D^b(U_{\l,\l'})\xr{(j_{U_{\l,\l'}})_!} \wt\cD_{\psi}\xr{\wh\D_{-\l'}\star_{\infty}}\wt\cD_{\psi}\xr{j_{V_{\l,\l'}}^!}D^b(V_{\l,\l'}).
\end{equation}

By construction the composition $D^b(U_{\l,\l'})\xr{(j_{U_{\l,\l'}})_!} \wt\cD_{\psi}\xr{\wh\D_{-\l'}\star_{\infty}}\wt\cD_{\psi}$ produces sheaves that are $!$ extended from $V_{\l,\l'}$. It follows that the last functor, $j_{V_{\l,\l'}}^!$ is conservative on the image of the previous functors. The action $\wh\D_{-\l'}\star_{\infty}$ is an invertible functor with inverse given by $\wh\D_{\l'}\star_{\infty}$ and so the composition in equation \eqref{compfunct} is conservative. 

Now let $\F\in \wt \cD_\psi$ with $\cH_\psi(\lambda)$ containing an open substack in its support. Then there exists $\l'$ such that $U_{\l,\l'}$ contains an open substack in the support of $\F$. We show $j_{V_{\l,\l'}}^!\mu_{\l'}(\F)\neq 0$. To see this, restrict the bundle $U_{l'}$ over $V_{\l,\l'}$ and note that this intersects the support of $\F$ only at the $G[\t]^{1}$-orbit of $U_{\l,\l'}$. It follows that $j_{V_{\l,\l'}}^!\mu_{\l'}(\F)\neq 0$ is the above conservative functor $D^b(U_{\l,\l'})\rightarrow D^b(V_{\l,\l'})$ applied to $j_{U_{\l,\l'}}^!\F$. But by assumption $U_{\l,\l'}$ contains an open part in the support of $\F$ and thus $j_{U_{\l,\l'}}^!\F\neq 0$. It follows that $j_{V_{\l,\l'}}^!\mu_{\l'}(\F)\neq 0$ and thus $\mu_{\l'}(\F)\neq 0$.

Assume $\F\in \wt\cD_{\psi}$ has finite combinatorial support. 
Lemma \ref{mu l}(1) shows that the functor $\mu_\l$ is exact in the perverse $t$-structure, so we can assume that $\F$ is a perverse sheaf.
It suffices to show that $\F$ has a subquotient in $\cD_\psi$, then the claim follows by induction on the length of $\F$. Using Lemma \ref{mu l}(2) we can use Hecke action at infinity to reduce to 
the case when $\Grot$-fixed points in $(G/B)_\lambda$ are isolated for all $\lambda$
in the combinatorial support of $\F$; notice that this action is $t$-exact, so perversity and length of $\F$ are preserved. It follows that the support of $\F$ contains one of the $\Grot$-fixed points in $(G/B)_\lambda$ as an open substack in its support. For an isolated fixed point $x$ in $(G/B)_\lambda$ the IC extension of the unique irreducible sheaf in the Kirillov category of its $G[\t]^1$ orbit $G[\t]^1\cdot x$ is in $\cD_\psi$. \qed


\begin{lemma}
    For an isolated $\Grot$-fixed point $x$ in $(G/B)_\lambda$ the IC extension of the unique irreducible sheaf in $\Kir(G[\t]^\psi\bs G[\t]^1x)$ (where $G[\t]^\psi=\ker(G[\t]^1\to\frg\xr{\psi_1}\Ga)$) is in $\cD_\psi$.
\end{lemma}
\begin{proof}
    For simplicity, we will denote by $\IC(G[\t]^1 w)$ the IC extension of the unique irreducible sheaf in $\Kir(G[\t]^\psi\bs G[\t]^1 w)$.
    
    For $s$ a simple reflection, let $w$ and $ws$ be both isolated fixed point in some $(G/B)_\lambda$, such that $ws>w$ in the Bruhat order. We have that $\IC(G[\t]^1 w)$ is a summand of $\IC(G[\t]^1 ws)*\IC_s$ as the restriction of $\IC(G[\t]^1 ws)*\IC_s$ to the orbit of $w$, which is open in its support, is the constant local system. It follows that if $\IC(G[\t]^1 ws)\in \cD_\psi$, then we also have $\IC(G[\t]^1 w)\in \cD_\psi$.

    Further note that for $\l$ such that $(G/B)_\lambda$ has isolated fixed points $\pi^*\delta_\lambda$ is indeed the IC-sheaf on the orbit of $wt^\l$ of minimal length for $w\in W$.

    We just need to show that for each isolated fixed point $x$ in $(G/B)_\lambda$, there exists $w>x$ such that $\IC(G[\t]^1 w)\in \cD_\psi$ and a chain $w=x_0,\dots x=x_k$ such that $x_i$ isolated fixed points and $x_{i+1}=x_is_i$. This follows after noting that $\Grot$-fixed points in $(G/B)_\lambda$ are isolated precisely when $\lambda$ pairs non-zero with every root and so this corresponds to some set of alcoves strictly contained in the $W$-cones. Thus for any $x$ isolated, there exist $\lambda$ far enough in the corresponding cone, such that the minimal element in $Wt^\lambda$ is bigger than $x$. These two elements are in the same cone and so connected by a chain $w=x_0,\dots x=x_k$ such that $x_i$ isolated fixed points and $x_{i+1}=x_is_i$.
\end{proof}

\section{The equivalence}
We prove the main equivalence (Theorem \ref{th:intro Dpsi coh}) and give its application to small quantum groups.

\subsection{Recollections}
Recall $\bG=G\tl{t}\subset G\lr{t}$ and $\bI\subset \bG$ is an Iwahori subgroup. 
Let $G^{\vee}$ be the Langlands dual group of $G$ defined over $E$ (the coefficient field of sheaves in $\cD_{\psi}$, assumed to be of characteristic zero). Let $\cB^{\vee}$ be the flag variety of $G^{\vee}$, and $\wt\dN=T^{*}\cB^{\vee}$ its cotangent bundle. 

\begin{theorem}\label{old equiv}
a) We have monoidal equivalences 
\begin{equation}\label{der Sat}
D(\bG\bs G\lr{t}/\bG)\simeq D^b\Coh^{\dG}(\pt\times^{\bR}_{\frg^{\vee}}\pt)\longleftrightarrow D^b\Coh^{\dG}(\dg),
\end{equation}
\begin{equation}\label{two Hk}
D(\bI\bs G\lr{t}/\bI)\simeq D^b\Coh^{\dG}(\wt\dN\times^{\bR}_{\frg^{\vee}}\wt\dN)
\longleftrightarrow D^b\Coh^{\dG}(\wt\dg\times^{\bR}_{\frg^{\vee}}\wt\dg),
\end{equation}
where $\longleftrightarrow$ denotes linear Koszul duality \cite{MR}.

b) We have an equivalence 
\begin{equation}\label{GI}
D(\bG\bs G\lr{t}/\bI)\simeq D^b\Coh^{\dG}(\pt\times^{\bR}_{\frg^{\vee}}\wt\dN)\longleftrightarrow D^{b}\Coh^{\dG}(\wt \dg)
\end{equation}
intertwining the convolution action of the two monoidal categories in a).

\end{theorem}
\proof The equivalence \eqref{der Sat} was constructed in \cite{BF}, the equivalence 
\eqref{two Hk}
in \cite[Theorem 1(4)]{B}. A special case of \cite[Theorem 6.1]{BeLo}\footnote{An alternative
reference is \cite{ABG}, we neither prove nor use the fact that the equivalences given by the two constructions coincide.} yields a monodromic version
of \eqref{GI}
\begin{equation}\label{GI_mon}
D(\bG\bs G\lr{t}/\bI^0)\simeq D^b\Coh^{\dG}(\pt\times^{\bR}_{\frg^{\vee}}\wt\dg)\longleftrightarrow D^{b}\Coh^{\dG}(\wt \dg).
\end{equation}
The equivalence in \eqref{GI} follows from \eqref{GI_mon} by the argument of 
\cite[\S 9.3]{B}. Compatibility of \eqref{GI_mon} with the action of monodromic affine Hecke category
is established in \cite{BeLo}, compatibility of the equivalence in \eqref{GI} with \eqref{two Hk}
follows by inspecting the constructions. Since linear Koszul duality is compatible with the convolution action, we get that Koszul duality in \eqref{GI}, \eqref{two Hk} is also compatible with the action. 
 It remains to check compatibility
with \eqref{der Sat}. 

Recall the tautological endomorphism $\tau$ of  identity in $\Coh^{\dG}(\dg)$:
it is characterized by the property that for $\cF\in \Coh^{\dG}( \dg)$
the action of $\tau$ on the fiber $\cF_x$ equals the action of $x$ coming from the equivariant
structure. We use the same notation for the similar endomorphism of the identity functor of
$\Coh^{\dG}(\wt \dg)$.
We first check that the restriction of the composed Koszul duality in \eqref{GI} to pure complexes carries multiplication $\bf{c}$ by the first Chern class of the determinant bundle to $\tau$.

For a weight $\lambda$ let $L_\lambda$ be the corresponding line bundle  $\Fl$, and let $h_\lambda$ be the  function on $\wt \dg$ pulled back from the corresponding linear function on $\dt$.
We claim that the composed Koszul duality in \eqref{GI} sends
multiplication by the first Chern class of  $L_\lambda$ 
to multiplication by $h_\lambda$. To see this, one can interpret $c_1(L_\lambda)$ as the deformation 
 class corresponding to the deformation of the category $D(\bG\bs G\lr{t}/\bI)$ to the category of monodromic sheaves on the punctured total space of $L_\lambda$. Likewise, it is easy to 
 see that $h_\lambda$ is Koszul dual to the deformation class of the deformation 
 of $\pt\times^{\bR}_{\frg^{\vee}}\wt\dN$ to
 $\pt\times^{\bR}_{\frg^{\vee}}\wt\dg_\lambda$ where $\wt\dg_\lambda=\wt\dg\times_{\dt} \{E\cdot \lambda\}$. Thus compatibility follows from \eqref{GI_mon}.
 
 Recall the Wakimoto object $J_\lambda\in D^{b}(\bI\bs G\lr{t}/\bI)$, let $j_\lambda$ be its direct
image to $D^{b}(\bG\bs G\lr{t}/\bI)$. It is easy to see that the action $\bf{c}$ on $j_\lambda$ coincides with the action of $c_1(L_{\kappa(\lambda)})$,
where $\kappa$ is the map from weights of $\dG$ to weights of $G$ coming from the quadratic form used to define the determinant bundle. The image of $j_\lambda$
under the composed functor in \eqref{GI} is the direct image of $\cO(\lambda)$ under the map $\wt \dg \to \dg$, the parallel compatibility between the tautological
endomorphism and the action of $\dt$ is also clear for that object. Thus we get the compability between $\bf{c}$ and $\tau$ for $j_\lambda$. An irreducible $\IC_\lambda \in D(\bG\bs G\lr{t}/\bI)$
pulled back from the Satake
category admits a "filtration" (a collection of distinguished triangles) by $j_\lambda$ and   $\Ext(\IC_\lambda, \IC_\lambda)$
maps injectively to the endomorphisms of the associated graded of that filtration, so compatibility between $\bf{c}$ and $\tau$ holds for $\IC_\lambda$. Now we see it holds for 
every pure object since push-forward from $\bG\bs G\lr{t}/\bI$ to the Satake category is faithful on pure complexes.

Now compatibility of \eqref{der Sat}, \eqref{GI} with the action of the category of pure complexes in the Satake category on pure complexes in $\bG\bs G\lr{t}/\bI$ follows 
 from compatibility between $\bf{c}$ and $\tau$ in view of the Tannakian formalism, 
 see \cite[\S 3.4]{AB}. Finally, compatibility for pure complexes implies the desired compatibility by the standard
 formality from weights argument, see e.g. \cite[\S 6.5]{BF}. \qed

Next, by Corollary \ref{c:free over lattice}, we have an equivalence
\begin{equation}\label{DpsiG}
 \DG \cong D^b(\Rep(\dT))=D^{b}\Coh^{\dT}(\pt).
\end{equation}
This equivalence intertwines the action of $\wh\D_\l\in \hHim$ on the left and tensoring with the one-dimensional representation $E(\l)$ on the right.

By Theorem \ref{th:rho comm}, the spherical action of $\Rep(\dG)\cong\Hsh$ on $\DG$ intertwines with the action by tensoring on $D^b\Coh^{\dT}(\pt)$ via
\begin{equation*}
\Rep(\dG)\xr{\Res}\Rep(\dT)\xr{\inv}\Rep(\dT).
\end{equation*}

Then $T^{\vee}$ naturally acts on $\cB^{\vee}$ and hence on $\wt\dN$. Our goal in this section is to prove Theorem \ref{th1} establishing an equivalence 
\begin{equation}\label{coh_descr}
\cD_{\psi}\simeq D^{b}\Coh^{\dT}_{\cB^{\vee}}(\wt\dN),
\end{equation}
where $\Coh^{T^{\vee}}_{\cB^{\vee}}(\wt\dN)\subset \Coh^{T^{\vee}}(\wt\dN)$ is the full subcategory of sheaves  set-theoretically supported on the zero section 
$\cB^{\vee}\subset T^{*}\cB^{\vee}=\wt\dN$.

\begin{remark}\label{heur}
Heuristically, \eqref{coh_descr} follows from the above equivalences in view of the full embeddings:
$$ \DG ^{ho}\ot_{D^{ho}(\bG\bs G\lr{t}/\bG)} D^{ho}(\bG\bs G\lr{t}/\bI)\incl  \wt\cD^{ho}_{\psi}$$ 
\begin{equation}\label{Coh G tensor}
D^{ho}\Coh^{\dT}(\pt)\ot_{D^{ho}\Coh^{\dG}(\pt\times^{\bR}_{\frg^{\vee}}\pt)}D^{ho}\Coh^{\dG}(\pt\times^{\bR}_{\frg^{\vee}}\wt\dN)\incl D^{ho}\Coh^{\dT}(\wt\dN),
\end{equation} 
whose existence follows from first principles, where the superscript $^{ho}$ refers to the appropriate homotopy lift of the triangulated categories.

However, making it into a rigorous argument
involves some subtleties of homotopical algebra, including a homotopy version of compatibilities of \eqref{der Sat}, \eqref{GI} with the action of the monoidal categories on the module categories and so forth.
The argument below follows that general idea, but instead of considering the tensor product of homotopy categories we reduce
the construction to tensor product of additive categories of projective objects in relevant abelian categories which is quite close 
to the familiar tensor product of rings over a given commutative ring. The reduction is achieved by invoking Koszul duality; as a byproduct of this approach
we get a stronger statement involving an equivalence of abelian categories.
\end{remark}

\subsection{Quantum groups and noncommutative Springer resolution}


Recall \cite{BeMi} that the {\em noncommutative Springer resolution} is the ring $A=\End_{\Coh(\wt\dg)}(\cE)^{op}$ where $\cE$ is a $\LG$-equivariant vector bundle, the projective generator of the unique exotic $t$-structure \cite[\S 1.4.1, Theorem 1.5.1]{BeMi} on $D^b(\Coh(\wt \dg))$. Let $\bar{A}= A\otimes _{E[\dt]}E=\End(\cE|_{\wt \dN})^{op}$.
We have equivalences $D^b(A\lmod)\cong D^b(\Coh(\wt \dg))$, $D^b(\bar{A}\lmod)\cong D^b(\Coh(\wt \dN))$, where $\lmod$ stands for the category of finitely
generated modules.
 
Let $U_q$ be the Kac-De Concini quantum group at a root of unity $\zeta$ of order $\ell$. We require $\ell$ to be odd and  larger than the Coxeter number. If $\frg$ has a factor of type $G_2$ we also require that $\ell$ be coprime to $3$. 

By \cite{KDC} the center of $U_q$ contains the Harish-Chandra center $Z_{HC}\cong E[\LT]^W$ and the $\ell$-center $Z_\ell\cong E[\LG_0]$ where $\LG_0=\LB_-\LB_+$
is the open Bruhat cell. Thus the quotient of $U_q$ by the trivial $\ell$-central character $U_q\otimes _{Z_\ell} E=u_q$ is the small quantum group. We let $u_q^{\wh{0}}$ be the summand
of $u_q$ corresponding to a regular block.

We record the relation between these categories following from the literature.

\begin{prop}\label{barA}
\begin{enumerate}
\item The algebra $\bar{A}$ can be equipped with a grading making it a Koszul quadratic algebra.
The Koszul dual algebra $\bar{A}^!$ is Morita equivalent  to $A_0:=A\otimes_{E[\dg]}E$; the equivalence can be chosen
to be compatible with the $\LG$ action.

\item We have compatible equivalences $A_0\lmod \cong u_q^{\wh{0}}\lmod$, $A_0\lmod^{\LT}\cong u_q^{\wh{0}}\lmod^{gr}$
where $\lmod^{\LT}$ denotes the category of $\LT$-equivariant $A_0$-modules and $u_q^{\wh{0}}\lmod^{gr}$ denotes the 
category of $u_q^{\wh{0}}$ modules graded by the weight lattice $\Lambda$ in a way compatible with the $\Lambda/\ell \Lambda$
action coming from the action of the Cartan generators $K_i\in u_q$.

\item We have: $$A_0\lmod^{\dG} \cong \Perv_{\bI^0}(\Gr)\cong U_q^{Lus}\lmod_0,$$
where $U_q^{Lus}\lmod_0$ is
a regular block in the category of modules over Lusztig's big quantum group. 
\end{enumerate}
\end{prop}

\proof The first statement in part (1) is a special case of \cite[Proposition 5.5]{BeMi}. The $\LG$-equivariant Morita equivalence  of the Koszul
dual algebra  $\bar{A}^!$ with $A_0$ follows from the arguments of \cite{Ri}. Namely, by \cite{MR} we have a Koszul type equivalence 
(i.e. an equivalence sending homological shift to the composition of the homological shift and twist by the tautological  character of $\Gm$)
\[ D^b(\Coh^{\Gm}(\wt \dN))\cong D^b(\Coh^{\Gm}(\wt \dg \times^{\bR} _{\Lg} \{ 0\})) \] where the fiber product in the right hand side is understood
to be derived. Recall the equivalences \[ D^b(\bar{A}\lmod)\cong D^b(\Coh(\wt \dN)) \text{ and } D^b(A_0\lmod)\cong D^b(\Coh(\wt \dg \times _{\Lg} \{ 0\})), \]
where the latter follows from the equivalence $D^b(A\lmod)\cong D^b(\Coh(\wt \dg))$ by base change, along with their $\Gm$ equivariant versions. It remains
to show that projective graded ($\Gm$-equivariant) 
$A_0$ modules are sent into semisimple complexes of $\bar{A}$-modules (i.e. a sum of simple modules and their homological shifts).
One such projective module $P_0$ corresponds to the structure sheaf of $\wt \dg \times _{\Lg} \{ 0\}$. Its image under Koszul duality is 
$i_*(\cO)$ where $i:\cB^\vee\to \dN$ is the zero section. This object corresponds to an irreducible $\bar{A}$-module. Other projective modules
$P_0$ are obtained from $P_0$ by {\em geometric reflection functors} \cite{BeMi}, thus we are reduced to checking that the Koszul dual
functors in the sense of \cite{MR} correspond to semi-simple 
endofunctors of $D^b(\bar{A}\lmod^{\Gm})$. 
For a base field of large positive characteristic the latter is shown in 
  \cite{Ri}. The present case of a characteristic zero base field follows.

The first equivalence in part (3) is \cite[Corollary 6.3]{BeLo}, while the second one is the main result of Lusztig's program. See overview in the introduction to \cite{ABG}; the main result of \cite{ABG} provides an alternative proof.

The composed equivalence between the first and the third category in part (3) is easily checked to be
compatible with the natural $\Rep(\LG)$ action. Using the description of modules over the small quantum group as $\Rep(\LG)$ eigenobjects in
representations of $U_q^{Lus}$ \cite{AG} we get  compatible equivalences in part (2). \qed

\begin{remark} Let $U_q^0$ be the quotient of $U_q$ by a regular integral character of $Z_{HC}$. Let $U_q^0\lmod_0$ be the category of (finitely generated) modules
where $Z_\ell$ acts with the trivial generalized central character. 
Using a result of \cite{Tanisaki}, one can construct compatible equivalences  $\bar{A}\lmod_0 \cong U_q^0\lmod_0$,
 $\bar{A}\lmod_0^{\LT}\cong U_q^0\lmod_0^{gr}$ where  $U_q^0\lmod_0^{gr}$ 
  denotes the 
category of $U_q^0$ modules graded by the weight lattice $\Lambda$ in a way compatible with the $\Lambda/\ell \Lambda$
grading by generalized eigenspaces of the Cartan generators $K_i\in u_q$.
This gives an alternative way to establish the above equivalences between derived
categories of coherent sheaves and quantum group modules, bypassing \cite{ABG}. See the discussion
in \cite{BeLa}.
\end{remark}

\begin{cor}\label{Abar0}
Consider the dga $\bar{A}_0=A_0\otimes^{\mathrm{L}}_{E[\frt]} E=\bar{A} \otimes^{\mathrm{L}}_{E[\dg]} E$.
We have a $t$-exact equivalence $D^b(\bar{A}_0\lmod^{\LG})\cong D_{\bI}(\Gr)$.
\end{cor}

\proof This follows from the first equivalence in Proposition \ref{barA}c) by base change, cf. \cite[\S 9.3.1]{B}. \qed

\begin{remark}\label{Koszul_rem}
The proof of Koszul duality between  $\bar{A}^!$ and $A_0$ can be adapted to show Koszul duality between the dga $\bar{A}_0$ and the ring $A$:
again we have a Koszul type equivalence $D^b(\Coh^{\Gm}(\wt \dg)) \cong D^b\Coh^{\Gm}(\wt \dN \times _{\Lg} \{0\})$ by \cite{MR}. This functor sends the structure
sheaf of $\wt \dg$ to $i_*(\cO)$ where $i:\cB^\vee\to\wt  \dN\times _{\dg} \{0\}$. The two objects correspond, respectively, to a projective $A$-module and an irreducible $\bar{A}_0$-module. Furthermore, the geometric reflection functors commute with semisimple functors on $D(\bar{A}_0\lmod)$ which yields the statement.
\end{remark}

\subsection{The equivalence}
Let $\cP_\psi=\cD_\psi^\hs\subset \cD_\psi$ be the subcategory of perverse sheaves.

\begin{theorem}  \label{th1} 
\begin{enumerate}
\item We have a canonical equivalence
$$\Th:  \cD_{\psi}\simeq D^{b}\Coh^{\dT}_{\cB^{\vee}}(\wt\dN). $$


It intertwines the action of $\Rep(\dT)$ coming from Corollary 
\ref{c:mon gr vec to Hinf} with the action of the affine Hecke category \eqref{two Hk}.

\item The equivalence intertwines the exotic $t$-structure on $D^b(\Coh^{\LT}(\wt\dN))$ with the
$t$-structure on $\cD_\psi$ induced by the perverse $t$-structure on constructible sheaves. 
Thus it induces an equivalence 
$$\Th^\hs: \cP_\psi\cong \bar{A}\lmod_0^{\LT}.$$
\end{enumerate}
\end{theorem}

\proof 
Consider the subcategory $\Irr_{\Gr}$ of semisimple complexes (i.e. of sums of irreducible perverse sheaves and their homological shifts) in $D(\bG\bs LG/\bI)$. 
Let $L_0\in  \cD_{\psi, \bG}$ be the irreducible object with full support. 
We have the functor $\ph_0: D_{\bI}(\Gr) \to \cD_\psi$ sending $\cF$ to the convolution $L_0\star_{\bG}\cF$. Its essential image generates $\cD_\psi$. 
Using Corollary \ref{Abar0} we can identify $\Irr_{\Gr}$ with the full subcategory $\Irr(\bar{A}_0\lmod^{\LG})$ of semisimple
complexes in $D^b(\bar{A}_0\lmod^{\LG})$. We have a functor $\ph_1: D^b(\bar{A}_0\lmod^{\LG}) \to D^b(\bar{A}\lmod^{\LT})$ given by pull back and restricting the equivariance. By inspection, the equivalences $$D^b(\bar{A}_0\lmod^{\LG})
\cong  D^b\Coh^{\LT}(\wt \dN\times^{\bR}_{\dg}\{0\}),$$
$$D^b(\bar{A}\lmod^{\LT})\cong D^b\Coh^{\LT}(\wt \dN)$$ intertwine $\ph_1$ with the natural
functor between the coherent sheaves categories. Hence, the essential image of $\ph_1$ generates $D^{b}\Coh^{\dT}_{\cB^{\vee}}(\wt\dN)$.
It thus suffices to identify the essential images $\ph_0(\Irr_{\Gr})$ and $\ph_1(\Irr(\bar{A}_0\lmod^{\LG}))$ and show formality of the corresponding $\RHom$ algebras.

By Proposition \ref{barA}(1), $\Irr(\bar{A}_0\lmod^{\LG})$ is equivalent to the category of $\LG$-equivariant projective graded $A$-modules.
By \eqref{der Sat} the category  $\Irr_S$ of semisimple objects in the Satake category $D(\bG\bs LG/\bG)$ is equivalent to the category $\Coh^{\LG\times \Gm}_{fr}(\Lg)$ of projective
objects in $\Coh^{\LG\times \Gm}(\Lg)$.

 By Theorem \ref{th:rho comm}(2), $\ph_0|_{\Irr_{\Gr}}$ intertwines the action of $\Irr_S \cong \Coh^{\LG\times \Gm}_{fr}(\Lg)$
on $\Irr_{\Gr}$  with an action on its target factoring through the functor
$\Coh^{\LG\times \Gm}_{fr}(\Lg)\to \Rep(\LT)$ given by pull back to $0\in \Lg$ followed by restriction of equivariance from
$\LG$ to $\LT$. It follows that $\ph_0|_{\Irr_{\Gr}}$ factors canonically through a functor from $\LT$-equivariant projective graded modules
over $A/\Lg A=A_0$, which by Koszul duality is identified with $\Irr((\bar{A}\lmod^{\LT}))$,
to $\cD_\psi$.
We now argue that the resulting functor 
$\ph: \Irr(\bar{A}\lmod^{\LT})\to  \cD_\psi$  is fully faithful. 

By Theorem \ref{old equiv}(b) the functor $\ph_0$ is compatible
with the action of the affine Hecke category. It follows that
 $\ph$ is compatible with the action of 
the category of semisimple complexes in the affine Hecke category $\Irr(\bI\bs LG/\bI)$.
This allows one to reduce checking that the map $\Hom(A,B)\to \Hom(\ph(A),\ph(B))$ is an isomorphism to the case
when $A=\cF_0=\pi_{\bG}^*(\delta_0)$. 

One also checks that $\ph$ fits in commutative diagrams with the equivalence \eqref{DpsiG}:

$$
\xymatrix{ \Irr(\bar{A}\lmod^{\LT}) \ar[d]^{\alpha}\ar[r]^-{\ph}  & \cD_\psi \ar[d]^{\pi_*}\\
D^b(\Rep(\LT)) \ar[r]^-{\eqref{DpsiG}}& \DG } 
$$
$$
\xymatrix{ 
D^b(\Rep(\LT)) \ar[d]^{\beta} \ar[r]^-{\eqref{DpsiG}} & \DG \ar[d]^{\pi^*} \\ 
   \Irr(\bar{A}\lmod^{\LT}) \ar[r]^-{\ph}  & \cD_\psi }
$$
where $\alpha$ and $\beta$ come from the functors 
 $ p_*\iota^*: D^b( \Coh^{\LG}(\wt\dN))\to D^b(\Rep(\LT))$ and 
 $\iota_*p^*: D^b(\Rep(\LT)) \to
D^b( \Coh^{\LG}(\wt\dN) )$ as follows. We have $\alpha= p_*\iota^*|_{\Irr(\bar{A}\lmod^{\LT}) }$, where we identify
 $\Irr(\bar{A}\lmod^{\LT})$ with the image of its full embedding in $D^b(\Coh^{\LT}(\wt\dN))$). On the other hand, $\beta$ post-composed with the full embedding
$  \Irr(\bar{A}\lmod^{\LT}) \incl D^b(\Coh^{\LT}(\wt\dN))$ is identified with $\iota_*p^*$.  
Here $\iota: \wt\dN\times^{\bR}_{\Lg} \{0\}\to \wt\dN$ is the embedding and $p: \wt\dN\times^{\bR}_{\Lg} \{0\}\to \{0\}$
is the projection. The commutativity follows from similar compatibilities between \eqref{der Sat}, \eqref{GI} mentioned after \eqref{GI}.

It follows that the essential image of $D^b(\Rep(\LT)) $ in $\Irr(\bar{A}\lmod^{\LT})$ is equivalent to the essential image of $\DG$ in $\cD_\psi$. 

This yields an equivalence between the essential images  $\ph_0(\Irr_{\Gr})$ and $\ph_1(\Irr(\bar{A}_0\lmod^{\LG}))$.

Finally, formality in both cases is implied by coincidence of the inner (dilation) grading with the homological grading on the corresponding $\Ext$ spaces, cf.
\cite[Proposition 6]{BF}. For the essential image of $\ph_1$ this coincidence follows from Koszul duality. Also, the identification between the essential images of $\ph_1$
and $\ph_0$ carries the inner grading into the weight grading, since the same is true for the equivalence $D^b(\bar{A}_0\lmod^{\LG})\cong D_{\bI}(\Gr)$, this yields the
grading compatibility for the essential image of $\ph_0$. This proves the equivalence in part (1). Compatibility with the action of $\Rep(\dT)$ and of
$\Irr(\bI\bs LG/\bI)$ on the subcategory of semi-simple complexes follows, respectively, from Theorem \ref{th:rho comm}(2) and Theorem \ref{old equiv}(b) as pointed out above. The extension to the full derived category via
formality of $\RHom$ is compatible with convolution, see e.g. \cite[\S 6.5]{BF},
this yields compatibility with monoidal actions claimed in part (1).

By Proposition \ref{barA}(3) we have an equivalence between the category
of $\LG$-equivariant $A_0$-modules and a regular block in the category of Lusztig's big quantum group $U_q^{Lus}\lmod_0$.
Now (2) follows from $t$-exactness of the $(\J,\psi)$ averaging functor $D_{\bI^0}(\Gr) \to \cD_\psi$ and 
the compatibility of the coherent description of $D_{\bI^0}(\Gr)$ with the equivalence of part (1). \qed 

\begin{cor}\label{main_cor}
\begin{enumerate}
\item The category $\cP_\psi$ is equivalent 
to the category of finitely generated modules with nilpotent central character 
over the Koszul dual to $u_q^{\wh{0}}\lmod^{gr}$.
\item We have an exact full embedding $u_q^0\lmod \to \cP_\psi$ inducing a bijection between isomorphism classes of irreducible objects,
where $u_q^0$ is the quotient of $u_q$ by the regular Harish-Chandra central character.
\end{enumerate}
\end{cor}

\begin{remark}\label{r:whole mush} As mentioned in \S\ref{sss:intro whole mush}, we expect the whole category $\muSh_{\L_\psi}(\cM_\psi)$ to be embedded into $\dT$-equivariant ind-coherent sheaves on $\wt\cW_0$, which can be identified with an open subset of $\wt\cN^\vee$. Note that the objects $\dc_\l$ from Section \ref{sec:Wak_action} are not objects in the category $\cD_{\psi}$, however under the fully faithful embedding in Conjecture \ref{conj:mush-IndCoh}, these objects correspond to the structure sheaf of the fibers of $\wt\cW_0\to \cB^\vee$ over a $\dT$-fixed point of $\cB^\vee$, twisted by the $T^\vee$ character $\l$. This should follow from the compatibility of the action of Wakimoto sheaves and the category $\cH_\infty$.
\end{remark}



\subsection{Consequences on Grothendieck groups}\label{ss:conseq on K}

By Lemma \ref{l:SS finite}, the characteristic cycles of objects in $\cD_\psi$ are finite linear combinations of irreducible components of $\L_\psi\cong \Fl_\psi$, which thus induces a map
\begin{equation*}
    CC: K(\cD_\psi)\to \homog{\tp}{\Fl_\psi,\ZZ}.
\end{equation*}

There are actions of the extended affine Weyl group $\tilW=\xcoch(T)\rtimes W$ on both sides of the map $CC$, which we explain below.

\sss{Right $\tilW$-action on $K(\cD_\psi)$ from affine Hecke action}

Consider the Grothendieck group $K(\cH_0)$ of the the affine Hecke category. It is well-known that we have an isomorphism of rings
\begin{equation}\label{KH0}
    K(\cH_0)\isom \ZZ[\tilW]
\end{equation}
sending the class of the standard sheaf $i_{w!}E$ on the $\bI$-orbit $\bI w \bI/\bI$ to $w\in \ZZ[\tilW]$ for each $w\in \tilW$. The right action of the affine Hecke category $\cH_0$, after passing to the Grothendieck group, induces a right action of $\tilW$ on $K(\cD_\psi)$.

\sss{Right $\tilW$-action on $\homog{\tp}{\Fl_\psi,\ZZ}$: Springer action}\label{sss:Spr action} 
The right action of the affine Weyl group $\Wa$ on the cohomology of affine Springer fibers is defined in \cite[\S5]{L-Aff}. The extension to $\tilW=\Wa\rtimes\Om$ (where $\Om$ is the subgroup of length zero elements in $\tilW$) is easy and can be found in \cite[2.4]{Yun}.

\begin{prop}\label{p:CC Springer}
    The map $CC$ is equivariant under the right $\tilW$-actions up to twisting by the sign character $w\mapsto (-1)^{\ell(w)}$ of $\tilW$.
\end{prop}
\begin{proof}
    It suffices to check for each simple reflection $s$ and each length zero element $\om\in \Om$.

    For $\om\in \Om$, both the Hecke action and the Springer action are induced by the automorphism of $\cM_\psi$ (resp. $\Fl_\psi$) coming from any lifting of $\om$ to $N_{LG}(\bI)$ on $\Bun_G(\J,\bI)$ (resp. $\Fl=LG/\bI$). Their agreement on characteristic cycles is clear.

    For a simple reflection $s$, consider the shifted constant sheaf $\IC_s=E_{\bP_s/\bI}\j{1}$ supported on the flag variety of the parahoric $\bP_s$ of type $s$, viewed as an object in $\cH_0$. Under the isomorphism  \eqref{KH0}, $\IC_s$ corresponds to $-s-1\in\ZZ[\tilW]$. We thus reduce to proving the following statement: for $\cF\in \cD_\psi$, we have    \begin{equation}\label{CC ICs}  CC(\cF\star_0 \IC_s)=CC(\cF)\cdot s-CC(\cF)\in \homog{\tp}{\L_\psi}.
    \end{equation}
    Using notation from \S\ref{sss:DpsiP}, $\star_0 \IC_s$ is the composition
    \begin{equation*}        \cD_\psi\xr{\pi_{s*}}\cD_{\psi,\bP}\xr{\pi_{s}^*[1]}\cD_\psi.
    \end{equation*}
    Here $\pi_s=\pi_{\bP_s}$. Referring to the diagram \eqref{Lag pi s}, using that $\pi_s$ is proper and smooth, we have
    \begin{equation}\label{CC under Cs}
        CC(\cF\star_0\IC_s)=(d\pi_s)_*(\pi_{s,\na})^*(\pi_{s,\na})_*(d\pi_s)^*CC(\cF)
    \end{equation}
    See \cite[Proposition 9.4.2, 9.4.3]{KS}.

    On the other hand, consider the moduli stack $\cM'_{\psi,\bP_s}$ that classifies $(\cE,\t_{\infty}, \cE_{\bP_s}, \ph)$ where $(\cE,\t_{\infty}, \cE_{\bP_s})\in\Bun_{G}(\bG^{1}_{\infty}, \bP_s)$, $\ph\in \Ad^*(\cE)\ot \om_{X}(2\cdot\infty+\un 0)$ satisfies the usual condition $\ph\in\psi_{1} dt + \Ad^*(\cE)dt/t$ at $\infty$, while at $0$, $\ph\in \Ad(\cE_{\bP_s})dt/t$ (i.e., in the lattice $(\Lie \bP_s)dt/t$ after trivializing $(\cE,\cE_{\bP_s})$ on the formal disk at $0$), and the image of $\ph$ in $\Ad(\cE_{L_s})dt/t$ (here $\cE_{L_s}$ is the reduction of the $\bP_s$-torsor $\cE_{\bP_s}$ to the reductive quotient $L_s$ of $\bP_s$) is nilpotent.
    We then have a Cartesian diagram
    \begin{eqnarray*}    \xymatrix{\cM_\psi\ar[d]^{\nu_s}\ar[r]^-{\wt\e_s}& [\wt\cN_{L_s}/L_s]\ar[d]^{\nu_{L_s}}\\
        \cM'_{\psi,\bP_s} \ar[r]^-{\e_s} & [\cN_{L_s}/L_s]}
    \end{eqnarray*}
    Here $\nu_{L_s}:\wt\cN_{L_s}\to \cN_{L_s}$ is the Springer resolution of the nilpotent cone of $L_s$. We have the Springer action of $\j{s}$ on 
    $\nu_{L_s*}\DD\in D_{L_s}(\cN_{L_s})$. By $!$-pullback to $\cM'_{\psi,\bP_s}$ and proper base change, it gives an action of $\j{s}$ on $\nu_{s*}\DD_{\cM_\psi}$, and hence on $\homog{*}{\cM_\psi}$ by taking global sections. Under the isomorphism $\homog{*}{\L_\psi}\isom\homog{*}{\cM_\psi}$, this gives the action of $\j{s}$ as part of the $\tilW$-action specified in \S\ref{sss:Spr action}.

    To compute the action of $s$ on $\homog{*}{\cM_\psi}$, we first recall that the $\j{s}$-action on $\nu_{L_s*}\DD$ is given by top homology classes in $\hBM{4}{\St_{L_s}}$ where $\St_{L_s}=\wt\cN_{L_s}\times_{\cN_{L_s}}\wt\cN_{L_s}$ is the Steinberg variety of $L_s$. Specifically, a simple calculation shows that the action of $s-1$ is given by the cycle class of the component $\PP^1\times \PP^1\subset \St_{L_s}$ (the preimage of the zero nilpotent orbit). Now we claim
    
    \begin{claim} The map  $\e_s:\cM'_{\psi, \bP_s}\to [\cN_{L_s}/L_s]$ is smooth.
    \end{claim}

    The claim can be proved using the same infinitesimal calculations as in \cite[\S2.9.2]{BBAMY}.  We omit details here.

    Using that $\e_s$ is smooth, the action of $s-1$ on $\nu_{s*}\DD_{\cM_\psi}$ is also given by the cycle class of the preimage of $(\PP^1\times\PP^1)/L_s$ in $\cM_{\psi}\times_{\cM'_{\psi,\bP_s}} \cM_{\psi}$ under the map
    \begin{eqnarray*}
        \s_s=(\wt\e_s, \wt\e_s): \cM_\psi\times_{\cM'_{\psi,\bP_s}} \cM_\psi\to \St_{L_s}/L_s.
    \end{eqnarray*}
    Finally we observe that
    \begin{eqnarray*}
        \s_s^{-1}((\PP^1\times\PP^1)/L_s)=\cM_{\psi,\bP_{s}}\times_{\Bun_{G}(\J,\bP_{s})}\Bun_{G}(\J,\bI)\times_{\Bun_{G}(\J,\bP_{s})}\Bun_{G}(\J,\bI).
    \end{eqnarray*}
    Therefore the action of $s-1$ on $\homog{*}{\cM_\psi}$ is given by pull-push operators on the correspondence
    \begin{equation}\label{P1P1}
        \xymatrix{\cM_\psi & \ar[l]_-{d\pi_s\c p_{12}} \cM_{\psi,\bP_{s}}\times_{\Bun_{G}(\J,\bP_{s})}\Bun_{G}(\J,\bI)\times_{\Bun_{G}(\J,\bP_{s})}\Bun_{G}(\J,\bI)\ar[r]^-{d\pi_s\c p_{13}} & \cM_\psi}.
    \end{equation}
    Here $p_{12}$ and $p_{13}$ refer to the projections to the indicated factors, and $d\pi_s: \cM_{\psi,\bP_{s}}\times_{\Bun_{G}(\J,\bP_{s})}\Bun_{G}(\J,\bI)\to \cM_\psi$ is the differential of $\pi_s$. The correspondence \eqref{P1P1} is the composition of the correspondence \eqref{Lag pi s} and its transpose, therefore its action on $\homog{*}{\cM_\psi}$ is given by the composition
    \begin{equation*}
        (d\pi_s)_*(\pi_{s,\na})^*(\pi_{s,\na})_*(d\pi_s)^*\in \End(\homog{*}{\cM_\psi}).
    \end{equation*}
    Here to make sense of $(d\pi_s)^*$ on homology we use Poincar\'e duality and closed restriction in compact support cohomology. Comparing with \eqref{CC under Cs}, we conclude that \eqref{CC ICs} holds.
\end{proof}

\sss{Left $\tilW$-action on $\hBM{\tp}{\Fl_\psi,\ZZ}$ from lattice and monodromy action}

The action of the lattice $\xcoch(T)$ on $\Fl_\psi$ has been explained in \S\ref{sss:lattice on ASF}. To see the $W$-action, we need to vary $\psi_1$ inside $\frt^{*,\rs}$. Let $\cX\to \frt^{*,\rs}$ be the family of affine Springer fibers for all $\psi=t\psi_1$, where $\psi_1$ varies in $\frt^{*,\rs}$. Using intersection with semi-infinite orbits as in \cite[\S5, Prop. 1]{KL}, we see that the restriction to any fiber induces an isomorphism
\begin{equation}\label{rest top homology}
    \homog{\tp}{\cX,\ZZ}\isom\homog{\tp}{\Fl_\psi,\ZZ}. 
\end{equation}
Now $\dot w\in N_G(T)$ acts on $\cX$ sending $(\psi_1, g\bI)$ to $(w\psi_1, \dot{w}g\bI)$. It induces a left $W$-action on $\homog{\tp}{\cX,\ZZ}$, hence on $\homog{\tp}{\Fl_\psi,\ZZ}$ via the isomorphism \eqref{rest top homology}.

\sss{Left $\tilW$-action on $K(\cD_\psi)$}
The action of the lattice part is induced by $\Rep(\dT)$-action on $\cD_\psi$ via the monoidal functor \eqref{phi RT} and the $\hHim$-action on $\cD_\psi$.

For the $W$-action, here we only give a sketch, as it is not used in the rest of the paper.  We consider a version of $\cD_\psi$ with a varying $\psi_1\in \frt^{*,\rs}$. Take the level subgroup $\Jinf=\ker(\bG_\infty^1\to\t\frg\to\t\frt)$, where $\frg\to \frt$ is the orthogonal projection under the Killing form on $\frg$.

Consider the stack
\begin{equation*}
    \Bun_G(\Jinf)\times \frt^{*,\rs}.
\end{equation*}
It carries an action of $\frt\times \frt^{*,\rs}$, viewed as a constant additive group over $\frt^{*,\rs}$ (where $\frt$ is identified with $\J/\Jinf$ hence acts by changing levels). We have a surjection of group schemes over $\frt^{*,\rs}$
\begin{equation*}
    \e: \frt\times \frt^{*,\rs}\to \Ga\times \frt^{*,\rs}
\end{equation*}
sending $(X,\psi_1)$ to $(\j{\psi_1,X}, \psi_1)$. Let $H$ be the kernel of $\e$. Consider the quotient
\begin{equation*}
    \cZ=(\Bun_G(\Jinf)\times \frt^{*,\rs})/H.
\end{equation*}
Then $\cZ$ carries the residual action of the constant group $\Ga\times \frt^{*,\rs}$, hence an action of $\Ga$. Together with the $\Grot$-action we get an $\Aff$-action $\cZ$. We can form the Kirillov category $\Kir(\cZ)$. It contains a subcategory $\cD_{\bG}^\univ\subset \Kir(\cZ)$ of objects that are constant along $\frt^{*,\rs}$, such that restriction to each fixed $\psi_1\in \frt^{*,\rs}$ gives a functor
\begin{equation*}
i_{\psi_1}^*: \cD_{\bG}^\univ\to \DG
\end{equation*} 
that induces an isomorphism after passing to the Grothendieck groups. Similarly we may consider the universal version $\cD^\univ$ of $\cD_\psi$ by taking the version $\cZ(\bI_0)$ of $\cZ$ with $\bI_0$-level structure at $0$ and by saturating the pullback of $\cD_{\bG}^\univ$ in $\Kir(\cZ(\bI_0))$ by the $\cH_0$-action. Again the restriction functor 
\begin{equation}\label{rest univ to Dpsi}
i_{\psi_1}^*: \cD^\univ\to \cD_\psi
\end{equation} 
induces an isomorphism on the Grothendieck groups. 

The level group $\Jinf$ is normalized by $N_G(T)$, inducing an action of $N_G(T)$ on $\Bun_G(\Jinf)$ and on $\cZ$ that covers the $W$-action on $\frt^{*,\rs}$. It induces an action of $W$ on $K(\cD_{\bG}^\univ)$, which can then be transported to an action of $W$ on $K(\cD_\psi)$ via the isomorphism induced by \eqref{rest univ to Dpsi}.

With the above construction, it can be checked that the map $CC$ is equivariant under the left $\tilW$-actions.

\sss{}
On the other hand, Theorem \ref{th1} induces an isomorphism on the level of Grothendieck groups
\begin{equation*}
\Th: K(\Coh^{\dT}_{\cB^\v}(\wt\cN^\v))\cong K(\cD_\psi).
\end{equation*}
The composition with $CC$ gives a map
\begin{equation*}
    CC\c\Th: K(\Coh^{\dT}_{\cB^\v}(\wt\cN^\v))\to \homog{\tp}{\Fl_\psi,\ZZ}.
\end{equation*}

\sss{Left $\tilW$-action on $K(\Coh^{\dT}_{\cB^\v}(\wt\cN^\v))$} For a translation element $\l\in \xcoch(T)$, it acts on $\Coh^{\dT}_{\cB^\v}(\wt\cN^\v)$ by tensoring with $E(\l)\in \Rep(\dT)$. On the other hand,  $W=N_{\dG}(\dT)/\dT$ acts on the stack $\dT\bs \cB^\v$, inducing an action of $K(\Coh^{\dT}_{\cB^\v}(\wt\cN^\v))\cong K(\Coh^{\dT}(\cB^\v))$. Together they give a left action of $\tilW$ on $K(\Coh^{\dT}_{\cB^\v}(\wt\cN^\v))$.

\sss{Right $\tilW$-action on $K(\Coh^{\dT}_{\cB^\v}(\wt\cN^\v))$} The action of the coherent version of the affine Hecke category $D^b\Coh^{\dG}(\St^\v)$, where $\St^\v=\wt\cN^\v\times_{\frg^\v}\wt\cN^\v$,  on $D^b\Coh^{\dT}_{\cB^\v}(\wt\cN^\v)$ induces a right action of $K(\Coh^{\dG}(\St^\v))$ on  $K(\Coh^{\dT}_{\cB^\v}(\wt\cN^\v))$. It follows from \cite[]{KL-DL} (specializing the $\Gm$-equivariant parameter to $1$) that we have an isomorphism of rings $K(\Coh^{\dG}(\St^\v))\cong \ZZ[\tilW]$, hence a right action of $\tilW$ on $K(\Coh^{\dT}_{\cB^\v}(\wt\cN^\v))$.

The following theorem is used in the work of Feng-Le Hung \cite{FLH} on the Breuil-M\'ezard Conjecture.

\begin{theorem}\label{th:conseq K} 
\begin{enumerate}
    \item The maps $CC$ and $CC\c\Th$ are equivariant under the right $\tilW$-actions.
    \item The map $CC$ (hence $CC\circ \Th$) is an isomorphism.
    \item The isomorphism $CC\c\Th$ is equivariant under the left $\tilW$-actions.
\end{enumerate} 
\end{theorem}
\begin{proof}
(1) For $CC$, it follows from Proposition \ref{p:CC Springer}. For $\Th$, it follows from Theorem \ref{th1} by the compatibility with the affine Hecke category actions.

(2) Surjectivity: We can assume that the center of $G$ is connected. If not consider a central isogeny $G\to \wt G$ where $\wt G$ has connected center. 
The flag variety $\Fl_{\widetilde{G}}$ of $\widetilde{G}$ contains $\Fl_G$ as a union of connected components, and as such, $\cD_{\psi}$ is the restriction of $\cD_{\psi, \wt G}$ (the counterpart of $\cD_{\psi}$ for $\wt G$) to $\Fl_G$. The map $CC$ splits as a sum along connected components. It is thus enough to show the surjectivity of $CC$ for $\widetilde{G}$. 

We use the compatibility of the map $CC$ with the right $\tilW$ action by (1).

Recall from Section \ref{Hess_var}, the Hessenberg varieties $\cH_\psi(\lambda)$. By \cite{GKM} the subvarieties $\Fl_\psi\cap(\bG t^\lambda\bI/\bI)$ are successive affine bundles over $\cH_\psi(\lambda)$ of maximal dimension. It follows that its closure is a union of components.

First note that $CC(\cF_{\l})=(d\pi_\bG)_*(\pi_{\bG,\na})^*CC(\d_{\l})$. The submodule of  $\homog{\tp}{\Fl_\psi,\ZZ}$ generated by $CC(\cF_{0})$ under the Springer action contains the sum of all components over each $\cH_\psi(\lambda)$. If $w_\lambda$ is the minimal length element in $Wt^\lambda\subset \tilW$, the Hessenberg variety $\cH_\psi(\lambda)$ is connected if and only if $w_\lambda$ lies in the fundamental box. These give $\frac{|W|}{|Z(G^{sc})|}$ many components which generate $\homog{\tp}{\Fl_\psi,\ZZ}$ up to the left $\xcoch(T)$-action together with the action of the length $0$ elements in the right $\tilW$-action, as $G$ has connected center.

As the map $CC$ is compatible with the right $\tilW$-actions, and the left $\xcoch(T)$-action sends $CC(\cF_{0})$ to $CC(\cF_{\l})$ the surjectivity follows as these generate under the right $\tilW$-action. Further as $CC$ surjective and compatible with the right $\tilW$-action and the left $\xcoch(T)$-action on generators, it follows that it is compatible with the left $\xcoch(T)$-action on all elements.

To show injectivity of $CC$, it suffices to show that $CC\c\Th$ is injective. By construction, the action of $\cH_\infty$ and the $\Rep(T^\vee)$ action are compatible under the equivalence of categories in Theorem \ref{th1}. It follows that $\Th$ is compatible with the left $\xcoch(T)$-action. By the compatibility of $CC$ with the actions, it then follows that $CC\c\Th$ is compatible with the right $\tilW$-action and the left $\xcoch(T)$-action. Both $K(\Coh^{\dT}_{\cB^\v}(\wt\cN^\v))$ and $ \homog{\tp}{\Fl_\psi}$ are free $\ZZ[\xcoch(T)]\cong K(\Rep(\dT))$-modules of rank equal to $|W|$. Since $CC\c\Th$ is surjective, it also has to be injective.

(3) By (2) $\Th$ and $CC\c\Th$ are compatible with the left $\xcoch(T)$-action.

It remains to show the compatibility of the left $W$-action. To show this, we again reduce to the case of $G$ having connected center. In this case by (1) both sides are generated under the left $\xcoch(T)$- and right $\tilW$-actions by a single element $[\cO_{\cB^{\vee}}]\in K(\Coh^{\dT}_{\cB^\v}(\wt\cN^\v))$ and $CC(\cF_0)=CC\c\Th([\cO_{\cB^{\vee}}])\in \homog{\tp}{\Fl_\psi,\ZZ}$. By compatibility of $CC\c\Th$ with the left $\xcoch(T)$- and right $\tilW$-actions, it is enough to check the compatibility of the $W$ action on $[\cO_{\cB^{\vee}}]$ and $CC(\cF_0)$. Both these elements are invariant under the left $W$-action and the compatibility follows.
\end{proof}

\sss{Other blocks of small quantum group} Let $W_P\subset W\subset\tilW$ a parabolic subgroup and denote by $\cB_P^\v$ the corresponding partial flag variety for $G^\v$ and $\Fl^P$ the partial affine flag variety for $G$. Then using the right $\tilW$ action, we have $K(\Coh^{\dT}_{\cB^\v}(\wt\cN^\v))^{W_P}\cong K(\Coh^{\dT}(\cB^\v))^{W_P}\cong K(\Coh^{\dT}(\cB_P^\v))$ and $\homog{\tp}{\Fl_\psi}^{W_P}\cong \homog{\tp}{\Fl^P_\psi}$, where $\Fl^P_\psi$ denotes the Spaltenstein fiber of $\psi$ in the partial affine flag variety $\Fl^P$.

\appendix

\section{Specialization and microlocalization}\label{as:mon}
In this appendix we review the Kashiwara-Schapira microlocalization construction in the algebro-geometric context over an algebraically closed field of any characteristic.  We also prove formulas for computing the specialization and microlocalization in the situation with a contracting $\Gm$-action.

\subsection{Specialization}

\sss{Monodromic sheaves}
Let $X$ be a stack with $\Gm$-action. Following Verdier \cite[Def.3.1 and \S5]{Ver}, an object  $\cF\in D(X)$ is called {\em monodromic} if for any $x\in X$ with the action map $a_x: \Gm\to X$ (sending $t\in \Gm$ to $t\cdot x$), the pullback $a_x^*\cF$ has tame locally constant cohomology sheaves. Let $D^{\mon}(X)\subset D(X)$ or $D^{\Gm\lmon}(X)\subset D(X)$ denote the full subcategory of monodromic objects.

\sss{Specialization} Suppose $Z\subset X$ is a closed subscheme, Verdier \cite[\S8]{Ver} has defined the specialization functor
\begin{equation*}
\Sp_{Z/X}: D(X)\to D^{\mon}(N_{Z}X)
\end{equation*}
where $N_{Z}X$ is the normal cone of $Z$ in $X$ with the natural $\Gm$-action fiberwise over $Z$. We may extend this definition where $Z$ is locally closed in $X$. In this case, let $j: U\incl X$ be an open neighborhood of $Z$ in which $Z$ is closed. Define the normal cone $N_{Z}X$ to be $N_{Z}U$ (which is clearly independent of the choice of $U$). Then define $\Sp_{Z/X}$ to be the composition
\begin{equation*}
\Sp_{Z/X}: D(X)\xr{j^{*}}D(U)\xr{\Sp_{Z/U}} D^{\mon}(N_{Z}U)=D^{\mon}(N_{Z}X).
\end{equation*}

Below we assume $Z\subset X$ is locally closed.

Recall one key property of $\Sp_{Z/X}$ \cite[\S8, (SP5)]{Ver}: let $i:Z\incl X$ be the inclusion and $i_{0}: Z\incl N_{Z}X$ be the inclusion of the zero section, then the natural transformations of functors $D(X)\to D(Z)$
\begin{equation}\label{Sp res vertex}
i_{0}^{*}\Sp_{Z/X}\to i^{*}, \quad i^{!}\to i_{0}^{!}\Sp_{Z/X}
\end{equation}
are isomorphisms.

Now suppose $Z\subset Y\subset X$, where $Y$ is another closed subscheme of $X$. Suppose there is a $\Gm$-action on a Zariski neighborhood $U$ of $Y$ that contracts $U$ to $Y$ (and in particular fixes $Y$ pointwise).  We denote this $\Gm$ by $\Gm^{v}$ (for vertical). Let $k: Y\incl U$ and $k_{0}: N_{Z}Y\incl N_{Z}U$ be the inclusions.

\begin{lemma}\label{l:sp contr} The natural transformation of functors $D^{\Gm^{v}\lmon}(U)\to D^{\Gm^{v}\times\Gm\lmon}(N_{Z}Y)$ (monodromic both under the $\Gm^{v}$-action and the scaling action in the fiber direction)
\begin{equation*}
k_{0}^{*}\c \Sp_{Z/U}\to \Sp_{Z/Y}\c k^{*}
\end{equation*}
is an isomorphism.
\end{lemma}
\begin{proof}
Let $\pi: U\to Y$ be the contraction map under the $\Gm^{v}$-action, and let $\pi_{0}: N_{Z}U\to N_{Z}Y$ be the induced contraction on the normal cones. Since $k$ is a section of $\pi$, we have a natural transformation
\begin{equation*}
c: \pi_{*}\to \pi_{*}k_{*}k^{*}=k^{*}.
\end{equation*}
By the contraction principle,  the above map is an isomorphism when applied to $\Gm^{v}$-monodromic sheaves on $U$. Similarly, we have a natural  transformation
\begin{equation*}
c_{0}: \pi_{0*}\to \pi_{0*}k_{0*}k^{*}_{0}=k_{0}^{*}
\end{equation*}
which is an isomorphism when applied to $\Gm^{v}$-monodromic sheaves on $N_{Z}U$. 
A diagram chase gives a commutative diagram
\begin{equation*}
\xymatrix{ \Sp_{Z/Y}\c \pi_{*}\ar[r]\ar[d]^{c} & \pi_{0*}\c\Sp_{Z/U}\ar[d]^{c_{0}}\\
\Sp_{Z/Y}\c k^{*} & \ar[l] k_{0}^{*}\c \Sp_{Z/U}
}
\end{equation*} 
When applied to objects from $D^{\Gm^{v}\lmon}(U)$, the vertical maps are isomorphisms, therefore the horizontal maps are inverse isomorphisms. 
\end{proof}

\subsection{Monodromic vanishing cycles}\label{ss:mon van}
Let $Y$ be a stack over $k$ and $X=Y\times \AA^1$, equipped with an $\Gm$-action only on the $\AA^1$-factor with nonzero weight. Let $p: X\to \AA^1$ and $\pi: X\to Y$ be the projections.

For $t\in \AA^{1}(k)$, let  $i_{t}: Y\cong p^{-1}(t)\incl X$ be the inclusion map. The unit of the adjunction $\id\to i_{t*}i_{t}^{*}$ and the counit $i_{t!}i_{t}^{!}\to \id$ induce natural transformations
\begin{equation*}
\a_{t}: \pi_{*}\to \pi_{*}i_{t*}i_{t}^{*}=i_{t}^{*}, \quad \b_{t}: i_{t}^{!}=\pi_{!}i_{t!}i_{t}^{!}\to \pi_{!}.
\end{equation*}


For $\cF\in D^{\Gm\lmon}(X)$, by the contraction principle, $\a_{0}$ and $\b_{0}$ applied to $\cF$ are isomorphisms
\begin{equation}\label{contraction}
\a_{0}: \pi_{*}\cF\isom i_{0}^{*}\cF, \quad \b_{0}: i_{0}^{!}\cF\isom \pi_{!}\cF.
\end{equation}
We {\em define} the functor $$\phi:D^{\Gm\lmon}(X)\to D(Y)$$ 
to fit into the distinguished triangle (using dg enhancements of the relevant triangulated categories)
\begin{equation}\label{star tri phi}
\phi(\cF)\to i^*_0\cF\xr{\a_{1}\c\a_0^{-1}} i^{*}_{1}\cF\to.
\end{equation}
We call $\phi$ the {\em monodromic vanishing cycles} functor.


\sss{Relation with usual vanishing cycles}
We now relate $\phi$ with the usual vanishing cycles functor for the projection  $p: X\to \AA^{1}$. Choose a geometric generic point $\ov\y$ of the henselization of $\AA^{1}$ at $0$, and a specialzation $\ov\y \rightsquigarrow 0$. The nearby and vanishing cycles
functors are defined
$$\Psi_{p}, \Phi_{p}: D(X)\to D(Y).$$
Recall that they fit into distinguished triangles that are related to one another by Verdier duality
\begin{eqnarray*}
\Psi_{p}(\cF)[-1]\xr{\can}\Phi_{p}(\cF)\to i_{0}^{*}\cF\to \\
i_{0}^{!}\cF\to \Phi_{p}(\cF)\xr{\textup{var}} \Psi_{p}(\cF)[1]\to
\end{eqnarray*}

Now choose  an \'etale path between $\ov\y$ and $1$ on $\AA^1$. For $\cF\in D^{\Gm\lmon}(X)$, the chosen path give isomorphisms
\begin{equation}\label{Psi fiber 1}
\Psi_{p}(\cF)\cong i_{1}^{*}\cF.
\end{equation}
Using the isomorphism \eqref{Psi fiber 1}, the rotations of the triangles above become
\begin{eqnarray*}
\Phi_{p}(\cF)\to i_0^*\cF\to i^{*}_{1}\cF\to,\\
i^{!}_{1}\cF\to i_0^!\cF\to \Phi_{p}(\cF)\to.
\end{eqnarray*}
Comparing the first isomorphism above with the definition of $\phi(\cF)$, we get a functorial isomorphism 
\begin{equation*}
\Phi_{p}(\cF)\cong \phi(\cF), \quad \mbox{ for }\cF\in D^{\Gm\lmon}(X)
\end{equation*}
which depends on the choice of the \'etale path between $1$ and $\ov\y$. We also get a functorial distinguished triangle
\begin{equation}\label{shriek tri phi}
i^{!}_{1}\cF\xr{\b_0^{-1}\c\b_{1}}i^!_{0}\cF\to \phi(\cF)\to .
\end{equation}

\subsection{Monodromic Fourier transform}\label{ss:FT}
Let $V$ be a vector bundle of rank $r$ over a stack $S$ and let $V^{\vee}$ be the dual vector bundle. Consider the diagram
\begin{equation}\label{FT diagram}
\xymatrix{ & V\times_{S}V^{\vee}\ar[dl]_{p}\ar[dr]^{\wt p^{\vee}=(p^{\vee}, \ev)}\\
V & & V^{\vee}\times\AA^{1}
}
\end{equation}
where $\ev: V\times_{S}V^{\vee}\to \AA^{1}$ is the pairing between the two vector bundles. We define the Fourier transform functor
\begin{eqnarray*}
\FT: D^{\mon}(V)&\to& D^{\mon}(V^{\vee})\\
\cF&\mt& \phi(\wt p^{\vee}_{!}p^{*}\cF)[r].
\end{eqnarray*}
Here $\phi: D^{\Gm^{2}\lmon}(V^{\vee}\times\AA^{1})\to D^{\mon}(V^{\vee})$ is the monodromic vanishing cycles functor defined in \S\ref{ss:mon van} (the source categoery consists of sheaf monodromic with respect to the two dilation $\Gm$-actions  on individual factors  of $V^{\vee}\times\AA^{1}$).

\sss{Wang's Fourier transform} In \cite{Wang} J.Wang defines a Fourier transform using Fourier kernel $u_{*}E$, where $u:\AA^{1}-\{1\}\incl \AA^{1}$. More precisely, Wang's Fourier transform is the functor
\begin{eqnarray*}
\FT_{\Wang}: D(V)&\to& D(V^{\vee})\\
\cF&\mt&  p^{\vee}_{!}(p^{*}\cF\ot \ev^{*}u_{*}E)[r].
\end{eqnarray*}
Below we explain the relation between $\FT$ and $\FT_{\Wang}$.

\begin{lemma}
There is a canonical isomorphism between $\FT$ and the restriction of $\FT_{\Wang}$  to $D^{\mon}(V)$.
\end{lemma}
\begin{proof} Letting $\d_{1}: \{1\}\incl \AA^{1}$ be the inclusion, and using the distinguished triangle $\d_{1*}\d^{!}_{1}E\to \un E\to u_{*}E\to $, $\FT_{\Wang}[-r]$ fits into a distinguished triangle
\begin{equation}\label{tri FTW}
p^{\vee}_{!}(p^{*}\cF\ot\ev^{*}\d_{1*}\d^{!}_{1}E)\xr{\g} p^{\vee}_{!}p^{*}\cF\to \FT_{\Wang}(\cF)[-r]\to .
\end{equation}
On the other hand, letting  $\pi: V^{\vee}\times\AA^{1}\to V^{\vee}$ be the projection, and $i_{1}: V^{\vee}\times\{1\}\incl V^{\vee}\times\AA^{1}$ be the inclusion,  using \eqref{shriek tri phi}, $\FT(\cF)[-r]=\phi(\wt p^{\vee}_{!}p^{*}\cF)$ fits into a distinguished triangle
\begin{equation}\label{tri FT}
i_{1}^{!}\wt p^{\vee}_{!}p^{*}\cF\xr{\b_{1}} \pi_{!}\wt p^{\vee}_{!}p^{*}\cF\to \FT(\cF)[-r]\to. 
\end{equation}
We have a canonical identification
\begin{equation}\label{FTterm2}
\pi_{!}\wt p^{\vee}_{!}p^{*}\cF=p^{\vee}_{!}p^{*}\cF
\end{equation}
We also have isomorphisms by the projection formula
\begin{equation}\label{FTterm11}
p^{\vee}_{!}(p^{*}\cF\ot \ev^{*}\d_{1*}\d^{!}_{1}E)=\pi_{!}(\wt p^{\vee}_{!}p^{*}\cF\ot \pr_{\AA^{1}}^{*}\d_{1*}\d^{!}_{1}E)\cong i_{1}^{*}\wt p^{\vee}_{!}p^{*}\cF\ot \d^{!}_{1}E
\end{equation}
Now we have a natural map 
\begin{equation}\label{FTterm12}
i_{1}^{*}(\wt p^{\vee}_{!}p^{*}\cF)\ot \d^{!}_{1}E\to i_{1}^{!}(\wt p^{\vee}_{!}p^{*}\cF)
\end{equation}
that is an isomorphism since $\wt p^{\vee}_{!}p^{*}\cF\in D(V^{\vee}\times \AA^{1})$ is monodromic with respect to the scaling action on $\AA^{1}$. The isomorphisms \eqref{FTterm11} an \eqref{FTterm12} give
\begin{equation}\label{FTterm1}
p^{\vee}_{!}(p^{*}\cF\ot\ev^{*}\d_{1*}\d^{!}_{1}E)\isom i_{1}^{!}(\wt p^{\vee}_{!}p^{*}\cF).
\end{equation}
The isomorphisms \eqref{FTterm1} and \eqref{FTterm2} intertwine the maps $\g$ and $\b_{1}$ in \eqref{tri FTW} and \eqref{tri FT}. We conclude with a functorial isomorphism $\FT(\cF)\cong \FT_{\Wang}(\cF)$.
\end{proof}

Because $\FT$ coincides with $\FT_{\Wang}$ on monodromic sheaves, Wang's theorem \cite[Theorem 3.8]{Wang} implies:
\begin{cor}The functor $\FT$ is an equivalence.
\end{cor}



\begin{lemma}\label{l:FT stalk}
Let $\xi: S\to V^{\vee}$ be a section. Then we have a canonical isomorphism for $\cF\in D^{\mon}(V)$
\begin{equation*}
i^{*}_{\xi}\FT(\cF)\cong \phi(\wt\xi_{!}\cF)[r]
\end{equation*}
where $\wt\xi: V\to S\times \AA^{1}$ is the map given by fiberwise pairing with $\xi$.
\end{lemma}
\begin{proof}
Let $\wt i_{\xi}: S\times\AA^{1}\xr{\xi\times\id}V^{\vee}\times\AA^{1}$. We use $i_{t}$ to denote both the inclusion $V^{\vee}=V^{\vee}\times\{t\}\incl V^{\vee}\times\AA^{1}$ and the inclusion $S=S\times\{t\}\incl S\times\AA^{1}$, for $t=0,1$.  Note that $\phi(\wt\xi_{!}\cF)[r]\cong \phi(\wt i_{\xi}^{*}\wt p^{\vee}_{!}p^{*}\cF)[r]$. The natural transformation
\begin{equation}\label{van i}
    u: i_{\xi}^{*}\phi\to \phi\wt i^{*}_{\xi}
\end{equation}
then induces a map $i^{*}_{\xi}\FT\to \phi\wt\xi_{!}[r]$. It suffices to show that $u$ is an isomorphism when applied to any object in $D^{\Gm^{2}\lmon}(V^{\vee}\times \AA^{1})$. Using the defining triangle \eqref{star tri phi} for $\phi$ (with $\pi_{*}$ replaced by $i_{0}^{*}$ using the isomorphism $\a_{0}$ in \eqref{contraction}), the map $u$ fits into a  commutative diagram of distinguished triangles
\begin{equation*}
\xymatrix{i_{\xi}^{*}\phi\ar[d]^{u} \ar[r] & i^{*}_{\xi}i^{*}_{0}\ar[d]^{u_{0}}\ar[r]^-{\a_{1}\a_{0}^{-1}} & i_{\xi}^{*}i_{1}^{*}\ar[d]^{u_{1}}\ar[r] & \\
\phi(\wt i^{*}_{\xi}) \ar[r]& i_{0}^{*} \wt i_{\xi}^{*} \ar[r]^-{\a_{1}\a_{0}^{-1}} & i_{1}^{*} \wt i_{\xi}^{*}\ar[r] &
}
\end{equation*}
The maps $u_{0}$ and $u_{1}$ are the isomorphisms given by the equality of maps $i_{t}i_{\xi}=\wt i_{\xi} i_{t}: S\times\{t\}\xr{\xi\times i_{t}} V^{\vee}\times\AA^{1}$ for $t=0,1$. Therefore $u$ is an isomorphism.
\end{proof}

\begin{lemma}\label{l:FT push}
Let $V$ and $U$ be vector bundles over $S$ of constant ranks $r_V$ and $r_U$ respectively. Let $t: V\to U$ be a linear map (but not necessarily of the same rank at every geometric point). Let $t^\v: U^\v\to V^\v$ be the dual map. There is a natural isomorphism of functors
\begin{equation*}
    (t^\v)^*\c \FT_V\cong \FT_U \c t_![r_V-r_U]: D^{\mon}(V)\to D^{\mon}(U^\v).
\end{equation*}
\end{lemma}
\begin{proof}
Let $\wt t^\v=t^\v\times \id_{\AA^1}: U^\v\times \AA^1\to V^\v\times \AA^1$. We use the notations $p, \wt p^\v$ as in the diagram \eqref{FT diagram}. By definition we have
\begin{equation*}
    (t^\v)^*\c \FT_V=(t^\v)^*\c\phi(\wt p^\v_!p^*)[r_V]
\end{equation*}
A diagram chase using the pairing $V\times_S U^\v\to \AA^1$ together with proper base change shows that
    \begin{equation*}
        \FT_U\c t_!\cong \phi((\wt t^\v)^*\wt p^\v_!p^*)[r_U].
    \end{equation*}
    Here $\phi$ denotes the monodromic vanishing cycles with respect to $U^\v\times \AA^1\to \AA^1$. Comparing the two formulas it suffices to show that the  natural transformation
    \begin{equation*}
        (t^\v)^*\c \phi\to\phi\c (\wt t^\v)^*: D^{\Gm^2\lmon}(V^\v\times \AA^1)\to D^{\mon}(U^\v)
    \end{equation*}
    is an isomorphism. 
    The proof of this last fact is the same as the proof that \eqref{van i} is an isomorphism, using the definition that $\phi$ is the cone of the cospecialization maps between restrictions to the $0$ and $1$ fibers. 
\end{proof}

\sss{Alternative definition} Here is an alternative definition of $\FT$: for $\cF\in D(V)$, consider the vanishing cycles of $p^{*}\cF$ with respect to the function $\ev: V\times_{S}V^{\vee}\to \AA^{1}$.

\begin{lemma}\label{l:Phi ev}
    For any $\cF\in D(V)$, $\Phi_{\ev}(p^{*}\cF)$ is concentrated on $0_S\times_S V^{\vee}$, where $0_S\subset V$ denotes the zero section. Moreover, $\Phi_{\ev}(p^{*}\cF)$ is monodromic with respect to the $\Gm$-scaling action on $V^\vee$.
\end{lemma}
\begin{proof}
The map $\wt p'=(p, \ev): (V-0_S)\times_S V^\vee\to (V-0_S)\times \AA^1$ is smooth. Therefore we have an isomorphism $\wt p'^*_0 \Phi_{\pr_2}(\cK)\cong \Phi_{\ev}(\wt p'^*\cK)$ for any $\cK\in D((V-0_S)\times \AA^1)$ (here $\wt p'_0$ is the restriction of $\wt p'$ to $\ev^{-1}(0)$, and $\pr_2: (V-0_S)\times \AA^1\to \AA^1$ is the projection). For $\cK=\cF|_{V-0_S}\bt E$, we have $\Phi_{\pr_2}(\cK)=0$, hence $\Phi_{\ev}(\wt p'^*\cK)=\Phi_\ev(\cF|_{V-0_S}\bt E)=0$, i.e., $\Phi_{\ev}(p^{*}\cF)$ is zero when restricted to $(V-0_S)\times_S V^\vee$.

Now the map $\ev: V\times_S V^\vee\to \AA^1$ is $\Gm$-equivariant with respect to the scaling actions on $V^\vee$ and $\Gm$, and $p^*\cF$ is $\Gm$-equivariant on $V\times_S V^\vee$. This implies the vanishing cycle $\Phi_\ev(p^*\cF)$ is $\Gm$-monodromic by \cite[Prop. 7.1]{Ver}
\end{proof}

We may then define
\begin{eqnarray*}
    \FT': D(V)&\to& D^{\mon}(V^\vee)\\
    \cF &\mt&\Phi_{\ev}(p^{*}\cF)[r],
\end{eqnarray*}
where we view the right side as an object in $D^{\mon}(V^\vee)$ by Lemma \ref{l:Phi ev}.

\begin{prop}
    For a choice of an \'etale path between $\ov \y$ and $1$ on $\AA^1$, there is a canonical isomorphism of functors $\FT'|_{D^{\mon}(V)}\isom \FT: D(V)\to D^{\mon}(V^\vee)$.
\end{prop}
\begin{proof}
    With the choice of \'etale path on $\AA^1$, we can identify the monodromic vanishing cycles on $V^\vee\times \AA^1$ in the definition of $\FT$ with the usual vanishing cycles $\Phi_{\pr}$, where $\pr: V^\vee\times\AA^1\to \AA^1$ is the projection. Functoriality of vanishing cycles with respect to pushforward along $\wt p^\vee: V\times_S V^\vee\to V^\vee\times \AA^1$ gives a natural transformation
    \begin{equation}\label{q!Phi}
        (p_0^\vee)_!\Phi_\ev\to \Phi_{\pr}\wt p^\vee_!.
    \end{equation}
    Here $p_0^\vee: \ev^{-1}(0)\to V^\vee$ is the projection. We claim that \eqref{q!Phi} is an isomorphism when evaluated on $p^*\cF$, for any $\cF\in D(V)$.

    Let $P=\PP(\cO_S\op V)$, considered as compactifications of $V$ over $S$ in the obvious way. Let $W\subset P\times_S V^\vee \times \AA^1$ be the subscheme defined in homogeneous coordinates $[x_0,x]\in P$ (where $x_+\in V$), $y\in V^\vee$ and $a\in \AA^1$ by
    \begin{equation*}
        \j{x,y}=ax_0.
    \end{equation*}
    Let $q:W\to P$ (which extends $p$), $q^\vee: W\to V^\vee$ (extending $p^\vee$) and $e: W\to \AA^1$ (extending $\ev$) be the projections.
    We have a commutative diagram
    \begin{equation*}
        \xymatrix{V\times_S V^\vee \ar[dr]_{\wt p^\vee}\ar@{^{(}->}[r]^-{j} & W'\ar[d]^{\wt q^\vee=(q^\vee, e)}\\
        &V^\vee\times\AA^1}
    \end{equation*}
    Let $\cK=j_!(p^*\cF)\cong q^*(\io_!\cF)\in D(W')$,  where $\io:V\incl P$ is the inclusion. Let $q^\vee_0: e^{-1}(0)\subset W\to V^\vee$ be the restriction of $\wt q^\vee$ to the zero fiber of $e$. Since $\wt q^\vee$ is proper, we have 
    \begin{equation*}
        (q^\vee_0)!\Phi_{e}(\cK)\cong \Phi_{\pr}(\wt q^\vee_!\cK)\cong \Phi_{\pr}(\wt p^\vee_!p^*\cF).
    \end{equation*}
    It remains to show that the natural map
    \begin{equation}\label{jPhi}
        (p_0^\vee)_!\Phi_\ev(p^*\cF)\to (q^\vee_0)_!\Phi_{e}(j_!p^*\cF)=(q^\vee_0)_!\Phi_{\ev'}(\cK)
    \end{equation}
    is an isomorphism. Let $j_0: \ev^{-1}(0)\incl e^{-1}(0)$ be the restriction of the $j$ over the zero locus of $e$. The map \eqref{jPhi} is induced from the natural map
    \begin{equation}\label{jPhi2}
    (j_0)_!\Phi_\ev(p^*\cF)\to\Phi_{e}(j_!p^*\cF)
    \end{equation}
    by applying $(\wt q^\vee_0)_!$. Thus it suffices to show that \eqref{jPhi2} is an isomorphism. Now the same argument of Lemma \ref{l:Phi ev} applied to $e: W'\to \AA^1$ and the sheaf $j_!p^*\cF\cong q^*\io_!\cF$ shows that $\Phi_{e}(j_!p^*\cF)$ is supported on the non-smooth locus of the map $\wt q=(q,e): W'\to P\times \AA^1$, which is still $0_S\times_S V^\vee\times\{0\}\subset W$. In particular, the support of $\Phi_{e}(j_!p^*\cF)$ is contained in $\ev^{-1}(0)\subset V\times_S V^\vee$. This implies \eqref{jPhi2} is an isomorphism and finishes the proof.
\end{proof}

\subsection{Microlocalization}\label{ss:mic}

In the case $X$ is a smooth scheme and $Z\subset X$ a smooth locally closed subscheme, $N_{Z}X$ is a vector bundle over $Z$ and its dual bundle is the conormal bundle $T^{*}_{Z}X$. The microlocalization functor along $Z$
\begin{equation*}
\mu_{Z/X}: D(X)\to D^{\mon}(T^{*}_{Z}X)
\end{equation*}
is defined as  the composition
\begin{equation*}
\mu_{Z/X}: D(X)\xr{\Sp_{Z/X}} D^{\mon}(N_{Z}X)\xr{\FT} D^{\mon}(T^{*}_{Z}X)
\end{equation*}
where $\FT: D^{\mon}(N_{Z}X)\to D^{\mon}(T^{*}_{Z}X)$ is the monodromic Fourier transform with base $Z$ (see \S\ref{ss:FT}) for sheaves monodromic with respect to the dilation $\Gm$-action on the fibers.

\begin{lemma}\label{l:res slice}
    Let $X$ be a scheme smooth over $k$ and $i: Y\incl X$ be a closed subscheme. Suppose $Y$ carries an action of an algebraic group $H$ and $Z$ is an $H$-orbit through $z\in Z$. Let $s: S\incl X$ be a transversal slice to $Z$ at $z$. In particular, $T_{z}S$ (resp. $T^{*}_{z}S$) is canonically identified with the fiber $(N_{Z}X)_{z}$ (resp. $(T^{*}_{Z}X)_{z}$). 
    Then for any $\cF\in D_H(Y)$ we have a canonical isomorphism
    \begin{equation}\label{mu res slice}
        \mu_{Z/X}(i_!\cF)|_{T^*_zS}\cong \mu_{\{z\}/S}(s^*i_!\cF)\in D(T^*_zS).
    \end{equation}
\end{lemma}
\begin{proof}
    Let $S^Y=\Sig\cap Y$ with inclusion $s_Y: S^Y\incl Y$. Let $M_{Z}S^Y$ be the total space of the deformation to the normal for the pair $Z\incl S^Y$. Similarly we have $M_ZY$. 
    We have a diagram of inclusions of normal cones
    \begin{equation*}
         \xymatrix{N_z(S^Y)\ari[d]^-{\io_S}\ari[r]^-{\s_Y} & N_ZY \ari[d]^-{\io} \\
         T_zS=N_zS\ari[r]^-{\s} & N_ZX=T_ZX}
    \end{equation*}
    The fact that $S$ is transversal to $Z$ implies that the map $S^Y\to [H\bs Y]$ is smooth. Therefore it induces a smooth map over $\AA^1$
    \begin{equation*}
        M_Z(S^Y)\to [H\bs M_ZY].
    \end{equation*}
    Since nearby cycles commute with smooth pullback, we conclude that canonical map
    \begin{equation*}
        \s_Y^*\Sp_{Z/Y}(\cF)\to \Sp_{\{z\}/S^Y}(s^*_Y\cF)\in D_H(N_z(S^Y))
    \end{equation*}
    is an isomorphism. Applying $\io_{S*}$ to both sides, and using that specialization commutes with closed embeddings, we get a canonical isomorphism
    \begin{eqnarray*}
        \s^*\Sp_{Z/X}(i_!\cF)\to \Sp_{\{z\}/S}(s^*i_!\cF)\in D(T_zS).
    \end{eqnarray*}
    Finally we apply Fourier transform on both sides, and using the fact that Fourier transform commutes with arbitrary base change to conclude \eqref{mu res slice}.
\end{proof}


\sss{}
Consider the following situation. Let $X$ be a smooth scheme with a $\Gm$-action contracting to its fixed point locus $Z=X^{\Gm}$. Let $\cF\in D^{\Gm\lmon}(X)$ be monodromic with respect to the action of $\Gm$. Let $f:X\to \AA^{1}$ be a regular function of weight $w\in \ZZ_{>0}$ under the $\Gm$-action, i.e., $f(t\cdot x)=t^{w}f(x)$ or $t\in \Gm$ and $x\in X$. Let $c: X\to Z$ be the contraction map $x\mt \lim_{s\to 0}s\cdot x$ and $\wt f=(c,f): X\to Z\times \AA^{1}$.  On the other hand, $\xi:=df|_{Z}$ gives a section of the conormal bundle $T^{*}_{Z}X\to Z$.

\begin{prop}\label{p:aff space van cycle} Suppose the weight $w$ of $f$ is the smallest among the weights of $\Gm$ on the fibers of the conormal bundle $T^{*}_{Z}X$. Then there is a functorial isomorphism in $\cF\in D^{\Gm\lmon}(X)$

\begin{equation*}
\xi^{*}\mu_{Z/X}(\cF)\isom\phi(\wt f_{!}\cF)[\dim X-\dim Z].
\end{equation*}
Here $\phi: D^{{\Gm\lmon}}(Z\times \AA^{1})\to D(Z)$ is the monodromic vanishing cycles functor defined in \S\ref{ss:mon van}.
\end{prop}
\begin{proof}

Let $p: M_{X}\to \AA^{1}_{t}$ be the total space of the deformation from $X$ to $V:=N_{Z}X$, where $\AA^{1}_{t}$ is a copy of $\AA^{1}$ with coordinate $t$. By construction,  $M_{X}$ is equipped with a map $a: M_{X}\to X$ such that over $\Gm$, we have a canonical isomorphism $(p,a): M^{\c}_{X}:=p^{-1}(\Gm)\cong \Gm\times X$. 

We have a $\Gm$-action on $M_{X}$ such that $p$ has $\Gm$-weight $1$ and $a$ is $\Gm$-invariant. To distinguish various $\Gm$-actions, we denote the above $\Gm$ by $\Gm^{\dil}$, and denote the original $\Gm$ that acts on $X$ by $\Gm^{X}$. The $\Gm^{X}$-action on $X$ induces a $\Gm^{X}$-action on $M_{X}$ such that $p$ is $\Gm^{X}$-invariant.

We denote the target of the function $f$ by $\AA^{1}_{f}$. The function $M^{\c}_{X}\to X\xr{f} \AA^{1}_{f}$ extends to a function $f_{M}: M_{X}\to \AA^{1}_{f}$ whose restriction to the zero fiber is $\xi: V\to \AA^{1}_{f}$. We also have a projection $c\c a: M_{X}\to Z$ whose restriction to $V$ is the projection $q: V=N_{Z}X\to Z$.  Let $\wt f_{M}=(c\c a, f_{M}): M_{X}\to Z\times\AA^{1}_{f}$ whose restriction to $V$ is $\wt\xi=(q,\xi): V=N_{Z}X\to Z\times \AA^{1}_{f}$. Together with $p$ we get a morphism over $\AA^{1}_{t}$
\begin{equation*}
\xymatrix{M_{X}\ar[rr]^-{F=(p, \wt f_{M})}\ar[dr]_{p} && \AA^{1}_{t}\times Z\times \AA^{1}_{f}\ar[dl]^{\pr_{1}}\\
& \AA^{1}_{t}}
\end{equation*}
We have a natural transformation
\begin{equation}\label{nearby push}
\wt\xi_{!}\Psi_{p}\to \Psi_{\pr_{1}} F_{!}.
\end{equation}
Apply both sides of the above  to $a^{*}\cF\in D^{\Gm^{\dil}}(M_{X})$. Using that $F_{!}a^{*}\cF$ is $\Gm$-equivariant, we may canonically identify $\Psi_{\pr_{1}} F_{!}a^{*}\cF$ with the restriction $(F_{!}a^{*}\cF)|_{\{1\}\times Z\times \AA^{1}_{f}}\cong \wt f_{!}\cF$. On the other hand we have $\Psi_{p}a^{*}\cF=\Sp(\cF)\in D(N_{Z}X)$ by definition. Therefore \eqref{nearby push} gives a functorial map
\begin{equation*}
u:  \wt\xi_{!}\Sp(\cF) \to \wt f_{!}\cF\in D^{\Gm^{X}\lmon}(Z\times \AA^{1}_{f}).
\end{equation*}
Applying the monodromic vanishing cycles $\phi$ from \S\ref{ss:mon van} we get a functorial map
\begin{equation*}
u_{\phi}: \phi(\wt \xi_{!}\Sp(\cF))\to \phi(\wt f_{!}\cF)
\end{equation*}
By Lemma \ref{l:FT stalk} we have a canonical isomorphism
\begin{equation*}
\xi^{*}\mu_{Z/X}(\cF)=\xi^{*}\FT(\Sp(\cF))\cong \phi(\wt \xi_{!}\Sp(\cF))[n]
\end{equation*}
where $n=\dim X-\dim Z$. It remains to show that $u_{\phi}$ is an isomorphism for all $\cF\in D^{\Gm^{X}\lmon}(X)$. 

For $s\in \AA^{1}_{f}(k)$, let $i_{s}: Z\times \{s\}\incl Z\times\AA^{1}_{f}$ be the inclusion. Let $\pi: Z\times \AA^{1}_{f}\to Z$ be the projection. By the functorial distinguished triangle \eqref{shriek tri phi}, the map $u_{\phi}$ fits into a map of distinguished triangles in $D(Z)$
\begin{equation*}
\xymatrix{ i_{1}^{!}\wt \xi_{!}\Sp(\cF)\ar[d]^{u_{1}}\ar[r] &  \pi_{!}\wt \xi_{!}\Sp(\cF)\ar[d]^{u_{\pi}}\ar[r] &  \phi(\wt \xi_{!}\Sp(\cF))\ar[d]^{u_{\phi}}\\
i_{1}^{!}\wt f_{!}\cF \ar[r] & \pi_{!}\wt f_{!}\cF\ar[r] & \phi(\wt f_{!}\cF)  
}
\end{equation*}
Here $u_{1}$ and $u_{\pi}$ are obtained from $u$ by applying $i_{1}^{!}$ and $\pi_{!}$. To show $u_{\phi}$ is an isomorphism, it suffices to show $u_{1}$ and $u_{\pi}$ are.  

First we show $u_{\pi}$ is an isomorphism. We have $\pi\c\wt\xi=q: V\to Z$ and $\pi\c\wt f=c: X\to Z$. Therefore $u_{\pi}: q_{!}\Sp(\cF)\to c_{!}\cF$. Let $i_{Z,V}: Z\incl V$ and $i_{Z,X}: Z\incl X$ be the inclusions. Using the contraction principle for the $\Gm^{X}$-actions on $X$ and on $V$, we have $q_{!}\Sp(\cF)\cong i_{Z,V}^{!}\Sp(\cF)$ and $c_{!}\cF\cong i_{Z,X}^{!}\cF$. The map $u_{\pi}$ then gets identified with the inverse of the isomorphism $i_{Z,X}^{!}\cF\cong i_{Z,V}^{!}\Sp(\cF)$ in \eqref{Sp res vertex}.  This shows $u_{\pi}$ is an isomorphism.

Next we show $u_{1}$ is an isomorphism. Let $X_{1}=f^{-1}(1)$ and $V_{1}=\xi^{-1}(1)$. Let $q_{1}: V_{1}\to Z$ be the restriction of the projection $q: V\to Z$, and $c_{1}: X_{1}\to Z$ be the restrictions of the contraction map $c:X\to Z$. Using that $\wt\xi_{!}\Sp(\cF)$ and $\wt f_{!}\cF$ are both $\Gm$-monodromic with respect to the weight $w>0$ action on the $\AA^{1}_{f}$-factor of $Z\times \AA^{1}_{f}$, we may rewrite $i_{1}^{!}$ as $i^{*}_{1}[-2](-1)$, and hence it suffices to show that 
\begin{equation*}
u'_{1}: i_{1}^{*}\wt \xi_{!}\Sp(\cF)\to i_{1}^{*}\wt f_{!}\cF
\end{equation*}
is an isomorphism. Using proper base change, $u'_{1}$ gets identified with the map
\begin{equation*}
u'_{1}: q_{1!}(\Sp(\cF)|_{V_{1}})\to c_{1!}(\cF|_{X_{1}}).
\end{equation*}
The conormal space  $T^{*}_{Z}X=V^{*}$ is decomposed into a direct sum of weight subbundles under $\Gm^{X}$ with positive weights by the contracting assumption. Since our question is Zariski local over $Z$, we may assume each $\Gm^{X}$-weight subbundle of $T^{*}_{Z}X$ is a trivial vector bundle over $Z$.  Fix a trivialization $N_{Z}X\cong \AA^{n}_{Z}$ where the coordinates $\xi_{1},\cdots, \xi_{n}$ are weight vectors with  $\Gm^{X}$-weights $w_{1},w_{2},\cdots, w_{n}>0$. We may assume $\xi=\xi_{1}$ so that $w_{1}=w$. Using the positivity of the weights of the $\Gm^{X}$-action on $\cO(X)$, each $\xi_{i}$ lifts uniquely to a regular function $x_{i}\in \cO(X)$ of weight $w_{i}$. In particular, $x_{1}=f$. Using $x_{i}$ as coordinates we get an isomorphism $X\cong \AA^{n}_{Z}$. We may identify $M_{X}$ with $\AA^{1}_{t}\times \AA^{n}_{Z}$, so that $p$ becomes the first projection. The map $a: M_{X}\to X$ takes the form 
\begin{eqnarray*}
a: \AA^{1}_{t}\times \AA^{n}_{Z}\cong M_{X}&\to& X\cong  \AA^{n}_{Z}\\
(t,z, \xi_{1},\cdots, \xi_{n})&\mapsto& (z, t\xi_{1},\cdots, t\xi_{n}).
\end{eqnarray*}
The map $a$ is invariant under the $\Gm^{\dil}$-action on $M_{X}$ which takes the form
\begin{equation*}
s\cdot (t, z, \xi_{1},\cdots, \xi_{n})=(st, z, s^{-1}\xi_{1},\cdots, s^{-1}\xi_{n}), \quad z\in Z.
\end{equation*}
The $\Gm^{X}$-action on $M_{X}$  takes the form
\begin{equation*}
s\cdot (t, z,\xi_{1},\cdots, \xi_{n})=(t,z, s^{w_{1}}\xi_{1}, s^{w_{2}}\xi_{2},\cdots, s^{w_{n}}\xi_{n}).
\end{equation*}
Combining the two actions, we consider another $\Gm$-action on $M_{X}$ given by
\begin{equation}\label{star action}
s\star (t, z,\xi_{1},\cdots, \xi_{n})=(s^{w}t, z, \xi_{1}, s^{w_{2}-w}\xi_{2},\cdots, s^{w_{n}-w}\xi_{n}).
\end{equation}
The monodromicity of $a^{*}\cF$ under both the $\Gm^{\dil}$-action and $\Gm^{X}$-action implies that $a^{*} \cF$ is also monodromic under the $\Gm$-action \eqref{star action}. 

Let $M_{1}\subset M_{X}$ be the subscheme defined by the equation $\xi_{1}=1$. Let $p_{1}: M_{1}\to \AA^{1}_{t}$ be the restriction of $p$. Then $p_{1}^{-1}(1)=X_{1}$ and $p_{1}^{-1}(0)=V_{1}$. The $\Gm$-action \eqref{star action} restricts to a $\Gm$-action on $M_{1}$. Since $w_{i}\ge w>0$ for all $i$, this action contracts to a subspace of the $0$-fiber $V_{1}$, and the function $p_{1}$ has positive weight $w$. Applying Lemma \ref{l:global section nearby} to the scheme $M_{1}$ with the $\Gm$-action given in \eqref{star action}, the function $p_{1}:M_{1}\to \AA^{1}_{t}$, the $\Gm$-invariant map $\Pi_{1}: M_{1}\to Z$ and the $\Gm$-monodromic sheaf $(a^{*}\cF)|_{M_{1}}$, we get an isomorphism
\begin{equation*}
u''_{1}: q_{1!}\Psi_{p_{1}}((a^{*}\cF)|_{M_{1}})\isom c_{1!}(\cF|_{X_{1}}).
\end{equation*}
By examining the construction of the isomorphisms in Lemma \ref{l:global section nearby}, $u'_{1}$ is the composition
\begin{equation*}
q_{1!}(\Sp(\cF)|_{V_{1}})\xr{v_{1}} q_{1!}\Psi_{p_{s}}((a^{*}\cF)|_{M_{1}})\xr{u''_{1}}c_{1!}(\cF|_{X_{1}}).\end{equation*}
Here the first map is induced from the map
\begin{equation*}
v_{1}: \Sp(\cF)|_{V_{1}}=(\Psi_{p}(a^{*}\cF))|_{V_{1}}\to \Psi_{p_{1}}((a^{*}\cF)|_{M_{1}})
\end{equation*}
which is an isomorphism because $a^{*}\cF$ is $\Gm^{X}$-monodromic, and the $\Gm^{X}$-action on $M_{X}$ scales the map to $\AA^{1}_{f}$ by weight $w>0$ and leaves the map $p$ invariant. Therefore $u'_{1}=u''_{1}\c v_{1}$ is also an isomorphism. This completes the proof.
\end{proof}

\begin{lemma}\label{l:global section nearby}
Let $M$ be a scheme with a contracting $\Gm$-action. Let $p: M\to \AA^{1}$ be a $\Gm$-equivariant map where $\Gm$ acts on $\AA^{1}$ with positive weight. Let $\pi: M\to N$ be a $\Gm$-invariant map. Let $\pi_{t}: M_{t}=p^{-1}(t)\to N$ and $\pi_{\ov\y}: M_{\ov\y}=p^{-1}(\ov\y)\to N_{\ov\y}:=N\times\ov\y$ be the restrictions of $\pi$, where $\ov\y$ is a geometric generic point of the hensalization of $\AA^{1}$ at $0$. Then there are functorial isomorphisms for $\cF\in D^{\Gm\lmon}(p^{-1}(\Gm))$ 
\begin{equation}\label{global section Mgen}
\pi_{\ov\y*}\cF_{\ov\y}\isom \pi_{0*}\Psi_{p}(\cF), \quad \pi_{0!} \Psi_{p}(\cF)\isom \pi_{\ov\y!}\cF_{\ov\y}.
\end{equation}
Here $\cF_{\ov\y}=\cF|_{M_{\ov\y}}$ and $\Psi_{p}$ denotes the nearby cycles towards the zero fiber $X_{0}$  (defined using the same geometric generic point $\ov\y$). The two sides of the isomorphisms take values in $D(N_{\ov\y})$ and $D(N)$ respectively, which are identified. 

In particular, if we choose an \'etale path between $1$ and $\ov\y$, we get  functorial isomorphisms
\begin{equation*}
\pi_{1*} (\cF|_{M_{1}})\isom \pi_{0*}\Psi_{p}(\cF), \quad \pi_{0!} \Psi_{p}(\cF)\isom \pi_{1!}(\cF|_{M_{1}}).
\end{equation*}
\end{lemma}
\begin{proof}
By the fact that $\Psi_{p}$ commutes with Verdier duality, it suffices to prove the first isomorphism in \eqref{global section Mgen}. By the contracting assumption, $Z=M^{\Gm}\subset M_{0}$. Let $D$ be the henselization of $\AA^{1}$ at $0$. 
Let $\ov j: M_{\ov \y} \to M_{D}=M\times_{\AA^{1}}D$. Let $i:M_{0}\incl M_{D}$ be the inclusion. By definition, $\Psi_{p}(\cF)=i^{*}\ov j_{*}\cF_{\ov\y}$.  Let $\pi_{Z}: Z\to N$ be the restriction of $\pi$. Using the contracting principle for the $\Gm$-monodromic sheaf $\Psi_{p}\cF$ on $M_{0}$, we see
\begin{equation}\label{res M0}
\pi_{0*}\Psi_{p}(\cF)\isom \pi_{Z*} i_{Z}^{*}\Psi_{p}(\cF)\cong \pi_{Z*} ((\ov j_{*}\cF_{\ov\y})|_{Z}).
\end{equation}
Let $\pi_{D}: M_{D}\to N$ be induced from $\pi$. Apply the contraction principle to the $\Gm$-monodromic sheaf $\ov j_{*}\cF_{\ov\y}$ on $M_{D}$, we have a canonical isomorphism
\begin{equation}\label{res Mgen}
\pi_{\ov\y*} \cF_{\ov \y}=\pi_{D*}\ov j_{*}\cF_{\ov \y}\isom \pi_{Z*}((\ov j_{*}\cF_{\ov\y})|_{Z})
\end{equation}
under the identification $D(N)\cong D(N_{\ov\y})$.
Combining \eqref{res M0} and \eqref{res Mgen} we get the desired isomorphism \eqref{global section Mgen}.
\end{proof}


\begin{thebibliography}{99}

\bibitem{AB}
S. Arkhipov,  R.Bezrukavnikov, 
Perverse sheaves on affine flags and Langlands dual group.
Israel J. Math. 170 (2009), 135--183.

\bibitem{ABG}
S. Arkhipov,  R.Bezrukavnikov, V. Ginzburg, 
Quantum groups, the loop Grassmannian, and the Springer resolution. 
J. Amer. Math. Soc. 17 (2004), no. 3, 595--678.

\bibitem{AG} 
S. Arkhipov, D. Gaitsgory, 
Another realization of the category of modules over the small quantum group.
Advances in Mathematics, Volume 173, Issue 1, 2003, Pages 114--143.

\bibitem{Bei}
A. Beilinson, 
On the derived category of perverse sheaves.
In: Manin, Y.I. (eds) K-Theory, Arithmetic and Geometry. Lecture Notes in Mathematics, vol 1289. Springer, Berlin, Heidelberg.

\bibitem{BZN} D. Ben-Zvi, D. Nadler, Loops spaces and Langlands parameters,
arXiv:0706.0322 

\bibitem{B-ASF}
R. Bezrukavnikov,
The dimension of the fixed point set on affine flag manifolds. 
Math. Res. Lett. 3 (1996), no. 2, 185–189, DOI 10.4310/MRL.1996.v3.n2.a5.

\bibitem{Be}
R. Bezrukavnikov,
On tensor categories attached to cells in affine Weyl groups, with an
appendix by D. Gaitsgory, 
in {\em Representation theory of algebraic groups and quantum groups}, 69--100, 
Adv. Stud. Pure Math. 40, Math. Soc. Japan, 2004.

\bibitem{B}
R. Bezrukavnikov,
On two geometric realizations of an affine Hecke algebra. 
Publ. Math. Inst. Hautes \'Etudes Sci. 123 (2016), 1--67.

\bibitem{BBAMY}
R.Bezrukavnikov, P.Boixeda Alvarez, M. McBreen, Z.Yun.
Non-abelian Hodge moduli spaces and homogeneous affine Springer fibers.
Pure Appl. Math. Q. 21 (2025), no. 1, 61--130. 

\bibitem{BBSV}
R.Bezrukavnikov, P.Boixeda Alvarez, P.Shan, E.Vasserot,
A geometric realization of the center of the small quantum group.
arXiv:2205.05951.


\bibitem{BF}
R. Bezrukavnikov, M. Finkelberg,
Equivariant Satake category and Kostant-Whittaker reduction. 
Mosc. Math. J. {\bf 8} (2008), no. 1, 39--72, 183.

\bibitem{BFM}
R. Bezrukavnikov, M. Finkelberg, I. Mirkovi\' c,
 Equivariant 
($K$)-homology of affine Grassmannian and Toda 
lattice, 
 Compos. Math.  {\bf 141}  (2005),  no. 3, 746--768. 


\bibitem{BFO}
R. Bezrukavnikov, M. Finkelberg, V. Ostrik, 
Character D-modules via Drinfeld center of Harsih-Chandra bimodules.
Inventiones Math. {\bf 188} (2012) no 3, 589--620.

\bibitem{BeLa} R. Bezrukavnikov, A. Lachowska, 
The small quantum group and the Springer resolution, in:
{\em Quantum groups,}  89--101, Contemp. Math., 433, Amer. Math. Soc., 
Providence, RI, 2007.



\bibitem{BeLo} 
R. Bezrukavnikov, I. Losev, 
Dimensions of modular representations of semisimple Lie algebras.
J. Amer. Math. Soc. 36 (2023), no. 4, 1235--1304.

\bibitem{BeMi} 
R. Bezrukavnikov, I. Mirkovic, 
Representations of semisimple Lie algebras in positive characteristic and noncommutative Springer resolution.
Annals of Mathematics 178 (2013), 835--919.

\bibitem{BRR}
R. Bezrukavnikov, S. Riche, L. Rider,
Modular affine Hecke category and regular unipotent centralizer, I.
arXiv:2005.05583.

\bibitem{BeRi}
R. Bezrukavnikov, S. Riche,
On two modular geometric realizations of an affine Hecke algebra.
arXiv:2402.08281.

\bibitem{BY}
R.Bezrukavnikov, Z.Yun,
On Koszul duality for Kac-Moody groups.
Represent. Theory 17, no. 1 (January 2, 2013): 1--98.


\bibitem{CKN}
J. D. Christensen, B. Keller, A, Neeman,
Failure of Brown representability in derived categories.
Topology, 40(6). p. 1339-1361, 2001.

\bibitem{CKNS}
L. C\^{o}t\'{e}, C. Kuo, D. Nadler, V. Shende,
Perverse Microsheaves.
arXiv:2209.12998.


\bibitem{EM}
S. Evans, I. Mirkovi\'c,
Characteristic cycles for the loop Grassmannian and nilpotent orbits, 
Duke Math. Journal (97)1999, pp.109--126.


\bibitem{FLH}
T.Feng, B.Le Hung,
Mirror symmetry and the Breuil-M\'ezard Conjecture.
arXiv:2310.07006.

\bibitem{FK} M. Finkelberg, D. Kubrak, Vanishing cycles on Poisson varieties, arXiv:1212.3051.

\bibitem{Ga}
D. Gaitsgory, 
Construction of central elements in the affine Hecke algebra via nearby cycles.
Invent. Math. 144 (2001), 253--280.

\bibitem{Ga-Wh}
D. Gaitsgory,
The local and global versions of the Whittaker category.
arxiv:1811.02468.

\bibitem{GKM}
M. Goresky, R. Kottwitz, and R. MacPherson, 
Purity of equivalued affine Springer fibers.
Representation Theory 10.6 (2006), 130--146.


\bibitem{GMW}
Gammage, B., McBreen, M., Webster, B. (2025). Homological mirror symmetry for hypertoric varieties, II. Geometry \& Topology, 29(8), 3921-3993.

\bibitem{ganatra2024microlocal}
S. Ganatra, J. Pardon, and V. Shende, 
Microlocal Morse theory of wrapped Fukaya categories.
Annals of Mathematics vol 199 no 3 (2024), 943--1042.



\bibitem{KDC} 
V. Kac, De Concini, 
Representations of quantum groups at roots of 1. 
In {\em Operator algebras, unitary representations, enveloping algebras, and invariant theory}, 471--506, 
Progr. Math., 92, Birkh\"auser Boston.

\bibitem{KamSch}
M. Kamgarpour, T. Schedler,
Geometrization of principal series representations of reductive groups.
Annales de l'Institut Fourier, Volume 65 (2015) no. 5, pp. 2273--2330.


\bibitem{KS}
M.Kashiwara, P.Schapira,
{\em Sheaves on manifolds}.
Springer-Verlag, Berlin, 1990, Fundamental Principles of Mathematical Sciences, 292, x+512.

\bibitem{KL}
D. Kazhdan, G. Lusztig,
Fixed point varieties on affine flag manifolds. 
Israel J. Math. 62 (1988), no. 2, 129--168.

\bibitem{KL-DL}
D. Kazhdan, G. Lusztig,
Proof of the Deligne-Langlands conjecture for Hecke algebras.
Invent. Math. 87 (1987), no. 1, 153--215.

\bibitem{LQ}
A.Lachowska,  Y. Qi, 
The center of small quantum groups I: The principal block in type A.
Int. Math. Res. Not. IMRN 2018, no. 20, 6349–-6405. 
II: singular blocks. Proc. Lond. Math. Soc. (3) 118 (2019), no. 3, 513--544.


\bibitem{Laumon}
G. Laumon, 
Transformation de Fourier homog\`ene.
Publications Math\'ematiques de l'IH\'ES, Volume 65 (1987), pp. 131--210.

\bibitem{LO} 
I. Losev, V. Ostrik,
Classification of finite-dimensional irreducible modules over $W$-algebras,
Compos. Math. {\bf 150} (2014), no. 6, 1024--1076.

\bibitem{L-Aff}
G. Lusztig, 
Affine Weyl groups and conjugacy classes in Weyl groups. 
Transform. Groups 1(1-2), 83--97 (1996).

\bibitem{MR} 
I. Mirkovi\' c, S. Riche, 
Linear Koszul duality.
Compos. Math. 146 (2010), no. 1, 233–258.
II: coherent sheaves on perfect sheaves.
J. Lond. Math. Soc. (2) 93 (2016), no. 1, 1--24.


\bibitem{Trinh}
M-T. Trinh, 
From the Hecke category to the unipotent locus.
arXiv: 2106.07444.

\bibitem{MV}
I.Mirkovi\'c, K.Vilonen,
Geometric Langlands duality and representations of algebraic groups over commutative rings.  
Ann. of Math. (2)  166  (2007),  no. 1, 95--143.

\bibitem{NS}
D. Nadler, V. Shende,
Sheaf quantization in Weinstein symplectic manifolds. arXiv: 2007.10154.

\bibitem{NS2}
D. Nadler, V. Shende,
Invariance of microsheaves on stable Higgs bundles. arXiv: 2301.01342.

\bibitem{Ngo}
B-C. Ng\^o,
Le lemme fondamental pour les alg\`ebres de Lie.
Publ. Math. Inst. Hautes \'Etudes Sci. No. 111 (2010), 1--169. 

\bibitem{OY}
A.Oblomkov, Z. Yun, 
Geometric representations of graded and rational Cherednik algebras.
Adv. Math. 292 (2016), 601--706.

\bibitem{Ri} 
S.Riche, 
Koszul duality and modular representations of semi-simple Lie algebras.
Duke Math. J. 154 (2010), no. 1, 31--134.

\bibitem{S}
V. Shende, 
Microlocal Category for Weinstein Manifolds via the h-Principle.
Publ. Res. Inst. Math. Sci. 57 (2021), no. 3/4, pp. 1041--1048.

\bibitem{Tanisaki}
T. Tanisaki, Quantized flag manifolds and non-restricted modules over quantum groups at roots of unity.
arXiv:2109.03319 

\bibitem{VV} M. Varagnolo, E. Vasserot, Double affine Hecke algebras and affine flag manifolds, I,
in: Affine flag manifolds and principal bundles, 233--289.
Trends Math.
Birkh\" auser/Springer Basel AG, Basel, 2010

\bibitem{Ver}
J.L.Verdier,
Sp\'ecialisation de faisceaux et monodromie mod\'er\'ee.
Ast\'erisque, 101-102 (1983), p. 332--364.

\bibitem{Wang}
J. Wang,
A new Fourier transform.
Math. Res. Lett. 22 (2015), no. 5, 1541--1562.

\bibitem{Xia}
J. Xia,
Harvard Ph.D. Thesis, 2024.

\bibitem{Yun}
Z.Yun,
The spherical part of local and global Springer actions. 
Math. Ann. 359 (2014), no. 3-4, 557--594.



\end{thebibliography}
\end{document}